\documentclass[10pt]{amsart}
\usepackage[latin1]{inputenc}
\usepackage{xspace,amssymb,amsfonts,euscript, enumitem}
\usepackage{amsthm,amsmath,amscd,stmaryrd,latexsym}

\usepackage{color}

\usepackage{graphicx}
\usepackage[all,dvips]{xy}
\xyoption{line}
\xyoption{arrow}
\xyoption{color}
\SelectTips{cm}{}

\usepackage{tikz}
\usetikzlibrary{decorations.markings}
\usetikzlibrary{decorations.pathreplacing}
\usetikzlibrary{arrows,shapes,positioning}
\tikzstyle directed=[postaction={decorate,decoration={markings,
    mark=at position #1 with {\arrow{>}}}}]
\tikzstyle rdirected=[postaction={decorate,decoration={markings,
    mark=at position #1 with {\arrow{<}}}}]

\IfFileExists{comments.sty}{\usepackage{comments}}

\usepackage{url}
\usepackage[bookmarks=true,%
    colorlinks=true,%
    linkcolor=blue,%
    citecolor=blue,%
    filecolor=blue,%
    menucolor=blue,%
    urlcolor=blue,%
    breaklinks=true]{hyperref}

\newcommand{\refequal}[1]{\xy {\ar@{=}^{#1}
(-1,0)*{};(1,0)*{}};
\endxy}

\newcommand{\refcong}[1]{\xy {\ar@{\cong}^{#1}
(-1,0)*{};(1,0)*{}};
\endxy}

\newcommand\nc{\newcommand}
\nc\calU{\mathcal{U}}
\nc\cU{\calU}
\nc\col{\colon\thinspace}
\nc\calA{\mathcal{A}}
 \nc\diag{\mathrm{d}}
\nc\modU {\mathcal{U}}
\nc\bfU{\mathbf{U}}
\nc\dU{\dot{\mathbf U}}
\nc\dUZ{{_\modZ\dot{\mathbf U}}}
\nc\UZ{{_\modZ \mathbf U} }

\nc\fsl{\mathfrak{sl}}

\newcommand{\U}{\dot{{\bf U}}}
\newcommand{\UA}{{_{\cal{A}}\dot{{\bf U}}}}

\newcommand{\Ucat}{\cal{U}}

\newcommand{\UcatD}{\dot{\cal{U}}}

\newcommand{\onell}[1]{{\mathbf 1}_{#1}}
\newcommand{\onelp}{{\mathbf 1}_{\lambda'}}
\def\l{\lambda}
\newcommand{\onel}{{\mathbf 1}_{\lambda}}
\newcommand{\onem}{{\mathbf 1}_{\mu}}
\newcommand{\onenu}{{\mathbf 1}_{\nu}}
\newcommand\sE{{\cal{E}}}
\newcommand\sF{{\cal{F}}}

\newcommand{\und}{\underline}

\newcommand{\xsum}[2]{
  \vcenter{\xy
  (0,.4)*{\sum};
  (0,3.8)*{\scs #2};
  (0,-3.2)*{\scs #1};
  \endxy}
}

\newcommand{\iccbub}[2]{
\xybox{%
 (-6,0)*{};
  (6,0)*{};
  (-4,0)*{}="t1";
  (4,0)*{}="t2";
  "t2";"t1" **\crv{(4,6) & (-4,6)}; ?(.7)*\dir{}+(-2,0)*{\scs #2}
  ?(.05)*\dir{>} ?(1)*\dir{>};
  "t2";"t1" **\crv{(4,-6) & (-4,-6)};
   ?(.3)*\dir{}+(0,0)*{\bullet}+(0,-3)*{\scs {#1}};
}}
\newcommand{\icbub}[2]{
\xybox{%
 (-6,0)*{};
  (6,0)*{};
  (-4,0)*{}="t1";
  (4,0)*{}="t2";
  "t2";"t1" **\crv{(4,6) & (-4,6)};?(.7)*\dir{}+(-2,0)*{\scs #2};
   ?(0)*\dir{<} ?(.95)*\dir{<};
  "t2";"t1" **\crv{(4,-6) & (-4,-6)};
   ?(.3)*\dir{}+(0,0)*{\bullet}+(0,-3)*{\scs {#1}};
}}

\newcommand{\BOX}{\hbox {$\sqcap$ \kern -1em $\sqcup$}}

\newcommand{\Hom}{{\rm Hom}}

\renewcommand{\to}{\rightarrow}
\newcommand{\maps}{\colon}
\newcommand{\op}{{\rm op}}
\newcommand{\co}{{\rm co}}

\newcommand{\im}{{\rm im\ }}

\newcommand{\la}{\langle}
\newcommand{\ra}{\rangle}

\newcommand{\scs}{\scriptstyle}

\newcommand{\Com}{\textsf{Com}}
\newcommand{\Kom}{\textsf{Kom}}

\theoremstyle{definition}
\newtheorem{thm}{Theorem}[section]
\newtheorem{cor}[thm]{Corollary}
\newtheorem{conj}[thm]{Conjecture}

\newtheorem{rem}[thm]{Remark}
\newtheorem{prop}[thm]{Proposition}
\newtheorem{definition}[thm]{Definition}

        \newcommand{\be}{\begin{equation}}
        \newcommand{\ee}{\end{equation}}
        \newcommand{\ba}{\begin{eqnarray}}
        \newcommand{\ea}{\end{eqnarray}}
        \newcommand{\ban}{\begin{eqnarray*}}
        \newcommand{\ean}{\end{eqnarray*}}
        \newcommand{\barr}{\begin{array}}
        \newcommand{\earr}{\end{array}}

\numberwithin{equation}{section}

\def\emph#1{{\sl #1\/}}
\def\ie{{\sl i.e. \/}}
\def\eg{{\sl e.g. \/}}

\let\tilde=\widetilde

\let\phi=\varphi
\let\epsilon=\varepsilon

\usepackage{bbm}
\def\C{{\mathbbm C}}

\def\Z{{\mathbbm Z}}
\def\Q{{\mathbbm Q}}

\def\cal#1{\mathcal{#1}}%
\def\1{\mathbbm{1}}%
\def\nn{\notag}

\def\Id{\mathrm{Id}}

\def\mf{\mathfrak}

\newcommand{\hackcenter}[1]{
 \xy (0,0)*{#1}; \endxy}

\newcommand{\sdotu}[1]{\xybox{%
  (-2,0)*{};
  (2,0)*{};
  (0,0)*{}; (0,-8)*{} **[blue]\dir{-}?(.5)*{\scs\bullet} ?(0)*[blue]\dir{<};
  (0,-10)*{\scs  #1};
}}
\newcommand{\sdotur}[1]{\xybox{%
  (-2,0)*{};
  (2,0)*{};
  (0,0)*{}; (0,-8)*{} **[black]\dir{-}?(.5)*{\scs\bullet} ?(0)*[black]\dir{<};
  (0,-10)*{\scs  #1};
}}
\newcommand{\sdotug}[1]{\xybox{%
  (-2,0)*{};
  (2,0)*{};
  (0,0)*{}; (0,-8)*{} **[green]\dir{-}?(.5)*{\scs\bullet} ?(0)*[green]\dir{<};
  (0,-10)*{\scs  #1};
}}
\newcommand{\sdotup}[1]{\xybox{%
  (-2,0)*{};
  (2,0)*{};
  (0,0)*{}; (0,-8)*{} **[magenta]\dir{-}?(.5)*{\scs\bullet} ?(0)*[magenta]\dir{<};
  (0,-10)*{\scs  #1};
}}
\newcommand{\sdotd}[1]{\xybox{%
  (-2,0)*{};
  (2,0)*{};
  (0,0)*{}; (0,-8)*{} **[blue]\dir{-}?(.5)*{\scs\bullet} ?(1)*[blue]\dir{>};
  (0,-10)*{\scs  #1};
}}
\newcommand{\sdotdr}[1]{\xybox{%
  (-2,0)*{};
  (2,0)*{};
  (0,0)*{}; (0,-8)*{} **[black]\dir{-}?(.5)*{\scs\bullet} ?(1)*[black]\dir{>};
  (0,-10)*{\scs  #1};
}}
\newcommand{\sdotdg}[1]{\xybox{%
  (-2,0)*{};
  (2,0)*{};
  (0,0)*{}; (0,-8)*{} **[green]\dir{-}?(.5)*{\scs\bullet} ?(1)*[green]\dir{>};
  (0,-10)*{\scs  #1};
}}
\newcommand{\sdot}[1]{\xybox{%
  (-2,0)*{};
  (2,0)*{};
  (0,0)*{}; (0,-8)*{} **[blue]\dir{-}?(.5)*{\scs\bullet} ?(1)*[blue]\dir{};
  (0,-10)*{\scs  #1};
}}
\newcommand{\sdotr}[1]{\xybox{%
  (-2,0)*{};
  (2,0)*{};
  (0,0)*{}; (0,-8)*{} **[black]\dir{-}?(.5)*{\scs\bullet} ?(1)*[black]\dir{};
  (0,-10)*{\scs  #1};
}}

\newcommand{\sdotp}[1]{\xybox{%
  (-2,0)*{};
  (2,0)*{};
  (0,0)*{}; (0,-8)*{} **[magenta]\dir{-}?(.5)*{\scs\bullet} ?(1)*[magenta]\dir{};
  (0,-10)*{\scs  #1};
}}
\newcommand{\slineu}[1]{\xybox{%
  (-2,0)*{};
  (2,0)*{};
  (0,0)*{}; (0,-8)*{} **[blue]\dir{-}; ?(0)*[blue]\dir{<};
  (0,-10)*{\scs  #1};
}}
\newcommand{\slineur}[1]{\xybox{%
  (-2,0)*{};
  (2,0)*{};
  (0,0)*{}; (0,-8)*{} **[black]\dir{-}; ?(0)*[black]\dir{<};
  (0,-10)*{\scs  #1};
}}
\newcommand{\slineup}[1]{\xybox{%
  (-2,0)*{};
  (2,0)*{};
  (0,0)*{}; (0,-8)*{} **[magenta]\dir{-}; ?(0)*[magenta]\dir{<};
  (0,-10)*{\scs  #1};
}}
\newcommand{\slineug}[1]{\xybox{%
  (-2,0)*{};
  (2,0)*{};
  (0,0)*{}; (0,-8)*{} **[green]\dir{-}; ?(0)*[green]\dir{<};
  (0,-10)*{\scs  #1};
}}
\newcommand{\slinen}[1]{\xybox{%
  (-2,0)*{};
  (2,0)*{};
  (0,0)*{}; (0,-8)*{} **[blue]\dir{-};
  (0,-10)*{\scs  #1};
}}
\newcommand{\slinenr}[1]{\xybox{%
  (-2,0)*{};
  (2,0)*{};
  (0,0)*{}; (0,-8)*{} **[black]\dir{-};
  (0,-10)*{\scs  #1};
}}
\newcommand{\slinenp}[1]{\xybox{%
  (-2,0)*{};
  (2,0)*{};
  (0,0)*{}; (0,-8)*{} **[magenta]\dir{-};
  (0,-10)*{\scs  #1};
}}
\newcommand{\slineng}[1]{\xybox{%
  (-2,0)*{};
  (2,0)*{};
  (0,0)*{}; (0,-8)*{} **[green]\dir{-};
  (0,-10)*{\scs  #1};
}}
\newcommand{\slined}[1]{\xybox{%
  (-2,0)*{};
  (2,0)*{};
  (0,0)*{}; (0,-8)*{} **[blue]\dir{-}; ?(1)*[blue]\dir{>};
  (0,-10)*{\scs  #1};
}}
\newcommand{\slinedp}[1]{\xybox{%
  (-2,0)*{};
  (2,0)*{};
  (0,0)*{}; (0,-8)*{} **[magenta]\dir{-}; ?(1)*[magenta]\dir{>};
  (0,-10)*{\scs  #1};
}}
\newcommand{\slinedr}[1]{\xybox{%
  (-2,0)*{};
  (2,0)*{};
  (0,0)*{}; (0,-8)*{} **[black]\dir{-}; ?(1)*[black]\dir{>};
  (0,-10)*{\scs  #1};
}}
\newcommand{\slinedg}[1]{\xybox{%
  (-2,0)*{};
  (2,0)*{};
  (0,0)*{}; (0,-8)*{} **[green]\dir{-}; ?(1)*[green]\dir{>};
  (0,-10)*{\scs  #1};
}}

\newcommand{\ucrosspb}[2]{
 \xybox{
    (-3,-4)*{};(3,4)*{} **[magenta]\crv{(-3,-1) & (3,1)}?(1)*[magenta]\dir{>} ;
    (3,-4)*{};(-3,4)*{} **[blue]\crv{(3,-1) & (-3,1)}?(1)*[blue]\dir{>};
    (-3,-6)*{\scs #1};
     (3.1,-6)*{\scs #2};
     (-8,0)*{};(8,0)*{};
     }}
\newcommand{\ucrossbp}[2]{
 \xybox{
    (-3,-4)*{};(3,4)*{} **[blue]\crv{(-3,-1) & (3,1)}?(1)*[blue]\dir{>} ;
    (3,-4)*{};(-3,4)*{} **[magenta]\crv{(3,-1) & (-3,1)}?(1)*[magenta]\dir{>};
    (-3,-6)*{\scs #1};
     (3.1,-6)*{\scs #2};
     (-8,0)*{};(8,0)*{};
     }}
\newcommand{\ucrossrb}[2]{
 \xybox{
    (-3,-4)*{};(3,4)*{} **[black]\crv{(-3,-1) & (3,1)}?(1)*[black]\dir{>} ;
    (3,-4)*{};(-3,4)*{} **[blue]\crv{(3,-1) & (-3,1)}?(1)*[blue]\dir{>};
    (-3,-6)*{\scs #1};
     (3.1,-6)*{\scs #2};
     (-8,0)*{};(8,0)*{};
     }}
 \newcommand{\ucrossrp}[2]{
 \xybox{
    (-3,-4)*{};(3,4)*{} **[black]\crv{(-3,-1) & (3,1)}?(1)*[black]\dir{>} ;
    (3,-4)*{};(-3,4)*{} **[magenta]\crv{(3,-1) & (-3,1)}?(1)*[magenta]\dir{>};
    (-3,-6)*{\scs #1};
     (3.1,-6)*{\scs #2};
     (-8,0)*{};(8,0)*{};
     }}
  \newcommand{\ucrosspr}[2]{
 \xybox{
    (-3,-4)*{};(3,4)*{} **[magenta]\crv{(-3,-1) & (3,1)}?(1)*[magenta]\dir{>} ;
    (3,-4)*{};(-3,4)*{} **[black]\crv{(3,-1) & (-3,1)}?(1)*[black]\dir{>};
    (-3,-6)*{\scs #1};
     (3.1,-6)*{\scs #2};
     (-8,0)*{};(8,0)*{};
     }}
\newcommand{\ucrossbr}[2]{
 \xybox{
    (-3,-4)*{};(3,4)*{} **[blue]\crv{(-3,-1) & (3,1)}?(1)*[blue]\dir{>} ;
    (3,-4)*{};(-3,4)*{} **[black]\crv{(3,-1) & (-3,1)}?(1)*[black]\dir{>};
    (-3,-6)*{\scs #1};
     (3.1,-6)*{\scs #2};
     (-8,0)*{};(8,0)*{};
     }}
\newcommand{\ucrossrr}[2]{
 \xybox{
    (-3,-4)*{};(3,4)*{} **[black]\crv{(-3,-1) & (3,1)}?(1)*[black]\dir{>} ;
    (3,-4)*{};(-3,4)*{} **[black]\crv{(3,-1) & (-3,1)}?(1)*[black]\dir{>};
    (-3,-6)*{\scs #1};
     (3.1,-6)*{\scs #2};
     (-8,0)*{};(8,0)*{};
     }}
\newcommand{\ucrossgr}[2]{
 \xybox{
    (-3,-4)*{};(3,4)*{} **[green]\crv{(-3,-1) & (3,1)}?(1)*[green]\dir{>} ;
    (3,-4)*{};(-3,4)*{} **[black]\crv{(3,-1) & (-3,1)}?(1)*[black]\dir{>};
    (-3,-6)*{\scs #1};
     (3.1,-6)*{\scs #2};
     (-8,0)*{};(8,0)*{};
     }}
\newcommand{\ucrossrg}[2]{
 \xybox{
    (-3,-4)*{};(3,4)*{} **[black]\crv{(-3,-1) & (3,1)}?(1)*[black]\dir{>} ;
    (3,-4)*{};(-3,4)*{} **[green]\crv{(3,-1) & (-3,1)}?(1)*[green]\dir{>};
    (-3,-6)*{\scs #1};
     (3.1,-6)*{\scs #2};
     (-8,0)*{};(8,0)*{};
     }}
\newcommand{\ucrossgb}[2]{
 \xybox{
    (-3,-4)*{};(3,4)*{} **[green]\crv{(-3,-1) & (3,1)}?(1)*[green]\dir{>} ;
    (3,-4)*{};(-3,4)*{} **[blue]\crv{(3,-1) & (-3,1)}?(1)*[blue]\dir{>};
    (-3,-6)*{\scs #1};
     (3.1,-6)*{\scs #2};
     (-8,0)*{};(8,0)*{};
     }}
\newcommand{\ucrossbg}[2]{
 \xybox{
    (-3,-4)*{};(3,4)*{} **[blue]\crv{(-3,-1) & (3,1)}?(1)*[blue]\dir{>} ;
    (3,-4)*{};(-3,4)*{} **[green]\crv{(3,-1) & (-3,1)}?(1)*[green]\dir{>};
    (-3,-6)*{\scs #1};
     (3.1,-6)*{\scs #2};
     (-8,0)*{};(8,0)*{};
     }}
\newcommand{\ucrossgg}[2]{
 \xybox{
    (-3,-4)*{};(3,4)*{} **[green]\crv{(-3,-1) & (3,1)}?(1)*[green]\dir{>} ;
    (3,-4)*{};(-3,4)*{} **[green]\crv{(3,-1) & (-3,1)}?(1)*[green]\dir{>};
    (-3,-6)*{\scs #1};
     (3.1,-6)*{\scs #2};
     (-8,0)*{};(8,0)*{};
     }}
\newcommand{\sucrossrr}[2]{
 \xybox{
    (-2,-4)*{};(2,4)*{} **[black]\crv{(-2,-1) & (2,1)}?(1)*[black]\dir{>} ;
    (2,-4)*{};(-2,4)*{} **[black]\crv{(2,-1) & (-2,1)}?(1)*[black]\dir{>};
    (-2,-6)*{\scs #1};
     (2.1,-6)*{\scs #2};
     (-4,0)*{};(4,0)*{};
     }}

\newcommand{\sucrossbp}[2]{
 \xybox{
    (-2,-4)*{};(2,4)*{} **[blue]\crv{(-2,-1) & (2,1)}?(1)*[blue]\dir{>} ;
    (2,-4)*{};(-2,4)*{} **[magenta]\crv{(2,-1) & (-2,1)}?(1)*[magenta]\dir{>};
    (-2,-6)*{\scs #1};
     (2.1,-6)*{\scs #2};
     (-4,0)*{};(4,0)*{};
     }}

\newcommand{\lcross}[2]{
 \xybox{
    (-3,-4)*{};(3,4)*{} **[blue]\crv{(-3,-1) & (3,1)}?(0)*[blue]\dir{<} ;
    (3,-4)*{};(-3,4)*{} **[blue]\crv{(3,-1) & (-3,1)}?(1)*[blue]\dir{>};
    (-3,-6)*{\scs #1};
     (3.1,-6)*{\scs #2};
     (-8,0)*{};(8,0)*{};
     }}
\newcommand{\lcrossrr}[2]{
 \xybox{
    (-3,-4)*{};(3,4)*{} **[black]\crv{(-3,-1) & (3,1)}?(0)*[black]\dir{<} ;
    (3,-4)*{};(-3,4)*{} **[black]\crv{(3,-1) & (-3,1)}?(1)*[black]\dir{>};
    (-3,-6)*{\scs #1};
     (3.1,-6)*{\scs #2};
     (-8,0)*{};(8,0)*{};
     }}
\newcommand{\lcrossrb}[2]{
 \xybox{
    (-3,-4)*{};(3,4)*{} **[black]\crv{(-3,-1) & (3,1)}?(0)*[black]\dir{<} ;
    (3,-4)*{};(-3,4)*{} **[blue]\crv{(3,-1) & (-3,1)}?(1)*[blue]\dir{>};
    (-3,-6)*{\scs #1};
     (3.1,-6)*{\scs #2};
     (-8,0)*{};(8,0)*{};
     }}
\newcommand{\lcrossbr}[2]{
 \xybox{
    (-3,-4)*{};(3,4)*{} **[blue]\crv{(-3,-1) & (3,1)}?(0)*[blue]\dir{<} ;
    (3,-4)*{};(-3,4)*{} **[black]\crv{(3,-1) & (-3,1)}?(1)*[black]\dir{>};
    (-3,-6)*{\scs #1};
     (3.1,-6)*{\scs #2};
     (-8,0)*{};(8,0)*{};
     }}
\newcommand{\lcrossrp}[2]{
 \xybox{
    (-3,-4)*{};(3,4)*{} **[black]\crv{(-3,-1) & (3,1)}?(0)*[black]\dir{<} ;
    (3,-4)*{};(-3,4)*{} **[magenta]\crv{(3,-1) & (-3,1)}?(1)*[magenta]\dir{>};
    (-3,-6)*{\scs #1};
     (3.1,-6)*{\scs #2};
     (-8,0)*{};(8,0)*{};
     }}

\newcommand{\lcrossbg}[2]{
 \xybox{
    (-3,-4)*{};(3,4)*{} **[blue]\crv{(-3,-1) & (3,1)}?(0)*[blue]\dir{<} ;
    (3,-4)*{};(-3,4)*{} **[green]\crv{(3,-1) & (-3,1)}?(1)*[green]\dir{>};
    (-3,-6)*{\scs #1};
     (3.1,-6)*{\scs #2};
     (-8,0)*{};(8,0)*{};
     }}
\newcommand{\lcrossgr}[2]{
 \xybox{
    (-3,-4)*{};(3,4)*{} **[green]\crv{(-3,-1) & (3,1)}?(0)*[green]\dir{<} ;
    (3,-4)*{};(-3,4)*{} **[black]\crv{(3,-1) & (-3,1)}?(1)*[black]\dir{>};
    (-3,-6)*{\scs #1};
     (3.1,-6)*{\scs #2};
     (-8,0)*{};(8,0)*{};
     }}
\newcommand{\lcrossgb}[2]{
 \xybox{
    (-3,-4)*{};(3,4)*{} **[green]\crv{(-3,-1) & (3,1)}?(0)*[green]\dir{<} ;
    (3,-4)*{};(-3,4)*{} **[blue]\crv{(3,-1) & (-3,1)}?(1)*[blue]\dir{>};
    (-3,-6)*{\scs #1};
     (3.1,-6)*{\scs #2};
     (-8,0)*{};(8,0)*{};
     }}
\newcommand{\lcrossbp}[2]{
 \xybox{
    (-3,-4)*{};(3,4)*{} **[blue]\crv{(-3,-1) & (3,1)}?(0)*[blue]\dir{<} ;
    (3,-4)*{};(-3,4)*{} **[magenta]\crv{(3,-1) & (-3,1)}?(1)*[magenta]\dir{>};
    (-3,-6)*{\scs #1};
     (3.1,-6)*{\scs #2};
     (-8,0)*{};(8,0)*{};
     }}
\newcommand{\lcrosspb}[2]{
 \xybox{
    (-3,-4)*{};(3,4)*{} **[magenta]\crv{(-3,-1) & (3,1)}?(0)*[magenta]\dir{<} ;
    (3,-4)*{};(-3,4)*{} **[blue]\crv{(3,-1) & (-3,1)}?(1)*[blue]\dir{>};
    (-3,-6)*{\scs #1};
     (3.1,-6)*{\scs #2};
     (-8,0)*{};(8,0)*{};
     }}
\newcommand{\lcrossrg}[2]{
 \xybox{
    (-3,-4)*{};(3,4)*{} **[black]\crv{(-3,-1) & (3,1)}?(0)*[black]\dir{<} ;
    (3,-4)*{};(-3,4)*{} **[green]\crv{(3,-1) & (-3,1)}?(1)*[green]\dir{>};
    (-3,-6)*{\scs #1};
     (3.1,-6)*{\scs #2};
     (-8,0)*{};(8,0)*{};
     }}
\newcommand{\lcrossgg}[2]{
 \xybox{
    (-3,-4)*{};(3,4)*{} **[green]\crv{(-3,-1) & (3,1)}?(0)*[green]\dir{<} ;
    (3,-4)*{};(-3,4)*{} **[green]\crv{(3,-1) & (-3,1)}?(1)*[green]\dir{>};
    (-3,-6)*{\scs #1};
     (3.1,-6)*{\scs #2};
     (-8,0)*{};(8,0)*{};
     }}
\newcommand{\rcross}[2]{
 \xybox{
    (-3,-4)*{};(3,4)*{} **[blue]\crv{(-3,-1) & (3,1)}?(1)*[blue]\dir{>} ;
    (3,-4)*{};(-3,4)*{} **[blue]\crv{(3,-1) & (-3,1)}?(0)*[blue]\dir{<};
    (-3,-6)*{\scs #1};
     (3.1,-6)*{\scs #2};
     (-8,0)*{};(8,0)*{};
     }}
\newcommand{\rcrossrr}[2]{
 \xybox{
    (-3,-4)*{};(3,4)*{} **[black]\crv{(-3,-1) & (3,1)}?(1)*[black]\dir{>} ;
    (3,-4)*{};(-3,4)*{} **[black]\crv{(3,-1) & (-3,1)}?(0)*[black]\dir{<};
    (-3,-6)*{\scs #1};
     (3.1,-6)*{\scs #2};
     (-8,0)*{};(8,0)*{};
     }}
\newcommand{\rcrossbr}[2]{
 \xybox{
    (-3,-4)*{};(3,4)*{} **[blue]\crv{(-3,-1) & (3,1)}?(1)*[blue]\dir{>} ;
    (3,-4)*{};(-3,4)*{} **[black]\crv{(3,-1) & (-3,1)}?(0)*[black]\dir{<};
    (-3,-6)*{\scs #1};
     (3.1,-6)*{\scs #2};
     (-8,0)*{};(8,0)*{};
     }}
\newcommand{\rcrossrb}[2]{
 \xybox{
    (-3,-4)*{};(3,4)*{} **[black]\crv{(-3,-1) & (3,1)}?(1)*[black]\dir{>} ;
    (3,-4)*{};(-3,4)*{} **[blue]\crv{(3,-1) & (-3,1)}?(0)*[blue]\dir{<};
    (-3,-6)*{\scs #1};
     (3.1,-6)*{\scs #2};
     (-8,0)*{};(8,0)*{};
     }}
\newcommand{\rcrossgb}[2]{
 \xybox{
    (-3,-4)*{};(3,4)*{} **[green]\crv{(-3,-1) & (3,1)}?(1)*[green]\dir{>} ;
    (3,-4)*{};(-3,4)*{} **[blue]\crv{(3,-1) & (-3,1)}?(0)*[blue]\dir{<};
    (-3,-6)*{\scs #1};
     (3.1,-6)*{\scs #2};
     (-8,0)*{};(8,0)*{};
     }}
\newcommand{\rcrossbg}[2]{
 \xybox{
    (-3,-4)*{};(3,4)*{} **[blue]\crv{(-3,-1) & (3,1)}?(1)*[blue]\dir{>} ;
    (3,-4)*{};(-3,4)*{} **[green]\crv{(3,-1) & (-3,1)}?(0)*[green]\dir{<};
    (-3,-6)*{\scs #1};
     (3.1,-6)*{\scs #2};
     (-8,0)*{};(8,0)*{};
     }}
\newcommand{\rcrosspb}[2]{
 \xybox{
    (-3,-4)*{};(3,4)*{} **[magenta]\crv{(-3,-1) & (3,1)}?(1)*[magenta]\dir{>} ;
    (3,-4)*{};(-3,4)*{} **[blue]\crv{(3,-1) & (-3,1)}?(0)*[blue]\dir{<};
    (-3,-6)*{\scs #1};
     (3.1,-6)*{\scs #2};
     (-8,0)*{};(8,0)*{};
     }}
\newcommand{\rcrossbp}[2]{
 \xybox{
    (-3,-4)*{};(3,4)*{} **[blue]\crv{(-3,-1) & (3,1)}?(1)*[blue]\dir{>} ;
    (3,-4)*{};(-3,4)*{} **[magenta]\crv{(3,-1) & (-3,1)}?(0)*[magenta]\dir{<};
    (-3,-6)*{\scs #1};
     (3.1,-6)*{\scs #2};
     (-8,0)*{};(8,0)*{};
     }}
\newcommand{\rcrossrg}[2]{
 \xybox{
    (-3,-4)*{};(3,4)*{} **[black]\crv{(-3,-1) & (3,1)}?(1)*[black]\dir{>} ;
    (3,-4)*{};(-3,4)*{} **[green]\crv{(3,-1) & (-3,1)}?(0)*[green]\dir{<};
    (-3,-6)*{\scs #1};
     (3.1,-6)*{\scs #2};
     (-8,0)*{};(8,0)*{};
     }}
\newcommand{\rcrossgr}[2]{
 \xybox{
    (-3,-4)*{};(3,4)*{} **[green]\crv{(-3,-1) & (3,1)}?(1)*[green]\dir{>} ;
    (3,-4)*{};(-3,4)*{} **[black]\crv{(3,-1) & (-3,1)}?(0)*[black]\dir{<};
    (-3,-6)*{\scs #1};
     (3.1,-6)*{\scs #2};
     (-8,0)*{};(8,0)*{};
     }}

\newcommand{\rcrosspr}[2]{
 \xybox{
    (-3,-4)*{};(3,4)*{} **[magenta]\crv{(-3,-1) & (3,1)}?(1)*[magenta]\dir{>} ;
    (3,-4)*{};(-3,4)*{} **[black]\crv{(3,-1) & (-3,1)}?(0)*[black]\dir{<};
    (-3,-6)*{\scs #1};
     (3.1,-6)*{\scs #2};
     (-8,0)*{};(8,0)*{};
     }}
\newcommand{\rcrossgg}[2]{
 \xybox{
    (-3,-4)*{};(3,4)*{} **[green]\crv{(-3,-1) & (3,1)}?(1)*[green]\dir{>} ;
    (3,-4)*{};(-3,4)*{} **[green]\crv{(3,-1) & (-3,1)}?(0)*[green]\dir{<};
    (-3,-6)*{\scs #1};
     (3.1,-6)*{\scs #2};
     (-8,0)*{};(8,0)*{};
     }}

\newcommand{\ncrossrr}[2]{
 \xybox{
    (-3,-4)*{};(3,4)*{} **[black]\crv{(-3,-1) & (3,1)}?(1)*[black]\dir{} ;
    (3,-4)*{};(-3,4)*{} **[black]\crv{(3,-1) & (-3,1)}?(1)*[black]\dir{};
    (-3,-6)*{\scs #1};
     (3.1,-6)*{\scs #2};
     (-8,0)*{};(8,0)*{};
     }}
 \newcommand{\ncrosspb}[2]{
 \xybox{
    (-3,-4)*{};(3,4)*{} **[magenta]\crv{(-3,-1) & (3,1)}?(1)*[magenta]\dir{} ;
    (3,-4)*{};(-3,4)*{} **[blue]\crv{(3,-1) & (-3,1)}?(1)*[blue]\dir{};
    (-3,-6)*{\scs #1};
     (3.1,-6)*{\scs #2};
     (-8,0)*{};(8,0)*{};
     }}
 \newcommand{\ncrossbp}[2]{
 \xybox{
    (-3,-4)*{};(3,4)*{} **[blue]\crv{(-3,-1) & (3,1)}?(1)*[blue]\dir{} ;
    (3,-4)*{};(-3,4)*{} **[magenta]\crv{(3,-1) & (-3,1)}?(1)*[magenta]\dir{};
    (-3,-6)*{\scs #1};
     (3.1,-6)*{\scs #2};
     (-8,0)*{};(8,0)*{};
     }}
  \newcommand{\ncrosspr}[2]{
 \xybox{
    (-3,-4)*{};(3,4)*{} **[magenta]\crv{(-3,-1) & (3,1)}?(1)*[magenta]\dir{} ;
    (3,-4)*{};(-3,4)*{} **[black]\crv{(3,-1) & (-3,1)}?(1)*[black]\dir{};
    (-3,-6)*{\scs #1};
     (3.1,-6)*{\scs #2};
     (-8,0)*{};(8,0)*{};
     }}
   \newcommand{\ncrossrp}[2]{
 \xybox{
    (-3,-4)*{};(3,4)*{} **[black]\crv{(-3,-1) & (3,1)}?(1)*[black]\dir{} ;
    (3,-4)*{};(-3,4)*{} **[magenta]\crv{(3,-1) & (-3,1)}?(1)*[magenta]\dir{};
    (-3,-6)*{\scs #1};
     (3.1,-6)*{\scs #2};
     (-8,0)*{};(8,0)*{};
     }}
\newcommand{\ncrossbr}[2]{
 \xybox{
    (-3,-4)*{};(3,4)*{} **[blue]\crv{(-3,-1) & (3,1)}?(1)*[blue]\dir{} ;
    (3,-4)*{};(-3,4)*{} **[black]\crv{(3,-1) & (-3,1)}?(1)*[black]\dir{};
    (-3,-6)*{\scs #1};
     (3.1,-6)*{\scs #2};
     (-8,0)*{};(8,0)*{};
     }}
\newcommand{\ncrossrb}[2]{
 \xybox{
    (-3,-4)*{};(3,4)*{} **[black]\crv{(-3,-1) & (3,1)}?(1)*[black]\dir{} ;
    (3,-4)*{};(-3,4)*{} **[blue]\crv{(3,-1) & (-3,1)}?(1)*[blue]\dir{};
    (-3,-6)*{\scs #1};
     (3.1,-6)*{\scs #2};
     (-8,0)*{};(8,0)*{};
     }}
\newcommand{\ncrossgb}[2]{
 \xybox{
    (-3,-4)*{};(3,4)*{} **[green]\crv{(-3,-1) & (3,1)}?(1)*[green]\dir{} ;
    (3,-4)*{};(-3,4)*{} **[blue]\crv{(3,-1) & (-3,1)}?(1)*[blue]\dir{};
    (-3,-6)*{\scs #1};
     (3.1,-6)*{\scs #2};
     (-8,0)*{};(8,0)*{};
     }}
\newcommand{\ncrossbg}[2]{
 \xybox{
    (-3,-4)*{};(3,4)*{} **[blue]\crv{(-3,-1) & (3,1)}?(1)*[blue]\dir{} ;
    (3,-4)*{};(-3,4)*{} **[green]\crv{(3,-1) & (-3,1)}?(1)*[green]\dir{};
    (-3,-6)*{\scs #1};
     (3.1,-6)*{\scs #2};
     (-8,0)*{};(8,0)*{};
     }}
\newcommand{\ncrossrg}[2]{
 \xybox{
    (-3,-4)*{};(3,4)*{} **[black]\crv{(-3,-1) & (3,1)}?(1)*[black]\dir{} ;
    (3,-4)*{};(-3,4)*{} **[green]\crv{(3,-1) & (-3,1)}?(1)*[green]\dir{};
    (-3,-6)*{\scs #1};
     (3.1,-6)*{\scs #2};
     (-8,0)*{};(8,0)*{};
     }}
\newcommand{\ncrossgr}[2]{
 \xybox{
    (-3,-4)*{};(3,4)*{} **[green]\crv{(-3,-1) & (3,1)}?(1)*[green]\dir{} ;
    (3,-4)*{};(-3,4)*{} **[black]\crv{(3,-1) & (-3,1)}?(1)*[black]\dir{};
    (-3,-6)*{\scs #1};
     (3.1,-6)*{\scs #2};
     (-8,0)*{};(8,0)*{};
     }}

\newcommand{\dcrossbp}[2]{
 \xybox{
    (-3,-4)*{};(3,4)*{} **[blue]\crv{(-3,-1) & (3,1)}?(0)*[blue]\dir{<} ;
    (3,-4)*{};(-3,4)*{} **[magenta]\crv{(3,-1) & (-3,1)}?(0)*[magenta]\dir{<};
    (-3,-6)*{\scs #1};
     (3.1,-6)*{\scs #2};
     (-8,0)*{};(8,0)*{};
     }}
\newcommand{\dcrosspr}[2]{
 \xybox{
    (-3,-4)*{};(3,4)*{} **[magenta]\crv{(-3,-1) & (3,1)}?(0)*[magenta]\dir{<} ;
    (3,-4)*{};(-3,4)*{} **[black]\crv{(3,-1) & (-3,1)}?(0)*[black]\dir{<};
    (-3,-6)*{\scs #1};
     (3.1,-6)*{\scs #2};
     (-8,0)*{};(8,0)*{};
     }}
\newcommand{\dcrossrp}[2]{
 \xybox{
    (-3,-4)*{};(3,4)*{} **[black]\crv{(-3,-1) & (3,1)}?(0)*[black]\dir{<} ;
    (3,-4)*{};(-3,4)*{} **[magenta]\crv{(3,-1) & (-3,1)}?(0)*[magenta]\dir{<};
    (-3,-6)*{\scs #1};
     (3.1,-6)*{\scs #2};
     (-8,0)*{};(8,0)*{};
     }}
\newcommand{\dcrossrr}[2]{
 \xybox{
    (-3,-4)*{};(3,4)*{} **[black]\crv{(-3,-1) & (3,1)}?(0)*[black]\dir{<} ;
    (3,-4)*{};(-3,4)*{} **[black]\crv{(3,-1) & (-3,1)}?(0)*[black]\dir{<};
    (-3,-6)*{\scs #1};
     (3.1,-6)*{\scs #2};
     (-8,0)*{};(8,0)*{};
     }}
\newcommand{\dcrossbr}[2]{
 \xybox{
    (-3,-4)*{};(3,4)*{} **[blue]\crv{(-3,-1) & (3,1)}?(0)*[blue]\dir{<} ;
    (3,-4)*{};(-3,4)*{} **[black]\crv{(3,-1) & (-3,1)}?(0)*[black]\dir{<};
    (-3,-6)*{\scs #1};
     (3.1,-6)*{\scs #2};
     (-8,0)*{};(8,0)*{};
     }}
\newcommand{\dcrossrb}[2]{
 \xybox{
    (-3,-4)*{};(3,4)*{} **[black]\crv{(-3,-1) & (3,1)}?(0)*[black]\dir{<} ;
    (3,-4)*{};(-3,4)*{} **[blue]\crv{(3,-1) & (-3,1)}?(0)*[blue]\dir{<};
    (-3,-6)*{\scs #1};
     (3.1,-6)*{\scs #2};
     (-8,0)*{};(8,0)*{};
     }}
\newcommand{\dcrossgb}[2]{
 \xybox{
    (-3,-4)*{};(3,4)*{} **[green]\crv{(-3,-1) & (3,1)}?(0)*[green]\dir{<} ;
    (3,-4)*{};(-3,4)*{} **[blue]\crv{(3,-1) & (-3,1)}?(0)*[blue]\dir{<};
    (-3,-6)*{\scs #1};
     (3.1,-6)*{\scs #2};
     (-8,0)*{};(8,0)*{};
     }}
\newcommand{\dcrossbg}[2]{
 \xybox{
    (-3,-4)*{};(3,4)*{} **[blue]\crv{(-3,-1) & (3,1)}?(0)*[blue]\dir{<} ;
    (3,-4)*{};(-3,4)*{} **[green]\crv{(3,-1) & (-3,1)}?(0)*[green]\dir{<};
    (-3,-6)*{\scs #1};
     (3.1,-6)*{\scs #2};
     (-8,0)*{};(8,0)*{};
     }}
\newcommand{\dcrossrg}[2]{
 \xybox{
    (-3,-4)*{};(3,4)*{} **[black]\crv{(-3,-1) & (3,1)}?(0)*[black]\dir{<} ;
    (3,-4)*{};(-3,4)*{} **[green]\crv{(3,-1) & (-3,1)}?(0)*[green]\dir{<};
    (-3,-6)*{\scs #1};
     (3.1,-6)*{\scs #2};
     (-8,0)*{};(8,0)*{};
     }}
\newcommand{\dcrossgr}[2]{
 \xybox{
    (-3,-4)*{};(3,4)*{} **[green]\crv{(-3,-1) & (3,1)}?(0)*[green]\dir{<} ;
    (3,-4)*{};(-3,4)*{} **[black]\crv{(3,-1) & (-3,1)}?(0)*[black]\dir{<};
    (-3,-6)*{\scs #1};
     (3.1,-6)*{\scs #2};
     (-8,0)*{};(8,0)*{};
     }}
\newcommand{\dcrossgg}[2]{
 \xybox{
    (-3,-4)*{};(3,4)*{} **[green]\crv{(-3,-1) & (3,1)}?(0)*[green]\dir{<} ;
    (3,-4)*{};(-3,4)*{} **[green]\crv{(3,-1) & (-3,1)}?(0)*[green]\dir{<};
    (-3,-6)*{\scs #1};
     (3.1,-6)*{\scs #2};
     (-8,0)*{};(8,0)*{};
     }}

\newcommand{\srcup}[1]{\xybox{%
  (-5,0)*{};
  (5,0)*{};
  (-3,0)*{}="t1";
  (3,0)*{}="t2";
  "t1";"t2" **[blue]\crv{(-3,-5) & (3,-5)}; ?(1)*[blue]\dir{>}
   ?(.5)*\dir{}+(0,-2)*{\scriptstyle{#1}};
}}
\newcommand{\srcupr}[1]{\xybox{%
  (-5,0)*{};
  (5,0)*{};
  (-3,0)*{}="t1";
  (3,0)*{}="t2";
  "t1";"t2" **[black]\crv{(-3,-5) & (3,-5)}; ?(1)*[black]\dir{>}
   ?(.5)*\dir{}+(0,-2)*{\scriptstyle{#1}};
}}
\newcommand{\srcupp}[1]{\xybox{%
  (-5,0)*{};
  (5,0)*{};
  (-3,0)*{}="t1";
  (3,0)*{}="t2";
  "t1";"t2" **[magenta]\crv{(-3,-5) & (3,-5)}; ?(1)*[magenta]\dir{>}
   ?(.5)*\dir{}+(0,-2)*{\scriptstyle{#1}};
}}
\newcommand{\ssrcupr}[1]{\xybox{%
  (-5,0)*{};
  (5,-6)*{};
  (-3,0)*{}="t1";
  (3,0)*{}="t2";
  "t1";"t2" **[black]\crv{(-3,-3.5) & (3,-3.5)}; ?(1)*[black]\dir{>}
   ?(.5)*\dir{}+(0,-2)*{\scriptstyle{#1}};
}}
\newcommand{\srcupg}[1]{\xybox{%
  (-5,0)*{};
  (5,0)*{};
  (-3,0)*{}="t1";
  (3,0)*{}="t2";
  "t1";"t2" **[green]\crv{(-3,-5) & (3,-5)}; ?(1)*[green]\dir{>}
   ?(.5)*\dir{}+(0,-2)*{\scriptstyle{#1}};
}}
\newcommand{\slcup}[1]{\xybox{%
  (-5,0)*{};
  (5,0)*{};
  (-3,0)*{}="t1";
  (3,0)*{}="t2";
  "t1";"t2" **[blue]\crv{(-3,-5) & (3,-5)}; ?(0)*[blue]\dir{<}
   ?(.5)*\dir{}+(0,-2)*{\scriptstyle{#1}};
}}
\newcommand{\slcupr}[1]{\xybox{%
  (-5,0)*{};
  (5,0)*{};
  (-3,0)*{}="t1";
  (3,0)*{}="t2";
  "t1";"t2" **[black]\crv{(-3,-5) & (3,-5)}; ?(0)*[black]\dir{<}
   ?(.5)*\dir{}+(0,-2)*{\scriptstyle{#1}};
}}
\newcommand{\slcupg}[1]{\xybox{%
  (-5,0)*{};
  (5,0)*{};
  (-3,0)*{}="t1";
  (3,0)*{}="t2";
  "t1";"t2" **[green]\crv{(-3,-5) & (3,-5)}; ?(0)*[green]\dir{<}
   ?(.5)*\dir{}+(0,-2)*{\scriptstyle{#1}};
}}
\newcommand{\srcap}[1]{\xybox{%
  (-5,6)*{};
  (5,0)*{};
  (-3,0)*{}="t1";
  (3,0)*{}="t2";
  "t1";"t2" **[blue]\crv{(-3,5) & (3,5)}; ?(1)*[blue]\dir{>}
  ?(.5)*\dir{}+(0,2)*{\scriptstyle{#1}};
}}
\newcommand{\srcapr}[1]{\xybox{%
  (-5,6)*{};
  (5,0)*{};
  (-3,0)*{}="t1";
  (3,0)*{}="t2";
  "t1";"t2" **[black]\crv{(-3,5) & (3,5)}; ?(1)*[black]\dir{>}
  ?(.5)*\dir{}+(0,2)*{\scriptstyle{#1}};
}}
\newcommand{\srcapp}[1]{\xybox{%
  (-5,6)*{};
  (5,0)*{};
  (-3,0)*{}="t1";
  (3,0)*{}="t2";
  "t1";"t2" **[magenta]\crv{(-3,5) & (3,5)}; ?(1)*[magenta]\dir{>}
  ?(.5)*\dir{}+(0,2)*{\scriptstyle{#1}};
}}
\newcommand{\ssrcapr}[1]{\xybox{%
  (-5,6)*{};
  (5,0)*{};
  (-3,0)*{}="t1";
  (3,0)*{}="t2";
  "t1";"t2" **[black]\crv{(-3,3.5) & (3,3.5)}; ?(1)*[black]\dir{>}
  ?(.5)*\dir{}+(0,2)*{\scriptstyle{#1}};
}}
\newcommand{\srcapg}[1]{\xybox{%
  (-5,6)*{};
  (5,0)*{};
  (-3,0)*{}="t1";
  (3,0)*{}="t2";
  "t1";"t2" **[green]\crv{(-3,5) & (3,5)}; ?(1)*[green]\dir{>}
  ?(.5)*\dir{}+(0,2)*{\scriptstyle{#1}};
}}
\newcommand{\slcap}[1]{\xybox{%
  (-5,6)*{};
  (5,0)*{};
  (-3,0)*{}="t1";
  (3,0)*{}="t2";
  "t1";"t2" **[blue]\crv{(-3,5) & (3,5)}; ?(0)*[blue]\dir{<}
  ?(.5)*\dir{}+(0,2)*{\scriptstyle{#1}};
}}
\newcommand{\slcapr}[1]{\xybox{%
  (-5,6)*{};
  (5,0)*{};
  (-3,0)*{}="t1";
  (3,0)*{}="t2";
  "t1";"t2" **[black]\crv{(-3,5) & (3,5)}; ?(0)*[black]\dir{<}
  ?(.5)*\dir{}+(0,2)*{\scriptstyle{#1}};
}}
\newcommand{\slcapg}[1]{\xybox{%
  (-5,6)*{};
  (5,0)*{};
  (-3,0)*{}="t1";
  (3,0)*{}="t2";
  "t1";"t2" **[green]\crv{(-3,5) & (3,5)}; ?(0)*[green]\dir{<}
  ?(.5)*\dir{}+(0,2)*{\scriptstyle{#1}};
}}

\newcommand{\lcupcuprb}[2]{\xybox{%
(-5,0)*{};
  (5,0)*{};
  (-3,0)*{}="t1";
  (3,0)*{}="t2";
  "t1";"t2" **[blue]\crv{(-3,-5) & (3,-5)}; ?(0)*[blue]\dir{<}
   ?(.5)*\dir{}+(0,-2)*{\scriptstyle{#1}};
  (-8,0)*{};
  (8,0)*{};
  (-6,0)*{}="t1";
  (6,0)*{}="t2";
  "t1";"t2" **[black]\crv{(-6,-8) & (6,-8)}; ?(0)*[black]\dir{<}
  ?(.5)*\dir{}+(0,-2)*{\scriptstyle{#1}};
}}
\newcommand{\lcupcupbr}[2]{\xybox{%
(-5,0)*{};
  (5,0)*{};
  (-3,0)*{}="t1";
  (3,0)*{}="t2";
  "t1";"t2" **[black]\crv{(-3,-5) & (3,-5)}; ?(0)*[black]\dir{<}
   ?(.5)*\dir{}+(0,-2)*{\scriptstyle{#1}};
  (-8,0)*{};
  (8,0)*{};
  (-6,0)*{}="t1";
  (6,0)*{}="t2";
  "t1";"t2" **[blue]\crv{(-6,-8) & (6,-8)}; ?(0)*[blue]\dir{<}
  ?(.5)*\dir{}+(0,-2)*{\scriptstyle{#1}};
}}
\newcommand{\rcupcuprb}[2]{\xybox{%
(-5,0)*{};
  (5,0)*{};
  (-3,0)*{}="t1";
  (3,0)*{}="t2";
  "t1";"t2" **[blue]\crv{(-3,-5) & (3,-5)}; ?(1)*[blue]\dir{>}
   ?(.5)*\dir{}+(0,-2)*{\scriptstyle{#1}};
  (-8,0)*{};
  (8,0)*{};
  (-6,0)*{}="t1";
  (6,0)*{}="t2";
  "t1";"t2" **[black]\crv{(-6,-8) & (6,-8)}; ?(1)*[black]\dir{>}
  ?(.5)*\dir{}+(0,-2)*{\scriptstyle{#1}};
}}
\newcommand{\rcupcupbr}[2]{\xybox{%
(-5,0)*{};
  (5,0)*{};
  (-3,0)*{}="t1";
  (3,0)*{}="t2";
  "t1";"t2" **[black]\crv{(-3,-5) & (3,-5)}; ?(1)*[black]\dir{>}
   ?(.5)*\dir{}+(0,-2)*{\scriptstyle{#1}};
  (-8,0)*{};
  (8,0)*{};
  (-6,0)*{}="t1";
  (6,0)*{}="t2";
  "t1";"t2" **[blue]\crv{(-6,-8) & (6,-8)}; ?(1)*[blue]\dir{>}
  ?(.5)*\dir{}+(0,-2)*{\scriptstyle{#1}};
}}
\newcommand{\rcapcapbr}[2]{\xybox{%
(-5,0)*{};
  (5,0)*{};
  (-3,0)*{}="t1";
  (3,0)*{}="t2";
  "t1";"t2" **[black]\crv{(-3,5) & (3,5)}; ?(1)*[black]\dir{>}
   ?(.5)*\dir{}+(0,2)*{\scriptstyle{#1}};
  (-8,0)*{};
  (8,0)*{};
  (-6,0)*{}="t1";
  (6,0)*{}="t2";
  "t1";"t2" **[blue]\crv{(-6,8) & (6,8)}; ?(1)*[blue]\dir{>}
  ?(.5)*\dir{}+(0,2)*{\scriptstyle{#1}};
}}
\newcommand{\rcapcaprb}[2]{\xybox{%
(-5,0)*{};
  (5,0)*{};
  (-3,0)*{}="t1";
  (3,0)*{}="t2";
  "t1";"t2" **[blue]\crv{(-3,5) & (3,5)}; ?(1)*[blue]\dir{>}
   ?(.5)*\dir{}+(0,2)*{\scriptstyle{#1}};
  (-8,0)*{};
  (8,0)*{};
  (-6,0)*{}="t1";
  (6,0)*{}="t2";
  "t1";"t2" **[black]\crv{(-6,8) & (6,8)}; ?(1)*[black]\dir{>}
  ?(.5)*\dir{}+(0,2)*{\scriptstyle{#1}};
}}
\newcommand{\lcapcapbr}[2]{\xybox{%
(-5,0)*{};
  (5,0)*{};
  (-3,0)*{}="t1";
  (3,0)*{}="t2";
  "t1";"t2" **[black]\crv{(-3,5) & (3,5)}; ?(0)*[black]\dir{<}
   ?(.5)*\dir{}+(0,2)*{\scriptstyle{#1}};
  (-8,0)*{};
  (8,0)*{};
  (-6,0)*{}="t1";
  (6,0)*{}="t2";
  "t1";"t2" **[blue]\crv{(-6,8) & (6,8)}; ?(0)*[blue]\dir{<}
  ?(.5)*\dir{}+(0,2)*{\scriptstyle{#1}};
}}
\newcommand{\lcapcaprb}[2]{\xybox{%
(-5,0)*{};
  (5,0)*{};
  (-3,0)*{}="t1";
  (3,0)*{}="t2";
  "t1";"t2" **[blue]\crv{(-3,5) & (3,5)}; ?(0)*[blue]\dir{<}
   ?(.5)*\dir{}+(0,2)*{\scriptstyle{#1}};
  (-8,0)*{};
  (8,0)*{};
  (-6,0)*{}="t1";
  (6,0)*{}="t2";
  "t1";"t2" **[black]\crv{(-6,8) & (6,8)}; ?(0)*[black]\dir{<}
  ?(.5)*\dir{}+(0,2)*{\scriptstyle{#1}};
}}

\newcommand{\lowrru}[1]{\xybox{%
  (-8,0)*{};
  (8,0)*{};
  (-6,-18)*{};(6,-9)*{} **[blue]\crv{(-6,-13) & (6,-15)} ?(1)*[blue]\dir{>};
  (6,-9)*{};(6,0)*{}  **[blue]\dir{-} ?(.3)*\dir{ }+(2,0)*{\scs {\bf j}};
}}

\newcommand{\lowllu}[1]{\xybox{%
  (-8,0)*{};
  (8,0)*{};
  (6,-18)*{};(-6,-9)*{} **[blue]\crv{(6,-13) & (-6,-15)} ?(1)*[blue]\dir{>};
  (-6,-9)*{};(-6,0)*{}  **[blue]\dir{-} ?(.3)*\dir{ }+(-2,0)*{\scs {\bf j}};
}}

\newcommand{\bber}[1]{\xybox{%
  (-2,0)*{};
  (2,0)*{};
  (0,0);(0,-18) **[black]\dir{-}; ?(.5)*[black]\dir{<}+(2.3,0)*{\scriptstyle{#1}};
}}

\newcommand{\ccbub}[2]{
\xybox{%
 (-6,0)*{};
  (6,0)*{};
  (-4,0)*{}="t1";
  (4,0)*{}="t2";
  "t2";"t1" **[blue]\crv{(4,6) & (-4,6)}; ?(.7)*\dir{}+(-2,1)*{\scs #2}
  ?(.05)*[blue]\dir{>} ?(1)*[blue]\dir{>};
  "t2";"t1" **[blue]\crv{(4,-6) & (-4,-6)};
   ?(.3)*\dir{}+(0,0)*{\bullet}+(0,-3)*{\scs {#1}};
}}
\newcommand{\ccbubr}[2]{
\xybox{%
 (-6,0)*{};
  (6,0)*{};
  (-4,0)*{}="t1";
  (4,0)*{}="t2";
  "t2";"t1" **[black]\crv{(4,6) & (-4,6)}; ?(.7)*\dir{}+(-2,0)*{\scs #2}
  ?(.05)*[black]\dir{>} ?(1)*[black]\dir{>};
  "t2";"t1" **[black]\crv{(4,-6) & (-4,-6)};
   ?(.3)*\dir{}+(0,0)*{\bullet}+(0,-3)*{\scs {#1}};
}}
\newcommand{\ccbubg}[2]{
\xybox{%
 (-6,0)*{};
  (6,0)*{};
  (-4,0)*{}="t1";
  (4,0)*{}="t2";
  "t2";"t1" **[green]\crv{(4,6) & (-4,6)}; ?(.7)*\dir{}+(-2,0)*{\scs #2}
  ?(.05)*[green]\dir{>} ?(1)*[green]\dir{>};
  "t2";"t1" **[green]\crv{(4,-6) & (-4,-6)};
   ?(.3)*\dir{}+(0,0)*{\bullet}+(0,-3)*{\scs {#1}};
}}
\newcommand{\cbub}[2]{
\xybox{%
 (-6,0)*{};
  (6,0)*{};
  (-4,0)*{}="t1";
  (4,0)*{}="t2";
  "t2";"t1" **[blue]\crv{(4,6) & (-4,6)};?(.7)*\dir{}+(-2,0)*{\scs #2};
   ?(0)*[blue]\dir{<} ?(.95)*[blue]\dir{<};
  "t2";"t1" **[blue]\crv{(4,-6) & (-4,-6)};
   ?(.3)*\dir{}+(0,0)*{\bullet}+(0,-3)*{\scs {#1}};
}}
\newcommand{\cbubr}[2]{
\xybox{%
 (-6,0)*{};
  (6,0)*{};
  (-4,0)*{}="t1";
  (4,0)*{}="t2";
  "t2";"t1" **[black]\crv{(4,6) & (-4,6)};?(.7)*\dir{}+(-2,0)*{\scs #2};
   ?(0)*[black]\dir{<} ?(.95)*[black]\dir{<};
  "t2";"t1" **[black]\crv{(4,-6) & (-4,-6)};
   ?(.3)*\dir{}+(0,0)*{\bullet}+(0,-3)*{\scs {#1}};
}}
\newcommand{\cbubg}[2]{
\xybox{%
 (-6,0)*{};
  (6,0)*{};
  (-4,0)*{}="t1";
  (4,0)*{}="t2";
  "t2";"t1" **[green]\crv{(4,6) & (-4,6)};?(.7)*\dir{}+(-2,0)*{\scs #2};
   ?(0)*[green]\dir{<} ?(.95)*[green]\dir{<};
  "t2";"t1" **[green]\crv{(4,-6) & (-4,-6)};
   ?(.3)*\dir{}+(0,0)*{\bullet}+(0,-3)*{\scs {#1}};
}}

\newcommand{\medccbub}[1]{
\xybox{%
 (-5,0)*{};
  (5,0)*{};
  (-2.2,0)*{}="t1";
  (2.2,0)*{}="t2";
  "t2";"t1" **[blue]\crv{(2.2,3.2) & (-2.2,3.2)};   ?(1)*[blue]\dir{>}; ?(.3)*\dir{}+(0,0)*{\bullet}+(1,3)*{\scriptscriptstyle
  {#1}};
  "t2";"t1" **[blue]\crv{(2.2,-3.2) & (-2.2,-3.2)};
}}
\newcommand{\medcbub}[1]{
\xybox{%
 (-5,0)*{};
  (5,0)*{};
  (-2.2,0)*{}="t1";
  (2.2,0)*{}="t2";
  "t2";"t1" **[blue]\crv{(2.2,3.2) & (-2.2,3.2)};   ?(1)*[blue]\dir{<}; ?(.3)*\dir{}+(0,0)*{\bullet}+(1,3)*{\scriptscriptstyle
  {#1}};
  "t2";"t1" **[blue]\crv{(2.2,-3.2) & (-2.2,-3.2)};
}}
\newcommand{\smccbubr}[1]{
\xybox{%
 (-5,0)*{};
  (5,0)*{};
  (-2,0)*{}="t1";
  (2,0)*{}="t2";
  "t2";"t1" **[black]\crv{(2,3) & (-2,3)}; ?(.05)*[black]\dir{>} ?(1)*[black]\dir{>}; ?(.3)*\dir{}+(0,0)*{\scs \bullet}+(1,2.5)*{\scriptscriptstyle
  {#1}};
  "t2";"t1" **[black]\crv{(2,-3) & (-2,-3)};
}}
\newcommand{\smccbubg}[1]{
\xybox{%
 (-5,0)*{};
  (5,0)*{};
  (-2,0)*{}="t1";
  (2,0)*{}="t2";
  "t2";"t1" **[green]\crv{(2,3) & (-2,3)}; ?(.05)*[green]\dir{>} ?(1)*[green]\dir{>}; ?(.3)*\dir{}+(0,0)*{\scs \bullet}+(1,2.5)*{\scriptscriptstyle
  {#1}};
  "t2";"t1" **[green]\crv{(2,-3) & (-2,-3)};
}}

\newcommand{\smcbubr}[1]{
\xybox{%
 (-5,0)*{};
  (5,0)*{};
  (-2,0)*{}="t1";
  (2,0)*{}="t2";
  "t2";"t1" **[black]\crv{(2,3) & (-2,3)}; ?(0)*[black]\dir{<} ?(.95)*[black]\dir{<} ?(.3)*\dir{}+(0,0)*{\scs \bullet}+(1,2.5)*{\scriptscriptstyle {#1}};
  "t2";"t1" **[black]\crv{(2,-3) & (-2,-3)};
}}
\newcommand{\smcbubg}[1]{
\xybox{%
 (-5,0)*{};
  (5,0)*{};
  (-2,0)*{}="t1";
  (2,0)*{}="t2";
  "t2";"t1" **[green]\crv{(2,3) & (-2,3)}; ?(0)*[green]\dir{<} ?(.95)*[green]\dir{<} ?(.3)*\dir{}+(0,0)*{\scs \bullet}+(1,2.5)*{\scriptscriptstyle {#1}};
  "t2";"t1" **[green]\crv{(2,-3) & (-2,-3)};
}}

\newcommand{\rcup}[1]{\xybox{%
  (-6,0)*{};
  (6,0)*{};
  (-4,0)*{}="t1";
  (4,0)*{}="t2";
  "t1";"t2" **[blue]\crv{(-4,-6) & (4,-6)}; ?(1)*[blue]\dir{>}
   ?(.5)*\dir{}+(0,-2)*{\scriptstyle{#1}};
}}
\newcommand{\rcupr}[1]{\xybox{%
  (-6,0)*{};
  (6,0)*{};
  (-4,0)*{}="t1";
  (4,0)*{}="t2";
  "t1";"t2" **[black]\crv{(-4,-6) & (4,-6)}; ?(1)*[black]\dir{>}
   ?(.5)*\dir{}+(0,-2)*{\scriptstyle{#1}};
}}
\newcommand{\rcupg}[1]{\xybox{%
  (-6,0)*{};
  (6,0)*{};
  (-4,0)*{}="t1";
  (4,0)*{}="t2";
  "t1";"t2" **[green]\crv{(-4,-6) & (4,-6)}; ?(1)*[green]\dir{>}
   ?(.5)*\dir{}+(0,-2)*{\scriptstyle{#1}};
}}
\newcommand{\lcup}[1]{\xybox{%
  (-6,0)*{};
  (6,0)*{};
  (-4,0)*{}="t1";
  (4,0)*{}="t2";
  "t2";"t1" **[blue]\crv{(4,-6) & (-4,-6)}; ?(1)*[blue]\dir{>}
  ?(.5)*\dir{}+(0,-2)*{\scriptstyle{#1}};
}}
\newcommand{\lcupr}[1]{\xybox{%
  (-6,0)*{};
  (6,0)*{};
  (-4,0)*{}="t1";
  (4,0)*{}="t2";
  "t2";"t1" **[black]\crv{(4,-6) & (-4,-6)}; ?(1)*[black]\dir{>}
  ?(.5)*\dir{}+(0,-2)*{\scriptstyle{#1}};
}}
\newcommand{\lcupg}[1]{\xybox{%
  (-6,0)*{};
  (6,0)*{};
  (-4,0)*{}="t1";
  (4,0)*{}="t2";
  "t2";"t1" **[green]\crv{(4,-6) & (-4,-6)}; ?(1)*[green]\dir{>}
  ?(.5)*\dir{}+(0,-2)*{\scriptstyle{#1}};
}}
\newcommand{\rcap}[1]{\xybox{%
  (-6,0)*{};
  (6,0)*{};
  (-4,0)*{}="t1";
  (4,0)*{}="t2";
  "t1";"t2" **[blue]\crv{(-4,6) & (4,6)}; ?(1)*[blue]\dir{>}
  ?(.5)*\dir{}+(0,2)*{\scriptstyle{#1}};
}}
\newcommand{\rcapr}[1]{\xybox{%
  (-6,0)*{};
  (6,0)*{};
  (-4,0)*{}="t1";
  (4,0)*{}="t2";
  "t1";"t2" **[black]\crv{(-4,6) & (4,6)}; ?(1)*[black]\dir{>}
  ?(.5)*\dir{}+(0,2)*{\scriptstyle{#1}};
}}
\newcommand{\rcapg}[1]{\xybox{%
  (-6,0)*{};
  (6,0)*{};
  (-4,0)*{}="t1";
  (4,0)*{}="t2";
  "t1";"t2" **[green]\crv{(-4,6) & (4,6)}; ?(1)*[green]\dir{>}
  ?(.5)*\dir{}+(0,2)*{\scriptstyle{#1}};
}}
\newcommand{\lcap}[1]{\xybox{%
  (-6,0)*{};
  (6,0)*{};
  (-4,0)*{}="t1";
  (4,0)*{}="t2";
  "t2";"t1" **[blue]\crv{(4,6) & (-4,6)}; ?(1)*[blue]\dir{>} ?(.5)*\dir{}+(0,2)*{\scriptstyle{#1}};
}}
\newcommand{\lcapr}[1]{\xybox{%
  (-6,0)*{};
  (6,0)*{};
  (-4,0)*{}="t1";
  (4,0)*{}="t2";
  "t2";"t1" **[black]\crv{(4,6) & (-4,6)}; ?(1)*[black]\dir{>} ?(.5)*\dir{}+(0,2)*{\scriptstyle{#1}};
}}
\newcommand{\lcapg}[1]{\xybox{%
  (-6,0)*{};
  (6,0)*{};
  (-4,0)*{}="t1";
  (4,0)*{}="t2";
  "t2";"t1" **[green]\crv{(4,6) & (-4,6)}; ?(1)*[green]\dir{>} ?(.5)*\dir{}+(0,2)*{\scriptstyle{#1}};
}}

\newcommand{\rbigcrosspb}[2]{\xybox{
  (-9,-8)*{};(9,8)*{} **[magenta]\crv{(-9,-2) & (9,2)}?(1)*[magenta]\dir{>} ;
    (9,-8)*{};(-9,8)*{} **[blue]\crv{(9,-2) & (-9,2)}?(0)*[blue]\dir{<};
    (-9,-12)*{\scs #1};
     (9.2,-12)*{\scs #2};
     (-11,0)*{};(11,0)*{};
 }}
\newcommand{\rbigcrossrr}[2]{\xybox{
  (-9,-8)*{};(9,8)*{} **[black]\crv{(-9,-2) & (9,2)}?(1)*[black]\dir{>} ;
    (9,-8)*{};(-9,8)*{} **[black]\crv{(9,-2) & (-9,2)}?(0)*[black]\dir{<};
    (-9,-12)*{\scs #1};
     (9.2,-12)*{\scs #2};
     (-11,0)*{};(11,0)*{};
 }}
\newcommand{\lbigcrossrr}[2]{\xybox{
  (-9,-8)*{};(9,8)*{} **[black]\crv{(-9,-2) & (9,2)}?(0)*[black]\dir{<} ;
    (9,-8)*{};(-9,8)*{} **[black]\crv{(9,-2) & (-9,2)}?(1)*[black]\dir{>};
    (-9,-12)*{\scs #1};
     (9.2,-12)*{\scs #2};
     (-11,0)*{};(11,0)*{};
 }}

\newcommand{\lbigcrossbp}[2]{\xybox{
  (-9,-8)*{};(9,8)*{} **[blue]\crv{(-9,-2) & (9,2)}?(0)*[blue]\dir{<} ;
    (9,-8)*{};(-9,8)*{} **[magenta]\crv{(9,-2) & (-9,2)}?(1)*[magenta]\dir{>};
    (-9,-12)*{\scs #1};
     (9.2,-12)*{\scs #2};
     (-11,0)*{};(11,0)*{};
 }}
\newcommand{\llrcup}[1]{\xybox{%
  (-12,0)*{};
  (12,0)*{};
  (-9,0)*{}="t1";
  (9,0)*{}="t2";
  "t1";"t2" **[blue]\crv{(-9,-12) & (9,-12)}; ?(1)*[blue]\dir{>}
  ?(.5)*\dir{}+(0,-2)*{\scriptstyle{#1}};
}}
\newcommand{\llrcupr}[1]{\xybox{%
  (-12,0)*{};
  (12,0)*{};
  (-9,0)*{}="t1";
  (9,0)*{}="t2";
  "t1";"t2" **[black]\crv{(-9,-12) & (9,-12)}; ?(1)*[black]\dir{>}
  ?(.5)*\dir{}+(0,-2)*{\scriptstyle{#1}};
}}
\newcommand{\llrcupp}[1]{\xybox{%
  (-12,0)*{};
  (12,0)*{};
  (-9,0)*{}="t1";
  (9,0)*{}="t2";
  "t1";"t2" **[magenta]\crv{(-9,-12) & (9,-12)}; ?(1)*[magenta]\dir{>}
  ?(.5)*\dir{}+(0,-2)*{\scriptstyle{#1}};
}}
\newcommand{\llrcupg}[1]{\xybox{%
  (-12,0)*{};
  (12,0)*{};
  (-9,0)*{}="t1";
  (9,0)*{}="t2";
  "t1";"t2" **[green]\crv{(-9,-12) & (9,-12)}; ?(1)*[green]\dir{>}
  ?(.5)*\dir{}+(0,-2)*{\scriptstyle{#1}};
}}
\newcommand{\lllcup}[1]{\xybox{%
  (-12,0)*{};
  (12,0)*{};
  (-9,0)*{}="t1";
  (9,0)*{}="t2";
  "t1";"t2" **[blue]\crv{(-9,-12) & (9,-12)}; ?(0)*[blue]\dir{<}
  ?(.5)*\dir{}+(0,-2)*{\scriptstyle{#1}};
}}
\newcommand{\lllcupr}[1]{\xybox{%
  (-12,0)*{};
  (12,0)*{};
  (-9,0)*{}="t1";
  (9,0)*{}="t2";
  "t1";"t2" **[black]\crv{(-9,-12) & (9,-12)}; ?(0)*[black]\dir{<}
  ?(.5)*\dir{}+(0,-2)*{\scriptstyle{#1}};
}}

\newcommand{\lllcap}[1]{\xybox{%
  (-12,0)*{};
  (12,0)*{};
  (-9,0)*{}="t1";
  (9,0)*{}="t2";
  "t1";"t2" **[blue]\crv{(-9,12) & (9,12)}; ?(0)*[blue]\dir{<}
  ?(.5)*\dir{}+(0,2)*{\scriptstyle{#1}};
}}
\newcommand{\lllcapr}[1]{\xybox{%
  (-12,0)*{};
  (12,0)*{};
  (-9,0)*{}="t1";
  (9,0)*{}="t2";
  "t1";"t2" **[black]\crv{(-9,12) & (9,12)}; ?(0)*[black]\dir{<}
  ?(.5)*\dir{}+(0,2)*{\scriptstyle{#1}};
}}

\newcommand{\llrcap}[1]{\xybox{%
  (-12,0)*{};
  (12,0)*{};
  (-9,0)*{}="t1";
  (9,0)*{}="t2";
  "t1";"t2" **[blue]\crv{(-9,12) & (9,12)}; ?(1)*[blue]\dir{>}
  ?(.5)*\dir{}+(0,2)*{\scriptstyle{#1}};
}}
\newcommand{\llrcapr}[1]{\xybox{%
  (-12,0)*{};
  (12,0)*{};
  (-9,0)*{}="t1";
  (9,0)*{}="t2";
  "t1";"t2" **[black]\crv{(-9,12) & (9,12)}; ?(1)*[black]\dir{>}
  ?(.5)*\dir{}+(0,2)*{\scriptstyle{#1}};
}}
\newcommand{\llrcapp}[1]{\xybox{%
  (-12,0)*{};
  (12,0)*{};
  (-9,0)*{}="t1";
  (9,0)*{}="t2";
  "t1";"t2" **[magenta]\crv{(-9,12) & (9,12)}; ?(1)*[magenta]\dir{>}
  ?(.5)*\dir{}+(0,2)*{\scriptstyle{#1}};
}}
\newcommand{\llrcapg}[1]{\xybox{%
  (-12,0)*{};
  (12,0)*{};
  (-9,0)*{}="t1";
  (9,0)*{}="t2";
  "t1";"t2" **[green]\crv{(-9,12) & (9,12)}; ?(1)*[green]\dir{>}
  ?(.5)*\dir{}+(0,2)*{\scriptstyle{#1}};
}}

 \newcommand{\sgmt}[1]{\xybox{%
(-2,0)*{};
  (2,0)*{};
  (0,0)*{}; (0,-8)*{} **[#1]\dir{-};
}}
 \newcommand{\lblsgmt}[2]{\xybox{%
(-2,0)*{};
  (2,0)*{};
  (0,0)*{}; (0,-8)*{} **[#1]\dir{-};
  (0,-10)*{\scs #2};
}}
 
\newcommand{\topsgmt}[1]{\xybox{%
  (-2,0)*{};
  (2,0)*{};
  (0,0)*{}; (0,-8)*{} **[#1]\dir{-}; ?(0)*[#1]\dir{<};
}}

 \newcommand{\cross}[2]{ \xybox{
    (-3,-4)*{};(3,4)*{} **[#1]\crv{(-3,-1) & (3,1)};
    (3,-4)*{};(-3,4)*{} **[#2]\crv{(3,-1) & (-3,1)};
     (-8,0)*{};(8,0)*{};
     }}
 \newcommand{\lblcross}[4]{ \xybox{
    (-3,-4)*{};(3,4)*{} **[#1]\crv{(-3,-1) & (3,1)};
    (3,-4)*{};(-3,4)*{} **[#2]\crv{(3,-1) & (-3,1)};
    (-3,-6)*{\scs #3};
    (3,-6)*{\scs #4};
     (-8,0)*{};(8,0)*{};
     }}

              \newcommand{\topcross}[2]{\xybox{
 (-3,-4)*{};(3,4)*{} **[#1]\crv{(-3,-1) & (3,1)}?(1)*[#1]\dir{>} ;
    (3,-4)*{};(-3,4)*{} **[#2]\crv{(3,-1) & (-3,1)}?(1)*[#2]\dir{>};
     (-8,0)*{};(8,0)*{};
     }}

      \newcommand{\Rthreel}[6]{\xybox{
 (0,-1)*{\lblcross{#1}{#2}{#4}{#5}}; (9,-1)*{\lblsgmt{#3}{#6}};
 (-3,8)*{\sgmt{#2}}; (6,8)*{\cross{#1}{#3}};
 (0,16)*{\topcross{#2}{#3}}; (9,16)*{\topsgmt{#1}};
 }}
      \newcommand{\Rthreer}[6]{\xybox{
 (0,-1)*{\lblcross{#2}{#3}{#5}{#6}}; (-9,-1)*{\lblsgmt{#1}{#4}};
 (3,8)*{\sgmt{#2}}; (-6,8)*{\cross{#1}{#3}};
 (0,16)*{\topcross{#1}{#2}}; (-9,16)*{\topsgmt{#3}};
 }}

\allowdisplaybreaks[1]

\nc\Udot{\dot{\mathcal{U}}}
\nc\Sym{\operatorname{Sym}}

\nc\tP{\tilde{P}}
\nc\tl{\tilde{\lambda}}
\nc\tm{\tilde{\mu}}

\makeatletter
\newtheorem*{rep@theorem}{\rep@title}
\newcommand{\newreptheorem}[2]{%
\newenvironment{rep#1}[1]{%
 \def\rep@title{#2 \ref{##1}}%
 \begin{rep@theorem}}%
 {\end{rep@theorem}}}
\makeatother

\newreptheorem{theorem}{Theorem}
\newtheorem{lemma}{Lemma}
\newreptheorem{lemma}{Lemma}
\newreptheorem{cor}{Corollary}

\begin{document}
%


\title[Categorification of the internal braid group action]{Categorification of the internal braid group action for quantum groups I: 2-functoriality}

\author{Michael T.\  Abram}
\thanks{M.T.A. was supported by a USC Graduate School Dissertation Completion Fellowship.}

\author{Laffite Lamberto-Egan}

\author{Aaron D. Lauda}
\address{Department of Mathematics and Department of Physics, University of Southern California, Los Angeles}
\email{lauda@usc.edu}
\thanks{ A.D.L. was partially supported by the NSF grants DMS-1255334, DMS-1664240, DMS-1902092 and Army Research Office W911NF-20-1-0075.}

\author{David E.V. Rose}
\address{Department of Mathematics, University of North Carolina at Chapel Hill}
\email{davidrose@unc.edu}
\thanks{D.E.V.R. was partially supported by an NSA Young Investigator Grant 
and a Simons Collaboration Grant.}

\subjclass[2010]{Primary }


\date{\today}

\begin{abstract}
We define 2-functors on the categorified quantum group of a  simply-laced  Kac-Moody algebra that induce Lusztig's internal braid group action at the level of the Grothendieck group.
\end{abstract}

\maketitle

 \setcounter{tocdepth}{3}


\allowdisplaybreaks

%
\section{Introduction}
%

Geometric representation theory has motivated the study of categorical representation theory.
Rather than studying the action of Lie algebras $\mf{g}$, or quantum groups $\mathbf{U}_q(\mf{g})$,
on $\C(q)$-vector spaces $V$ with weight decompositions $V = \oplus_{\lambda}V_{\lambda}$,
categorical representation theory studies the action of these algebras on graded additive categories $\mathcal{V}$ with decomposition into graded additive subcategories $\mathcal{V}= \oplus_{\lambda} \mathcal{V}_{\lambda}$.
Rather than linear maps between spaces, Chevalley generators act by functors  $\mathcal{E}_i\onel \maps \mathcal{V}_{\lambda} \to \mathcal{V}_{\lambda+\alpha_i}$,
$\mathcal{F}_i\onel \maps \mathcal{V}_{\lambda} \to \mathcal{V}_{\lambda-\alpha_i}$
satisfying quantum group relations up to natural isomorphism of functors.
The novel and distinguishing feature of higher representation theory is that the natural transformations between such functors contain a wealth of information that is inaccessible within the realm of traditional representation theory.

Indeed, the essence of \emph{categorification} is to uncover this higher level structure and use it to further our understanding of traditional representation theory, as well as related fields.
In this article we will focus our attention on the categorical representation theory of the quantum group $\mathbf{U}_q(\mf{g})$ associated to a \emph{simply-laced} Kac-Moody algebra $\mf{g}$.
Categorified quantum groups are the objects that govern the higher structure and explicitly describe the natural transformations that arise in categorical representations.
More precisely, we focus on the higher representation theory of Lusztig's idempotent form $\U:=\dot{\mathbf{U}}_q(\mf{g})$.
This is a version of the quantum group that arises in geometric representation theory and is most appropriate for studying representations with integral weight decompositions.
For the precise definition of $\U$, see Section~\ref{sec:Udot}.

In most instances when $\U$ admits a categorical action as described above, the natural transformations between functors arise via the action of a categorified quantum group.
The latter is a graded, additive, linear $2$-category $\UcatD_Q$ associated to $\mf{g}$.
The objects in $\UcatD_Q$ are elements of the weight lattice $\lambda \in X$ of $\mf{g}$,
and the 1-morphisms are generated by Chevalley generators $\mathcal{E}_i\onel \maps \lambda \to \lambda+\alpha_i$, $\mathcal{F}_i\onel \maps \lambda \to \lambda-\alpha_i$ and identity 1-morphisms $\onel \maps \lambda \to \lambda$,
\ie any 1-morphism is given by a finite direct sum of grading shifts of composites of these generators.
The 2-morphisms specify maps between composites of Chevalley generators.  For example, there are 2-morphisms
\[
 \mathsf{X}_i \maps \cal{E}_i\onel \to \cal{E}_i \onel \la2\ra, \quad \text{and} \quad
 \mathsf{T}_{ij} \maps \cal{E}_i\cal{E}_j \onel \to \cal{E}_j\cal{E}_i \onel \la -\alpha_i\cdot \alpha_j\ra
\]
where here, and for the duration, $\cdot$ denotes the symmetric bilinear form specified by the Catan datum for $\frak{g}$ (see Section \ref{sec:Udot}).
A novel feature of the categorified quantum group is its diagrammatic generators-and-relations description in which all 2-morphisms are conveniently encoded in a 2-dimensional graphical calculus,
\eg the generating 2-morphisms above have the following depiction:
\[
\mathsf{X}_i \;\;:= \;\;\;\;
\vcenter{ \xy 0;/r.18pc/:
  (0,0)*{\sdotur{i}};
 (6,3)*{ \lambda};
 (-9,3)*{ \lambda +\alpha_j};
 (-10,0)*{};(10,0)*{};
 \endxy }
\qquad \quad
\mathsf{T}_{ij}\;\; := \;\; \;\;
   \vcenter{\xy (0,0)*{\ucrossrb{i}{j}};
 (9,3)*{ \lambda};
 (-13,3)*{ \lambda+\alpha_i+\alpha_j};\endxy }
\]
Key features are that $\cal{F}_i$ and $\cal{E}_i$ are biadjoint,
and endomorphisms of compositions of $\cal{E}_i$'s are given by the so-called \emph{KLR algebras} developed in ~\cite{CR,KL1,KL2,Rou2,RouQH}.
Taken together, the relations on 2-morphisms provide explicit isomorphisms lifting relations in $\U$, and further guarantee that $K_0(\UcatD_Q) \cong \U$,
where, $K_0$ denotes taking the split Grothendieck ring to \emph{decategorify}.
Otherwise, only shadows of this structure are visible at the decategorified level,
\eg Lusztig's canonical basis of $\U$ is recovered by taking the classes in $K_0(\UcatD_Q)$ of indecomposable 1-morphisms in $\UcatD_Q$.

Pioneering work of Chuang and Rouquier demonstrated the importance of the higher structure in categorical representation theory~\cite{CR}.
At the heart of their work is a beautiful categorification of the familiar fact that,
in any integrable representation $V=\oplus_{\lambda} V_\lambda$ of $\mf{sl}_2$, the Weyl group action gives rise to an isomorphism
\[
 \mathsf{t}1_{\l}  \maps V_{\lambda }  \xrightarrow{\cong} V_{-\lambda}
\]
between opposite weight spaces.
In the quantum setting, the Weyl group for $\mf{sl}_2$ (\ie the symmetric group $\frak{S}_2$) deforms to the two strand braid group $B_2$,
and the isomorphism $\mathsf{t}1_{\l}$ can be written in a completion of $\U(\mf{sl}_2)$ as the infinite sum
\begin{equation} \label{eq:intro-tau}
 \mathsf{t}1_{\l} =
\left\{
  \begin{array}{ll}
    \sum_{b\geq 0} (-q)^{b}F^{(\l+b)}E^{(b)} 1_{\l}, & \hbox{if $\l \geq 0$}, \\
    \sum_{a\geq 0} (-q)^{-\l+a}E^{(-\l+a)} F^{(a)}1_{\l}, & \hbox{if $\l \leq 0$},
  \end{array}
\right.
\end{equation}
where $E^{(a)}=E^a/[a]!$, $F^{(a)}=F^a/[a]!$ are the so-called \emph{divided powers},
and $[a]!=\prod_{m=1}^a\frac{q^m-q^{-m}}{q-q^{-1}}$ are quantum factorials.
Note that, when acting on an integrable module, only finitely many terms in this infinite sum are non-zero.
From the perspective of categorification, the crucial observation about equation \eqref{eq:intro-tau} is the occurrence of minus signs.

For those initiated in the categorification doctrine, the occurrence of minus signs immediately necessitates the departure from strictly additive categorification.
That is, we can no longer work with additive categories $\cal{V}_{\lambda}$, as there is no categorical analogue of subtraction therein.
To accommodate such minus signs within a categorical framework, one typically passes to derived, or more-generally triangulated, categories,
where the translation functor gives a categorical notion of multiplying by $-1$.
One manner for doing so is to take the categories of chain complexes $\Kom(\cal{V}_{\l})$ of the weight categories $\cal{V}_\l$ in an additive categorification,
and pass to their homotopy categories of complexes $\Com(\cal{V}_{\l})$.
See Section \ref{sec_KomU} for more details on homotopy categories of additive categories and their Grothendieck groups;
we note that we follow \cite{BKL-Casimir} in using the non-standard notation $\Com$ to denote the homotopy category,
so as not to confuse with our notation $K_0$ for taking the Grothendieck group/ring.
Under decategorification, the classes of such complexes are equal to the alternating sum of the classes of their terms in $K_0(\cal{V}_\l)$.

The alternating sum in \eqref{eq:intro-tau} suggests that a categorification of $\mathsf{t} 1_{\l}$ might be achieved using a chain complex whose differential is built from the 2-morphisms in $\UcatD_{Q}(\mf{sl}_2)$.
Indeed, Chuang and Rouquier's work determines chain complexes $\tau \onel$ and $\tau^{-1} \onel$, the so-called \emph{Rickard complexes},
that lift $\mathsf{t}1_{\l}$ and its inverse $\mathsf{t}^{-1}1_{\l}$ to the categorical setting~\cite{CR}.
The composition of complexes $\tau\tau^{-1} \onel$ and $\tau^{-1}\tau\onel$ are both isomorphic to the identity in $\Com(\UcatD_{Q}(\mf{sl}_2))$,
\ie the complexes are homotopy equivalent to (but, in fact, not equal to) $\onel$ in $\Kom(\UcatD_{Q}(\mf{sl}_2))$.
Using this, Chuang and Rouquier lifted the Weyl group action of $\mf{sl}_2$ to define equivalences
\[
\tau\onel \maps \Com(\cal{V}_{\l}) \xrightarrow{\cong} \Com(\cal{V}_{-\l})
\]
lifting $\mathsf{t}1_{\l}$
(to be precise, Chuang-Rouquier originally worked in the non-quantum and abelian/derived setting,
with the extension to the quantum and triangulated setting given in work of Rouquier~\cite{RouQH} and Cautis-Kamnitzer~\cite{CautisKam}).

For general $\mf{g}$,
the corresponding Weyl group action on integrable representations deforms to an action of the type-$\frak{g}$ braid group $B_{\mf{g}}$ in the quantum setting;
we will follow the standard terminology in referring to this as the \emph{quantum Weyl group} action.
Analogous to the $\frak{g}=\frak{sl}_2$ case, this action lifts to highly non-trivial braid group actions in categorical representation theory~\cite{CautisKam,RouQH}.
To illustrate their far reaching impact in mathematics, we recall just a handful of their many applications.
\begin{itemize}
\item Chuang and Rouquier use the equivalence induced by categorical $\mf{sl}_2$ actions on derived categories
of modules over the symmetric group in positive characteristic
to resolve Brou\'{e}'s Abelian deffect group conjecture for the symmetric group $\frak{S}_n$~\cite{CR}.

\item Cautis, Kamnitzer, and Licata use categorical $\mf{sl}_2$ actions to resolve a conjecture of Namikawa~\cite{Namikawa}
asserting the existence a of derived equivalence between
cotangent bundles of complementary Grassmannians $T^{\ast}G(k,N)$ and $T^{\ast}G(N-k,N)$~\cite{CKL,Cautis-equiv}.
These varieties are related by a stratified Mukai flop,
and the problem of constructing such equivalences had previously only been resolved in the $k=1$ case \cite{Kawamata0,Namikawa0}
and for $G(2,4)$ in work of Kawamata \cite{Kawamata}.
More generally, Cautis, Kamnitzer, and Licata construct categorical braid group actions on cotangent bundles to partial flag varieties
and Nakajima quiver varieties \cite{CautisKam,CKLquiver,Cautis-rigid}

\item Categorical representations of $\mf{sl}_m$, and the associated braid group actions,
can be used to categorify the $\mf{sl}_n$ Reshetikhin-Turaev quantum link invariants
via a categorical analogue of the skew Howe duality between $\mf{gl}_m$ and $\mf{gl}_n$~\cite{CKL,LQR,QR,Cautis}.
This perspective has led to the solution of a number of conjectures in link homology~\cite{RoseW,RoseQ},
and provides a framework for connecting link homologies deffined using wildly different machinery~\cite{Cautis,LQR,MacWeb}.
\end{itemize}

At the decategorified level, the braid group action on integrable modules of $\U_q(\mf{g})$ comes in several flavors
\begin{equation}\label{eq:QWG}
\begin{aligned}
\mathsf{t}_{i,e}'1_{\l} &=
 \sum_{a,b; a-b=\lambda_i} (-q)^{eb}F_i^{(a)}E_i^{(b)} 1_{\l}
  =\sum_{a,b; a-b=\lambda_i} (-q)^{eb}E_i^{(b)} F_i^{(a)}1_{\l}, \\
\mathsf{t}_{i,e}''1_{\l} &=
  \sum_{a,b; -a+b=\lambda_i} (-q)^{eb}E_i^{(a)}F_i^{(b)} 1_{\l}
  =  \sum_{a,b; -a+b=\lambda_i} (-q)^{eb}F_i^{(b)} E_i^{(a)} 1_{\l},
\end{aligned}
\end{equation}
where $e=\pm 1$, see Section \ref{sec:braid-int-modules} for more details.
Given the importance of these braid group actions, it is natural to ask how the braid group action
$B_{\mf{g}}$ on an integrable module interacts with the $\U_q(\mf{g})$ action.
This was answered by Lusztig~\cite[Proposition 37.1.2]{Lus4},
who showed that, for each node in the Dynkin diagram $i\in I$ and $e=\pm1$,
there exist algebra automorphisms $T'_{i,e}$ and $T''_{i,e}$ of $\U=\U_q(\mf{g})$ uniquely defined by the condition that, for any integrable $\U$-module
$V$, any $z \in V$,
and $u \in 1_{\nu} \U 1_{\l}$,
the following equations hold
\begin{equation}\label{eq:QWGcompat}
\begin{split}
    T_{i,e}' (u)\mathsf{t}_{i,e}'1_{\l}(z) &=   \mathsf{t}_{i,e}'1_{\nu}(uz), \\
    T_{i,e}''(u) \mathsf{t}_{i,e}''1_{\l}(z) &=   \mathsf{t}_{i,e}''1_{\nu}(uz) .
\end{split}
\end{equation}
Related operators were studied in finite type in \cite{Soil1,Soil2,Soil3},
then generalized to simply-laced Cartan data in \cite{Lus-T} and general Cartan data in \cite{Lus-T2}.
See Section~\ref{sec:braid-quant-group} for more details.

The algebra automorphisms $T_{i,e}'$ and $T_{i,e}''$ each define braid group actions on the algebra $\U$ itself that we call the internal braid group action.
This internal action plays an important role, \eg in the construction of the PBW basis for $\U$.
Lusztig goes on to give precise formulas for the action of $T_{i,e}'$ and $T_{i,e}''$ on the generators of $\U$,
that unsurprisingly involve minus signs,
\eg
\begin{equation} \label{eq:intro2}
 T'_{i,+1}(E_j 1_{\lambda})=
E_j E_i 1_{s_i(\lambda)}-q E_i E_j 1_{s_i(\lambda)}   \quad \text{if $i\cdot j=-1$.}
\end{equation}
where here $s_i$ are the simple reflections in the Weyl group.

We now describe the results contained in this article.  Throughout we let $\mf{g}$ be a simply-laced Kac-Moody algebra.


\subsection{Categorifying \texorpdfstring{$T_{i,e}'$}{T'i,e} and \texorpdfstring{$T_{i,e}''$}{T''i,e}}


In Section~\ref{sec:defTi} we define graded, additive 2-functors $\cal{T}_{i,e}', \cal{T}_{i,e}'' \maps  \UcatD_Q \to \Com(\UcatD_Q)$.
To do so, we first assign explicit chain complexes to generating 1-morphisms in $\UcatD_Q$ that lift the formulae defining $T_{i,e}'$ and $T_{i,e}''$,
\eg equation \eqref{eq:intro2} lifts to the assignment
\[
\cal{T}'_{i,+1}(\cal{E}_j \onell{\lambda})
=
\clubsuit \;\cal{E}_{j}\cal{E}_i \onell{s_i(\lambda)}
\xrightarrow{\ucrossbr{j}{i}}
\cal{E}_i\cal{E}_{j} \onell{s_i(\lambda)} \la 1 \ra
\quad \text{if $i \cdot j =-1$,}
\]
where (here, and throughout) $\clubsuit$ denotes the term in homological degree zero.
Functoriality then requires that the composite $xy\onel$ of composable $1$-morphisms $y\mathbf{1}_{\lambda'}$ and $x\onel$ is sent to the
composition of chain complexes $\cal{T}'_{i,+1}(y\onelp)\cal{T}'_{i,+1}(x\onel)$,
defined using composition of $1$-morphisms in $\UcatD_Q$ in a manner similar to taking tensor product of chain complexes.
To complete the definition of $\cal{T}_{i,e}'$ and $\cal{T}_{i,e}''$,
we then assign an explicit chain map $\cal{T}'_{i,+1}(\alpha) \maps \cal{T}'_{i,+1}(x\onel) \to \cal{T}'_{i,+1}(x'\onel)$
to each generating 2-morphism $\alpha \maps x\onel \to x'\onel$ in $\UcatD$,
\eg the 2-morphism $X_j \maps \cal{E}_j\onel \to \cal{E}_j \onel \la2\ra$ is sent by $T_{i,+1}'$ to
\begin{align*}
   \cal{T}'_i \left(    \xy 0;/r.18pc/:
  (0,0)*{\sdotu{j}};
 (6,3)*{ \lambda};
 (-9,3)*{ \lambda +\alpha_j};
 (-10,0)*{};(10,0)*{};
 \endxy \right)
 :=
  \vcenter{\xy 0;/r.18pc/:
  (-20,15)*+{\clubsuit \;\cal{E}_j\cal{E}_i \onell{s_i(\lambda)}\la 2\ra}="1";
  (-20,-15)*+{\clubsuit \;\cal{E}_j\cal{E}_i \onell{s_i(\lambda)}}="2";
   {\ar^{\xy (-12,0)*{\sdotu{j}}; (-6,0)*{\slineur{i}};  \endxy} "2";"1"};
  (25,15)*+{\cal{E}_i\cal{E}_j \onell{s_i(\lambda)}\la 3\ra }="3";
  (25,-15)*+{\cal{E}_i\cal{E}_j \onell{s_i(\lambda)}\la 1\ra}="4";
    {\ar_{\xy (-12,0)*{\slineur{i}}; (-6,0)*{\sdotu{j}}; \endxy} "4";"3"};
   {\ar^{\xy (0,0)*{\ucrossbr{j}{i}};\endxy   } "1";"3"};
   {\ar_{\xy (0,0)*{\ucrossbr{j}{i}};\endxy   } "2";"4"};
 \endxy}
\end{align*}
which is a chain map by the $i\neq j$ dot sliding relation, see (5) in Definition~\ref{defU_cat-cyc}.
Finally, we show that the images of relations in $\UcatD_Q$ are satisfied in $\Com(\UcatD_Q)$, up to homotopy.

Proving that $\cal{T}'_{i,+1}$ is a well-defined 2-functor requires an immense number of verifications.
The diagrammatic relations defining $\UcatD_Q$ involve strands colored by the Dynkin nodes of $\mf{g}$,
and depend on the adjacency of the colors involved.
For example, the relation involving the greatest number of strands is:
\begin{equation*}
\vcenter{\xy 0;/r.17pc/:
    (-4,-4)*{};(4,4)*{} **\crv{(-4,-1) & (4,1)}?(1)*\dir{};
    (4,-4)*{};(-4,4)*{} **\crv{(4,-1) & (-4,1)}?(1)*\dir{};
    (4,4)*{};(12,12)*{} **\crv{(4,7) & (12,9)}?(1)*\dir{};
    (12,4)*{};(4,12)*{} **\crv{(12,7) & (4,9)}?(1)*\dir{};
    (-4,12)*{};(4,20)*{} **\crv{(-4,15) & (4,17)}?(1)*\dir{};
    (4,12)*{};(-4,20)*{} **\crv{(4,15) & (-4,17)}?(1)*\dir{};
    (-4,4)*{}; (-4,12) **\dir{-};
    (12,-4)*{}; (12,4) **\dir{-};
    (12,12)*{}; (12,20) **\dir{-};
    (4,20); (4,21) **\dir{-}?(1)*\dir{>};
    (-4,20); (-4,21) **\dir{-}?(1)*\dir{>};
    (12,20); (12,21) **\dir{-}?(1)*\dir{>};
   (18,8)*{\lambda};  (-6,-3)*{\scs \ell};
  (6,-3)*{\scs j};
  (15,-3)*{\scs k};
\endxy}
 \;\; - \;\;
\vcenter{\xy 0;/r.17pc/:
    (4,-4)*{};(-4,4)*{} **\crv{(4,-1) & (-4,1)}?(1)*\dir{};
    (-4,-4)*{};(4,4)*{} **\crv{(-4,-1) & (4,1)}?(1)*\dir{};
    (-4,4)*{};(-12,12)*{} **\crv{(-4,7) & (-12,9)}?(1)*\dir{};
    (-12,4)*{};(-4,12)*{} **\crv{(-12,7) & (-4,9)}?(1)*\dir{};
    (4,12)*{};(-4,20)*{} **\crv{(4,15) & (-4,17)}?(1)*\dir{};
    (-4,12)*{};(4,20)*{} **\crv{(-4,15) & (4,17)}?(1)*\dir{};
    (4,4)*{}; (4,12) **\dir{-};
    (-12,-4)*{}; (-12,4) **\dir{-};
    (-12,12)*{}; (-12,20) **\dir{-};
    (4,20); (4,21) **\dir{-}?(1)*\dir{>};
    (-4,20); (-4,21) **\dir{-}?(1)*\dir{>};
    (-12,20); (-12,21) **\dir{-}?(1)*\dir{>};
  (10,8)*{\lambda};
  (-14,-3)*{\scs \ell};
  (-6,-3)*{\scs j};
  (6,-3)*{\scs k};
\endxy}
=
\left\{\\
 \begin{array}{ccc}
 t_{\ell j} \;\;
\xy 0;/r.17pc/:
  (4,12);(4,-12) **\dir{-}?(.5)*\dir{<};
  (-4,12);(-4,-12) **\dir{-}?(.5)*\dir{<} ;
  (12,12);(12,-12) **\dir{-}?(.5)*\dir{<} ;
  (-6,-9)*{\scs \ell};     (6.1,-9)*{\scs j};
  (14,-9)*{\scs \ell};
 \endxy & & \text{if $\ell=k$ and $\ell \cdot j=-1$}\\ \\
 0 & & \text{if $\ell \neq k$ or $\ell \cdot j\neq -1$}\\
\end{array}\right.
\end{equation*}
where $t_{\ell j}$ is a scalar defined in Section~\ref{sec:choice-scalars}.
Showing that $\cal{T}_{i,+1}'$ preserves this relation for all $i$ and all triples $j,k,\ell$ requires considering all possible types of adjacency relations between the nodes corresponding to
$i,j,k,\ell$, requiring $27$ essentially distinct case that need to be verified.
The complexity is further exacerbated by the fact that $\cal{T}_{i,+1}'$ often only preserves a relation up to homotopy.

Unfortunately, we are not aware of a means to define the 2-functors lifting Lusztig's formulae without explicitly constructing the chain homotopies for each relation and each possible coloring by nodes $i \in I$.
We have made every attempt to provide sufficient detail in this work to aid in any future applications of these 2-functors,
and in particular provide sufficient detail so that the relevant homotopies can be easily extracted.

Our main result in this article is the following theorem.
\begin{thm} \label{thm:A}
Let $\mf{g}$ be a simply-laced Kac-Moody algebra, then there is an explicitly defined 2-functor
\[
\cal{T}_{i,+1}'   \maps  \UcatD_Q(\mf{g})  \to \Com(\UcatD_Q(\mf{g}))
\]
so that the induced map $[\cal{T}_{i,+1}'] : \U_q(\mf{g}) \cong K_0(\UcatD_Q(\mf{g})) \to K_0(\Com(\UcatD_Q(\mf{g}))) \cong \U_q(\mf{g})$
agrees with $T_{i,+1}'$.
\end{thm}

At the level of 1-morphisms, such functors have already appeared at the categorical level in \cite{Cautis,CautisKam} and were given a geometric interpretation in
\cite{
Kato2,
Kato,
XZ,
Zhao
};
however, to our knowledge, no information about extending these maps to 2-morphisms has appeared previously.
As such, Theorem \ref{thm:A} initiates the study of Lusztig's operators at the 2-categorical level.
In fact, we conjecture much more.
At the decategorified level, Lusztig's operators are invertible and satisfy the braid relations.
These properties, combined with our forthcoming work, stated in Theorem \ref{thm:compat} below,
suggest the following:
\begin{conj}
Let $\mf{g}$ be a (simply-laced) Kac-Moody algebra, then $\cal{T}_{i,+1}'$ extends to an autoequivalence of $\Com(\UcatD_Q(\mf{g}))$
so that the induced automorphism $[\cal{T}_{i,+1}']$ of $\U_q(\mf{g}) \cong K_0(\Com(\UcatD_Q(\mf{g})))$ agrees with $T_{i,+1}'$.
Moreover, the $\cal{T}_{i,+1}'$ satisfy the braid relations.
\end{conj}
The extension (of domain) to the homotopy category is a problem in obstruction theory that
we plan to attack in future work. Having done so, the proof of braid relations will be a straightforward (but tedious) check.

%
\subsection{Symmetries and the internal braid group action}
%
There are a number of other (anti)linear (anti)automorphisms $\und{\sigma},\und{\omega},\und{\psi}$ defined on $\U$, see section~\ref{sec_Usymm} for their definitions.
These (anti)involutions allow one to pass between the variants $T_{i,e}'$ and $T_{i,e}''$ of the internal braid group generators via conjugation,
\ie
\begin{align} \label{intro-T'_rel_T''}
\und{\sigma} T'_{i,e}\und{\sigma} &=T''_{i,-e}       &
\und{\omega} T'_{i,e}\und{\omega} =&T''_{i,e} \\
 \und{\psi} T'_{i,e}\und{\psi}        &=T'_{i,-e}   &
 \und{\psi} T''_{i,e} \und{\psi}      =&T''_{i,-e}.
\end{align}
In \cite{KL3} these symmetries were lifted to define 2-functors $\sigma,\omega,\psi$ on a certain version of the categorified quantum group.
Each has a natural interpretation in terms of symmetries of the graphical calculus for $\UcatD_{Q}$,
and, in the $\mf{sl}_2$ case, were extended to the homotopy category of complexes in \cite{BKL-Casimir}.

Recall (or see Section~\ref{sec:choice-scalars} below) that the definition of $\UcatD_Q$ requires a choice of scalar parameters $Q$;
it was recently shown that there is a natural normalization for the categorified quantum group associated to an arbitrary KLR algebra and choice of $Q$~\cite{BHLW2}.
This so-called \emph{cyclic version} of $\UcatD_Q$ satisfies the property that diagrams that are planar isotopic relative to their boundaries specify the same 2-morphism in $\UcatD_Q$,
a property that only holds up to scalars in previous formulations.
Given the utility of the cyclic version, we also prove the following result, which defines these symmetries in this setting.

\begin{thm} \label{thm:B}
There are  invertible 2-functors
$\sigma,\omega,\psi$ defined on the cyclic version of the categorified quantum group $\UcatD_Q$
that categorify the symmetries $\und{\sigma},\und{\omega},\und{\psi}$,
\ie
\[
[\sigma]=\und{\sigma}, \quad [\omega]=\und{\omega}, \quad
[\psi] = \und{\psi}
\]
in $K_0(\UcatD_Q(\mf{g})) \cong \U_q(\mf{g})$.
\end{thm}

Defining these 2-functors requires several subtle aspects involving the choice of scalars $Q$, so we include the details below in Section \ref{sec_symm}.
Using these symmetries, we use the categorical analogue of \eqref{intro-T'_rel_T''} to define the variants $\cal{T}_{i,-1}'$ and $\cal{T}_{i,e}''$ of the internal braid group action.

%
\subsection{Compatibility with Rickard complexes}
%

As noted above,
the defining feature of the internal braid group action at the decategorified level is its compatibility with the quantum Weyl group action,
given in equation \eqref{eq:QWGcompat}.
In a sequel to this paper \cite{ALLR2}, we show that our 2-functors $\cal{T}_{i,+1}'$ satisfy an analogous compatibility with the Rickard complexes.

To be precise, note that the first equality in equation \eqref{eq:QWGcompat} asserts that the actions of the elements
$T_{i,e}' (u)\mathsf{t}_{i,e}'1_{\l}$ and $\mathsf{t}_{i,e}'1_{\nu} u$ on the $\lambda$ weight space of any integrable representation agree
for all $u \in 1_{\nu} \U 1_{\l}$.
Equivalently, for any integrable representation $V=\oplus_{\lambda} V_\lambda$ there is an equality between the corresponding linear maps
$1_{\nu} \U 1_{\l} \to \Hom(V_\lambda,V_{s_i(\nu)})$.
At the categorical level,
the operation of composing with the complex $\tau_{i,+1}'$ defines a functor
\[
\tau_{i,+1}' \onenu (-)\onel \maps \Hom_{\Ucat_Q}(\l,\nu) \to \Hom_{\Com(\Ucat_Q)}(\l,s_i(\nu))
\]
and we can similarly consider the functor $\cal{T}_{i,e}'(-)\tau_{i,e}' \onel$,
which maps between the same $\Hom$-categories.
The main result of \cite{ALLR2} is the following:

\begin{thm} \label{thm:compat}
For all objects $\l, \nu$ in $\Ucat_Q$,
there is an isomorphism of functors
\begin{equation}
	\beth\maps \tau_{i,+1}' \onenu (-)\onel \cong \cal{T}_{i,e}'(-)\tau_{i,e}' \onel
\end{equation}
between $\Hom$-categories $\Hom_{\Ucat_Q}(\l,\nu) \to \Hom_{\Com(\Ucat_Q)}(\l,s_i(\nu))$.
\end{thm}

%
\subsection{Applications of the internal braid group action}
%

\subsubsection{PBW basis and their categorifications}
In finite type, Lusztig's internal braid group action can be used to deduce the quantum PBW theorem for $\U^+(\mf{g})$,
providing a basis of monomials that are useful in many applications.
The KLR algebra provides a categorification of $\U^+(\mf{g})$ via its category of projective/finitely generated modules~\cite{KL1,KL2,RouQH}.
Therein, the indecomposable projective modules correspond to the canonical basis of $\U^+(\mf{g})$~\cite{VV}, while the simple modules corresponds to the dual canonical basis~\cite{BK2,Web4}.
At the categorical level, the analogues of PBW monomials lead to a rich theory of standard modules for KLR algebras.
In finite type, standard modules were first described in \cite{KR2} (see also \cite{BSS,BKlM,HMM,McNam1,McNam4,Kato}), and in affine type they were studied in \cite{Klesh4,KleshMuth,TingWeb,McNam2};
in these studies, the focus has been on finding specific modules over KLR algebras that lift a given PBW monomial.
In forthcoming work \cite{McNam3},
McNamara plans to use our 2-functors $\cal{T}_{i,+1}'$ to build projective resolutions of standard KLR modules,
producing a categorical lift of Lusztig's internal braid group construction of the PBW basis,
and giving a strengthening of Kato's results on reflection functors for KLR algebras~\cite{Kato}.

\subsubsection{Quantum Affine algebras}
There is no obstruction to defining the 2-functors $\cal{T}_{i,+1}'$ in arbitrary symmetrizable type,
except that the check of well-definedness is much-more involved.
For example, Lusztig provides the explicit formula
\begin{equation}\label{eq:TarbType}
T_{i,+1}'(E_{\ell} 1_\l) = \sum_{j=0}^{-i\cdot \ell}(-q)^{j}E_i^{(j)} E_{\ell}E_{i}^{(-i\cdot \ell-j)} 1_{s_i(\l)}
\end{equation}
in arbitrary type (compare to equation \eqref{eq:intro2} above),
which suggests that the categorified Lusztig operator $\cal{T}'_{i,1}$ should send $\cal{E}_{\ell} \onel$ to a complex of length $1-i\cdot\ell$.
It is not difficult to specify a complex lifting equation \eqref{eq:TarbType},
\eg we could set
\[
\cal{T}'_{i,1}(\cal{E}_\ell \onel) :=
\clubsuit \cal{E}_{\ell}\cal{E}_i^{(-i\cdot\ell)} \onell{s_i(\lambda)}
\xrightarrow{\xy 0;/r.15pc/:
 (0,-8);(0,0) **[black][|(3)]\dir{-} ?(.5)*[black][|(2)]\dir{>};
 (0,0);(-4,8)*{} **[black][|(1)]\crv{(-4,1)} ?(1)*[black][|(1)]\dir{>};
 (0,0);(4,8)*{} **[black][|(3)]\crv{(4,1)} ?(1)*[black][|(2)]\dir{>};
 (-4,-8);(0,8)*{} **[magenta][|(1)]\crv{(-4,0)} ?(1)*[magenta][|(1)]\dir{>};
 (-6,0)*{}; (6,0)*{};\endxy}
 \cal{E}_i\cal{E}_\ell\cal{E}_{i}^{(-i\cdot\ell-1)} \onell{s_i(\lambda)} \la 1 \ra
 \xrightarrow{\xy 0;/r.10pc/:
 (0,-8);(0,0) **[black][|(3)]\dir{-} ?(.5)*[black][|(2)]\dir{>};
 (0,0);(-4,8)*{} **[black][|(1)]\crv{(-4,1)} ?(1)*[black][|(1)]\dir{};
 (0,0);(4,8)*{} **[black][|(3)]\crv{(4,1)} ?(1)*[black][|(2)]\dir{>};
 (-8,16);(-8,24) **[black][|(3)]\dir{-} ?(1)*[black][|(2)]\dir{>};
 (-12,8);(-8,16)*{} **[black][|(1)]\crv{(-12,15)} ?(.3)*[black][|(2)]\dir{};
 (-4,8);(-8,16)*{} **[black][|(1)]\crv{(-4,15)} ?(.3)*[black][|(2)]\dir{};
 (-12,8);(-12,-8)*{} **[black][|(1)]\dir{-} ?(.3)*[black][|(2)]\dir{};
 (4,8);(4,24)*{} **[black][|(3)]\dir{-} ?(1)*[black][|(2)]\dir{>};
 (-8,-8);(-2,24)*{} **[magenta]\crv{(-8,16),(-2,0)} ?(1)*[magenta]\dir{>};
 (-16,0)*{}; (8,0)*{};
\endxy}
\cdots
\xrightarrow{\xy 0;/r.15pc/:
 (-8,0);(-8,8) **[black][|(3)]\dir{-} ?(1)*[black][|(2)]\dir{>};
 (-12,-8);(-8,0)*{} **[black][|(3)]\crv{(-12,-1)} ?(.3)*[black][|(2)]\dir{>};
 (-4,-8);(-8,0)*{} **[black][|(1)]\crv{(-4,-1)} ?(.3)*[black][|(2)]\dir{};
 (-8,-8);(-2,8)*{} **[magenta]\crv{(-2,0)} ?(1)*[magenta]\dir{>};
 (-16,-8)*{}; (0,-8)*{};
\endxy}
\cal{E}_i^{(-i\cdot\ell)}\cal{E}_\ell \onell{s_i(\lambda)} \la -i\cdot\ell \ra.
\]
Here, the terms in the differential are given using the \emph{thick calculus} from \cite{KLMS}, and an easy computation therein verifies that they square to zero.
The appearance of complexes containing more than two non-zero terms suggests that even more of the defining relations in $\UcatD_Q$
may be preserved by $\cal{T}'_{i,1}$ only up to homotopy, exacerbating the difficulty of checking that these 2-functors are well-defined.
Despite this, we note one interesting application of an extension of our 2-functors to non-simply-laced type:
it may be possible to promote Beck's description~\cite{Beck} of the loop presentation of affine algebras in terms of the internal braid group action to the categorical level,
giving a categorification of affine algebras in their loop realization.

\subsubsection{Link invariants and skew Howe duality}

As referenced above, one can study the $\mf{sl}_n$ quantum link invariants via $\U(\mf{sl}_m)$ representation theory using quantum skew Howe duality.
The latter is the quantum analogue of the duality arising from the commuting actions of $\U(\mf{sl}_n)$ and $\U(\mf{sl}_m)$ on the quantum exterior power $\wedge^N(\C^m_q \otimes \C^n_q)$.
The $\mf{sl}_n$ link invariants admit a formulation in terms of MOY calculus~\cite{MOY} and $\mf{sl}_n$ webs~\cite{Kuperberg,Kim,Mor},
certain trivalent graphs which specify the morphisms in a diagrammatic description of the category of $\U(\mf{sl}_n)$ representations.

Cautis, Kamnitzer, and Morrison show that skew Howe duality admits a graphical description in terms of so-called ladder webs,
and use this to give an entirely diagrammatic description of the full subcategory of quantum $\mf{sl}_n$ representations tensor generated by the fundamental representations~\cite{CKM}.
In this formulation, skew Howe duality specifies a representation of $\U(\mf{sl}_m)$ in which an $\mf{sl}_m$ weight $\lambda = (\lambda_1, \lambda_2, \dots, \lambda_{m-1})$ is sent to the
to the $m$-tuple $(a_1,a_2, \dots, a_m)$ that satisfies $0 \leq a_i\leq n$, $\sum_{i=1}^m a_i = N$ and $\lambda_i = a_{i}-a_{i+1}$,
and weights not satisfying these conditions are sent to zero.
This representation maps the generators of $\U(\mf{sl}_m)$ as follows:
\[
1_{\l } \mapsto
\hackcenter{
\begin{tikzpicture}[scale=.3]
\node at (-4.125,0) { $\dots$};
	\draw [very thick, ->] (-5.5,-2) to (-5.5,2);
	\draw [very thick, ->] (-3,-2) to (-3,2);
\node at (-5.5,-2.55) {\tiny $a_{1}$};	
\node at (-3,-2.55) {\tiny $a_{m}$};	
\end{tikzpicture}}
\;\; , \;\;
E_i^{(r)}1_{\l} \mapsto
\hackcenter{
\begin{tikzpicture}[scale=.3]
\node at (-4.375,0) { $\dots$};
\node at (4.5,0) { $\dots$};
	\draw [very thick, ->] (-5.5,-2) to (-5.5,2);
	\draw [very thick, ->] (-3.5,-2) to (-3.5,2);
	\draw [very thick, ->] (3.5,-2) to (3.5,2);
	\draw [very thick, ->] (5.5,-2) to (5.5,2);
	\draw [very thick, directed=.55] (1.25,-.5) to (1.25,2);
	\draw [very thick, directed=.55] (1.25,-.5) to (-1.25,.5);
	\draw [very thick, directed=.55] (-1.25,-2) to (-1.25,.5);
	\draw [very thick, directed=.55] (-1.25,.5) to (-1.25,2);
	\draw [very thick, directed=.55] (1.25,-2) to (1.25,-.5);
\node at (-5.5,-2.55) {\tiny $a_{1}$};	
\node at (-3.5,-2.55) {\tiny $a_{i-1}$};	
\node at (-1.25,-2.55) {\tiny $a_i$};	
\node at (1.25,-2.55) {\tiny $a_{i+1}$};
\node at (3.5,-2.55) {\tiny $a_{i+2}$};
\node at (5.5,-2.55) {\tiny $a_{m}$};	
	\node at (-1.5,2.5) {\tiny $a_i{+}r$};
	\node at (1.5,2.5) {\tiny $a_{i+1}{-}r$};
	\node at (0,0.75) {\tiny $r$};
\end{tikzpicture}}
\;\; , \;\;
F_i^{(r)}1_{\l} \mapsto
\hackcenter{
\begin{tikzpicture}[scale=.3]
\node at (-4.375,0) { $\dots$};
\node at (4.5,0) { $\dots$};
	\draw [very thick, ->] (-5.5,-2) to (-5.5,2);
	\draw [very thick, ->] (-3.5,-2) to (-3.5,2);
	\draw [very thick, ->] (3.5,-2) to (3.5,2);
	\draw [very thick, ->] (5.5,-2) to (5.5,2);
	\draw [very thick, directed=.55] (1.25,.5) to (1.25,2);
	\draw [very thick, directed=.55] (-1.25,-.5) to (1.25,.5);
	\draw [very thick, directed=.55] (-1.25,-2) to (-1.25,-.5);
	\draw [very thick, directed=.55] (-1.25,-.5) to (-1.25,2);
	\draw [very thick, directed=.55] (1.25,-2) to (1.25,.5);
\node at (-5.5,-2.55) {\tiny $a_{1}$};	
\node at (-3.5,-2.55) {\tiny $a_{i-1}$};	
\node at (-1.25,-2.55) {\tiny $a_i$};	
\node at (1.25,-2.55) {\tiny $a_{i+1}$};
\node at (3.5,-2.55) {\tiny $a_{i+2}$};
\node at (5.5,-2.55) {\tiny $a_{m}$};	
	\node at (-1.5,2.5) {\tiny $a_i{-}r$};
	\node at (1.5,2.5) {\tiny $a_{i+1}{+}r$};
	\node at (0,0.75) {\tiny $r$};
\end{tikzpicture}}
\]
Under this representation, the braiding on the category of $\U(\mf{sl}_n)$ is given by the quantum Weyl group action,
\ie diagrammatically, we have:
\begin{equation}\label{eq:qWgCrossing}
\mathsf{t}_i 1_{\l}  \mapsto
\hackcenter{
\begin{tikzpicture}[scale=.75]
\draw[very thick, directed=1] (-2.4,0) -- (-2.4,1.5);
\draw[very thick, directed=1] (-1.2,0) -- (-1.2,1.5);
\draw[very thick, directed=1] (-.4,0) -- (.4,1.5);
\draw[very thick ] (.4,0) -- (.1,.6);
\draw[very thick, directed=1 ] (-.1,.9) -- (-.4,1.5);
\draw[very thick, directed=1] (1.2,0) -- (1.2,1.5);
\draw[very thick, directed=1] (2.4,0) -- (2.4,1.5);
\node at (-2.4,-.3) {\tiny $a_{1}$};
\node at (-1.2,-.3) {\tiny $a_{i-1}$};
\node at (-.4,-.3) {\tiny $a_i$};
\node at (.4,-.3) {\tiny $a_{i+1}$};
\node at (2.4,-.3) {\tiny $a_{m}$};
\node at (1.25,-.3) {\tiny $a_{i+2}$};
\node at (-1.8,.75) {$\dots$};
\node at ( 1.8,.75) {$\dots$};
\end{tikzpicture}}
\qquad \quad
\mathsf{t}_i^{-1} 1_{\l}  \mapsto
\hackcenter{
\begin{tikzpicture}[scale=.75]
\draw[very thick, directed=1] (-2.4,0) -- (-2.4,1.5);
\draw[very thick, directed=1] (-1.2,0) -- (-1.2,1.5);
\draw[very thick, directed=1] (.4,0) -- (-.4,1.5);
\draw[very thick ] (-.4,0) -- (-.1,.6);
\draw[very thick, directed=1 ] (.1,.9) -- (.4,1.5);
\draw[very thick, directed=1] (1.2,0) -- (1.2,1.5);
\draw[very thick, directed=1] (2.4,0) -- (2.4,1.5);
\node at (-2.4,-.3) {\tiny $a_{1}$};
\node at (-1.2,-.3) {\tiny $a_{i-1}$};
\node at (-.4,-.3) {\tiny $a_i$};
\node at (.4,-.3) {\tiny $a_{i+1}$};
\node at (2.4,-.3) {\tiny $a_{m}$};
\node at (1.25,-.3) {\tiny $a_{i+2}$};
\node at (-1.8,.75) {$\dots$};
\node at ( 1.8,.75) {$\dots$};
\end{tikzpicture}}
\end{equation}
In this way, these link invariants can be computed and studied via the elements in $\U(\mf{sl}_m)$ corresponding to a given link diagram.

Under this correspondence, the internal braid group action plays an interesting role in the diagrammatic description of quantum $\mf{sl}_n$ link invariants,
as equation~\eqref{eq:QWGcompat} shows how to slide the image of an arbitrary element $u \in \U(\mf{sl}_m)$ through a crossing,
\ie it gives the equality:
\[
\hackcenter{
\begin{tikzpicture} [scale=.75]
\draw[very thick, directed=1] (-2.4,0) -- (-2.4,2.5);
\draw[very thick, directed=1] (-1.2,0) -- (-1.2,2.5);
\draw[very thick, directed=1] (-.4,1) -- (.4,2.5);
\draw[very thick ]            (.4,1) -- (.1,1.6);
\draw[very thick, directed=1 ] (-.1,1.9) -- (-.4,2.5);
\draw[very thick, directed=1] (1.2,0) -- (1.2,2.5);
\draw[very thick, directed=1] (2.4,0) -- (2.4,2.5);
\draw[very thick, directed=1] (.4,0) -- (.4,1);
\draw[very thick, directed=1] (-.4,0) -- (-.4,1);
\node at (-2.4,-.3) {\tiny $a_{1}$};
\node at (-1.2,-.3) {\tiny $a_{i-1}$};
\node at (-.4,-.3) {\tiny $a_i$};
\node at (.4,-.3) {\tiny $a_{i+1}$};
\node at (2.4,-.3) {\tiny $a_{m}$};
\node at (1.25,-.3) {\tiny $a_{i+2}$};
\draw [fill=white] (-2.75,0.3) rectangle (2.75,1.2);
\node at ( 0,.7) {$u$};
\end{tikzpicture}}
\;\; =\;\;
\hackcenter{
\begin{tikzpicture}[scale=.75]
\draw[very thick, directed=1] (-2.4,0) -- (-2.4,2.5);
\draw[very thick, directed=1] (-1.2,0) -- (-1.2,2.5);
\draw[very thick, directed=1] (-.4,0) -- (.4,1.5);
\draw[very thick ] (.4,0) -- (.1,.6);
\draw[very thick, directed=1 ] (-.1,.9) -- (-.4,1.5);
\draw[very thick, directed=1] (1.2,0) -- (1.2,2.5);
\draw[very thick, directed=1] (2.4,0) -- (2.4,2.5);
\draw[very thick, directed=1] (.4,2) -- (.4,2.5);
\draw[very thick, directed=1] (-.4,2) -- (-.4,2.5);
\node at (-2.4,-.3) {$\scs a_{1}$};
\node at (-1.2,-.3) {$\scs a_{i-1}$};
\node at (-.4,-.3) {$\scs a_i$};
\node at (.4,-.3) {$\scs a_{i+1}$};
\node at (2.4,-.3) {$\scs a_{m}$};
\node at (1.25,-.3) {$\scs a_{i+2}$};
\node at (-1.8,.75) {$\dots$};
\node at ( 1.8,.75) {$\dots$};
\draw [fill=white] (-2.75,1.2) rectangle (2.75,2.1);
\node at ( 0,1.65) {$T_{i,+1}'(u)$};
\end{tikzpicture}}
\]
where he we abuse notation in denoting elements in $\U(\mf{sl}_m)$ and their images under the skew Howe representation via the same symbols.

This entire story lifts to the categorical level, allowing for the study of Khovanov~\cite{LQR} and Khovanov-Rozansky homology~\cite{Cautis,QR}
following Cautis,Kamnitzer, and Licata's pioneering work in using categorical skew Howe duality to study algebro-geometric categorifications of the $\mf{sl}_n$ link polynomials~\cite{CKL}.
The crucial point is that equation \eqref{eq:qWgCrossing} lifts to map the Rickard complexes to the chain complexes assigned to crossings in $\mf{sl}_n$ link homology.

In the foam-based description of link homology~\cite{Dror,MSVFoam,QR},
categorical skew Howe duality maps generators in $\UcatD_Q(\mf{sl}_m)$ to explicit $\mf{sl}_n$ foams,
certain singular surfaces that categorify $\mf{sl}_n$ webs.
Theorem~\ref{thm:compat} then explicitly shows how to slide not only webs, but also foams mapping between them, through crossings in $\mf{sl}_n$ link homology.
At the level of 1-morphisms (webs), this interaction is key to the stability results used to define $\mf{sl}_n$ analogues of Jones Wenzl projectors~\cite{Roz,Cautis,Rose0},
and we anticipate that our extension to the level of $2$-morphisms will prove useful for future arguments in link homology.

\medskip

\noindent {\bf Acknowledgments.}
This project began as a collaboration between the third author and Mikhail Khovanov; the authors sincerely thank him for his generosity with ideas.
We also thank Andrea Appel, Ben Elias, and Peter McNamara for valuable discussions.
During the course of this work, A~.D.~L.~ was partially supported by the NSF grants DMS-1255334, DMS-1664240, DMS-1902092 and Army Research Office W911NF-20-1-0075,
and D.~E.~V.~R. was partially supported by NSA grants H-98230-16-1-0003 and H-98230-17-1-0211, and Simons Collaboration Grant 523992.

%
\section{The quantum group and Lusztig symmetries}
%

%
\subsection{The quantum group \texorpdfstring{$\mathbf{U}_q(\mf{g})$}{U}} \label{sec:Udot}
%

%
\subsubsection{ Root datum }
%

For  the remainder of  this article we restrict our attention to simply-laced Kac-Moody algebras.
These algebras are associated to a choice of  simply-laced Cartan datum
consisting of
\begin{itemize}
\item a finite set $I$, and
\item a $\Z$-valued symmetric bilinear form $\cdot$ on $\Z I$ satisfying $i \cdot i = 2$ for all $i \in I$ and $i \cdot j \in \{0,-1\}$ for $i \neq j$
\end{itemize}
and root datum given by:
\begin{itemize}
\item a free $\Z$-module $X$, called the \emph{weight lattice},
\item a choice of \emph{simple roots} $\{\alpha_i\}_{i \in I} \subset X$ and \emph{simple coroots} $\{h_i\}_{i \in I} \subset X^\vee = \Hom_{\Z}(X,\Z)$
that satisfy $\la h_i, \alpha_j \ra = 2\frac{i \cdot j}{i \cdot i}$, where here $\langle \cdot, \cdot \rangle \maps X^{\vee} \times X
\to \Z$ is the canonical pairing.
\end{itemize}
In this case, $a_{ij} := \langle h_i, \alpha_j \rangle = i \cdot j$, so $(a_{ij})_{i,j\in I}$ is a symmetric generalized Cartan matrix.
Given an arbitrary \emph{weight} $\l \in X$, we will often abbreviate $\la h_i, \lambda \ra$ by either $\la i,\lambda\ra$ or simply by $\lambda_i$.
We let $\{\Lambda_i\}_{i \in I} \subset X$ denote the \emph{fundamental weights}, which are characterized by the property that
$\langle h_i, \Lambda_j \rangle =\delta_{ij}$ for all $i,j \in I$.

We let $X^+ \subset X$ denote the \emph{dominant weights}, which are those of the form $\sum_i \lambda_i \Lambda_i$ for $\lambda_i \ge 0$.
Associated to a simply-laced Cartan datum is a graph $\Gamma$ without loops or multiple edges,
with a vertex for each $i \in I$ and an edge from vertex $i$ to vertex $j$ if and only if  $i \cdot j =-1$.

%
\subsubsection{The simply-laced quantum group}
%

The quantum group ${\bf U}={\bf U}_q(\mf{g})$ associated to a simply-laced root datum
is the unital, associative $\Q(q)$-algebra given by generators $E_i$, $F_i$, $K_{\mu}$ for
$i\in I$ and $\mu \in X^{\vee}$, subject to the relations:
\begin{center}
\begin{enumerate}[label=(\alph*)]
\item $K_0=1$ and $K_{\mu}K_{\mu'}=K_{\mu+\mu'}$ for all $\mu,\mu' \in X^{\vee}$,
\item $K_{\mu}E_i = q^{\la \mu,\alpha_i\ra}E_iK_{\mu}$ for all $i \in I$, $\mu \in X^{\vee}$,
\item $K_{\mu}F_i = q^{-\la \mu, \alpha_i\ra}F_iK_{\mu}$ for all $i \in I$, $\mu \in X^{\vee}$,
\item $E_iF_j - F_jE_i = \delta_{ij} \frac{K_{h_i}-K_{h_i}^{-1}}{q-q^{-1}}$,  where we set $K_i := K_{h_i}$, and
\item for all $i\neq j$
\[
\sum_{a+b=-  \la i, \alpha_j \ra +1 }(-1)^{a} E_i^{(a)}E_jE_i^{(b)} = 0 \quad \text{and} \quad \sum_{a+b=- \la i, \alpha_j \ra +1}(-1)^{a} F_i^{(a)}F_jF_i^{(b)} = 0
\]
where $E_i^{(a)}=E_i^a/[a]!$, $F_i^{(a)}=F_i^a/[a]!$, and $[a]!=\prod_{m=1}^a\frac{q^m-q^{-m}}{q-q^{-1}}$.
\end{enumerate} \end{center}

%
\subsubsection{The integral idempotented form of quantum group}
%

We will work with the idempotent form of ${\bf U}$,
which is adapted to the study of ${\bf U}$-modules with weight space decompositions.
This non-unital algebra is equipped with a collection of orthogonal idempotents,
hence can be described as a $\Q(q)$-linear category $\dot{{\bf U}}=\dot{{\bf U}}_q(\mf{g})$,  defined as follows.
The objects of $\dot{{\bf U}}$ are elements of $X$,
and the $\Hom$-space between $\l, \nu \in X$ is defined to be
\[
\dot{{\bf U}}(\l, \nu):={\bf U}/\left( \sum_{\mu\in X^{\vee}} {\bf U} (K_\mu-q^{\la\mu,\l\ra})+ \sum_{\mu\in X^{\vee}}(K_\mu-q^{\la\mu,\nu\ra})  {\bf U}\right).
\]

The identity morphism of $\l\in X$ is denoted by $1_\l$
and we will typically abbreviate the element $1_\mu x1_\l \in \dot{\bfU}(\l,\mu)$ determined by $x \in {\bfU}$ by either $1_\mu x$ or $x1_\l$,
\eg we have $E_i1_{\lambda} = 1_{\lambda+\alpha_i}E_i$ and $F_i1_{\lambda} = 1_{\lambda-\alpha_i}F_i$.
Composition in $\dot{\bfU}$ is induced by multiplication in ${\bf U}$, \ie
\[
(1_\mu x1_\nu) (1_\nu y 1_\l)=1_\mu xy 1_\l
\]
for $x,y \in \bfU$, $\l,\mu,\nu \in X$.
 The idempotent form $\dot{{\bf U}}$ admits an integral form,
defined as the $\Z[q,q^{-1}]$-lattice $\UA \subset \dot{{\bf U}}$
spanned by products of divided powers $E_i^{(a)}1_{\lambda}$ and $F_i^{(a)}1_{\lambda}$.

%
\subsection{(Anti)linear (anti)automorphisms of \texorpdfstring{$\U$}{U}} \label{sec_Usymm}
%

There are several $\Z[q,q^{-1}]$-(anti)linear (anti)automorphisms that will be used in this paper.
 For $f \in \Q(q)$, let $f \mapsto \bar{f}$ be the $\Q$-linear involution of $\Q(q)$ that sends $q$ to $q^{-1}$.
\begin{itemize}
	\item  The $\Q(q)$-linear algebra anti-involution $\und{\sigma} \maps {\bf U} \to {\bf U}$ is given by
\begin{align*}
\und{\sigma}(E_i) &= E_i \;\; , \;\; \und{\sigma}(F_i) =F_i \;\; , \;\; \und{\sigma}(K_i) = K_i^{-1}, \\
   \und{\sigma}(fx) &= f\und{\sigma}(x) \;\; \text{for} \; f\in \Q(q) \; \text{and} \; x \in {\bf U},\\
  \und{\sigma}(xy)& = \und{\sigma}(y)\und{\sigma}(x) \;\; \text{for} \; x,y \in {\bf U}.
\end{align*}
	\item The $\Q(q)$-linear algebra involution $\und{\omega}\maps {\bf U} \to {\bf U}$ is given by
\begin{align*}
\und{\omega}(E_i) &=F_i \;\; , \;\; \und{\omega}(F_i)=E_i \;\; , \;\; \und{\omega}(K_i) = K_i^{-1}, \\
\und{\omega}(fx) &=f\und{\omega}(x) \;\; \text{for} \; f\in \Q(q) \; \text{and} \; x \in {\bf U},\\
\und{\omega}(xy) &=\und{\omega}(x)\und{\omega}(y) \;\; \text{for} \; x,y \in {\bf U}.
\end{align*}
	\item The $\Q(q)$-antilinear algebra involution $\und{\psi} \maps {\bf U} \to {\bf U}$ is given by
\begin{align*}
\und{\psi}(E_i) &=E_i \;\; , \;\; \und{\psi}(F_i)=F_i \;\; , \;\; \und{\psi}(K_i) = K_i^{-1}, \\
\und{\psi}(fx) &=\bar{f}\und{\psi}(x) \;\; \text{for} \; f\in \Q(q) \; \text{and} \; x \in {\bf U},\\
\und{\psi}(xy) &=\und{\psi}(x)\und{\psi}(y) \;\; \text{for} \; x,y \in {\bf U}.
\end{align*}
\end{itemize}

These (anti)linear (anti)involutions pairwise commute and generate the group $G=(\Z_2)^3$ of (anti)linear (anti)automorphisms acting on ${\bf U}$.
The (anti)involutions $\und{\sigma}$, $\und{\omega}$, and  $\und{\psi}$  all extend to $\U$ and $\UA$ by setting
\[
\und{\sigma}(1_{\l}) = 1_{-\l} \;\; , \;\; \und{\omega}(1_{\l}) = 1_{-\l} \;\; , \;\; \und{\psi}(1_{\l})=1_{\l}.
\]
and taking the induced maps on each summand $1_{\l'}\U1_{\l}$.

%
\subsection{Quantum Weyl group action on integrable \texorpdfstring{$\U$}{U}-modules} \label{sec:braid-int-modules}
%

Let $V=\oplus_{\lambda} V_\lambda$ be an integrable $\U$-module, then, for $e=\pm 1$,
Lusztig \cite[5.2.1]{Lus4} defines linear maps $\mathsf{t}_{i,e}'$, $\mathsf{t}_{i,e}'' \maps V \to V$ by
\begin{align*}
 \mathsf{t}_{i,e}'(z) &= \sum_{a,b,c; a-b+c=\lambda_i} (-1)^b q^{e(-ac+b)}F_i^{(a)}E_i^{(b)}F_i^{(c)} z,   \\
 \mathsf{t}_{i,e}''(z) &= \sum_{a,b,c; -a+b-c=\lambda_i} (-1)^b q^{e(-ac+b)}E_i^{(a)}F_i^{(b)}E_i^{(c)} z,
\end{align*}
for $z \in V_\lambda$ that are commonly called the \emph{quantum Weyl group} elements.
They are mutually inverse automorphisms (specifically, they satisfy $\mathsf{t}_{i,e}'\mathsf{t}_{i,-e}'' = \Id = \mathsf{t}_{i,-e}''\mathsf{t}_{i,e}'$)
that satisfy the relations
\begin{align*}
\mathsf{t}_{i,e}'\mathsf{t}_{j,e}'\mathsf{t}_{i,e}' = \mathsf{t}_{j,e}' \mathsf{t}_{i,e}' \mathsf{t}_{j,e}'   \;\; &\text{and} \;\;
\mathsf{t}_{i,e}''\mathsf{t}_{j,e}''\mathsf{t}_{i,e}'' = \mathsf{t}_{j,e}''\mathsf{t}_{i,e}''\mathsf{t}_{j,e}'' \;\;
\text{if $i\cdot j=-1$}\\
\mathsf{t}_{i,e}'\mathsf{t}_{j,e}' = \mathsf{t}_{j,e}'\mathsf{t}_{i,e}' \;\; &\text{and} \;\;
\mathsf{t}_{i,e}''\mathsf{t}_{j,e}'' = \mathsf{t}_{j,e}''\mathsf{t}_{i,e}'' \;\;
\text{if $i\cdot j=0$}
\end{align*}
and thus define an action of type $\mf{g}$ braid group on any integrable module \cite[Theorem 39.4.3]{Lus4}.
This action on a particular weight space can be conveniently described by the infinite sums
\begin{align*}
 \mathsf{t}_{i,e}'1_{\l} &= \sum_{a,b,c; a-b+c=\lambda_i} (-1)^b q^{e(-ac+b)}F_i^{(a)}E_i^{(b)}F_i^{(c)} 1_{\l}   \\
 \mathsf{t}_{i,e}''1_{\l} &= \sum_{a,b,c; -a+b-c=\lambda_i} (-1)^b q^{e(-ac+b)}E_i^{(a)}F_i^{(b)}E_i^{(c)} 1_{\l}
\end{align*}
of elements in $\U$, from which the maps $\mathsf{t}_{i,e}'$ ,$\mathsf{t}_{i,e}''$ can be recovered by taking the sum over all $\l \in X$.
It was shown in \cite[Lemma 6.1.1]{CKM} that these elements admit the simpler form given in equation \eqref{eq:QWG} above,
\ie in fact all terms with $c \neq 0$ cancel.

%
\subsection{Lusztig's internal braid group action} \label{sec:braid-quant-group}
%

For each $i\in I$ and $e=\pm1$, Lusztig defines algebra automorphisms $T'_{i,e}$ and $T''_{i,e}$ of $\dot{{\bf U}}=\dot{{\bf U}}_q(\mf{g})$
defined uniquely by the compatibility with the quantum Weyl group action given in equation \eqref{eq:QWGcompat} above.
They are given explicitly in \cite[41.1.2]{Lus4} by
\begin{equation}\label{eq_def-Ti}
\begin{aligned}
T'_{i,e}(1_{\lambda}) &= 1_{s_i(\lambda)} \\
T'_{i,e}(E_\ell 1_{\lambda}) &=
\begin{cases}
-q^{-e(2+\lambda_i)}F_i1_{s_i(\lambda)} &\text{if } i=\ell \\
E_\ell E_i1_{s_i(\lambda)}-q^{e}E_iE_\ell1_{s_i(\lambda)} &\text{if } i\cdot\ell=-1\\
E_\ell1_{s_i(\lambda)} &\text{if } i\cdot\ell=0
\end{cases} \\
T'_{i,e}(F_\ell1_{\lambda})&=
\begin{cases}
-q^{e(\lambda_i)}E_i1_{s_i(\lambda)} &\text{if } i=\ell \\
F_iF_\ell 1_{s_i(\lambda)}-q^{-e}F_\ell F_i1_{s_i(\lambda)} &\text{if } i\cdot\ell=-1\\
F_\ell1_{s_i(\lambda)} &\text{if } i\cdot\ell=0
\end{cases}
\end{aligned}
\end{equation}
and
\begin{align*}
T''_{i,e}(1_{\lambda}) &= 1_{s_i(\lambda)} \\
T''_{i,e}(E_\ell 1_{\lambda}) &=
\begin{cases}
-q^{-e(\lambda_i)}F_i1_{s_i(\lambda)} &\text{if } i=\ell \\
E_iE_\ell1_{s_i(\lambda)}-q^{-e}E_\ell E_i1_{s_i(\lambda)} &\text{if } i\cdot\ell=-1\\
E_\ell1_{s_i(\lambda)} &\text{if } i\cdot\ell=0
\end{cases} \\
T''_{i,e}(F_\ell1_{\lambda})&=
\begin{cases}
-q^{e(\lambda_i-2)}E_i1_{s_i(\lambda)} &\text{if } i=\ell \\
F_\ell F_i1_{s_i(\lambda)}-q^{e}F_iF_\ell1_{s_i(\lambda)} &\text{if } i\cdot\ell=-1\\
F_\ell1_{s_i(\lambda)} &\text{if } i\cdot\ell=0
\end{cases}
\end{align*}
where $s_i$ is the Weyl group element corresponding to the simple root $\alpha_i$,
\ie $s_i(\lambda)=\lambda-\la i,\lambda \ra\alpha_i$.
Lusztig further shows \cite[41.1.1]{Lus4} that $(T'_{i,e})^{-1} = T''_{i,-e}$,
and that these automorphisms interact with the automorphisms from Section \ref{sec_Usymm} as in equation \eqref{intro-T'_rel_T''} above.
As a consequence, we see that both $T'_{i,e}$ and $T''_{i,e}$ are invariant under conjugation by the triple composite $\und{\sigma}\und{\omega}\und{\psi}$.

In what follows, we focus our attention on the automorphisms $T'_{i,1}$, since similar results can be deduced for $T''_{i,-1}$, $T''_{i,1}$, and $T'_{i,-1}$
using equation \eqref{intro-T'_rel_T''}. When the context is clear, we will abbreviate $T'_{i,1}$ by $T'_i$.
In \cite[39.2.4 and 39.2.5]{Lus4}, Lusztig shows that the $T'_i$ satisfy
\begin{align*}
T_i'T_\ell'T_i' &= T_\ell'T_i'T_\ell' \;\; \text{if } i\cdot \ell=-1\\
T_i'T_\ell' &= T_\ell'T_i' \;\; \text{if } i\cdot \ell=0
\end{align*}
and hence defined a type $\mf{g}$ braid group action on $\dot{{\bf U}}$.

%
\section{The categorified quantum group}
%

In this section, we recall the definition of the categorified quantum group $\Ucat_Q(\mf{g})$,
specifically the cyclic version from \cite{BHLW2}, and establish a number of additional properties needed for our arguments.

%
\subsection{Choice of scalars \texorpdfstring{$Q$}{Q}} \label{sec:choice-scalars}
%

Let $\Bbbk$ be a field, not necessarily algebraically closed, or characteristic zero.

\begin{definition}
A {\em choice of scalars $Q$} associated to a simply-laced Cartan datum, consist of elements
$\{ t_{ij} \}_{i,j \in I}$ satisfying:
\begin{itemize}
\item $t_{ii}=1$ for all $i \in I$ and $t_{ij} \in \Bbbk^{\times}$ for $i\neq j$,
 \item $t_{ij}=t_{ji}$ when $a_{ij}=0$.
\end{itemize}
We say that a choice of scalars $Q$ is {\em integral} if $t_{ij} =\pm 1$ for all $i,j \in I$.
\end{definition}

The choice of scalars $Q$ controls the form of the KLR algebra $R_Q$ that
categorifies the positive half of the quantum group $\U$, and
the 2-category $\Ucat_Q(\mf{g})$ is governed by the products $v_{ij}=t_{ij}^{-1}t_{ji}$ taken over all pairs $i,j\in I$,
which can be viewed as a $\Bbbk^{\times}$-valued 1-cocycle on the graph $\Gamma$ associated to the Cartan datum.

\begin{definition}
A choice of {\em bubble parameters $C$} consists of elements $c_{i,\l} \in \Bbbk^{\times}$ for $i\in I$ and $\lambda \in X$.
We say that they are {\em compatible} with the scalars $Q$ if
\begin{equation}\label{eq:BubbCompat}
c_{i,\lambda+\alpha_j}/c_{i,\lambda}=t_{ij}.
\end{equation}
\end{definition}

Given any choice of scalars $Q$, we obtain a compatible choice of bubble parameters by
fixing $c_{i,\lambda}$ for a representative in every coset of the root lattice in the weight lattice,
and then extending to entire weight lattice using equation \eqref{eq:BubbCompat}.
For a compatible choice, note that the bubble parameters remain constant along an $\mf{sl}_2$-string
since
\[
 c_{i,\l+n\alpha_i} = t_{ii}^n c_{i,\l} = c_{i,\l}.
\]

%
\subsection{Definition of the 2-category \texorpdfstring{$\Ucat_Q(\mf{g})$}{Ug}}
%

Recall that a \emph{graded linear} category is an additive category equipped with an auto-equivalence $\la 1 \ra$ called the \emph{shift} (see \eg \cite{Dror}),
and a graded additive 2-category is a category enriched in graded linear categories.
Throughout, we will use $\la t \ra$ to denote the auto-equivalence given by applying $\la 1 \ra$ $t$ times,
and $\la -t\ra$ to denote the auto-equivalence obtained by applying the inverse of $\la 1 \ra$ $t$ times.

\begin{definition} \label{defU_cat-cyc}
Fix a choice of scalars $Q$ and compatible bubble parameters $C$,
then the 2-category $\Ucat_Q:= \Ucat_Q^{cyc}(\mf{g})$ is the graded linear 2-category with:
\begin{itemize}
\item Objects: $\lambda \in X$,

\item $1$-morphisms: formal direct sums of shifts of compositions of the generating $1$-morphisms:
\[
\onel \;\; , \;\; \onell{\l+\alpha_i} \sE_i = \onell{\l+\alpha_i} \sE_i\onel = \sE_i \onel \;\; , \;\; \onell{\lambda-\alpha_i} \sF_i = \onell{\lambda-\alpha_i} \sF_i\onel = \sF_i\onel
\]
for $i \in I$ and $\l \in X$.

\item $2$-morphisms: $\Hom$-spaces are $\Bbbk$-vector spaces
spanned by (horizontal and vertical) compositions of the following decorated tangle-like diagrams.
\begin{align}
  \xy 0;/r.17pc/:
 (0,0)*{\sdotur{i}};
 (7,3)*{ \scs \lambda};
 (-9,3)*{\scs  \lambda+\alpha_i};
 (-10,0)*{};(10,0)*{};
 \endxy &\maps \cal{E}_i\onel \to \cal{E}_i\onel\la 2 \ra  & \quad
 &
    \xy 0;/r.17pc/:
 (0,0)*{\sdotdr{i}};
 (7,3)*{ \scs \lambda};
 (-9,3)*{\scs  \lambda-\alpha_i};
 (-10,0)*{};(10,0)*{};
 \endxy\maps \cal{F}_i\onel \to \cal{F}_i\onel\la 2 \ra  \nn \\
   & & & \nn \\
   \xy 0;/r.17pc/:
  (0,0)*{\xybox{
(0,0)*{\ucrossrr{i}{j}};
     (9,1)*{\scs  \lambda};
     (-10,0)*{};(10,0)*{};
     }};
  \endxy \;\;&\maps \cal{E}_i\cal{E}_j\onel  \to \cal{E}_j\cal{E}_i\onel\la -i\cdot j \ra  &
  &
   \xy 0;/r.17pc/:
  (0,0)*{\xybox{(0,0)*{\dcrossrr{i}{j}};
     (9,1)*{\scs  \lambda};
     (-10,0)*{};(10,0)*{};
     }};
  \endxy\;\; \maps \cal{F}_i\cal{F}_j\onel  \to \cal{F}_j\cal{F}_i\onel\la - i\cdot j \ra  \nn \\
     & & & \nn \\
   \xy 0;/r.17pc/:
  (0,0)*{\xybox{(0,0)*{\lcrossrr{i}{j}};
     (9,1)*{\scs  \lambda};
     (-10,0)*{};(10,0)*{};
     }};
  \endxy \;\;&\maps \cal{F}_i\cal{E}_j\onel  \to \cal{E}_j\cal{F}_i\onel   &
  &
   \xy 0;/r.17pc/:
  (0,0)*{\xybox{(0,0)*{\rcrossrr{i}{j}};
     (9,1)*{\scs  \lambda};
     (-10,0)*{};(10,0)*{};
     }};
  \endxy\;\; \maps \cal{E}_i\cal{F}_j\onel  \to \cal{F}_j\cal{E}_i\onel   \nn \\
  & & & \nn \\
     \xy 0;/r.17pc/:
    (0,0)*{\rcupr{i}};
    (8,-5)*{\scs  \lambda};
    (-10,0)*{};(10,0)*{};
    \endxy &\maps \onel  \to \cal{F}_i\cal{E}_i\onel\la 1 + \l_i \ra   &
    &
   \xy 0;/r.17pc/:
    (0,0)*{\lcupr{i}};
    (8,-5)*{\scs \lambda};
    (-10,0)*{};(10,0)*{};
    \endxy \maps \onel  \to\cal{E}_i\cal{F}_i\onel\la 1 - \l_i \ra  \nn \\
      & & & \nn \\
  \xy 0;/r.17pc/:
    (0,0)*{\lcapr{i}};
    (8,4)*{\scs  \lambda};
    (-10,0)*{};(10,0)*{};
    \endxy & \maps \cal{F}_i\cal{E}_i\onel \to\onel\la 1 + \l_i \ra  &
    &
 \xy 0;/r.17pc/:
    (0,0)*{\rcapr{i}};
    (8,4)*{\scs  \lambda};
    (-10,0)*{};(10,0)*{};
    \endxy \maps\cal{E}_i\cal{F}_i\onel  \to\onel\la 1 - \l_i \ra \nn
\end{align}
\end{itemize}
\end{definition}

Note that we follow the grading conventions in \cite{CLau,LQR}, which are opposite to those from \cite{KL3}.
We read such diagrams from right to left and bottom to top, and the identity 2-morphisms of the 1-morphisms $\cal{E}_i \onel$ and $\cal{F}_i \onel$ are
depicted by upward and downward oriented segments labeled by $i$, respectively.

The following local relations are imposed on the 2-morphisms.
\begin{enumerate}
\item \label{def:lr-adj}
Right and left adjunction:
\[
 \xy   0;/r.17pc/:
    (-8,0)*{}="1";
    (0,0)*{}="2";
    (8,0)*{}="3";
    (-8,-10);"1" **\dir{-};
    "1";"2" **\crv{(-8,8) & (0,8)} ?(0)*\dir{>} ?(1)*\dir{>};
    "2";"3" **\crv{(0,-8) & (8,-8)}?(1)*\dir{>};
    "3"; (8,10) **\dir{-};
    (12,-9)*{\lambda};
    (-6,9)*{\lambda+\alpha_i};
    \endxy
    \; =
    \;
\xy   0;/r.17pc/:
    (-8,0)*{}="1";
    (0,0)*{}="2";
    (8,0)*{}="3";
    (0,-10);(0,10)**\dir{-} ?(.5)*\dir{>};
    (5,8)*{\lambda};
    (-9,8)*{\lambda+\alpha_i};
    \endxy
\;\; , \qquad
    \xy  0;/r.17pc/:
    (8,0)*{}="1";
    (0,0)*{}="2";
    (-8,0)*{}="3";
    (8,-10);"1" **\dir{-};
    "1";"2" **\crv{(8,8) & (0,8)} ?(0)*\dir{<} ?(1)*\dir{<};
    "2";"3" **\crv{(0,-8) & (-8,-8)}?(1)*\dir{<};
    "3"; (-8,10) **\dir{-};
    (12,9)*{\lambda+\alpha_i};
    (-6,-9)*{\lambda};
    \endxy
    \; =
    \;
\xy  0;/r.17pc/:
    (8,0)*{}="1";
    (0,0)*{}="2";
    (-8,0)*{}="3";
    (0,-10);(0,10)**\dir{-} ?(.5)*\dir{<};
    (9,-8)*{\lambda+\alpha_i};
    (-6,-8)*{\lambda};
    \endxy
\]

\[
 \xy   0;/r.17pc/:
    (8,0)*{}="1";
    (0,0)*{}="2";
    (-8,0)*{}="3";
    (8,-10);"1" **\dir{-};
    "1";"2" **\crv{(8,8) & (0,8)} ?(0)*\dir{>} ?(1)*\dir{>};
    "2";"3" **\crv{(0,-8) & (-8,-8)}?(1)*\dir{>};
    "3"; (-8,10) **\dir{-};
    (12,9)*{\lambda};
    (-5,-9)*{\lambda+\alpha_i};
    \endxy
    \; =
    \;
    \xy 0;/r.17pc/:
    (8,0)*{}="1";
    (0,0)*{}="2";
    (-8,0)*{}="3";
    (0,-10);(0,10)**\dir{-} ?(.5)*\dir{>};
    (5,-8)*{\lambda};
    (-9,-8)*{\lambda+\alpha_i};
    \endxy
\;\; , \qquad
\xy   0;/r.17pc/:
    (-8,0)*{}="1";
    (0,0)*{}="2";
    (8,0)*{}="3";
    (-8,-10);"1" **\dir{-};
    "1";"2" **\crv{(-8,8) & (0,8)} ?(0)*\dir{<} ?(1)*\dir{<};
    "2";"3" **\crv{(0,-8) & (8,-8)}?(1)*\dir{<};
    "3"; (8,10) **\dir{-};
    (12,-9)*{\lambda+\alpha_i};
    (-6,9)*{\lambda};
    \endxy
    \; =
    \;
\xy   0;/r.17pc/:
    (-8,0)*{}="1";
    (0,0)*{}="2";
    (8,0)*{}="3";
    (0,-10);(0,10)**\dir{-} ?(.5)*\dir{<};
   (9,8)*{\lambda+\alpha_i};
    (-6,8)*{\lambda};
    \endxy
\]

\item \label{def:dot-cyc}
Dot cyclicity:
\begin{equation*}
 \xy 0;/r.17pc/:
    (-8,5)*{}="1";
    (0,5)*{}="2";
    (0,-5)*{}="2'";
    (8,-5)*{}="3";
    (-8,-10);"1" **\dir{-};
    "2";"2'" **\dir{-} ?(.5)*\dir{<};
    "1";"2" **\crv{(-8,12) & (0,12)} ?(0)*\dir{<};
    "2'";"3" **\crv{(0,-12) & (8,-12)}?(1)*\dir{<};
    "3"; (8,10) **\dir{-};
    (17,-9)*{\lambda+\alpha_i};
    (-12,9)*{\lambda};
    (0,4)*{\bullet};
    (10,8)*{\scs };
    (-10,-8)*{\scs };
    \endxy
\;\; = \;\;
      \xy 0;/r.17pc/:
 (0,10);(0,-10); **\dir{-} ?(.75)*\dir{<}+(2.3,0)*{\scriptstyle{}}
 ?(.1)*\dir{ }+(2,0)*{\scs };
 (0,0)*{\bullet};
 (-6,5)*{\lambda};
 (10,5)*{\lambda+\alpha_i};
 (-10,0)*{};(10,0)*{};(-2,-8)*{\scs };
 \endxy
\;\; = \;\;
   \xy 0;/r.17pc/:
    (8,5)*{}="1";
    (0,5)*{}="2";
    (0,-5)*{}="2'";
    (-8,-5)*{}="3";
    (8,-10);"1" **\dir{-};
    "2";"2'" **\dir{-} ?(.5)*\dir{<};
    "1";"2" **\crv{(8,12) & (0,12)} ?(0)*\dir{<};
    "2'";"3" **\crv{(0,-12) & (-8,-12)}?(1)*\dir{<};
    "3"; (-8,10) **\dir{-};
    (17,9)*{\lambda+\alpha_i};
    (-12,-9)*{\lambda};
    (0,4)*{\bullet};
    (-10,8)*{\scs };
    (10,-8)*{\scs };
    \endxy
\end{equation*}

\item  \label{def:cross-cyc}
Crossing cyclicity:
\begin{align*}
   \xy 0;/r.17pc/:
  (0,0)*{\xybox{
    (-4,4)*{};(4,-4)*{} **\crv{(-4,1) & (4,-1)}?(1)*\dir{>} ;
    (4,4)*{};(-4,-4)*{} **\crv{(4,1) & (-4,-1)}?(1)*\dir{>};
    (-6.5,-3)*{\scs i};
     (6.5,-3)*{\scs j};
     (9,1)*{\scs  \lambda};
     (-10,0)*{};(10,0)*{};
     }};
  \endxy
\;\; &= \;\;
  \xy 0;/r.17pc/:
  (0,0)*{\xybox{
    (4,-4)*{};(-4,4)*{} **\crv{(4,-1) & (-4,1)}?(1)*\dir{>};
    (-4,-4)*{};(4,4)*{} **\crv{(-4,-1) & (4,1)};
     (-4,4)*{};(18,4)*{} **\crv{(-4,16) & (18,16)} ?(1)*\dir{>};
     (4,-4)*{};(-18,-4)*{} **\crv{(4,-16) & (-18,-16)} ?(1)*\dir{<}?(0)*\dir{<};
     (-18,-4);(-18,12) **\dir{-};(-12,-4);(-12,12) **\dir{-};
     (18,4);(18,-12) **\dir{-};(12,4);(12,-12) **\dir{-};
     (22,1)*{ \lambda};
     (-10,0)*{};(10,0)*{};
     (-4,-4)*{};(-12,-4)*{} **\crv{(-4,-10) & (-12,-10)}?(1)*\dir{<}?(0)*\dir{<};
      (4,4)*{};(12,4)*{} **\crv{(4,10) & (12,10)}?(1)*\dir{>}?(0)*\dir{>};
      (-20,11)*{\scs j};(-10,11)*{\scs i};
      (20,-11)*{\scs j};(10,-11)*{\scs i};
     }};
  \endxy
\;\; = \;\;
\xy 0;/r.17pc/:
  (0,0)*{\xybox{
    (-4,-4)*{};(4,4)*{} **\crv{(-4,-1) & (4,1)}?(1)*\dir{>};
    (4,-4)*{};(-4,4)*{} **\crv{(4,-1) & (-4,1)};
     (4,4)*{};(-18,4)*{} **\crv{(4,16) & (-18,16)} ?(1)*\dir{>};
     (-4,-4)*{};(18,-4)*{} **\crv{(-4,-16) & (18,-16)} ?(1)*\dir{<}?(0)*\dir{<};
     (18,-4);(18,12) **\dir{-};(12,-4);(12,12) **\dir{-};
     (-18,4);(-18,-12) **\dir{-};(-12,4);(-12,-12) **\dir{-};
     (22,1)*{ \lambda};
     (-10,0)*{};(10,0)*{};
      (4,-4)*{};(12,-4)*{} **\crv{(4,-10) & (12,-10)}?(1)*\dir{<}?(0)*\dir{<};
      (-4,4)*{};(-12,4)*{} **\crv{(-4,10) & (-12,10)}?(1)*\dir{>}?(0)*\dir{>};
      (20,11)*{\scs i};(10,11)*{\scs j};
      (-20,-11)*{\scs i};(-10,-11)*{\scs j};
     }};
  \endxy \\
  \xy 0;/r.18pc/:
  (0,0)*{\xybox{
    (-4,-4)*{};(4,4)*{} **\crv{(-4,-1) & (4,1)}?(1)*\dir{>} ;
    (4,-4)*{};(-4,4)*{} **\crv{(4,-1) & (-4,1)}?(0)*\dir{<};
    (-5,-3)*{\scs j};
     (6.5,-3)*{\scs i};
     (9,2)*{ \lambda};
     (-12,0)*{};(12,0)*{};
     }};
  \endxy
\;\; &= \;\;
 \xy 0;/r.17pc/:
  (0,0)*{\xybox{
    (4,-4)*{};(-4,4)*{} **\crv{(4,-1) & (-4,1)}?(1)*\dir{>};
    (-4,-4)*{};(4,4)*{} **\crv{(-4,-1) & (4,1)};
     (-4,4);(-4,12) **\dir{-};
     (-12,-4);(-12,12) **\dir{-};
     (4,-4);(4,-12) **\dir{-};(12,4);(12,-12) **\dir{-};
     (16,1)*{\lambda};
     (-10,0)*{};(10,0)*{};
     (-4,-4)*{};(-12,-4)*{} **\crv{(-4,-10) & (-12,-10)}?(1)*\dir{<}?(0)*\dir{<};
      (4,4)*{};(12,4)*{} **\crv{(4,10) & (12,10)}?(1)*\dir{>}?(0)*\dir{>};
      (-14,11)*{\scs i};(-2,11)*{\scs j};
      (14,-11)*{\scs i};(2,-11)*{\scs j};
     }};
  \endxy
\;\; = \;\;
 \xy 0;/r.17pc/:
  (0,0)*{\xybox{
    (-4,-4)*{};(4,4)*{} **\crv{(-4,-1) & (4,1)}?(1)*\dir{<};
    (4,-4)*{};(-4,4)*{} **\crv{(4,-1) & (-4,1)};
     (4,4);(4,12) **\dir{-};
     (12,-4);(12,12) **\dir{-};
     (-4,-4);(-4,-12) **\dir{-};(-12,4);(-12,-12) **\dir{-};
     (16,1)*{\lambda};
     (10,0)*{};(-10,0)*{};
     (4,-4)*{};(12,-4)*{} **\crv{(4,-10) & (12,-10)}?(1)*\dir{>}?(0)*\dir{>};
      (-4,4)*{};(-12,4)*{} **\crv{(-4,10) & (-12,10)}?(1)*\dir{<}?(0)*\dir{<};
     }};
     (12,11)*{\scs j};(0,11)*{\scs i};
      (-17,-11)*{\scs j};(-5,-11)*{\scs i};
  \endxy \\
  \xy 0;/r.18pc/:
  (0,0)*{\xybox{
    (-4,-4)*{};(4,4)*{} **\crv{(-4,-1) & (4,1)}?(0)*\dir{<} ;
    (4,-4)*{};(-4,4)*{} **\crv{(4,-1) & (-4,1)}?(1)*\dir{>};
    (5.1,-3)*{\scs i};
     (-6.5,-3)*{\scs j};
     (9,2)*{ \lambda};
     (-12,0)*{};(12,0)*{};
     }};
  \endxy
\;\; &= \;\;
 \xy 0;/r.17pc/:
  (0,0)*{\xybox{
    (-4,-4)*{};(4,4)*{} **\crv{(-4,-1) & (4,1)}?(1)*\dir{>};
    (4,-4)*{};(-4,4)*{} **\crv{(4,-1) & (-4,1)};
     (4,4);(4,12) **\dir{-};
     (12,-4);(12,12) **\dir{-};
     (-4,-4);(-4,-12) **\dir{-};(-12,4);(-12,-12) **\dir{-};
     (16,-6)*{\lambda};
     (10,0)*{};(-10,0)*{};
     (4,-4)*{};(12,-4)*{} **\crv{(4,-10) & (12,-10)}?(1)*\dir{<}?(0)*\dir{<};
      (-4,4)*{};(-12,4)*{} **\crv{(-4,10) & (-12,10)}?(1)*\dir{>}?(0)*\dir{>};
      (14,11)*{\scs j};(2,11)*{\scs i};
      (-14,-11)*{\scs j};(-2,-11)*{\scs i};
     }};
  \endxy
\;\; = \;\;
  \xy 0;/r.17pc/:
  (0,0)*{\xybox{
    (4,-4)*{};(-4,4)*{} **\crv{(4,-1) & (-4,1)}?(1)*\dir{<};
    (-4,-4)*{};(4,4)*{} **\crv{(-4,-1) & (4,1)};
     (-4,4);(-4,12) **\dir{-};
     (-12,-4);(-12,12) **\dir{-};
     (4,-4);(4,-12) **\dir{-};(12,4);(12,-12) **\dir{-};
     (16,6)*{\lambda};
     (-10,0)*{};(10,0)*{};
     (-4,-4)*{};(-12,-4)*{} **\crv{(-4,-10) & (-12,-10)}?(1)*\dir{>}?(0)*\dir{>};
      (4,4)*{};(12,4)*{} **\crv{(4,10) & (12,10)}?(1)*\dir{<}?(0)*\dir{<};
      (-14,11)*{\scs i};(-2,11)*{\scs j};(14,-11)*{\scs i};(2,-11)*{\scs j};
     }};
  \endxy
\end{align*}

The next three relations imply that the $\cal{E}$'s (and $\cal{F}'s$) carry an action of the KLR algebra associated to $Q$.
\item \label{def:KLR-R2}
Quadratic KLR:
\[
 \vcenter{\xy 0;/r.17pc/:
    (-4,-4)*{};(4,4)*{} **\crv{(-4,-1) & (4,1)}?(1)*\dir{};
    (4,-4)*{};(-4,4)*{} **\crv{(4,-1) & (-4,1)}?(1)*\dir{};
    (-4,4)*{};(4,12)*{} **\crv{(-4,7) & (4,9)}?(1)*\dir{};
    (4,4)*{};(-4,12)*{} **\crv{(4,7) & (-4,9)}?(1)*\dir{};
    (8,8)*{\lambda};
    (4,12); (4,13) **\dir{-}?(1)*\dir{>};
    (-4,12); (-4,13) **\dir{-}?(1)*\dir{>};
  (-5.5,-3)*{\scs i};
     (5.5,-3)*{\scs j};
 \endxy}
\;\; = \;\;
\left\{
 \begin{array}{cl}
 0 & \text{if $i=j$, } \\
     t_{ij}\;\xy 0;/r.17pc/:
  (3,9);(3,-9) **\dir{-}?(0)*\dir{<}+(2.3,0)*{};
  (-3,9);(-3,-9) **\dir{-}?(0)*\dir{<}+(2.3,0)*{};
  (-5,-6)*{\scs i};     (5.1,-6)*{\scs j};
 \endxy &  \text{if $i\cdot j=0$,}\\
t_{ij} \vcenter{\xy 0;/r.17pc/:
  (3,9);(3,-9) **\dir{-}?(.5)*\dir{<}+(2.3,0)*{};
  (-3,9);(-3,-9) **\dir{-}?(.5)*\dir{<}+(2.3,0)*{};
  (-3,4)*{\bullet};
  (-5,-6)*{\scs i};     (5.1,-6)*{\scs j};
 \endxy} \;\; + \;\; t_{ji}
  \vcenter{\xy 0;/r.17pc/:
  (3,9);(3,-9) **\dir{-}?(.5)*\dir{<}+(2.3,0)*{};
  (-3,9);(-3,-9) **\dir{-}?(.5)*\dir{<}+(2.3,0)*{};
  (3,4)*{\bullet};
  (-5,-6)*{\scs i};     (5.1,-6)*{\scs j};
 \endxy} &  \text{if $i\cdot j=-1$}
 \end{array}
\right.
\]

\item \label{def:dot-slide}
Dot slide:
\[
\xy 0;/r.18pc/:
  (0,0)*{\xybox{
    (-4,-4)*{};(4,6)*{} **\crv{(-4,-1) & (4,1)}?(1)*\dir{>}?(.25)*{\bullet};
    (4,-4)*{};(-4,6)*{} **\crv{(4,-1) & (-4,1)}?(1)*\dir{>};
    (-5,-3)*{\scs i};
     (5.1,-3)*{\scs j};
     (-10,0)*{};(10,0)*{};
     }};
  \endxy
-
\xy 0;/r.18pc/:
  (0,0)*{\xybox{
    (-4,-4)*{};(4,6)*{} **\crv{(-4,-1) & (4,1)}?(1)*\dir{>}?(.75)*{\bullet};
    (4,-4)*{};(-4,6)*{} **\crv{(4,-1) & (-4,1)}?(1)*\dir{>};
    (-5,-3)*{\scs i};
     (5.1,-3)*{\scs j};
     (-10,0)*{};(10,0)*{};
     }};
  \endxy
=
\xy 0;/r.18pc/:
  (0,0)*{\xybox{
    (-4,-4)*{};(4,6)*{} **\crv{(-4,-1) & (4,1)}?(1)*\dir{>};
    (4,-4)*{};(-4,6)*{} **\crv{(4,-1) & (-4,1)}?(1)*\dir{>}?(.75)*{\bullet};
    (-5,-3)*{\scs i};
     (5.1,-3)*{\scs j};
     (-10,0)*{};(10,0)*{};
     }};
  \endxy
-
  \xy 0;/r.18pc/:
  (0,0)*{\xybox{
    (-4,-4)*{};(4,6)*{} **\crv{(-4,-1) & (4,1)}?(1)*\dir{>} ;
    (4,-4)*{};(-4,6)*{} **\crv{(4,-1) & (-4,1)}?(1)*\dir{>}?(.25)*{\bullet};
    (-5,-3)*{\scs i};
     (5.1,-3)*{\scs j};
     (-10,0)*{};(12,0)*{};
     }};
  \endxy
=
 \left\{\\
 \begin{array}{ccc}
     \xy 0;/r.18pc/:
  (0,0)*{\xybox{
    (-4,-4)*{};(-4,6)*{} **\dir{-}?(1)*\dir{>} ;
    (4,-4)*{};(4,6)*{} **\dir{-}?(1)*\dir{>};
    (-5,-3)*{\scs i};
     (5.1,-3)*{\scs i};
     (-10,0)*{};(12,0)*{};
     }};
  \endxy& & \text{if $i=j$}\\ \\
 0 & & \text{if $i\neq j$} \\
 \end{array}\right.
\]

\item \label{def:KLR-R3}
Cubic KLR:
\[
\vcenter{\xy 0;/r.17pc/:
    (-4,-4)*{};(4,4)*{} **\crv{(-4,-1) & (4,1)}?(1)*\dir{};
    (4,-4)*{};(-4,4)*{} **\crv{(4,-1) & (-4,1)}?(1)*\dir{};
    (4,4)*{};(12,12)*{} **\crv{(4,7) & (12,9)}?(1)*\dir{};
    (12,4)*{};(4,12)*{} **\crv{(12,7) & (4,9)}?(1)*\dir{};
    (-4,12)*{};(4,20)*{} **\crv{(-4,15) & (4,17)}?(1)*\dir{};
    (4,12)*{};(-4,20)*{} **\crv{(4,15) & (-4,17)}?(1)*\dir{};
    (-4,4)*{}; (-4,12) **\dir{-};
    (12,-4)*{}; (12,4) **\dir{-};
    (12,12)*{}; (12,20) **\dir{-};
    (4,20); (4,21) **\dir{-}?(1)*\dir{>};
    (-4,20); (-4,21) **\dir{-}?(1)*\dir{>};
    (12,20); (12,21) **\dir{-}?(1)*\dir{>};
   (18,8)*{\lambda};  (-6,-3)*{\scs i};
  (6,-3)*{\scs j};
  (15,-3)*{\scs k};
\endxy}
 \;\; - \;\;
\vcenter{\xy 0;/r.17pc/:
    (4,-4)*{};(-4,4)*{} **\crv{(4,-1) & (-4,1)}?(1)*\dir{};
    (-4,-4)*{};(4,4)*{} **\crv{(-4,-1) & (4,1)}?(1)*\dir{};
    (-4,4)*{};(-12,12)*{} **\crv{(-4,7) & (-12,9)}?(1)*\dir{};
    (-12,4)*{};(-4,12)*{} **\crv{(-12,7) & (-4,9)}?(1)*\dir{};
    (4,12)*{};(-4,20)*{} **\crv{(4,15) & (-4,17)}?(1)*\dir{};
    (-4,12)*{};(4,20)*{} **\crv{(-4,15) & (4,17)}?(1)*\dir{};
    (4,4)*{}; (4,12) **\dir{-};
    (-12,-4)*{}; (-12,4) **\dir{-};
    (-12,12)*{}; (-12,20) **\dir{-};
    (4,20); (4,21) **\dir{-}?(1)*\dir{>};
    (-4,20); (-4,21) **\dir{-}?(1)*\dir{>};
    (-12,20); (-12,21) **\dir{-}?(1)*\dir{>};
  (10,8)*{\lambda};
  (-14,-3)*{\scs i};
  (-6,-3)*{\scs j};
  (6,-3)*{\scs k};
\endxy}
\;\; = \;\;
 \left\{\\
 \begin{array}{cl}
 t_{ij} \;
\xy 0;/r.17pc/:
  (4,12);(4,-12) **\dir{-}?(.5)*\dir{<};
  (-4,12);(-4,-12) **\dir{-}?(.5)*\dir{<} ;
  (12,12);(12,-12) **\dir{-}?(.5)*\dir{<} ;
  (-6,-9)*{\scs i};     (6.1,-9)*{\scs j};
  (14,-9)*{\scs i};
 \endxy & \text{if $i=k$ and $i\cdot j=-1$}\\ \\
 0 & \text{if $i\neq k$ or $i\cdot j\neq -1$}\\
\end{array}\right.
\]

\item \label{def:mixed}
Mixed $EF$: for $i \neq j$
\begin{equation*}
 \vcenter{   \xy 0;/r.18pc/:
    (-4,-4)*{};(4,4)*{} **\crv{(-4,-1) & (4,1)}?(1)*\dir{>};
    (4,-4)*{};(-4,4)*{} **\crv{(4,-1) & (-4,1)}?(1)*\dir{<};?(0)*\dir{<};
    (-4,4)*{};(4,12)*{} **\crv{(-4,7) & (4,9)};
    (4,4)*{};(-4,12)*{} **\crv{(4,7) & (-4,9)}?(1)*\dir{>};
  (8,8)*{\lambda};(-6,-3)*{\scs i};
     (6,-3)*{\scs j};
 \endxy}
 \;\; = \;\;
\xy 0;/r.18pc/:
  (3,9);(3,-9) **\dir{-}?(.55)*\dir{>}+(2.3,0)*{};
  (-3,9);(-3,-9) **\dir{-}?(.5)*\dir{<}+(2.3,0)*{};
  (8,2)*{\lambda};(-5,-6)*{\scs i};     (5.1,-6)*{\scs j};
 \endxy
\qquad , \qquad
    \vcenter{\xy 0;/r.18pc/:
    (-4,-4)*{};(4,4)*{} **\crv{(-4,-1) & (4,1)}?(1)*\dir{<};?(0)*\dir{<};
    (4,-4)*{};(-4,4)*{} **\crv{(4,-1) & (-4,1)}?(1)*\dir{>};
    (-4,4)*{};(4,12)*{} **\crv{(-4,7) & (4,9)}?(1)*\dir{>};
    (4,4)*{};(-4,12)*{} **\crv{(4,7) & (-4,9)};
  (8,8)*{\lambda};(-6,-3)*{\scs i};
     (6,-3)*{\scs j};
 \endxy}
 \;\;=\;\;
\xy 0;/r.18pc/:
  (3,9);(3,-9) **\dir{-}?(.5)*\dir{<}+(2.3,0)*{};
  (-3,9);(-3,-9) **\dir{-}?(.55)*\dir{>}+(2.3,0)*{};
  (8,2)*{\lambda};(-5,-6)*{\scs i};     (5.1,-6)*{\scs j};
 \endxy
\end{equation*}

\item \label{def:UQ-bub}
Bubble relations:
\[
\xy 0;/r.18pc/:
(-12,0)*{\icbub{\l_i-1+m}{i}};
(-8,8)*{\lambda};
\endxy =
\begin{cases}
c_{i,\l} \Id_{\onel} & \text{if } m=0 \\
0 & \text{if } m<0
\end{cases}
\qquad , \qquad
\xy 0;/r.18pc/: (-12,0)*{\iccbub{-\l_i-1+m}{i}};
 (-8,8)*{\lambda};
 \endxy =
\begin{cases}
c_{i,\l}^{-1} \Id_{\onel} & \text{if } m=0 \\
0 & \text{if } m<0
\end{cases}
\]

\item \label{def:EF}
Extended ${\mathfrak{sl}}_2$ relations:

These final relations are the most-involved, and require the introduction of \emph{fake bubbles} --
positive degree endomorphisms of $\onel$ that are denoted by a bubble carrying a formal label by a negative number of dots.
They are defined by
\[
  \vcenter{\xy 0;/r.18pc/:
    (2,-11)*{\icbub{\l_i-1+j}{i}};
  (12,-2)*{\l};
 \endxy} \;\; =
 \left\{
 \begin{array}{cl}
  \;\; -\; c_{i,\l}\;\;
\displaystyle\sum_{\substack{a+b=j\\ b\geq 1}}
\;\; \vcenter{\xy 0;/r.18pc/:
    (2,0)*{\icbub{\l_i-1+a}{i}};
    (24,0)*{\iccbub{-\l_i-1+b}{i}};
  (12,8)*{\lambda};
 \endxy}  & \text{if $0 < j < -\l_i+1$} \\
   0 & \text{if $j \leq 0$. }
 \end{array}
\right.
\]
when $\l_i<0$, and by
\[
  \vcenter{\xy 0;/r.18pc/:
    (2,-11)*{\iccbub{-\l_i-1+j}{i}};
  (12,-2)*{\l};
 \endxy} \;\; =
 \left\{
 \begin{array}{cl}
  \;\; -\;c_{i,\l}^{-1}\;\;
\displaystyle\sum_{\substack{a+b=j\\ a\geq 1}}
\;\; \vcenter{\xy 0;/r.18pc/:
    (2,0)*{\icbub{\l_i-1+a}{i}};
    (24,0)*{\iccbub{-\l_i-1+b}{i}};
  (12,8)*{\lambda};
 \endxy}  & \text{if $0 <j < \l_i+1$} \\
   0 & \text{if $j \leq 0$. }
 \end{array}
\right.
\]
when $\l_i>0$.
The extended $\mf{sl}_2$ relations are then as follows, where we
employ the convention here (and throughout) that all summations are ``increasing'', \ie $\displaystyle \sum_{\substack{a+b+c\\=\mu}} X_{a,b,c}$ is zero if $\mu<0$.
\[
 \vcenter{\xy 0;/r.17pc/:
  (-8,0)*{};
  (8,0)*{};
  (-4,10)*{}="t1";
  (4,10)*{}="t2";
  (-4,-10)*{}="b1";
  (4,-10)*{}="b2";(-6,-8)*{\scs i};(6,-8)*{\scs i};
  "t1";"b1" **\dir{-} ?(.5)*\dir{<};
  "t2";"b2" **\dir{-} ?(.5)*\dir{>};
  (10,2)*{\l};
  \endxy}
\;\; = \;\; -\;\;
 \vcenter{   \xy 0;/r.17pc/:
    (-4,-4)*{};(4,4)*{} **\crv{(-4,-1) & (4,1)}?(1)*\dir{>};
    (4,-4)*{};(-4,4)*{} **\crv{(4,-1) & (-4,1)}?(1)*\dir{<};?(0)*\dir{<};
    (-4,4)*{};(4,12)*{} **\crv{(-4,7) & (4,9)};
    (4,4)*{};(-4,12)*{} **\crv{(4,7) & (-4,9)}?(1)*\dir{>};
  (8,8)*{\l};
     (-6,-3)*{\scs i};
     (6.5,-3)*{\scs i};
 \endxy}
  \;\; + \;\;
   \sum_{ \xy  (0,3)*{\scs a+b+c}; (0,0)*{\scs =\l_i-1};\endxy}
    \vcenter{\xy 0;/r.17pc/:
    (-12,10)*{\l};
    (-8,0)*{};
  (8,0)*{};
  (-4,-15)*{}="b1";
  (4,-15)*{}="b2";
  "b2";"b1" **\crv{(5,-8) & (-5,-8)}; ?(.05)*\dir{<} ?(.93)*\dir{<}
  ?(.8)*\dir{}+(0,-.1)*{\bullet}+(-3,2)*{\scs c};
  (-4,15)*{}="t1";
  (4,15)*{}="t2";
  "t2";"t1" **\crv{(5,8) & (-5,8)}; ?(.15)*\dir{>} ?(.95)*\dir{>}
  ?(.4)*\dir{}+(0,-.2)*{\bullet}+(3,-2)*{\scs \; a};
  (0,0)*{\iccbub{\scs \quad -\l_i-1+b}{i}};
  (7,-13)*{\scs i};
  (-7,13)*{\scs i};
  \endxy}
\]
\[
 \vcenter{\xy 0;/r.17pc/:
  (-8,0)*{};(-6,-8)*{\scs i};(6,-8)*{\scs i};
  (8,0)*{};
  (-4,10)*{}="t1";
  (4,10)*{}="t2";
  (-4,-10)*{}="b1";
  (4,-10)*{}="b2";
  "t1";"b1" **\dir{-} ?(.5)*\dir{>};
  "t2";"b2" **\dir{-} ?(.5)*\dir{<};
  (10,2)*{\l};
  \endxy}
\;\; = \;\;
  -\;\;\vcenter{\xy 0;/r.17pc/:
    (-4,-4)*{};(4,4)*{} **\crv{(-4,-1) & (4,1)}?(1)*\dir{<};?(0)*\dir{<};
    (4,-4)*{};(-4,4)*{} **\crv{(4,-1) & (-4,1)}?(1)*\dir{>};
    (-4,4)*{};(4,12)*{} **\crv{(-4,7) & (4,9)}?(1)*\dir{>};
    (4,4)*{};(-4,12)*{} **\crv{(4,7) & (-4,9)};
  (8,8)*{\l};(-6.5,-3)*{\scs i};  (6,-3)*{\scs i};
 \endxy}
  \;\; + \;\;
    \sum_{ \xy  (0,3)*{\scs a+b+c}; (0,0)*{\scs =-\l_i-1};\endxy}
    \vcenter{\xy 0;/r.17pc/:
    (-8,0)*{};
  (8,0)*{};
  (-4,-15)*{}="b1";
  (4,-15)*{}="b2";
  "b2";"b1" **\crv{(5,-8) & (-5,-8)}; ?(.1)*\dir{>} ?(.95)*\dir{>}
  ?(.8)*\dir{}+(0,-.1)*{\bullet}+(-3,2)*{\scs c};
  (-4,15)*{}="t1";
  (4,15)*{}="t2";
  "t2";"t1" **\crv{(5,8) & (-5,8)}; ?(.15)*\dir{<} ?(.9)*\dir{<}
  ?(.4)*\dir{}+(0,-.2)*{\bullet}+(3,-2)*{\scs a};
  (0,0)*{\icbub{\scs \quad\; \l_i-1 + b}{i}};
    (7,-13)*{\scs i};
  (-7,13)*{\scs i};
  (-10,10)*{\l};
  \endxy}
\]
\end{enumerate}

\begin{rem}\label{rem:presentation}
We will find it helpful to work with the reduced presentation for $\Ucat_Q$ where we restrict to the following generating 2-morphisms:
\begin{align*}
  \xy 0;/r.17pc/:
 (0,0)*{\sdotur{i}};
 (7,3)*{ \scs \lambda};
 (-9,3)*{\scs  \lambda+\alpha_i};
 (-10,0)*{};(10,0)*{};
 \endxy &\maps \cal{E}_i\onel \to \cal{E}_i\onel\la 2 \ra
 & \quad &
   \xy 0;/r.17pc/:
  (0,0)*{\xybox{
(0,0)*{\ucrossrr{i}{j}};
     (9,1)*{\scs  \lambda};
     (-10,0)*{};(10,0)*{};
     }};
  \endxy\maps \cal{E}_i\cal{E}_j\onel  \to \cal{E}_j\cal{E}_i\onel\la -i\cdot j \ra  \\
     \xy 0;/r.17pc/:
    (0,0)*{\rcupr{i}};
    (8,-5)*{\scs  \lambda};
    (-10,0)*{};(10,0)*{};
    \endxy &\maps \onel  \to \cal{F}_i\cal{E}_i\onel\la 1 + \l_i \ra   & &
   \xy 0;/r.17pc/:
    (0,0)*{\lcupr{i}};
    (8,-5)*{\scs \lambda};
    (-10,0)*{};(10,0)*{};
    \endxy \maps \onel  \to\cal{E}_i\cal{F}_i\onel\la 1 - \l_i \ra   \\
      & & &  \\
  \xy 0;/r.17pc/:
    (0,0)*{\lcapr{i}};
    (8,4)*{\scs  \lambda};
    (-10,0)*{};(10,0)*{};
    \endxy & \maps \cal{F}_i\cal{E}_i\onel \to\onel\la 1 + \l_i \ra  & &
 \xy 0;/r.17pc/:
    (0,0)*{\rcapr{i}};
    (8,4)*{\scs  \lambda};
    (-10,0)*{};(10,0)*{};
    \endxy \maps\cal{E}_i\cal{F}_i\onel  \to\onel\la 1 - \l_i \ra
\end{align*}
Indeed, the downward dot $2$-morphism and sideways and downward crossings can be defined in various ways by composing the upward versions with caps and cups,
and the cyclicity relations show that they do not depend on the choices made in doing so.
Further, Brundan~\cite{Brundan2} has shown that this presentation can be further simplified to agree with the one given by Rouquier~\cite{Rou2} that requires a smaller set of axioms,
together with the requirement that certain $2$-morphisms are (abstractly) invertible.
Although this further reduced presentation is helpful in checking that biadjointness and cyclicity hold in various $2$-representations,
it is not useful in our present work, as showing that the required maps are invertible essentially requires verifying the omitted axioms in $\Ucat_Q$.
\end{rem}

%
\subsection{Additional relations in \texorpdfstring{$\Ucat_Q$}{UQ}}
%

Here, we collect additional useful relations that will be used in later sections.

%
\subsubsection{Curl relations}\label{sec:curl}
%

Dotted curls can be reduced to simpler diagrams using the following.
\[
  \xy 0;/r.17pc/:
  (14,8)*{\l};
  (-3,-10)*{};(3,5)*{} **\crv{(-3,-2) & (2,1)}?(1)*\dir{>};?(.15)*\dir{>};
    (3,-5)*{};(-3,10)*{} **\crv{(2,-1) & (-3,2)}?(.85)*\dir{>} ?(.1)*\dir{>};
  (3,5)*{}="t1";  (9,5)*{}="t2";
  (3,-5)*{}="t1'";  (9,-5)*{}="t2'";
   "t1";"t2" **\crv{(4,8) & (9, 8)};
   "t1'";"t2'" **\crv{(4,-8) & (9, -8)};
   "t2'";"t2" **\crv{(10,0)} ;
(9.5,0)*{\bullet}+(3,2)*{\scs m};
   (-6,-8)*{\scs i};
 \endxy \;\; =
 \;\; -
   \sum_{ \xy  (0,3)*{\scs f_1+f_2 }; (0,0)*{\scs =m-\l_i};\endxy}
 \xy 0;/r.17pc/:
  (19,4)*{\l};
  (0,0)*{\bber{}};(-2,-8)*{\scs };
  (-2,-8)*{\scs i};
  (12,-2)*{\icbub{\l_i-1+f_2}{i}};
  (0,6)*{\bullet}+(3,-1)*{\scs f_1};
 \endxy
\quad , \quad
  \xy 0;/r.17pc/:
  (-14,8)*{\l};
  (3,-10)*{};(-3,5)*{} **\crv{(3,-2) & (-2,1)}?(1)*\dir{>};?(.15)*\dir{>};
    (-3,-5)*{};(3,10)*{} **\crv{(-2,-1) & (3,2)}?(.85)*\dir{>} ?(.1)*\dir{>};
  (-3,5)*{}="t1";  (-9,5)*{}="t2";
  (-3,-5)*{}="t1'";  (-9,-5)*{}="t2'";
   "t1";"t2" **\crv{(-4,8) & (-9, 8)};
   "t1'";"t2'" **\crv{(-4,-8) & (-9, -8)};
   "t2'";"t2" **\crv{(-10,0)} ;
(-9.5,0)*{\bullet}+(-3,2)*{\scs m};
   (6,-8)*{\scs i};
 \endxy \;\; = \;\;
 \sum_{ \xy  (0,3)*{\scs g_1+g_2 }; (0,0)*{\scs =m+\l_i};\endxy}
  \xy 0;/r.17pc/:
  (5,-8)*{\scs i};
  (-12,8)*{\l};
  (3,0)*{\bber{}};(2,-8)*{\scs};
  (-12,-2)*{\iccbub{-\l_i-1+g_2}{i}};
  (3,6)*{\bullet}+(3,-1)*{\scs g_1};
 \endxy
\]
Note that in \cite{Lau1,CLau} the $m=0$ cases of these relations were included in the defining list of relations,
but it was shown in \cite[Lemma 3.2]{BHLW2} that these relations (for arbitrary $m$) follow from the relations presented above.

%
\subsubsection{Infinite Grassmannian relations}\label{sec:inf}
%
These relations are obtained by equating the terms homogeneous in $t$ in the expression below.
\[
 \makebox[0pt][c]{ $
\left( \xy 0;/r.15pc/:
 (0,0)*{\iccbub{-\l_i-1}{i}};
  (4,8)*{\l};
 \endxy
 +
 \xy 0;/r.15pc/:
 (0,0)*{\iccbub{-\l_i-1+1}{i}};
  (4,8)*{\l};
 \endxy t
 + \cdots +
\xy 0;/r.15pc/:
 (0,0)*{\iccbub{-\l_i-1+\alpha}{i}};
  (4,8)*{\l};
 \endxy t^{\alpha}
 + \cdots
\right)
\left(\xy 0;/r.15pc/:
 (0,0)*{\icbub{\l_i-1}{i}};
  (4,8)*{\l};
 \endxy
 + \xy 0;/r.15pc/:
 (0,0)*{\icbub{\l_i-1+1}{i}};
  (4,8)*{\l};
 \endxy t
 +\cdots +
\xy 0;/r.15pc/:
 (0,0)*{\icbub{\l_i-1+\alpha}{i}};
 (4,8)*{\l};
 \endxy t^{\alpha}
 + \cdots
\right) = \Id_{\onel}$ }
\]
For low powers of $t$, these relations encode the definition of fake bubbles in terms of (real) bubbles,
and, for higher powers of $t$, they follow from the curl and extended $\mf{sl}_2$-relations.

%
\subsubsection{Bubble slides}\label{sec:bubble}
%
In what follows, we make use of the shorthand notation  for bubbles from \cite{KLMS}
\[
    \xy 0;/r.18pc/:
  (4,8)*{\lambda};
  (2,-2)*{\icbub{\spadesuit+\alpha}{i}};
 \endxy \;\; := \;\;
   \xy 0;/r.18pc/:
  (4,8)*{\lambda};
  (2,-2)*{\icbub{\la i,\lambda\ra-1+\alpha}{i}};
 \endxy
 \qquad
 \qquad
    \xy 0;/r.18pc/:
  (4,8)*{\lambda};
  (2,-2)*{\iccbub{\spadesuit+\alpha}{i}};
 \endxy \;\; := \;\;
   \xy 0;/r.18pc/:
  (4,8)*{\lambda};
  (2,-2)*{\iccbub{-\la i,\lambda\ra-1+\alpha}{i}};
 \endxy
\]
As long as $\alpha \geq 0$, this notation makes sense even when $\spadesuit+\alpha <0$,
in which case these are the fake bubbles defined in the previous section.

Counterclockwise bubbles can be slid through upward oriented lines via the following relations:
\begin{align*}
    \xy 0;/r.18pc/:
  (14,8)*{\lambda};
  (0,0)*{\bber{}};
  (0,-12)*{\scs j};
  (12,-2)*{\iccbub{\spadesuit+\alpha}{i}};
  (0,6)*{ }+(7,-1)*{\scs  };
 \endxy
&=
 \left\{
 \begin{array}{cl}
  \xsum{f=0}{\alpha}(\alpha+1-f)
   \xy 0;/r.18pc/:
  (0,8)*{\lambda+\alpha_j};
  (12,0)*{\bber{}};
  (12,-12)*{\scs j};
  (0,-2)*{\iccbub{\spadesuit+f}{i}};
  (12,6)*{\bullet}+(5,-1)*{\scs \alpha-f};
 \endxy
& \text{if } i=j \\ \\
     t_{ij}\;   \xy 0;/r.18pc/:
  (0,8)*{\lambda+\alpha_j};
  (12,0)*{\bber{}};
  (11,-12)*{\scs j};
  (0,-2)*{\iccbub{\spadesuit+\alpha}{i}};
 \endxy
 \quad + \quad  t_{ji} \;\;
  \xy 0;/r.18pc/:
  (0,8)*{\lambda+\alpha_j};
  (12,0)*{\bber{}};
  (12,-12)*{\scs j};
  (0,-2)*{\iccbub{\spadesuit+\alpha-1}{i}};
  (12,6)*{\bullet}+(5,-1)*{\scs };
 \endxy
& \text{if } a_{ij} =-1 \\
 \qquad \qquad t_{ij} \; \xy 0;/r.18pc/:
  (0,8)*{\lambda+\alpha_j};
  (12,0)*{\bber{}};
  (12,-12)*{\scs j};
  (0,-2)*{\iccbub{\spadesuit+\alpha}{i}};
 \endxy & \text{if } a_{ij}=0
 \end{array}
 \right. \\
\xy 0;/r.18pc/:
  (0,8)*{\lambda+\alpha_j};
  (12,0)*{\bber{}};
  (12,-12)*{\scs j};
  (0,-2)*{\iccbub{\spadesuit+\alpha}{i}};
  (12,6)*{}+(8,-1)*{\scs };
 \endxy
&=
\left\{
\begin{array}{cl}
    \xy 0;/r.18pc/:
  (15,8)*{\lambda};
  (0,0)*{\bber{}};
  (0,-12)*{\scs j};
  (12,-2)*{\iccbub{\spadesuit+(\alpha-2)}{i}};
  (0,6)*{\bullet }+(3,1)*{\scs 2};
 \endxy
  -2 \;
      \xy 0;/r.18pc/:
  (15,8)*{\lambda};
  (0,0)*{\bber{}};
  (0,-12)*{\scs j};
  (12,-2)*{\iccbub{\spadesuit+(\alpha-1)}{i}};
  (0,6)*{\bullet }+(5,-1)*{\scs };
 \endxy
 + \;\;
      \xy 0;/r.18pc/:
  (15,8)*{\lambda};
  (0,0)*{\bber{}};
  (0,-12)*{\scs j};
  (12,-2)*{\iccbub{\spadesuit+\alpha}{i}};
 \endxy
  &   \text{if } i=j \\
t_{ij}^{-1}\;
   \xsum{f=0}{\alpha}(-t_{ij}^{-1}t_{ji})^f
    \xy 0;/r.18pc/:
  (15,8)*{\lambda};
  (0,0)*{\bber{}};
  (0,-12)*{\scs j};
  (14,-2)*{\iccbub{\spadesuit+(\alpha-f)}{i}};
  (0,6)*{\bullet }+(3,1)*{\scs f};
 \endxy
    &   \text{if } a_{ij} =-1
\end{array}
\right.
\end{align*}
and we have similar relations involving clockwise bubbles:
\begin{align*}
    \xy 0;/r.18pc/:
  (15,8)*{\lambda};
  (11,0)*{\bber{}};
  (11,-12)*{\scs j};
  (0,-2)*{\icbub{\spadesuit+\alpha\quad }{i}};
 \endxy
&=
  \left\{\begin{array}{cl}
     \xsum{f=0}{\alpha}(\alpha+1-f)
     \xy 0;/r.18pc/:
  (18,8)*{\lambda};
  (0,0)*{\bber{}};
  (0,-12)*{\scs j};
  (14,-4)*{\icbub{\spadesuit+f}{i}};
  (0,6)*{\bullet }+(5,-1)*{\scs \alpha-f};
 \endxy
         & \text{if } i=j  \\
  t_{ji}\;\;  \xy 0;/r.18pc/:
  (18,8)*{\lambda};
  (0,0)*{\bber{}};
  (0,-12)*{\scs j};
  (12,-2)*{\icbub{\spadesuit+ \alpha-1}{i}};
    (0,6)*{\bullet }+(5,-1)*{\scs};
 \endxy
 \quad + \quad t_{ij} \;
\xy 0;/r.18pc/:
  (18,8)*{\lambda};
  (0,0)*{\bber{}};
  (0,-12)*{\scs j};
  (12,-2)*{\icbub{\spadesuit+ \alpha}{i}};
 \endxy
       &  \text{if } a_{ij} =-1\\
t_{ji} \;
    \xy 0;/r.18pc/:
  (15,8)*{\lambda};
  (0,0)*{\bber{}};
  (0,-12)*{\scs j};
  (12,-2)*{\icbub{\spadesuit+\alpha}{i}};
 \endxy          & \text{if } a_{ij} = 0
         \end{array}
 \right. \\
\xy 0;/r.18pc/:
  (15,8)*{\lambda};
  (0,0)*{\bber{}};
  (0,-12)*{\scs j};
  (12,-2)*{\icbub{\spadesuit+\alpha}{i}};
 \endxy
   &=
  \left\{
  \begin{array}{cl}
     \xy 0;/r.18pc/:
  (0,8)*{\lambda+\alpha_i};
  (12,0)*{\bber{}};
  (12,-12)*{\scs j};
  (0,-2)*{\icbub{\spadesuit+(\alpha-2)}{i}};
  (12,6)*{\bullet}+(3,-1)*{\scs 2};
 \endxy
   -2 \;
         \xy 0;/r.18pc/:
  (0,8)*{\lambda+\alpha_i};
  (12,0)*{\bber{}};
  (12,-12)*{\scs j};
  (0,-2)*{\icbub{\spadesuit+(\alpha-1)}{i}};
  (12,6)*{\bullet}+(8,-1)*{\scs };
 \endxy
 + \;\;
     \xy 0;/r.18pc/:
  (0,8)*{\lambda+\alpha_i};
  (12,0)*{\bber{}};
  (12,-12)*{\scs j};
  (0,-2)*{\icbub{\spadesuit+\alpha}{i}};
  (12,6)*{}+(8,-1)*{\scs };
 \endxy
  &   \text{if } i = j \\
  t_{ij}^{-1}\;\xsum{f=0}{\alpha} (-t_{ij}^{-1}t_{ji})^{f}
  \xy 0;/r.18pc/:
  (0,8)*{\lambda+\alpha_j};
  (14,0)*{\bber{}};
  (12,-12)*{\scs j};
  (0,-2)*{\icbub{\spadesuit+\alpha-f}{i}};
  (14,6)*{\bullet}+(3,-1)*{\scs f};
 \endxy &   \text{if } a_{ij} =-1.
  \end{array}
 \right.
\end{align*}
Both types of bubbles can then be slid through downward oriented lines
using these relations and the cyclicity of $\Ucat_Q(\mf{g})$.

%
\subsubsection{Triple intersections}\label{sec:triple}
%
We have
\begin{equation}\label{eq:otherR3}
\vcenter{\xy 0;/r.17pc/:
    (-4,-4)*{};(4,4)*{} **\crv{(-4,-1) & (4,1)}?(0)*\dir{};
    (4,-4)*{};(-4,4)*{} **\crv{(4,-1) & (-4,1)}?(0)*\dir{<};
    (4,4)*{};(12,12)*{} **\crv{(4,7) & (12,9)}?(1)*\dir{};
    (12,4)*{};(4,12)*{} **\crv{(12,7) & (4,9)}?(1)*\dir{};
    (-4,12)*{};(4,20)*{} **\crv{(-4,15) & (4,17)}?(1)*\dir{};
    (4,12)*{};(-4,20)*{} **\crv{(4,15) & (-4,17)}?(1)*\dir{};
    (-4,4)*{}; (-4,12) **\dir{-};
    (12,-4)*{}; (12,4) **\dir{-};
    (12,12)*{}; (12,20) **\dir{-};
    (4,20); (4,21) **\dir{-}?(1)*\dir{};
    (-4,20); (-4,21) **\dir{-}?(1)*\dir{>};
    (12,20); (12,21) **\dir{-}?(1)*\dir{>};
   (18,8)*{\lambda};  (-6,-3)*{\scs i};
  (6.5,-3)*{\scs j};
  (15,-3)*{\scs k};
\endxy}
\;\; - \;\;
\vcenter{\xy 0;/r.17pc/:
    (4,-4)*{};(-4,4)*{} **\crv{(4,-1) & (-4,1)}?(0)*\dir{};
    (-4,-4)*{};(4,4)*{} **\crv{(-4,-1) & (4,1)}?(0)*\dir{<};
    (-4,4)*{};(-12,12)*{} **\crv{(-4,7) & (-12,9)}?(1)*\dir{};
    (-12,4)*{};(-4,12)*{} **\crv{(-12,7) & (-4,9)}?(1)*\dir{};
    (4,12)*{};(-4,20)*{} **\crv{(4,15) & (-4,17)}?(1)*\dir{};
    (-4,12)*{};(4,20)*{} **\crv{(-4,15) & (4,17)}?(1)*\dir{};
    (4,4)*{}; (4,12) **\dir{-};
    (-12,-4)*{}; (-12,4) **\dir{-};
    (-12,12)*{}; (-12,20) **\dir{-};
    (4,20); (4,21) **\dir{-}?(1)*\dir{>};
    (-4,20); (-4,21) **\dir{-}?(1)*\dir{};
    (-12,20); (-12,21) **\dir{-}?(1)*\dir{>};
  (10,8)*{\lambda};
  (-14,-3)*{\scs i};
  (-6.5,-3)*{\scs j};
  (6,-3)*{\scs k};
\endxy}
=
\begin{cases}
\displaystyle
\sum_{\substack{a+b+c+d \\ = \l_i}}
\vcenter{\xy 0;/r.17pc/:
    (20,10)*{\l};
    (-8,0)*{};
  (8,0)*{};
  (-4,-15)*{}="b1";
  (4,-15)*{}="b2";
  "b2";"b1" **\crv{(5,-8) & (-5,-8)}; ?(.05)*\dir{<} ?(.93)*\dir{<}
  ?(.8)*\dir{}+(0,-.1)*{\bullet}+(-3,2)*{\scs c};
  (-4,15)*{}="t1";
  (4,15)*{}="t2";
  "t2";"t1" **\crv{(5,8) & (-5,8)}; ?(.15)*\dir{>} ?(.95)*\dir{>}
  ?(.4)*\dir{}+(0,-.2)*{\bullet}+(3,-2)*{\scs \; a};
  (0,0)*{\iccbub{\scs \quad \spadesuit+b}{i}};
  (7,-13)*{\scs i};
  (-7,13)*{\scs i};
(15,-15); (15,15) **\dir{-}?(1)*\dir{>};
 (17,-13)*{\scs i};
 (15,0)*{\bullet}+(3,2)*{\scs d};
  \endxy}
+
\sum_{\substack{a+b+c+d \\ = -\l_i-2}}
\vcenter{\xy 0;/r.17pc/:
    (12,10)*{\l};
    (-8,0)*{};
  (8,0)*{};
  (-4,-15)*{}="b1";
  (4,-15)*{}="b2";
  "b2";"b1" **\crv{(5,-8) & (-5,-8)}; ?(.07)*\dir{>} ?(.97)*\dir{>}
  ?(.8)*\dir{}+(0,-.1)*{\bullet}+(-3,2)*{\scs c};
  (-4,15)*{}="t1";
  (4,15)*{}="t2";
  "t2";"t1" **\crv{(5,8) & (-5,8)}; ?(.05)*\dir{<} ?(.9)*\dir{<}
  ?(.4)*\dir{}+(0,-.2)*{\bullet}+(3,-2)*{\scs \; a};
  (0,0)*{\iccbub{\scs \quad \spadesuit+b}{i}};
  (7,-13)*{\scs i};
  (-7,13)*{\scs i};
(-15,-15); (-15,15) **\dir{-}?(1)*\dir{>};
 (-17,-13)*{\scs i};
 (-15,0)*{\bullet}+(3,2)*{\scs d};
  \endxy}
  & \text{if } i=j=k \\ \\
0 & \text{else}
\end{cases}
\end{equation}
which is \cite[Proposition 5.8]{Lau1} when $i=j=k$, and follows from
cyclicity, the mixed $EF$ relation, and the cubic KLR relation in the other case.

%
\subsection{The 2-categories \texorpdfstring{$\UcatD_Q$}{UQ}, \texorpdfstring{$\Kom(\UcatD_Q)$}{Kom U}, and \texorpdfstring{$\Com(\UcatD_Q)$}{Com U}} \label{sec_KomU}
%


\subsubsection{Categories of complexes}
Given an additive category $\cal{M}$, we let $\Kom(\cal{M})$ denote the category of bounded
complexes in $\cal{M}$.
By convention, we work with \emph{cochain} complexes, so an object $(X,d)$ of $\Kom(\cal{M})$ is a collection of objects $X^i$ in $\cal{M}$
together with maps
\[
\cdots \xrightarrow{d_{i-2}} X^{i-1} \xrightarrow{d_{i-1}} X^i \xrightarrow{d_i} X^{i+1} \xrightarrow{d_{i+1}} \cdots
\]
such that $d_{i+1} d_i =0$ and only finitely many of the $X^i$'s are nonzero.
A morphism $f \maps (X,d) \to (Y,d')$ in $\Kom(\cal{M})$ consists of
a collection of morphisms $f_i \maps X^i \to Y^i$ in $\cal{M}$ such that $f_{i+1} d_i = d'_i f_i$.

Recall that morphisms $f,g \maps (X,d) \to (Y,d')$ in $\Kom(\cal{M})$ are called \emph{(chain) homotopic}, denoted by $f \sim g$,
if there exist morphisms $h^i \maps X^{i} \to Y^{i-1}$ such that
$f_i-g_i = h^{i+1}d_i+d'_{i-1}h^i$ for all $i$.
A morphism of complexes is said to be null-homotopic if it is homotopic to the zero map.

\begin{definition}
The homotopy category $\Com(\cal{M})$ is the addditive category with the same objects as $\Kom(\cal{M})$
with morphisms given by morphisms in $\Kom(\cal{M})$ modulo null-homotopic morphisms.
\end{definition}
We say that two complexes $(X,d_X)$ and $(Y,d_Y)$ are homotopy equivalent provided they are isomorphic in $\Com(\cal{M})$,
and denote this by $X \simeq Y$.

If $\cal{M}$ is monoidal, then $\Kom(\cal{M})$ is also monoidal, with the tensor product $(XY,d)$ of
$(X,d_X)$ and $(Y,d_Y)$ defined as follows:
\begin{equation}\label{eq:CompOfCom}
(XY)^i = \bigoplus_{r+s=i} X^{r}Y^{s} \;\; , \;\; d_i := \displaystyle\sum_{r+s=i} (d_X)_r \Id_{Y^s}+(-1)^{r} \Id_{X^r}(d_Y)_s
\end{equation}
Here, we denote the tensor product of objects and morphisms in $\cal{M}$ by juxtaposition.
Given chain maps $f \maps (X,d_X) \to (X',d_{X'})$ and $g \maps (Y,d_Y) \to (Y',d_{Y'})$ define the tensor product $fg \maps (XY,d) \to (X'Y',d')$ of chain maps by setting
\begin{equation} \label{eq_tensor_maps}
f_i =  \bigoplus_{r+s=i} f_r g_s.
\end{equation}
It is straightforward to check that if $f\sim f'$ and $g \sim g'$, then $fg \sim f'g'$,
so $\Com(\cal{M})$ inherits a monoidal structure from $\Kom(\cal{M})$.

\begin{rem}\label{rem:KomC}
More generally\footnote{Recall that a monoidal category can be interpreted as a $2$-category with only one object.},
if $\cal{C}$ is an additive $2$-category, we can consider the $2$-categories $\Kom(\cal{C})$ and $\Com(\cal{C})$ obtained by
taking complexes in each $\Hom$-category.
The above description of tensor product of complexes specifies how to take horizontal composition in $\Kom(\cal{C})$ and $\Com(\cal{C})$.
\end{rem}

\subsubsection{Karoubi envelope}

The Karoubi envelope $Kar(\cal{M})$ of a category $\cal{M}$ is the universal enlargement of $\cal{M}$ in which all idempotents split.
Recall that we say an idempotent $e \maps b\to b$ in a category $\cal{M}$ splits if there
exist morphisms $b \xrightarrow{g} b' \xrightarrow{h} b$ such that $e=hg$ and $gh = \Id_{b'}$.
The Karoubi envelope $Kar(\cal{M})$ admits an explicit description as the
category whose objects are pairs $(b,e)$, where $e \maps b \to b$ is an idempotent of $\cal{M}$,
and where morphisms are triples of the form
\[
 (e,f,e') \maps (b,e) \to (b',e')
\]
for $f \maps b \to b'$ in $\cal{M}$ satisfying $f=e'f=fe$.
Composition is induced from composition in $\cal{M}$, and the identity morphism is $(e,e,e) \maps (b,e) \to (b,e)$.

The identity map $\Id_b\maps b \to b$ is an idempotent,
and the assignment $b \mapsto (b, \Id_b)$ defines a fully faithful functor $\cal{M} \hookrightarrow Kar(\cal{M})$,
and this functor is universal among functors from $\cal{M}$ to idempotent split categories.
If $\cal{M}$ is additive then so is $Kar(\cal{M})$ and this embedding is additive;
in this case, for $(b,e)\in Kar(\cal{M})$, we have that $b \cong \im e \oplus \im (\Id_b-e)$ where $\im e := (b,e)$.
See \cite[Section 9]{Lau1} and references therein for more details.

The following result shows that the Karoubi envelope interacts nicely with passage to
(homotopy) categories of complexes.

\begin{prop}[{\cite[Propositions 3.6 and 3.7]{BKL-Casimir}}]\label{prop_Kar}
For any additive category $\cal{M}$ there is a canonical equivalence
$
\Kom\left(Kar(\cal{M})\right) \cong Kar\left(\Kom(\cal{M})\right).
$
Moreover, if $\cal{M}$ is $\Bbbk$-linear with finite-dimensional $\Hom$-spaces, then there is a canonical equivalence
$
\Com\left(Kar(\cal{M})\right) \cong Kar\left(\Com(\cal{M})\right).
$
\end{prop}

%
\subsubsection{Karoubi envelope of \texorpdfstring{$\Ucat_Q$}{U}}
%

\begin{definition}
The additive 2-category $\UcatD_Q$ has the same objects as $\Ucat_Q$ and has $\Hom$-categories given by
$\UcatD_Q(\lambda,\lambda') = Kar\left(\Ucat_Q(\lambda,\lambda')\right)$.
\end{definition}

Horizontal composition in $\UcatD_Q$ is induced from composition in $\Ucat_Q$ using the universal property of the
Karoubi envelope, and we similarly obtain an additive, fully-faithful 2-functor $\Ucat_Q \to \UcatD_Q$
that is universal with respect to splitting idempotents in the $\Hom$-categories $\UcatD_Q(\lambda,\lambda')$.
The significance of the 2-category $\UcatD_Q(\mf{g})$ is given by the following theorem.

\begin{thm} \label{thm_Groth}
(\cite{Lau1,KL3,Web5})
There is an isomorphism
$
\gamma\maps  \UA \xrightarrow{\cong} K_0(\UcatD_Q(\mf{g}))
$
where $K_0(\UcatD_Q)$ denotes the split Grothendieck ring of $\UcatD_Q$.
\end{thm}
For $\frak{g} = \mf{sl}_2$, this theorem also holds over $\Z$ by the results in \cite{KLMS}.

%
\subsubsection{Karoubian envelopes of \texorpdfstring{$\Kom(\Ucat)$}{Kom U} and \texorpdfstring{$\Com(\Ucat)$}{Com U}}
%

Following Remark \ref{rem:KomC} above, we consider the $2$-categories $\Kom(\Ucat_Q)$ and $\Com(\Ucat_Q)$.
Noting that the 2-$\Hom$-spaces $\Ucat_Q(x,y\la t\ra)$ are finite-dimensional $\Bbbk$-vector space for each $t\in \Z$,
Proposition\ref{prop_Kar} gives equivalences
\[
  Kar(\Kom(\Ucat_Q)) \cong \Kom(\UcatD_Q) \;\; , \;\; Kar(\Com(\Ucat_Q)) \cong \Com(\UcatD_Q).
\]
We arrange the various $2$-categories built from $\Ucat_Q$ into the following organizational diagram,
wherein the horizontal arrows denote passage to the Karoubian envelope, and vertical arrows denote the canonical maps between the various categories of complexes.
\[
 \xy
  (-25,15)*+{\Ucat_Q}="tl";
  (25,15)*+{\UcatD = Kar(\Ucat_Q)}="tr";
  (-25,0)*+{\Kom(\Ucat_Q)}="ml";
  (25,0)*+{\Kom(\UcatD_Q) \cong Kar(\Kom(\Ucat))}="mr";
  (-25,-15)*+{\Com(\Ucat_Q)}="bl";
  (25,-15)*+{\Com(\UcatD_Q) \cong Kar(\Com(\Ucat))}="br";
    {\ar@{^{(}->} "tl";"ml"};
    {\ar@{->>} "ml";"bl"};
    {\ar@{^{(}->} "tr";"mr"};
    {\ar@{->>} "mr";"br"};
    {\ar@{^{(}->} "tl";"tr"};
    {\ar@{^{(}->} "ml";"mr"};
    {\ar@{^{(}->} "bl";"br"};
 \endxy
\]

Theorem~\ref{thm_Groth} and the main result of \cite{Rose} imply that
\[
K_0\left( Kar\left( \Com(\Ucat_Q) \right)\right) \cong
K_0\left( \Com\left( Kar(\Ucat_Q) \right)\right) \cong
K_0\left( Kar(\Ucat_Q) \right) \cong K_0(\UcatD_Q) \cong \UA
\]
where we employ the triangulated Grothendieck group for the categories of complexes.
We can hence view the Karoubi envelope of the homotopy category $\Com(\Ucat_Q)$ as a categorification of the integral idempotent form $\UA$ of the quantum group.

%
\subsection{Symmetries of Categorified Quantum Groups} \label{sec_symm}
%

In this section, we use symmetries of the diagrammatic relations in $\cal{U}_Q$ to define 2-functors $\sigma$, $\omega$, and $\psi$ (for a general choice of scalars $Q$ and bubble parameters $C$)
that lift the symmetries of quantum groups from Section \ref{sec_Usymm}
This extends the work of Khovanov and Lauda in \cite{KL3},
who defined such functors in the specific case where $t_{ij}=1=c_{i,\lambda}$ for all $i,j\in I$ and $\lambda\in X$.
These 2-functors extend naturally to 2-functors on $\UcatD_Q$, $\Kom(\UcatD_Q)$, and $\Com(\UcatD_Q)$ \cite{BKL-Casimir},
and induce the corresponding quantum group symmetries $\underline{\sigma}$, $\underline{\omega}$, and $\underline{\psi}$ on $\UA$ upon passing to $K_0$.
For this reason, we refer to them as symmetry $2$-functors.

Rather than being 2-endofunctors of $\cal{U}_Q$, some of these symmetries map between
versions $\cal{U}_Q$ and $\cal{U}'_Q$ of the categorified quantum group corresponding to \emph{different} bubble parameters.
(\emph{Caveat lector}: $\cal{U}'_Q$ should \textbf{not} be confused with $\cal{U}_{Q'}$ from \cite{CLau} which instead corresponds to a different choice of scalars $Q$.)
We define $\cal{U}'_Q$ to be the 2-category given in Definition \ref{defU_cat-cyc}, but with the bubble parameters for $\cal{U}_Q$ replaced by primed bubble parameters $(c_{i,\lambda})':=c_{i,-\lambda}^{-1}$.
The primed bubble parameters are still compatible with the choice of scalars $Q$ (used for both $\cal{U}_Q$ and $\cal{U}'_Q$),
since
\[
\frac{(c_{i,\lambda+\alpha_j})'}{(c_{i,\lambda})'} =
\frac{c_{i,-(\lambda+\alpha_j)}^{-1}} {c_{i,-\lambda}^{-1}} =
\frac{c_{i,-\lambda}}{c_{i,-\lambda-\alpha_j}}=t_{ij}.
\]
In addition to mapping between versions of the categorified quantum group corresponding to different bubble parameters,
the symmetry 2-functors possess various flavors of contravariance.
Nevertheless, they are morally pairwise-commuting involutions, as the double application of a given symmetry is the identity and the result of a composition does not depend on the order,
despite the domain and codomain being different versions of the categorified quantum group.
Given this, we will slightly abuse notation and refer to the symmetry and its inverse by the same symbol.

\subsubsection{Forms of 2-categorical contravariance }

Recall that a contravariant functor $\mathbf{C} \to \mathbf{D}$ can be rephrased in terms of a (covariant) functor $\mathbf{C} \to \mathbf{D}^{\op}$,
where $\mathbf{D}^{\op}$ is the opposite category, defined to have the same objects as in $\mathbf{D}$,
but with $\mathbf{D}^{\op}(x,y) := \mathbf{D}(y,x)$, \ie the direction of the morphisms is opposite to that in $\mathbf{D}$.
For a 2-category $\cal{C}$, we can take the opposite 2-category in various ways,
depending on whether we take the opposite at the 1-morphism or 2-morphism level (or both).
Denote by $\cal{C}^{\op}$ the 2-category with the same objects as $\cal{C}$, and where we've taken the opposite with respect to 1-morphisms,
\ie for objects $x,y$ in $\cal{C}$, we let the $\Hom$-categories be given by $\cal{C}^{\op}(x,y) := \cal{C}(y,x)$.
Let $\cal{C}^{\co}$ denote the 2-category with the same objects and 1-morphisms as $\cal{C}$, but with opposite 2-morphisms,
\ie for objects $x,y$ in $\cal{C}$, we let the $\Hom$-categories be given by $\cal{C}^{\co}(x,y) := \cal{C}(x,y)^{\op}$.
Finally, $\cal{C}^{\co\op}$ is the 2-category in which we've taken opposite 1-morphisms and 2-morphisms,
\ie $\cal{C}^{\co\op}(x,y) := \cal{C}(y,x)^{\op}$.

In the case of $\cal{U}_Q$, functors between these opposite 2-categories correspond to $\Z[q,q^{-1}]$-(anti)linear algebra (anti)automorphisms of $\UA$ upon taking the Grothendieck group,
as summarized in the following table:
\begin{align*}
\begin{tabular}{|l|l|}
  \hline
  {\bf 2-functor} & {\bf Induced map on $\UA$} \\ \hline \hline
  $\UcatD_Q \to \UcatD_Q$ &  $\Z[q,q^{-1}]$-linear
 homomorphism\\
  $\UcatD_Q \to \UcatD_Q^{\op}$ & $\Z[q,q^{-1}]$-linear
antihomomorphism \\
  $\UcatD_Q \to \UcatD_Q^{\co}$ & $\Z[q,q^{-1}]$-antilinear
 homomorphism \\
  $\UcatD_Q \to \UcatD_Q^{\co\op}$ & $\Z[q,q^{-1}]$-antilinear
antihomomorphism \\
  \hline
\end{tabular}
\end{align*}

In the following sections, we will explicitly describe $\sigma$, $\omega$, and $\psi$.
To do so, we will use the notational convention from \cite{KL3} that $\cal{E}_{-i} := \cal{F}_i$.

%
\subsubsection{The 2-functor \texorpdfstring{$\sigma \maps \Ucat_Q  \to \left(\Ucat'_Q\right)^{\op}$}{sigma} }
%
Consider the operation on the diagrammatic calculus or $\Ucat_Q$ that
reflects a diagram across a vertical axis,
replaces $\lambda \leftrightarrow -\lambda$,
and scales all $ii$-crossings by $-1$.
This operation is contravariant for composition of 1-morphisms, covariant for composition of 2-morphisms,
preserves the degree of a diagram,
and takes relations in $\Ucat_Q$ to those in $\cal{U}'_Q$.
As such, it defines an invertible 2-functor given explicitly as follows:
\begin{align*}
  \sigma \maps \Ucat_Q &\to \left(\Ucat'_Q\right)^{\op} \\
  \lambda &\mapsto  -\lambda \\
  \onem\cal{E}_{\pm i_1} \cal{E}_{\pm i_2} \cdots
 \cal{E}_{\pm i_m}\onel\la t\ra
 &\mapsto
 \mathbf{1}_{-\lambda} \cal{E}_{\pm i_m}\cdots
 \cal{E}_{\pm i_2}\cal{E}_{\pm i_1}\mathbf{1}_{-\mu}\la t\ra \\ \\
\vcenter{\xy 0;/r.15pc/:(0,0)*{\ucrossrr{i}{j}};\endxy} \mapsto (-1)^{\delta_{ij}}\vcenter{\xy 0;/r.15pc/: (0,0)*{\ucrossrr{j}{i}};\endxy} \; , \;
\vcenter{\xy 0;/r.15pc/:(0,0)*{\dcrossrr{i}{j}};\endxy} \mapsto (-1)^{\delta_{ij}}\vcenter{\xy 0;/r.15pc/: (0,0)*{\dcrossrr{j}{i}};\endxy} \; &, \;
\vcenter{\xy 0;/r.15pc/:(0,0)*{\lcrossrr{i}{j}};\endxy} \mapsto (-1)^{\delta_{ij}}\vcenter{\xy 0;/r.15pc/: (0,0)*{\rcrossrr{j}{i}};\endxy} \; , \;
\vcenter{\xy 0;/r.15pc/:(0,0)*{\rcrossrr{i}{j}};\endxy} \mapsto (-1)^{\delta_{ij}}\vcenter{\xy 0;/r.15pc/: (0,0)*{\lcrossrr{j}{i}};\endxy} \\ \\
\vcenter{\xy 0;/r.15pc/:(0,0)*{\rcapr{}};(-4,-5)*{\scs i};(5,5)*{\lambda};\endxy} \mapsto \vcenter{\xy 0;/r.15pc/: (-4,-5)*{\scs i}; (0,0)*{\lcapr{}};(5,5)*{-\lambda};\endxy} \quad , \quad
\vcenter{\xy 0;/r.15pc/:(0,0)*{\lcupr{}};(-4,6)*{\scs i}; (5,-3)*{\lambda};\endxy} \mapsto \vcenter{\xy 0;/r.15pc/: (-4,6)*{\scs i}; (0,0)*{\rcupr{}};(5,-3)*{-\lambda};\endxy} \quad &, \quad
\vcenter{\xy 0;/r.15pc/:(0,0)*{\rcupr{}};(-4,6)*{\scs i}; (5,-3)*{\lambda};\endxy} \mapsto \vcenter{\xy 0;/r.15pc/: (-4,6)*{\scs i}; (0,0)*{\lcupr{}};(5,-3)*{-\lambda};\endxy} \quad , \quad
\vcenter{\xy 0;/r.15pc/:(0,0)*{\lcapr{}};(-4,-5)*{\scs i}; (5,5)*{\lambda};\endxy} \mapsto \vcenter{\xy 0;/r.15pc/: (-4,-5)*{\scs i}; (0,0)*{\rcapr{}};(5,5)*{-\lambda};\endxy} \\ \\
\vcenter{\xy 0;/r.15pc/:(0,0)*{\slineur{}};\endxy} \mapsto \vcenter{\xy 0;/r.15pc/: (0,0)*{\slineur{}};\endxy} \quad , \quad
\vcenter{\xy 0;/r.15pc/:(0,0)*{\sdotur{}};\endxy} \mapsto \vcenter{\xy 0;/r.15pc/: (0,0)*{\sdotur{}};\endxy} \quad &, \quad
\vcenter{\xy 0;/r.15pc/:(0,0)*{\slinedr{}};\endxy} \mapsto \vcenter{\xy 0;/r.15pc/: (0,0)*{\slinedr{}};\endxy} \quad , \quad
\vcenter{\xy 0;/r.15pc/:(0,0)*{\sdotdr{}};\endxy} \mapsto \vcenter{\xy 0;/r.15pc/: (0,0)*{\sdotdr{}};\endxy}
\end{align*}
This extends to a 2-functor $\sigma\maps \Kom(\Ucat_Q)\to \Kom(\Ucat'_Q)^{\op}$
defined on 1-morphisms via
\[
(X,d) \mapsto \Big( \cdots \to \sigma(X^{i-1}) \xrightarrow{(-1)^{i-1}\sigma(d_{i-1})} \sigma(X^i) \xrightarrow{(-1)^i\sigma(d_i)} \sigma(X^{i+1}) \to \cdots \Big)
\]
and on 2-morphisms by applying $\sigma$ component-wise.
The alternating differential is essential here to preserve composition of 1-morphisms (contravariantly), due to the sign conventions in taking horizontal composition of complexes.

%
\subsubsection{The 2-functor \texorpdfstring{$\omega \maps \Ucat_Q \to \Ucat'_Q$}{omega} }
%

Consider the operation on the diagrammatic calculus for $\Ucat_Q$ that
reverses the orientation of each strand,
replaces $\lambda \leftrightarrow -\lambda$,
and scales all $ii$-crossings by $-1$.
This operation is covariant for composition of both 1-morphisms and 2-morphisms,
preserves the degree of a diagram,
and takes relations in $\Ucat_Q$ to those in $\cal{U}'_Q$.
This defines an invertible 2-functor given explicitly as follows:
\begin{align*}
  \omega \maps \Ucat_Q &\to \Ucat'_Q  \\
  \lambda &\mapsto  -\lambda \\
  \onem\cal{E}_{\pm i_1}\cal{E}_{\pm i_2} \cdots
 \cal{E}_{\pm i_m}\onel\la t\ra
 &\mapsto
 \mathbf{1}_{-\mu} \cal{E}_{\mp i_1} \cal{E}_{\mp i_2}\cdots
\cal{E}_{\mp i_m}\mathbf{1}_{-\lambda}\la t\ra \\ \\
\vcenter{\xy 0;/r.15pc/:(0,0)*{\ucrossrr{i}{j}};\endxy} \mapsto (-1)^{\delta_{ij}}\vcenter{\xy 0;/r.15pc/: (0,0)*{\dcrossrr{i}{j}};\endxy} \; , \;
\vcenter{\xy 0;/r.15pc/:(0,0)*{\dcrossrr{i}{j}};\endxy} \mapsto (-1)^{\delta_{ij}}\vcenter{\xy 0;/r.15pc/: (0,0)*{\ucrossrr{i}{j}};\endxy} \; &, \;
\vcenter{\xy 0;/r.15pc/:(0,0)*{\lcrossrr{i}{j}};\endxy} \mapsto (-1)^{\delta_{ij}}\vcenter{\xy 0;/r.15pc/: (0,0)*{\rcrossrr{i}{j}};\endxy} \; , \;
\vcenter{\xy 0;/r.15pc/:(0,0)*{\rcrossrr{i}{j}};\endxy} \mapsto (-1)^{\delta_{ij}}\vcenter{\xy 0;/r.15pc/: (0,0)*{\lcrossrr{i}{j}};\endxy} \\ \\
\vcenter{\xy 0;/r.15pc/:(0,0)*{\rcapr{}};(-4,-6)*{\scs i};(5,5)*{\lambda};\endxy} \mapsto \vcenter{\xy 0;/r.15pc/: (-4,-6)*{\scs i}; (0,0)*{\lcapr{}};(5,5)*{-\lambda};\endxy} \quad , \quad
\vcenter{\xy 0;/r.15pc/:(0,0)*{\lcupr{}};(-4,6)*{\scs i}; (5,-3)*{\lambda};\endxy} \mapsto \vcenter{\xy 0;/r.15pc/: (-4,6)*{\scs i}; (0,0)*{\rcupr{}};(5,-3)*{-\lambda};\endxy} \quad &, \quad
\vcenter{\xy 0;/r.15pc/:(0,0)*{\rcupr{}};(-4,6)*{\scs i}; (5,-3)*{\lambda};\endxy} \mapsto \vcenter{\xy 0;/r.15pc/: (-4,6)*{\scs i}; (0,0)*{\lcupr{}};(5,-3)*{-\lambda};\endxy} \quad , \quad
\vcenter{\xy 0;/r.15pc/:(0,0)*{\lcapr{}};(-4,-6)*{\scs i}; (5,5)*{\lambda};\endxy} \mapsto \vcenter{\xy 0;/r.15pc/: (-4,-6)*{\scs i}; (0,0)*{\rcapr{}};(5,5)*{-\lambda};\endxy} \\ \\
\vcenter{\xy 0;/r.15pc/:(0,0)*{\slineur{}};\endxy} \mapsto \vcenter{\xy 0;/r.15pc/: (0,0)*{\slinedr{}};\endxy} \quad , \quad
\vcenter{\xy 0;/r.15pc/:(0,0)*{\sdotur{}};\endxy} \mapsto \vcenter{\xy 0;/r.15pc/: (0,0)*{\sdotdr{}};\endxy}\quad &, \quad
\vcenter{\xy 0;/r.15pc/:(0,0)*{\slinedr{}};\endxy} \mapsto \vcenter{\xy 0;/r.15pc/: (0,0)*{\slineur{}};\endxy}\quad , \quad
\vcenter{\xy 0;/r.15pc/:(0,0)*{\sdotdr{}};\endxy} \mapsto \vcenter{\xy 0;/r.15pc/: (0,0)*{\sdotur{}};\endxy}
\end{align*}
This again extends to a 2-functor $\omega\maps \Kom(\Ucat_Q)\to \Kom(\Ucat'_Q)$ defined on 1-morphisms via
\[
(X,d) \mapsto \Big( \cdots \to \omega(X^{i-1}) \xrightarrow{\omega(d_{i-1})} \omega(X^i) \xrightarrow{\omega(d_i)} \omega(X^{i+1}) \to \cdots \Big)
\]
and on 2-morphisms by applying $\omega$ component-wise.

%
\subsubsection{The 2-functor  \texorpdfstring{$\psi \maps \Ucat_Q \to \left(\Ucat_Q\right)^{\co}$}{psi} }
%

Consider the operation on the diagrammatic calculus for $\Ucat_Q$ that
reflects a diagram across a horizontal axis,
and reverses the orientation.
This operation is covariant for composition of 1-morphisms, contravariant for composition of 2-morphisms, and preserves the relations in $\cal{U}_Q$.
It determines an invertible 2-functor given explicitly as follows:
\begin{align*}
  \psi \maps \Ucat_Q &\to \left(\Ucat_Q\right)^{\co} \\
  \lambda &\mapsto  \lambda \\
  \onem\cal{E}_{\pm i_1} \cal{E}_{\pm i_2} \cdots
 \cal{E}_{\pm i_m}\onel\la t\ra
 &\mapsto
 \onem \cal{E}_{\pm i_1}\cal{E}_{\pm i_2}\cdots
 \cal{E}_{\pm i_m}\onel\la -t\ra \\ \\
\vcenter{\xy 0;/r.15pc/:(0,0)*{\ucrossrr{i}{j}};\endxy} \mapsto \vcenter{\xy 0;/r.15pc/: (0,0)*{\ucrossrr{j}{i}};\endxy} \;\; , \;\;
\vcenter{\xy 0;/r.15pc/:(0,0)*{\dcrossrr{i}{j}};\endxy} \mapsto \vcenter{\xy 0;/r.15pc/: (0,0)*{\dcrossrr{j}{i}};\endxy} \;\; &, \;\;
\vcenter{\xy 0;/r.15pc/:(0,0)*{\lcrossrr{i}{j}};\endxy} \mapsto \vcenter{\xy 0;/r.15pc/: (0,0)*{\rcrossrr{j}{i}};\endxy} \;\; , \;\;
\vcenter{\xy 0;/r.15pc/:(0,0)*{\rcrossrr{i}{j}};\endxy} \mapsto \vcenter{\xy 0;/r.15pc/: (0,0)*{\lcrossrr{j}{i}};\endxy} \\ \\
\vcenter{\xy 0;/r.15pc/:(0,0)*{\rcapr{}};(-4,-5)*{\scs i};(5,5)*{\lambda};\endxy} \mapsto \vcenter{\xy 0;/r.15pc/: (-4,-5)*{\scs i}; (0,0)*{\lcupr{}};(5,-3)*{\lambda};\endxy} \quad , \quad
\vcenter{\xy 0;/r.15pc/:(0,0)*{\lcupr{}};(-4,6)*{\scs i}; (5,-3)*{\lambda};\endxy} \mapsto \vcenter{\xy 0;/r.15pc/: (-4,6)*{\scs i}; (0,0)*{\rcapr{}};(5,5)*{\lambda};\endxy} \quad &, \quad
\vcenter{\xy 0;/r.15pc/:(0,0)*{\rcupr{}};(-4,6)*{\scs i}; (5,-3)*{\lambda};\endxy} \mapsto \vcenter{\xy 0;/r.15pc/: (-4,6)*{\scs i}; (0,0)*{\lcapr{}};(5,5)*{\lambda};\endxy} \quad , \quad
\vcenter{\xy 0;/r.15pc/:(0,0)*{\lcapr{}};(-4,-5)*{\scs i}; (5,5)*{\lambda};\endxy} \mapsto \vcenter{\xy 0;/r.15pc/: (-4,-5)*{\scs i}; (0,0)*{\rcupr{}};(5,-3)*{\lambda};\endxy} \\ \\
\vcenter{\xy 0;/r.15pc/:(0,0)*{\slineur{}};\endxy} \mapsto \vcenter{\xy 0;/r.15pc/: (0,0)*{\slineur{}};\endxy} \quad , \quad
\vcenter{\xy 0;/r.15pc/:(0,0)*{\sdotur{}};\endxy} \mapsto \vcenter{\xy 0;/r.15pc/: (0,0)*{\sdotur{}};\endxy} \quad &, \quad
\vcenter{\xy 0;/r.15pc/:(0,0)*{\slinedr{}};\endxy} \mapsto \vcenter{\xy 0;/r.15pc/: (0,0)*{\slinedr{}};\endxy} \quad , \quad
\vcenter{\xy 0;/r.15pc/:(0,0)*{\sdotdr{}};\endxy} \mapsto \vcenter{\xy 0;/r.15pc/: (0,0)*{\sdotdr{}};\endxy}
\end{align*}
Note that $\psi$ must negate grading shift in order to be degree-preserving, due to 2-morphism contravariance.
As such, it descends to an antilinear map on the Grothendieck group.
This extends to a 2-functor $\psi\maps \Kom(\Ucat_Q) \to \Kom(\Ucat_Q)^{\co}$ given on 1-morphisms by
\[
(X,d) \mapsto \Big( \cdots \to \psi(X^{i+1}) \xrightarrow{\psi(d_{i})} \psi(X^i) \xrightarrow{\psi(d_{i-1})} \psi(X^{i-1}) \to \cdots \Big)
\]
and on 2-morphisms by applying $\psi$ component-wise.
Implicit in this formula is that $\psi$ negates the homological degree,
\ie for $(X,d)$ in $\Kom(\Ucat)$ we have $\psi(X)^{i}=\psi(X^{-i})$.

\subsubsection{Properties of symmetries of categorified quantum groups}

The symmetries $\sigma$, $\omega$, and $\psi$ are graded, additive, $\Bbbk$-linear 2-functors,
and it is immediate from their definitions that each squares to the identity.
Moreover, the induced 2-functors on categories of complexes descend to homotopy categories.
The following result is immediate from the above definitions.

\begin{reptheorem}{thm:B}
Under the isomorphism $K_0(\UcatD_Q) \cong \UA \cong K_0(\UcatD'_Q)$ (see Theorem \ref{thm_Groth}),
the 2-functors defined above descend to the corresponding symmetries:
$
[\sigma]=\und{\sigma} , [\omega]=\und{\omega} ,
[\psi] = \und{\psi} .
$
\end{reptheorem}

\begin{rem}
The symmetry $\omega\psi$
(which reflects a diagram across a horizontal axis,
sends $\lambda$ to $-\lambda$,
and scales all $ii$-crossings by $-1$)
is closely related  to the Chevalley involution introduced in \cite{Brundan2}.
There, Brundan uses this to move between the 2-categories $\cal{U}_Q^{\co}$ and $\cal{U}_{Q'}$ associated to different choices of scalars.
In the cyclic setting, changing the choice of scalars from $Q$ to $Q'$ is no longer necessary,
provided we change the choice of bubble parameters from $C$ to $C'$ as above.
\end{rem}

%
\section{Defining the categorical Lusztig operator \texorpdfstring{$\cal{T}'_{i,1}$}{T'i,1}} \label{sec:defTi}
%

In this section, we explicitly define additive, $\Bbbk$-linear 2-functors $\cal{T}'_{i,1}\maps \cal{U}_Q \to \Com(\cal{U}_Q)$ for each $i\in I$.
In Section \ref{sec:Ti1mor} we define $\cal{T}'_{i,1}$ on objects and generating 1-morphisms, and extend via additive 2-functoriality to all $1$-morphisms,
\ie we send the horizontal composition of generators to the appropriate horizontal composition of the complexes giving their images,
via equation \eqref{eq:CompOfCom}, and map direct sums to the corresponding direct sums.
In Section \ref{sec:Ti2mor}, we extend this definition to the 2-morphisms in $\cal{U}_Q$,
assigning explicit chain maps to generating 2-morphisms, again extending to all 2-morphisms as required by additive 2-functoriality.

Section \ref{sec_proof_func} is then devoted to showing that $\cal{T}'_{i,1}$ is well-defined,
\ie showing that it preserves all defining relations on 2-morphisms of $\Ucat$, up to chain homotopy.
We also explicitly compute the chain homotopies involved.
We note that this check is considerably lengthy due to the many relations that must be checked,
and the piecewise nature of the definition of the (categorified) Lusztig operator, specifically, its dependency on the value of the bilinear form on $I$.

\begin{reptheorem}{thm:A}
Let $\mf{g}$ be a simply-laced Kac-Moody algebra,
then the data given below defines a 2-functor
\[
\cal{T}_{i,+1}'   \maps   \UcatD_Q(\mf{g})  \to \Com(\UcatD_Q(\mf{g}))
\]
such that induces the map on $K_0(\UcatD_Q(\mf{g})) \cong \U_q(\mf{g})$ satisfies
$
[\cal{T}_{i,+1}']=T_{i,+1}'   \maps \U_q(\mf{g}) \to \U_q(\mf{g}).
$
\end{reptheorem}

Given this, we then define the other versions of the categorified Lusztig operators using the symmetries of categorified quantum groups from Section \ref{sec_symm}.
\begin{definition}  \label{def:other-Ts} Let
\[
\cal{T}''_{i,-1}:=\sigma\cal{T}'_{i,1}\sigma \;\; , \;\;
\cal{T}''_{i,1}:=\omega\cal{T}'_{i,1}\omega \;\; , \;\;
\cal{T}'_{i,-1}:=\psi\cal{T}'_{i,1}\psi
\]
where in each case we apply $\cal{T}'_{i,1}$ on the appropriate version of the categorified quantum group, as determined by the codomain of the categorified symmetry.
\end{definition}

The following result now follows from Theorems \ref{thm:A} and \ref{thm:B}.

\begin{cor}
Upon passing to $K_0(\UcatD_Q(\mf{g}))$, we have:
\begin{align*}
[\cal{T}''_{i,-1}]=&\ [\sigma\cal{T}'_{i,1}\sigma]=[\sigma][\cal{T}'_{i,1}][\sigma]=\und{\sigma}T'_{i,1}\und{\sigma}=T''_{i,-1}\\
[\cal{T}''_{i,1}]=&\ [\omega\cal{T}'_{i,1}\omega]=[\omega][\cal{T}'_{i,1}][\omega]=\und{\omega}T'_{i,1}\und{\omega}=T''_{i,1}\\
[\cal{T}'_{i,-1}]=&\ [\psi\cal{T}'_{i,1}\psi]=[\psi][\cal{T}'_{i,1}][\psi]=\und{\psi}T'_{i,1}\und{\psi}=T'_{i,-1}
\end{align*}
\end{cor}

Recall from the introduction that, while a similar categorification has previously been defined on 1-morphisms \cite{Cautis},
our definition extends to the 2-morphisms in $\UcatD_Q(\mf{g})$, meaning that our categorified Lusztig operators help illuminate the higher structure of categorified quantum groups.

We now proceed with the definition, regularly abbreviating $\cal{T}'_{i,1}$ simply by $\cal{T}'_i$.
In addition, we will make use of color in the diagrammatic calculus for $\UcatD_Q$ in specifying $\cal{T}'_i$ as follows:
strands which are $i$-labeled (\ie their label agrees with subscript on $\cal{T}'_i$) will be black,
those whose labels $j$ and $j'$ satisfy $i \cdot j = -1 = i \cdot j'$ will be {\color{blue}blue} and {\color{magenta}magenta} (respectively),
and those with label $k$ satisfying $i \cdot k = 0$ will be {\color{green}green},
unless stated otherwise.

%
\subsection{\texorpdfstring{$\cal{T}'_{i,1}$}{T'i,1} on objects and 1-morphisms}\label{sec:Ti1mor}
%

On objects, we define the 2-functor $\cal{T}'_{i,1}$ by
\[
\cal{T}_i'(\lambda) = s_i(\lambda)
\]
where $s_i$ is the corresponding Weyl group element, defined by $s_i(\l)=\l-\l_i\alpha_i$.
On generating 1-morphisms, we define
\begin{align*}
\cal{T}_i'(\onel) &= \clubsuit \;\onell{s_i(\lambda)} \\ \\
\cal{T}_i'(\cal{E}_{\ell}\onel) &=
\left\{
\begin{array}{cl}
\cal{F}_i\onell{s_i(\lambda)} \la -2-\lambda_i\ra \longrightarrow \clubsuit\ 0 & \text{if } i =\ell \\ \\
\clubsuit \;\cal{E}_{\ell}\cal{E}_i\onell{s_i(\lambda)} \xrightarrow{\xy 0;/r.15pc/:(0,0)*{\ucrossbr{\ell}{i}};\endxy} \cal{E}_i\cal{E}_{\ell}\onell{s_i(\lambda)} \la 1 \ra & \text{if }i \cdot \ell =-1 \\ \\
\clubsuit \; \cal{E}_{\ell}\onell{s_i(\lambda)} & \text{if }i \cdot \ell =0
\end{array}
\right. \\ \\
\cal{T}_i'(\cal{F}_{\ell}\onel) &=
\left\{
\begin{array}{cl}
\clubsuit\ 0 \longrightarrow \cal{E}_i\onell{s_i(\lambda)}  \la \lambda_i \ra & \text{if } i =\ell \\ \\
\cal{F}_{\ell}\cal{F}_i\onell{s_i(\lambda)} \la -1 \ra \xrightarrow{\xy 0;/r.15pc/:(0,0)*{\dcrossbr{\ell}{i}};\endxy} \clubsuit \;\cal{F}_i\cal{F}_{\ell}\onell{s_i(\lambda)} & \text{if }i \cdot \ell =-1 \\ \\
\clubsuit \; \cal{F}_\ell\onell{s_i(\lambda)} & \text{if }i \cdot \ell =0
\end{array}
\right.
\end{align*}
where we have omitted all non-zero terms in these complexes, and we follow our convention in denoting homological degree zero with a $\clubsuit$.
Since each of these complexes has at most two nonzero terms, it is trivial that the square of the differential is zero.

%
\subsection{Definition of \texorpdfstring{$\cal{T}'_{i,1}$}{T'i,1} on 2-morphisms}\label{sec:Ti2mor}
%

The 2-functor $\cal{T}'_{i,1}$ is given on generating 2-morphisms as follows. In these equations,
we let our strand labels satisfy $i \cdot j = -1 = i\cdot j'$ and $i \cdot k =0$, and follow the color conventions specified above.
We will omit labelling the weight $s_i(\lambda)$ in the far right region of the diagrams in the codomain,
and in most cases will also only show the non-zero terms in our complexes.
Additionally, we will depict complexes of the form
\[
W \xrightarrow{\left(\begin{smallmatrix} \alpha \\ \beta \end{smallmatrix}\right)} X \oplus Y \xrightarrow{\left(\begin{smallmatrix} \gamma & \delta \end{smallmatrix}\right)} Z
\]
as anti-commutative squares with arrows labeled by the corresponding maps, 
\eg equation \eqref{eq:cubemap} depicts a chain map between such complexes.
In all cases, the chain map condition easily follows from the defining relations in $\UcatD_Q$.

\subsubsection{Definition of \texorpdfstring{$\cal{T}'_{i,1}$}{T'i,1} on upwards dot 2-morphisms}\label{def:updot}

\[
  \cal{T}'_i \left(    \xy 0;/r.18pc/:
 (0,0)*{\sdotur{i}};
 (6,3)*{ \lambda};
 (-9,3)*{ \lambda +\alpha_i};
 (-4,0)*{};(10,0)*{};
 \endxy \right)
 :=
\vcenter{ \xy 0;/r.18pc/:
 (0,12)*+{\cal{F}_i\onell{s_i(\lambda)}\la -\lambda_i \ra}="1";
 (0,-12)*+{\cal{F}_i\onell{s_i(\lambda)}\la -2-\lambda_i \ra}="2";
 {\ar^{\xy
 (0,0)*{\sdotdr{i}}; \endxy\;\;}"2";"1" };
 \endxy}
\quad , \quad
   \cal{T}'_i \left(    \xy 0;/r.18pc/:
  (0,0)*{\sdotug{k}};
 (6,3)*{ \lambda};
 (-9,3)*{ \lambda +\alpha_k};
 (-10,0)*{};(10,0)*{};
 \endxy \right)
 :=
\vcenter{\xy 0;/r.18pc/:
 (0,12)*+{\clubsuit\ \cal{E}_k\onell{s_i(\lambda)}\la 2 \ra}="1";
 (0,-12)*+{\clubsuit\ \cal{E}_k\onell{s_i(\lambda)}}="2";
 {\ar^{\xy
 (0,0)*{\sdotug{k}}; \endxy\;\;}"2";"1" };
 \endxy}
\]

\[
   \cal{T}'_i \left(    \xy 0;/r.18pc/:
  (0,0)*{\sdotu{j}};
 (6,3)*{ \lambda};
 (-9,3)*{ \lambda +\alpha_j};
 (-10,0)*{};(10,0)*{};
 \endxy \right)
 :=
  \vcenter{\xy 0;/r.18pc/:
  (-20,12)*+{\clubsuit \;\cal{E}_j\cal{E}_i\onell{s_i(\lambda)}\la 2\ra}="1";
  (-20,-12)*+{\clubsuit \;\cal{E}_j\cal{E}_i\onell{s_i(\lambda)}}="2";
   {\ar^{\xy (-12,0)*{\sdotu{j}}; (-6,0)*{\slineur{i}};  \endxy} "2";"1"};
  (25,12)*+{\cal{E}_i\cal{E}_j\onell{s_i(\lambda)}\la 3\ra }="3";
  (25,-12)*+{\cal{E}_i\cal{E}_j\onell{s_i(\lambda)}\la 1\ra}="4";
    {\ar_{\xy (-12,0)*{\slineur{i}}; (-6,0)*{\sdotu{j}}; \endxy} "4";"3"};
   {\ar^{\xy (0,0)*{\ucrossbr{j}{i}};\endxy   } "1";"3"};
   {\ar^{\xy (0,0)*{\ucrossbr{j}{i}};\endxy   } "2";"4"};
 \endxy}
\]

\subsubsection{Definition of \texorpdfstring{$\cal{T}'_{i,1}$}{T'i,1} on upwards crossing 2-morphisms}\label{def:upcross}

\begin{align*}
    \cal{T}'_i \left(  \xy 0;/r.18pc/:
  (0,0)*{\ucrossrr{i}{i}};
 (6,3)*{ \lambda};
 \endxy  \right) :=     \vcenter{\xy 0;/r.18pc/:
 (0,12)*+{\cal{F}_i\cal{F}_i\onell{s_i(\lambda)}\la -8-2\lambda_i \ra}="1";
 (0,-12)*+{\cal{F}_i\cal{F}_i\onell{s_i(\lambda)}\la -6-2\lambda_i\ra}="2";
 {\ar^{\xy (-7,0)*{-};
  (0,0)*{\dcrossrr{i}{i}}; \endxy}"2";"1" };
 \endxy}
 \quad &, \quad
 \cal{T}'_i \left(  \xy 0;/r.18pc/:
  (0,0)*{\ucrossgg{k}{k'}};
 (6,3)*{ \lambda};
 \endxy  \right)  :=     \vcenter{\xy 0;/r.18pc/:
 (0,12)*+{\clubsuit\ \cal{E}_{k'}\cal{E}_k\onell{s_i(\lambda)}\la -k\cdot k' \ra}="1";
 (0,-12)*+{\clubsuit\ \cal{E}_k\cal{E}_{k'}\onell{s_i(\lambda)}}="2";
 {\ar^{ \xy
 (0,0)*{\ucrossgg{k}{k'}}; \endxy}"2";"1" };
 \endxy}
 \\ \\
  \cal{T}'_i \left(  \xy 0;/r.18pc/:
  (0,0)*{\ucrossrg{i}{k}};
 (6,3)*{ \lambda};
 \endxy  \right)
 :=\vcenter{
    \xy 0;/r.18pc/:
  (-30,12)*+{\cal{E}_k\cal{F}_i\onell{s_i(\lambda)}\la -2-\lambda_i \ra}="1";
  (-30,-12)*+{\cal{F}_i\cal{E}_k\onell{s_i(\lambda)}\la-2-\lambda_i \ra}="2";
   {\ar^{ \vcenter{\xy 0;/r.18pc/:
 (3,0)*{\lcrossrg{i}{k}}; (-5,0)*{t_{ki}};
 \endxy} } "2";"1"};
 \endxy}
 \quad &, \quad
\cal{T}'_i \left(  \xy 0;/r.18pc/:
  (0,0)*{\ucrossgr{k}{i}};
 (6,3)*{ \lambda};
 \endxy  \right)
  :=\vcenter{
\xy 0;/r.18pc/:
  (-30,12)*+{\cal{F}_i\cal{E}_k\onell{s_i(\lambda)}\la-2-\lambda_i \ra }="1";
  (-30,-12)*+{\cal{E}_k\cal{F}_i\onell{s_i(\lambda)}\la-2-\lambda_i \ra}="2";
   {\ar^{ \vcenter{\xy 0;/r.18pc/:
 (3,0)*{\rcrossgr{k}{i}};
 \endxy} } "2";"1"};
 \endxy}
 \end{align*}
\[
  \cal{T}'_i \left(  \xy 0;/r.18pc/:
  (0,0)*{\ucrossrb{i}{j}};
 (6,3)*{ \lambda};
 \endxy  \right)
 :=
   \xy 0;/r.18pc/:
  (-35,15)*+{\cal{E}_j\cal{E}_i\cal{F}_i\onell{s_i(\lambda)}\la-1-\lambda_i \ra}="1";
  (-35,-15)*+{\cal{F}_i\cal{E}_j\cal{E}_i\onell{s_i(\lambda)}\la-1-\lambda_i \ra}="2";
   {\ar^{\vcenter{\xy 0;/r.15pc/:
 (-3,0)*{\lcrossrb{i}{j}};
  (6,0)*{\slinenr{i}};
 (3,8.5)*{\lcrossrr{}{}};
 (-6,8.5)*{\slineu{}};
 \endxy }} "2";"1"};
  (35,15)*+{\clubsuit\ \cal{E}_i\cal{E}_j\cal{F}_i\onell{s_i(\lambda)}\la -\lambda_i \ra}="3";
  (35,-15)*+{\clubsuit\ \cal{F}_i\cal{E}_i\cal{E}_j\onell{s_i(\lambda)}\la -\lambda_i \ra}="4";
   {\ar_{\vcenter{ \xy 0;/r.15pc/:
 (-12,4)*{-}; (-3,0)*{\lcrossrr{i}{i}};
  (6,0)*{\slinen{j}};
 (3,8.5)*{\lcrossrb{}{}};
 (-6,8.5)*{\slineur{}};
 \endxy}} "4";"3"};
   {\ar^-{\xy  (3,0)*{\slinedr{i}};(-6,0)*{\ucrossbr{j}{i}};\endxy   } "1";"3"};
   {\ar^-{\xy (-9,0)*{\slinedr{i}};(0,0)*{\ucrossbr{j}{i}}; (-13,1)*{-}; \endxy   } "2";"4"};
 \endxy
 \]
\[
\cal{T}'_i \left(  \xy 0;/r.18pc/:
  (0,0)*{\ucrossbr{j}{i}};
 (6,3)*{ \lambda};
 \endxy  \right)
  :=
     \xy 0;/r.18pc/:
  (-35,15)*+{\cal{F}_i\cal{E}_j\cal{E}_i\onell{s_i(\lambda)}\la -\lambda_i \ra}="1";
  (-35,-15)*+{\cal{E}_j\cal{E}_i\cal{F}_i\onell{s_i(\lambda)}\la -2 -\lambda_i \ra}="2";
   {\ar^{\vcenter{\xy 0;/r.15pc/:
 (3,0)*{\rcrossrr{i}{i}};
  (-6,0)*{\slinen{j}};
 (-3,8.5)*{\rcrossbr{}{}};
 (6,8.5)*{\sdotur{}}; (-12,4)*{t_{ij}};
 \endxy}\vcenter{\xy 0;/r.15pc/:
 (-12,4)*{-t_{ij}}; (3,0)*{\rcrossrr{i}{i}};
  (-6,0)*{\slinen{j}};
 (-3,8.5)*{\rcrossbr{}{}};
 (6,8.5)*{\slineur{}}; (-4.5,11)*[black]{\scs\bullet};
 \endxy}}  "2";"1"};
  (32,15)*+{\clubsuit\ \cal{F}_i\cal{E}_i\cal{E}_j\onell{s_i(\lambda)}\la 1-\lambda_i \ra }="3";
  (32,-15)*+{\clubsuit\ \cal{E}_i\cal{E}_j\cal{F}_i\onell{s_i(\lambda)}\la -1-\lambda_i \ra }="4";
  {\ar_{\vcenter{\xy 0;/r.15pc/:
    (3,0)*{\rcrossbr{j}{i}};
    (-6,0)*{\slinenr{i}};
    (-3,8.5)*{\rcrossrr{}{}}; (-12,4)*{t_{ij}};
    (6,8.5)*{\slineu{}}; (-5,11.3)*[black]{\scs \bullet};
    \endxy}\vcenter{\xy 0;/r.15pc/:
    (-12,4)*{-t_{ij}}; (3,0)*{\rcrossbr{j}{i}};
    (-6,0)*{\slinenr{i}};
    (-3,8.5)*{\rcrossrr{}{}};
    (6,8.5)*{\slineu{}};(-1,11)*[black]{\scs \bullet}; \endxy}} "4";"3"};
   {\ar^-{\xy  (3,0)*{\slinedr{i}};(-6,0)*{\ucrossbr{j}{i}};\endxy   } "2";"4"};
   {\ar^-{\xy (-9,0)*{\slinedr{i}};(0,0)*{\ucrossbr{j}{i}}; (-13,1)*{-}; \endxy   } "1";"3"};
 \endxy
\]
\[
     \cal{T}'_i \left(  \xy 0;/r.18pc/:
  (0,0)*{\ucrossbg{j}{k}};
 (6,3)*{ \lambda};
 \endxy  \right)  :=
 \xy 0;/r.18pc/:
  (-35,15)*+{\clubsuit\;\cal{E}_k\cal{E}_j\cal{E}_i\onell{s_i(\lambda)}
  \la -j\cdot k\ra}="1";
  (-35,-15)*+{\clubsuit \;\cal{E}_j\cal{E}_i\cal{E}_k\onell{s_i(\lambda)}}="2";
   {\ar^{ \vcenter{\xy 0;/r.15pc/:
 (-13,4)*{t_{ki}^{-1}}; (3,0)*{\ncrossrg{i}{k}};
  (-6,0)*{\slinen{j}};
 (-3,8.5)*{\ucrossbg{}{}};
 (6,8.5)*{\slineur{}};
 \endxy}} "2";"1"};
  (35,15)*+{\cal{E}_k\cal{E}_i\cal{E}_j\onell{s_i(\lambda)}\la 1-j\cdot k\ra}="3";
  (35,-15)*+{\cal{E}_i\cal{E}_j\cal{E}_k\onell{s_i(\lambda)}\la1\ra}="4";
  {\ar_{ \vcenter{\xy 0;/r.15pc/:
 (-13,4)*{t_{ki}^{-1}}; (3,0)*{\ncrossbg{j}{k}};
  (-6,0)*{\slinenr{i}};
 (-3,8.5)*{\ucrossrg{}{}};
 (6,8.5)*{\slineu{}};
 \endxy}} "4";"3"};
   {\ar^-{\xy  (3,0)*{\slineug{k}};(-6,0)*{\ucrossbr{j}{i}};\endxy   } "2";"4"};
   {\ar^-{\xy (-9,0)*{\slineug{k}};(0,0)*{\ucrossbr{j}{i}};\endxy   } "1";"3"};
 \endxy
\]
\[
      \cal{T}'_i \left(  \xy 0;/r.18pc/:
  (0,0)*{\ucrossgb{k}{j}};
 (6,3)*{ \lambda};
 \endxy  \right)  :=\vcenter{
 \xy 0;/r.18pc/:
  (-35,15)*+{\clubsuit\;\cal{E}_j\cal{E}_i\cal{E}_k\onell{s_i(\lambda)}
  \la -k\cdot j\ra}="1";
  (-35,-15)*+{\clubsuit \;\cal{E}_k\cal{E}_j\cal{E}_i\onell{s_i(\lambda)}}="2";
   {\ar^{ \vcenter{\xy 0;/r.15pc/:
 (-3,0)*{\ncrossgb{k}{j}};
  (6,0)*{\slinenr{i}};
 (3,8.5)*{\ucrossgr{}{}};
 (-6,8.5)*{\slineu{}};
 \endxy}} "2";"1"};
  (35,15)*+{\cal{E}_i\cal{E}_j\cal{E}_k\onell{s_i(\lambda)}\la 1-k\cdot j\ra}="3";
  (35,-15)*+{\cal{E}_k\cal{E}_i\cal{E}_j\onell{s_i(\lambda)}\la1\ra}="4";
  {\ar_{\vcenter{\xy 0;/r.15pc/:
 (-3,0)*{\ncrossgr{k}{i}};
  (6,0)*{\slinen{j}};
 (3,8.5)*{\ucrossgb{}{}};
 (-6,8.5)*{\slineur{}};
 \endxy} } "4";"3"};
   {\ar^-{\xy  (-3,0)*{\slineug{k}};(6,0)*{\ucrossbr{j}{i}};\endxy   } "2";"4"};
   {\ar^-{\xy (9,0)*{\slineug{k}};(0,0)*{\ucrossbr{j}{i}};\endxy   } "1";"3"};
 \endxy}
\]
$\cal{T}'_i \left(  \xy 0;/r.18pc/:
(0,0)*{\ucrossbp{j}{j'}};
(6,3)*{ \lambda};
\endxy  \right)
:=$
\begin{center}
\begin{equation}\label{eq:cubemap}
 \xy 0;/r.20pc/:
 (-60,-40)*+{\clubsuit \;\cal{E}_j\cal{E}_i\cal{E}_{j'}\cal{E}_i \onell{s_i(\lambda)}}="bl";
 (-20,-20)*+{\cal{E}_i\cal{E}_j\cal{E}_{j'}\cal{E}_i \onell{s_i(\lambda)}\la1\ra}="bt";
 (25,-60)*+{\cal{E}_j\cal{E}_i\cal{E}_i\cal{E}_{j'} \onell{s_i(\lambda)}\la1\ra}="bb";
 (60,-35)*+{\cal{E}_i\cal{E}_j\cal{E}_i\cal{E}_{j'} \onell{s_i(\lambda)}\la2\ra}="br";
  {\ar^<<<<<<<<<<<<{\xy 0;/r.14pc/:
    (-6,0)*{\ucrossbr{j}{i}}; (9,0)*{\slineur{i}}; (3,0)*{\slineup{j'}}; \endxy \;\;\;
  } "bl";"bt"};
  {\ar_{\xy 0;/r.14pc/:(6,0)*{\ucrosspr{j'}{i}}; (-3,0)*{\slineur{i}};
    (-9,0)*{\slineu{j}};\endxy \;\;
  } "bl";"bb"};
  {\ar^>>>>>>>>{\xy 0;/r.14pc/:
    (6,0)*{\ucrosspr{j'}{i}}; (-9,0)*{\slineur{i}}; (-13,1)*{-}; (-3,0)*{\slineu{j}};\endxy
  } "bt";"br"};
  {\ar_{\xy 0;/r.14pc/:
    (-6,0)*{\ucrossbr{j}{i}}; (9,0)*{\slineup{j'}}; (3,0)*{\slineur{i}};\endxy
 } "bb";"br"};
 (-60,40)*+{\clubsuit \;\cal{E}_{j'}\cal{E}_i\cal{E}_j\cal{E}_i \onell{s_i(\lambda)}\la-j\cdot j'\ra}="tl";
 (-20,60)*+{\cal{E}_i\cal{E}_{j'}\cal{E}_j\cal{E}_i \onell{s_i(\lambda)}\la1-j\cdot j'\ra}="tt";
 (25,20)*+{\cal{E}_{j'}\cal{E}_i\cal{E}_i\cal{E}_j \onell{s_i(\lambda)}\la1-j\cdot j'\ra}="tb";
 (60,45)*+{\cal{E}_i\cal{E}_{j'}\cal{E}_i\cal{E}_j \onell{s_i(\lambda)}\la2-j\cdot j'\ra}="tr";
 {\ar^<<<<<<<<<<<<{\xy 0;/r.18pc/:
    \endxy \;\;\;
  } "tl";"tt"};
  {\ar^<<<<<<<<<<<<{\;\;\;\xy 0;/r.18pc/:\endxy \;\;} "tl";"tb"};
  {\ar^{\;\;\;\xy 0;/r.18pc/:
   \endxy
  } "tt";"tr"};
  {\ar_{\xy 0;/r.18pc/:
    \endxy
 } "tb";"tr"};
  {\ar^{\vcenter{\xy 0;/r.14pc/:
        (0,0)*{\ncrossrp{i}{j'}};
        (9,0)*{\slinenr{i}};
        (-9,0)*{\slinen{j}};
        (-6,8.5)*{\ncrossbp{}{}};
        (6,8.5)*{\ncrossrr{}{}};
        (0,16.5)*{\ucrossbr{}{}};
        (9,16.5)*{\slineur{}};
        (-9,16.5)*{\slineup{}};
        (-14,6)*{\scs t_{ij}^{-1}}; \endxy}
  } "bl";"tl"};
  {\ar_{\vcenter{\xy 0;/r.14pc/:
        (0,0)*{\ncrossbr{j}{i}};
        (9,0)*{\slinenp{j'}};
        (-9,0)*{\slinenr{i}};
        (-6,8.5)*{\ncrossrr{}{}};
        (6,8.5)*{\ncrossbp{}{}};
        (0,16.5)*{\ucrossrp{}{}};
        (9,16.5)*{\slineu{}};
        (-9,16.5)*{\slineur{}};
        (-16,6)*{\scs -t_{ij}^{-1}}; \endxy}
 } "br";"tr"};
  {\ar_>>>>>>>>>>>>>>{\xy 0;/r.14pc/:
        (0,0)*{\ucrossrr{i}{i}};
        (7,0)*{\slineu{j}};
        (-7,0)*{\slineu{j}};
        (-20,0)*{\scs -\delta_{jj'}v_{ij}};
    \endxy} "bb";"tb"};
  {\ar^-{\xy 0;/r.15pc/:
        (0,0)*{\ucrossbp{j}{j'}};
        (7,0)*{\slineur{i}};
        (-7,0)*{\slineur{i}};
        (-16,0)*{\scs t_{ij}^{-1}t_{ij'}};
    \endxy} "bt";"tt"};
  {\ar^{} "bt";"tb"};
  {\ar@/_1.3pc/@{->} "bb";"tt"};
  (-8,9)*{\xy 0;/r.14pc/:
        (-14,6)*{\scs t_{ij}^{-1}};
        (0,0)*{\ncrossbp{j}{j'}};
        (9,0)*{\slinenr{i}};
        (-9,0)*{\slinenr{i}};
        (-6,8.5)*{\ncrossrp{}{}};
        (6,8.5)*{\ncrossbr{}{}};
        (0,16.5)*{\ucrossrr{}{}};
        (9,16.5)*{\slineu{}};
        (-9,16.5)*{\slineup{}};
    \endxy};
   (5,41)*{\xy 0;/r.14pc/:
        (-14,6)*{\scs t_{ij}^{-1}};
        (0,0)*{\ncrossrr{i}{i}};
        (9,0)*{\slinenp{j'}};
        (-9,0)*{\slinen{j}};
        (-6,8.5)*{\ncrossbr{}{}};
        (6,8.5)*{\ncrossrp{}{}};
        (0,16.5)*{\ucrossbp{}{}};
        (9,16.5)*{\slineur{}};
        (-9,16.5)*{\slineur{}};
      \endxy};
 \endxy
\end{equation}
\end{center}
In this last diagram, we have omitted the differentials on the codomain,
so as not to overcrowd it;
they are given analogously to those in the domain, with $j \leftrightarrow j'$.
Recall also that $v_{ij} := t_{ij}^{-1} t_{ji}$.

\subsubsection{Definition of \texorpdfstring{$\cal{T}'_{i,1}$}{T'i,1} on cap and cup 2-morphisms}\label{def:capcup}
\begin{align*}
\cal{T}_i'\left( \xy 0;/r.18pc/:
    (0,0)*{\rcapr{i}};
    (8,5)*{ \lambda};
    \endxy\right)
:= \vcenter{ \xy 0;/r.18pc/:
 (0,12)*+{\clubsuit\ \onell{s_i(\lambda)}\la1-\lambda_i\ra}="1";
 (0,-12)*+{\clubsuit\ \cal{F}_i\cal{E}_i\onell{s_i(\lambda)}}="2";
 {\ar^{ \xy
    (0,0)*{\lcapr{i}};
    (-10,0)*{c_{i,\lambda}};
    \endxy}"2";"1" };
 \endxy}
\quad &, \quad
\cal{T}_i'\left( \xy 0;/r.18pc/:
    (0,0)*{\rcupr{i}};
    (8,5)*{ \lambda};
    \endxy\right)
:= \vcenter{ \xy 0;/r.18pc/:
 (0,-12)*+{\clubsuit\ \onell{s_i(\lambda)}}="1";
 (0,12)*+{\clubsuit\ \cal{E}_i\cal{F}_i\onell{s_i(\lambda)}\la1 +\lambda_i\ra}="2";
 {\ar^{ \xy
    (0,0)*{\lcupr{i}};
    (-10,0)*{c_{i,\lambda}^{-1}};
    \endxy}"1";"2" };
 \endxy} \\ \\
\cal{T}_i'\left( \xy 0;/r.18pc/:
    (0,0)*{\lcupr{i}};
    (8,5)*{ \lambda};
    \endxy\right)
:= \vcenter{\xy 0;/r.18pc/:
 (0,-12)*+{\clubsuit\ \onell{s_i(\lambda)}}="1";
 (0,12)*+{\clubsuit\ \cal{F}_i\cal{E}_i\onell{s_i(\lambda)}\la 1-\lambda_i\ra}="2";
 {\ar^{ \xy
    (0,0)*{\rcupr{i}};
    (-10,0)*{c_{i,\lambda}};
    \endxy}"1";"2" };
 \endxy}
\quad &, \quad
 \cal{T}_i'\left( \xy 0;/r.18pc/:
    (0,0)*{\lcapr{i}};
    (8,5)*{ \lambda};
    \endxy\right)
:= \vcenter{ \xy 0;/r.18pc/:
 (0,12)*+{\clubsuit\ \onell{s_i(\lambda)}\la1+\lambda_i \ra}="1";
 (0,-12)*+{\clubsuit\ \cal{E}_i\cal{F}_i\onell{s_i(\lambda)}}="2";
 {\ar^{ \xy
    (0,0)*{\rcapr{i}};
    (-10,0)*{c_{i,\lambda}^{-1}};
    \endxy}"2";"1" };
 \endxy}\end{align*}
Note that the maps have the correct degree since
the rightmost region in all the images is labeled by $s_i(\l)$, and
$1 \pm \la i, s_i(\lambda) \ra =
1 \pm \la i,
\lambda -\lambda_i\alpha_i\ra
= 1\pm\lambda_i\mp 2\lambda_i
=1 \mp \lambda_i$.

\begin{align*}
\cal{T}_i'\left( \xy 0;/r.18pc/:
    (0,0)*{\rcapg{k}};
    (8,5)*{ \lambda};
    \endxy\right)
:= \vcenter{\xy 0;/r.18pc/:
 (0,12)*+{\clubsuit\ \onell{s_i(\lambda)}\la1-\lambda_k\ra}="1";
 (0,-12)*+{\clubsuit\ \cal{E}_k\cal{F}_k\onell{s_i(\lambda)}}="2";
 {\ar^{ \xy
    (0,0)*{\rcapg{k}};(-10,0)*{t_{ki}^{\lambda_i}};
    \endxy}"2";"1" };
 \endxy}
\quad &, \quad
\cal{T}_i'\left( \xy 0;/r.18pc/:
    (0,0)*{\rcupg{k}};
    (8,5)*{ \lambda};
    \endxy\right)
:= \vcenter{\xy 0;/r.18pc/:
 (0,-12)*+{\clubsuit\ \onell{s_i(\lambda)}}="1";
 (0,12)*+{\clubsuit\ \cal{F}_k\cal{E}_k\onell{s_i(\lambda)}\la1 +\lambda_k\ra}="2";
 {\ar^{ \xy
    (0,0)*{\rcupg{k}};(-10,0)*{t_{ki}^{-\lambda_i}};
    \endxy}"1";"2" };
 \endxy} \\ \\
\cal{T}_i'\left( \xy 0;/r.18pc/:
    (0,0)*{\lcupg{k}};
    (8,5)*{ \lambda};
    \endxy\right)
:= \vcenter{\xy 0;/r.18pc/:
 (0,-12)*+{\clubsuit\ \onell{s_i(\lambda)}}="1";
 (0,12)*+{\clubsuit\ \cal{E}_k\cal{F}_k\onell{s_i(\lambda)}\la 1-\lambda_k\ra}="2";
 {\ar^{ \xy
    (0,0)*{\lcupg{k}};
    \endxy}"1";"2" };
 \endxy}
\quad &, \quad
\cal{T}_i'\left( \xy 0;/r.18pc/:
    (0,0)*{\lcapg{k}};
    (8,5)*{ \lambda};
    \endxy\right)
:= \vcenter{ \xy 0;/r.18pc/:
 (0,12)*+{\clubsuit\ \onell{s_i(\lambda)}\la1+\lambda_k \ra}="1";
 (0,-12)*+{\clubsuit\ \cal{F}_k\cal{E}_k\onell{s_i(\lambda)}}="2";
 {\ar^{ \xy
    (0,0)*{\lcapg{k}};
    \endxy}"2";"1" };
 \endxy}
 \end{align*}
Again, the maps have the correct degree since
$1 \pm \la k, s_i(\lambda) \ra =
1 \pm \la k,
\lambda -\lambda_i\alpha_i\ra
=1 \pm \lambda_k$.

\[
\cal{T}_i'\left( \xy
    (0,0)*{\rcup{j}};
    (6,-2)*{ \lambda};
    \endxy\right):=\vcenter{\xy 0;/r.20pc/:
 (-55,40)*+{\cal{F}_j\cal{F}_i\cal{E}_j\cal{E}_i \onell{s_i(\lambda)}\la \lambda_j\ra}="bl";
 (-15,60)*+{\clubsuit \;\cal{F}_i\cal{F}_j\cal{E}_j\cal{E}_i \onell{s_i(\lambda)}\la 1+\lambda_j\ra}="bt";
 (15,20)*+{\clubsuit \;\cal{F}_j\cal{F}_i\cal{E}_i\cal{E}_j \onell{s_i(\lambda)}\la 1+\lambda_j\ra}="bb";
 (55,40)*+{\cal{F}_i\cal{F}_j\cal{E}_i\cal{E}_j \onell{s_i(\lambda)}\la 2+\lambda_j\ra}="br";
  {\ar^-<<<<<<<{\xy 0;/r.15pc/:
    (-6,0)*{\dcrossbr{}{}}; (9,0)*{\slineur{}}; (3,0)*{\slineu{}}; \endxy
  } "bl";"bt"};
  {\ar_{\xy 0;/r.15pc/:(6,0)*{\ucrossbr{}{}}; (-3,0)*{\slinedr{}}; (-14,0)*{\scs -};
    (-9,0)*{\slined{}};\endxy \;\;
  } "bl";"bb"};
  {\ar^->>>>>>>>>>>>>>>{\xy 0;/r.15pc/:
    (6,0)*{\ucrossbr{}{}}; (-9,0)*{\slinedr{}}; (-3,0)*{\slined{}};\endxy
  } "bt";"br"};
  {\ar_<<<<<<<<<<<<<<<<{\xy 0;/r.15pc/:
    (-6,0)*{\dcrossbr{}{}}; (9,0)*{\slineu{}}; (3,0)*{\slineur{}};\endxy
 } "bb";"br"};
 (-55,-5)*+{0}="1";
 (0,-5)*+{\clubsuit \;\onell{s_i(\lambda)}}="2";
 (55,-5)*+{0}="3";
 {\ar "1";"2"};
 {\ar "2";"3"};
 {\ar "bl";"1"};{\ar "br";"3"};
 {\ar@/^1.2pc/"bt";"2";};
 {\ar@/_1pc/ "bb";"2";};
   (-28,5)*{\xy
    (0,0)*{\rcupcuprb{}{}}; (-17,0)*{(-1)^{\lambda_j}c_{j,\lambda}};
   \endxy};
   (32,5)*{\xy
 (0,0)*{\rcupcupbr{}{}}; (-17,0)*{(-1)^{\lambda_j+1}c_{j,\lambda}};
     \endxy};
 \endxy}
\]
\[
 \cal{T}_i'\left( \xy
    (0,0)*{\rcap{j}};
    (6,5)*{ \lambda};
    \endxy\right):=
 \vcenter{\xy 0;/r.20pc/:
 (-55,-40)*+{\cal{E}_j\cal{E}_i\cal{F}_j\cal{F}_i \onell{s_i(\lambda)}\la \lambda_j-2\ra}="bl";
 (-15,-20)*+{\clubsuit \;\cal{E}_i\cal{E}_j\cal{F}_j\cal{F}_i \onell{s_i(\lambda)}\la \lambda_j-1\ra}="bt";
 (15,-60)*+{\clubsuit \;\cal{E}_j\cal{E}_i\cal{F}_i\cal{F}_j \onell{s_i(\lambda)}\la \lambda_j-1\ra}="bb";
 (55,-40)*+{\cal{E}_i\cal{E}_j\cal{F}_i\cal{F}_j \onell{s_i(\lambda)}\la \lambda_j\ra}="br";
  {\ar^-<<<<<<<{\xy 0;/r.15pc/:
    (-6,0)*{\ucrossbr{}{}}; (9,0)*{\slinedr{}}; (3,0)*{\slined{}}; \endxy
  } "bl";"bt"};
  {\ar_{\xy 0;/r.15pc/:(6,0)*{\dcrossbr{}{}}; (-3,0)*{\slineur{}};
    (-9,0)*{\slineu{}};\endxy \;\;
  } "bl";"bb"};
  {\ar^->>>>>>>>>>>>>>>{\xy 0;/r.15pc/:
    (6,0)*{\dcrossbr{}{}}; (-9,0)*{\slineur{}}; (-14,0)*{\scs -}; (-3,0)*{\slineu{}};\endxy
  } "bt";"br"};
  {\ar_-<<<<<<<<<<<<<<<<{\xy 0;/r.15pc/:
    (-6,0)*{\ucrossbr{}{}}; (9,0)*{\slined{}}; (3,0)*{\slinedr{}};\endxy
 } "bb";"br"};
 (-55,5)*+{0}="1";
 (0,5)*+{\clubsuit \;\onell{s_i(\lambda)}}="2";
 (55,5)*+{0}="3";
 {\ar "1";"2"};
 {\ar "2";"3"};
 {\ar "bl";"1"};{\ar "br";"3"};
 {\ar@/^1pc/ "bt";"2"};
 (-34,-5)*{\xy (0,0)*{\rcapcaprb{}{}};
  (-19,0)*{(-1)^{\lambda_j+1}c_{j,\lambda}^{-1}}; \endxy};
 {\ar@/_1.3pc/ "bb";"2"};
 (26,-5)*{\xy (0,0)*{\rcapcapbr{}{}};
  (-17,0)*{(-1)^{\lambda_j}c_{j,\lambda}^{-1}}; \endxy};
 \endxy}
\]
\[
 \cal{T}_i'\left( \xy
    (0,0)*{\lcup{j}};
    (6,-2)*{ \lambda};
    \endxy\right):=\vcenter{\xy 0;/r.20pc/:
 (-55,40)*+{\cal{E}_j\cal{E}_i\cal{F}_j\cal{F}_i \onell{s_i(\lambda)}\la -\lambda_j\ra}="bl";
 (-15,60)*+{\clubsuit \; \cal{E}_i\cal{E}_j\cal{F}_j\cal{F}_i \onell{s_i(\lambda)}\la 1-\lambda_j\ra}="bt";
 (15,20)*+{\clubsuit \;\cal{E}_j\cal{E}_i\cal{F}_i\cal{F}_j \onell{s_i(\lambda)}\la 1-\lambda_j\ra}="bb";
 (55,40)*+{\cal{E}_i\cal{E}_j\cal{F}_i\cal{F}_j \onell{s_i(\lambda)}\la 2-\lambda_j\ra}="br";
  {\ar^-<<<<<<<{\xy 0;/r.15pc/:
    (-6,0)*{\ucrossbr{}{}}; (9,0)*{\slinedr{}};
     (3,0)*{\slined{}}; \endxy
  } "bl";"bt"};
  {\ar_{\xy 0;/r.15pc/:(6,0)*{\dcrossbr{}{}}; (-3,0)*{\slineur{}};
    (-9,0)*{\slineu{}};\endxy \;\;
  } "bl";"bb"};
  {\ar^->>>>>>>>>>>>>>>{\xy 0;/r.15pc/:
    (6,0)*{\dcrossbr{}{}}; (-9,0)*{\slineur{}};(-13,0)*{\scs -}; (-3,0)*{\slineu{}};\endxy
  } "bt";"br"};
  {\ar_<<<<<<<<<<<<<<<<{\xy 0;/r.15pc/:
    (-6,0)*{\ucrossbr{}{}}; (9,0)*{\slined{}}; (3,0)*{\slinedr{}};\endxy
 } "bb";"br"};
 (-55,-5)*+{0}="1";
 (0,-5)*+{\clubsuit \;\onell{s_i(\lambda)}}="2";
 (55,-5)*+{0}="3";
 {\ar "1";"2"};
 {\ar "2";"3"};
 {\ar "1";"bl"};{\ar "3";"br"};
 {\ar@/^1.2pc/ "2";"bt"};
 {\ar@/^.5pc/ "2";"bb"};
   (-25,0)*{\xy
    (0,0)*{\lcupcuprb{}{}}; \endxy};
 (-31,8)*{(-t_{ij})^{\lambda_i}c_{i,\lambda-\alpha_j}^{-1}
    (-1)^{\lambda_j}c_{j,\lambda}};
  (25,0)*{\xy
   (0,0)*{\lcupcupbr{}{}}; \endxy};
 (30,8)*{(-t_{ij})^{\lambda_i}c_{i,\lambda-\alpha_j}^{-1}
   (-1)^{\lambda_j}c_{j,\lambda}};
 \endxy}
\]
\[
 \cal{T}_i'\left( \xy
    (0,0)*{\lcap{j}};
    (6,5)*{ \lambda};
    \endxy\right):=\vcenter{\xy 0;/r.20pc/:
 (-55,-40)*+{\cal{F}_j\cal{F}_i\cal{E}_j\cal{E}_i \onell{s_i(\lambda)}\la -2-\lambda_j\ra}="bl";
 (-15,-20)*+{\clubsuit \; \cal{F}_i\cal{F}_j\cal{E}_j\cal{E}_i \onell{s_i(\lambda)}\la -1-\lambda_j\ra}="bt";
 (15,-60)*+{\clubsuit \; \cal{F}_j\cal{F}_i\cal{E}_i\cal{E}_j \onell{s_i(\lambda)}\la -1-\lambda_j\ra}="bb";
 (55,-40)*+{\cal{F}_i\cal{F}_j\cal{E}_i\cal{E}_j \onell{s_i(\lambda)}\la -\lambda_j\ra}="br";
  {\ar^-<<<<<<<{\xy 0;/r.15pc/:
    (-6,0)*{\dcrossbr{}{}}; (9,0)*{\slineur{}};
     (3,0)*{\slineu{}}; \endxy
  } "bl";"bt"};
  {\ar_{\xy 0;/r.15pc/:(6,0)*{\ucrossbr{}{}}; (-3,0)*{\slinedr{}};(-13,0)*{\scs -};
    (-9,0)*{\slined{}};  \endxy \;\;
  } "bl";"bb"};
  {\ar^->>>>>>>>>>>>>>>{\xy 0;/r.15pc/:
    (6,0)*{\ucrossbr{}{}}; (-9,0)*{\slinedr{}}; (-3,0)*{\slined{}};\endxy
  } "bt";"br"};
  {\ar_-<<<<<<<<<<<<<<<<{\xy 0;/r.15pc/:
    (-6,0)*{\dcrossbr{}{}}; (9,0)*{\slineu{}}; (3,0)*{\slineur{}};\endxy
 } "bb";"br"};
 (-55,5)*+{0}="1";
 (0,5)*+{\clubsuit \;\onell{s_i(\lambda)}}="2";
 (55,5)*+{0}="3";
 {\ar "1";"2"};
 {\ar "2";"3"};
 {\ar "bl";"1"};{\ar "br";"3"};
 {\ar@/_.5pc/ "bt";"2"};
 (-25,0)*{\xy (0,0)*{\lcapcaprb{}{}};\endxy };
  (-32,-8)*{(-t_{ij})^{1-\lambda_i}c_{i,\lambda}
 (-1)^{\lambda_j}c_{j,\lambda}^{-1}};
 {\ar@/_1.3pc/ "bb";"2"};
 (25,0)*{\xy (0,0)*{\lcapcapbr{}{}};\endxy};
 (31,-8)*{(-t_{ij})^{1-\lambda_i}c_{i,\lambda}
 (-1)^{\lambda_j}c_{j,\lambda}^{-1}};
 \endxy}
 \]

As above, a simple computation shows that the maps have the correct degree, \eg
\begin{align*}
\deg \left( \;\; \vcenter{\xy 0;/r.15pc/:
   (0,0)*{\lcapcapbr{}{}}; (5,6)*{\scs s_i({\lambda})};
 \endxy } \;\;\right) &= 1 + \la j, s_i(\lambda) \ra + 1+ \la i, s_i(\lambda)+\alpha_j \\
 &= 2 + \lambda_j - \lambda_i (j\cdot i) +\lambda_i-\lambda_i (i\cdot i)+ i\cdot j = 1 + \lambda_j
\end{align*}

%
\section{Proof that categorified Lusztig operators are well-defined} \label{sec_proof_func}
%

In this section we show that $\cal{T}'_{i,1}$ is well-defined,
\ie that $\cal{T}'_{i,1}$ preserves the defining relations in $\Ucat_Q$,
up to chain homotopy.
We'll see, however, that many cases do not require a chain homotopy.
\eg $\cal{T}'_{i,1}$ holds on the nose for any relation that does not involve
$j$-labeled strands (for $i\cdot j=-1$), since
here the complexes involved have only one non-zero term (in the same homological degree),
precluding the existence of non-trivial chain homotopies.
A complete proof consists of checking many cases for each relation,
since $\cal{T}'_{i,1}$ is defined in a piecewise manner
that depends on the connectivity of the graph associated to the simply-laced root datum.

To simplify this task, we'll work with the presentation of $\Ucat_Q$ implicit in Remark \ref{rem:presentation}.
Specifically, we view downward dot and sideways and downward crossing 2-morphisms as defined in terms of
cap/cup 2-morphisms and their upward analogues (in the case of downward dots, we choose the presentation
in terms of right-oriented caps/cups).
It follows that $\cal{T}'_{i,1}$ is already fixed on these 2-morphisms (by 2-functoriality),
and we record its value on these composite 2-morphisms in Appendix~\ref{sec:composite}.
We make extensive use of these computations in the sections that follow.

Throughout, we will continue with our convention that the labels $j,j' \in I$ satisfy $i \cdot j = -1 = i \cdot j'$
and correspond to {\color{blue}blue} and {\color{magenta}magenta} strands,
while the labels $k,k' \in I$ satisfy $i \cdot k = 0 = i \cdot k'$ and correspond to {\color{green}green} strands.
We also let $\ell \in I$ denote an arbitrary label.

%
\subsection{Adjunction relations}
%

We verify the right and left adjunction relations given in
Definition~\ref{defU_cat-cyc} relation \eqref{def:lr-adj}.
\begin{prop}
For all $\ell \in I$ the equalities
\begin{align*}
\cal{T}'_i\left({\xy 0;/r.15pc/:
(-8,0)*{\xy \rcapr{} \endxy};
(-8,0)*{\xy \slinenr{} \endxy};
(0,0)*{\xy \rcupr{} \endxy};
(8,8)*{\xy \slineur{} \endxy};
 (17,4)*{ \lambda};
 (-9,-5)*{\scs \ell};(20,0)*{};
\endxy}\right)
&=\cal{T}'_i\left({\xy (0,0)*{\slineur{}}; (-2,-3)*{\scs \ell};(4,2)*{\lambda};\endxy}\right)
=
\cal{T}'_i\left({\xy 0;/r.15pc/:
(0,0)*{\xy \lcapr{} \endxy};
(8,0)*{\xy \slinenr{} \endxy};
(-8,0)*{\xy \lcupr{} \endxy};
(-8,8)*{\xy \slineur{} \endxy};
 (17,4)*{ \lambda};
  (13,-5)*{\scs \ell};(16,0)*{};
\endxy}\right)
\\
\cal{T}'_i\left({\xy 0;/r.15pc/:
(-8,0)*{\xy \lcapr{} \endxy};
(-8,0)*{\xy \slinedr{} \endxy};
(0,0)*{\xy \lcupr{} \endxy};
(8,8)*{\xy \slinenr{} \endxy};
 (17,4)*{ \lambda};
 (-9,-5)*{\scs \ell};(20,0)*{};
\endxy}\right)
&=
\cal{T}'_i\left({\xy (0,0)*{\slinedr{}}; (-2,-3)*{\scs \ell};(4,2)*{\lambda};\endxy}\right)
=
\cal{T}'_i\left({\xy 0;/r.15pc/:
(0,0)*{\xy \rcapr{} \endxy};
(8,0)*{\xy \slinedr{} \endxy};
(-8,0)*{\xy \rcupr{} \endxy};
(-8,8)*{\xy \slinenr{} \endxy};
 (17,4)*{ \lambda};
 (13,-5)*{\scs \ell};(16,0)*{};
\endxy}\right)
\end{align*}
hold in $\Com(\Ucat_Q)$.
\end{prop}

\begin{proof}
When $\ell=i$ or $\ell=k$ with $i \cdot k =0$,
these relations follow from a straightforward computation,
provided one is careful with the relevant parameters.
For example,
the first equality follows from the computation:
\[
\cal{T}'_i\left({\xy 0;/r.15pc/:
(-8,0)*{\xy \rcapr{} \endxy};
(-8,0)*{\xy \slinenr{} \endxy};
(0,0)*{\xy \rcupr{} \endxy};
(8,8)*{\xy \slineur{} \endxy};
 (17,4)*{ \lambda};
  (-9,-5)*{\scs i};
 (-9,0)*{};(20,0)*{};
\endxy}\right)
= c_{i,\lambda+\alpha_i}c_{i,\lambda}^{-1}
{\xy 0;/r.15pc/:
(-8,0)*{\xy \lcapr{} \endxy};
(-8,0)*{\xy \slinedr{i} \endxy};
(0,0)*{\xy \lcupr{} \endxy};
(8,8)*{\xy \slinenr{} \endxy};
 (20,4)*{s_i(\lambda)};
 (-9,0)*{};(16,0)*{};
\endxy}={\xy (0,0)*{\slinedr{i}}; (7,2)*{s_i(\lambda)}; \endxy}
=\cal{T}'_i\left({\xy (0,0)*{\slineur{i}}; (4,2)*{\lambda};\endxy}\right)
\]
when $\ell=i$,
and from
\[
\cal{T}'_i\left({\xy 0;/r.15pc/:
(-8,0)*{\xy \rcapg{} \endxy};
(-8,0)*{\xy \slineng{} \endxy};
(0,0)*{\xy \rcupg{} \endxy};
(8,8)*{\xy \slineug{} \endxy};
 (17,4)*{ \lambda};
   (-9,-5)*{\scs k};
 (-9,0)*{};(20,0)*{};
\endxy}\right)
= t_{ki}^{(\lambda_i+i\cdot k)-\lambda_i}{\xy 0;/r.15pc/:
(-8,0)*{\xy \rcapg{} \endxy};
(-8,0)*{\xy \slineng{k} \endxy};
(0,0)*{\xy \rcupg{} \endxy};
(8,8)*{\xy \slineug{} \endxy};
 (20,4)*{s_i(\lambda)};
 (-9,0)*{};(16,0)*{};
\endxy}={\xy (0,0)*{\slineug{k}}; (7,2)*{s_i(\lambda)}; \endxy}
=\cal{T}'_i\left({\xy (0,0)*{\slineug{k}}; (4,2)*{\lambda};\endxy}\right).
\]
when $\ell=k$.
We omit the other checks, as they are completely analogous.

For $\ell=j$, the coefficients are more delicate.
As $\cal{T}_i'(\cal{E}_j\onel)$ is a 2-term chain complex,
we will use an ordered pair to describe its chain endomorphisms
(with the convention that the first term in the pair corresponds to lower homological degree).
\begin{align*}
\cal{T}'_i\left({\xy 0;/r.15pc/:
(-8,0)*{\xy \rcap{} \endxy};
(-8,0)*{\xy \slinen{} \endxy};
(0,0)*{\xy \rcup{} \endxy};
(8,8)*{\xy \slineu{} \endxy};
  (-9,-5)*{\scs j};
 (10,-7)*{ \lambda}; (-3,11)*{\lambda+\alpha_j};
 (-13,0)*{};(15,0)*{};
\endxy}\right)
=&\left({\xy 0;/r.15pc/:
(-3,-4)*{\xy \rcapr{} \endxy};
(-3,-4)*{\xy \slinenr{i} \endxy};
(-7,-4)*{\xy \slinen{j} \endxy};
(4,-4)*{\xy \llrcupr{} \endxy};
(-8,-4)*{\xy \llrcap{} \endxy};
(11,-4)*{\xy \rcup{} \endxy};
(19,4)*{\xy \slineu{} \endxy};
(23,4)*{\xy \slineur{} \endxy};
 (-12,0)*{};(28,0)*{};
 (-29,2)*{(-1)^{\lambda_j+2}c_{j,\lambda+\alpha_j}^{-1}};
 (-29,-6)*{(-1)^{\lambda_j}c_{j,\lambda}};
\endxy}
\; , \;
{\xy 0;/r.15pc/:
(-3,-4)*{\xy \rcap{} \endxy};
(-3,-4)*{\xy \slinen{j} \endxy};
(-7,-4)*{\xy \slinenr{i} \endxy};
(4,-4)*{\xy \llrcup{} \endxy};
(-8,-4)*{\xy \llrcapr{} \endxy};
(11,-4)*{\xy \rcupr{} \endxy};
(19,4)*{\xy \slineur{} \endxy};
(23,4)*{\xy \slineu{} \endxy};
 (-12,0)*{};(28,0)*{};
 (-30,2)*{(-1)^{(\lambda_j+2)+1}c_{j,\lambda+\alpha_j}^{-1}};
 (-30,-6)*{(-1)^{\lambda_j+1}c_{j,\lambda}};
\endxy}\right) \\
=&\left({\xy (0,0)*{\slineu{j}}; (4,0)*{\slineur{i}}; (10,2)*{s_i(\lambda)}; \endxy}
\; , \;
{\xy (0,0)*{\slineur{i}}; (4,0)*{\slineu{j}}; (10,2)*{s_i(\lambda)}; \endxy}\right)
=\cal{T}'_i\left({\xy (0,0)*{\slineu{j}}; (4,2)*{\lambda};\endxy}\right)
\end{align*}

\begin{align*}
\cal{T}'_i\left({\xy 0;/r.15pc/:
(0,0)*{\xy \rcap{} \endxy};
(8,0)*{\xy \slined{} \endxy};
(-8,0)*{\xy \rcup{} \endxy};
(-8,8)*{\xy \slinen{} \endxy};
  (-9,5)*{\scs j};
 (13,7)*{ \lambda};  (-3,-8)*{\lambda-\alpha_j};
 (-13,0)*{};(17,0)*{};
\endxy}\right)
=&\left({\xy 0;/r.15pc/:
(11,-4)*{\xy \rcap{} \endxy};
(19,-4)*{\xy \slined{} \endxy};
(23,-4)*{\xy \slinedr{} \endxy};
(-8,-4)*{\xy \llrcup{} \endxy};
(4,-4)*{\xy \llrcapr{} \endxy};
(-3,-4)*{\xy \rcupr{} \endxy};
(-3,4)*{\xy \slinenr{} \endxy};
(-7,4)*{\xy \slinen{} \endxy};
(-1,8)*{\scs i};
(-5,8)*{\scs j};
 (-12,0)*{};(28,0)*{};
 (-28,2)*{(-1)^{\lambda_j+1}c_{j,\lambda}^{-1}};
 (-32,-6)*{(-1)^{(\lambda_j-2)+1}c_{j,\lambda-\alpha_j}};\endxy}
 \; , \;
{\xy 0;/r.15pc/:
(11,-4)*{\xy \rcapr{} \endxy};
(19,-4)*{\xy \slinedr{} \endxy};
(23,-4)*{\xy \slined{} \endxy};
(-8,-4)*{\xy \llrcupr{} \endxy};
(4,-4)*{\xy \llrcap{} \endxy};
(-3,-4)*{\xy \rcup{} \endxy};
(-3,4)*{\xy \slinen{} \endxy};
(-7,4)*{\xy \slinenr{} \endxy};
(-1,8)*{\scs j};
(-5,8)*{\scs i};
 (-12,0)*{};(28,0)*{};
 (-22,2)*{(-1)^{\lambda_j}c_{j,\lambda}^{-1}};
 (-30,-6)*{(-1)^{\lambda_j-2}c_{j,\lambda-\alpha_j}};
 \endxy}\right)\\
=&\left({\xy (0,0)*{\slined{j}}; (4,0)*{\slinedr{i}}; (10,2)*{s_i(\lambda)}; \endxy}
\; , \;
{\xy (0,0)*{\slinedr{i}}; (4,0)*{\slined{j}}; (10,2)*{s_i(\lambda)}; \endxy}\right)
=\cal{T}'_i\left({\xy (0,0)*{\slined{j}}; (4,2)*{\lambda};\endxy}\right)
\end{align*}
\begin{align*}
\cal{T}'_i\left({\xy 0;/r.15pc/:
(0,0)*{\xy \lcap{} \endxy};
(8,0)*{\xy \slinen{} \endxy};
(-8,0)*{\xy \lcup{} \endxy};
(-8,8)*{\xy \slineu{} \endxy};
  (-9,5)*{\scs j};
 (13,7)*{ \lambda};  (-3,-8)*{\lambda+\alpha_j};
 (-13,0)*{};(17,0)*{};
\endxy}\right)
=&
{\xy
(0,3)*{(-t_{ij})^{1-\lambda_i}c_{i,\lambda} (-1)^{\lambda_j}c_{j,\lambda}^{-1}};
(0,-3)*{(-t_{ij})^{\lambda_i-1}c_{i,\lambda}^{-1} (-1)^{\lambda_j+2}c_{j,\lambda+\alpha_j}};
\endxy}
\left({\xy 0;/r.15pc/:
 (-8,-10)*{};(28,10)*{};
(11,-4)*{\xy \lcap{} \endxy};
(19,-4)*{\xy \slinen{j} \endxy};
(23,-4)*{\xy \slinenr{i} \endxy};
(-8,-4)*{\xy \lllcup{} \endxy};
(4,-4)*{\xy \lllcapr{} \endxy};
(-3,-4)*{\xy \lcupr{} \endxy};
(-3,4)*{\xy \slineur{} \endxy};
(-7,4)*{\xy \slineu{} \endxy};
\endxy}
\; , \;
{\xy 0;/r.15pc/:
 (-8,-10)*{};(28,10)*{};
(11,-4)*{\xy \lcapr{} \endxy};
(19,-4)*{\xy \slinenr{i} \endxy};
(23,-4)*{\xy \slinen{j} \endxy};
(-8,-4)*{\xy \lllcupr{} \endxy};
(4,-4)*{\xy \lllcap{} \endxy};
(-3,-4)*{\xy \lcup{} \endxy};
(-3,4)*{\xy \slineu{} \endxy};
(-7,4)*{\xy \slineur{} \endxy};
\endxy}\right) \\
&=\left({\xy (0,0)*{\slineu{j}}; (4,0)*{\slineur{i}}; (10,2)*{s_i(\lambda)}; \endxy}
\; , \;
{\xy (0,0)*{\slineur{i}}; (4,0)*{\slineu{j}}; (10,2)*{s_i(\lambda)}; \endxy}\right)
=\cal{T}'_i\left({\xy (0,0)*{\slineu{j}}; (4,2)*{\lambda};\endxy}\right)
\end{align*}
\begin{align*}
\cal{T}'_i\left(
{\xy 0;/r.15pc/:
(-8,0)*{\xy \lcap{} \endxy};
(-8,0)*{\xy \slined{} \endxy};
(0,0)*{\xy \lcup{} \endxy};
(8,8)*{\xy \slinen{} \endxy};
  (-9,-5)*{\scs j};
 (10,-7)*{ \lambda}; (-3,11)*{\lambda-\alpha_j};
 (-13,0)*{};(15,0)*{};
\endxy}\right)
&=
{\xy
(0,3)*{(-t_{ij})^{1-(\lambda_i+1)}c_{i,\lambda-\alpha_j} (-1)^{\lambda_j-2}c_{j,\lambda-\alpha_j}^{-1}};
(0,-3)*{(-t_{ij})^{\lambda_i}c_{i,\lambda-\alpha_j}^{-1} (-1)^{\lambda_j}c_{j,\lambda}};
\endxy}
\left({\xy 0;/r.15pc/:
 (-8,-10)*{};(28,10)*{};
(-3,-4)*{\xy \lcapr{} \endxy};
(-3,-4)*{\xy \slinedr{} \endxy};
(-7,-4)*{\xy \slined{} \endxy};
(4,-4)*{\xy \lllcupr{} \endxy};
(-8,-4)*{\xy \lllcap{} \endxy};
(11,-4)*{\xy \lcup{} \endxy};
(19,4)*{\xy \slinen{} \endxy};
(23,4)*{\xy \slinenr{} \endxy};
(21,8)*{\scs j};
(25,8)*{\scs i};
\endxy}
\; , \;
{\xy 0;/r.15pc/:
 (-8,-10)*{};(28,10)*{};
(-3,-4)*{\xy \lcap{} \endxy};
(-3,-4)*{\xy \slined{} \endxy};
(-7,-4)*{\xy \slinedr{} \endxy};
(4,-4)*{\xy \lllcup{} \endxy};
(-8,-4)*{\xy \lllcapr{} \endxy};
(11,-4)*{\xy \lcupr{} \endxy};
(19,4)*{\xy \slinenr{} \endxy};
(23,4)*{\xy \slinen{} \endxy};
(21,8)*{\scs i};
(25,8)*{\scs j};
\endxy}\right) \\
&=\left({\xy (0,0)*{\slineu{j}}; (4,0)*{\slineur{i}}; (10,2)*{s_i(\lambda)}; \endxy}
\; , \;
{\xy (0,0)*{\slinedr{i}}; (4,0)*{\slined{j}}; (10,2)*{s_i(\lambda)}; \endxy}\right)
=\cal{T}'_i\left({\xy (0,0)*{\slined{j}}; (4,2)*{\lambda};\endxy}\right)
\end{align*}
\end{proof}


%
\subsection{Dot cyclicity}\label{proof:dotcyc}
%

We verify
the dot cyclicity relation given in Definition~\ref{defU_cat-cyc} relation \eqref{def:dot-cyc}.
Recall that, in our presentation given by Remark \ref{rem:presentation},
the downward dot morphism is defined in terms of the upward dot morphism and \emph{rightward}
cap/cup morphisms. Dot cyclicity is then equivalent to the following.

\begin{prop}
For $\ell \in I$, the relation
\[
\cal{T}'_i \left( \xy 0;/r.12pc/:
(0,0)*{\sdotur{}};
(-3,6)*{\slcapr{}};
(-6,0)*{\slinenr{}};
(-6,-8)*{\slinedr{}};
(3,-5)*{\slcupr{}};
(6,0)*{\slinedr{}};
(6,8)*{\slinenr{}};
 (12,3)*{ \lambda};
 (-9,-3)*{\scs \ell};
 (15,0)*{};
\endxy \right)
=\cal{T}'_i \left(\xy 0;/r.18pc/:
 (0,0)*{\sdotdr{\ell}};
 (6,3)*{ \lambda};
 (-3,0)*{};  (10,0)*{};
 \endxy \right)
\]
holds in $\Com(\Ucat_Q)$.
\end{prop}

\begin{proof} We compute the left-hand side, and verify the relations by comparing to the
results of Section~\ref{sec:downdot}, which give the value of $\cal{T}'_i$ on downward oriented dot 2-morphisms.
\begin{align*}
\cal{T}'_i \left( \xy 0;/r.12pc/:
(0,0)*{\sdotur{}};
(-3,6)*{\slcapr{}};
(-6,0)*{\slinenr{}};
(-6,-8)*{\slinedr{}};
(3,-5)*{\slcupr{}};
(6,0)*{\slinedr{}};
(6,8)*{\slinenr{}};
 (12,3)*{ \lambda};
  (-9,-3)*{\scs i};
 (-9,0)*{};(15,0)*{};
\endxy \right)
&= c_{i,\l}c_{i,\l-\alpha_i}^{-1}\vcenter{\xy 0;/r.12pc/:
(0,0)*{\sdotdr{}};
(-3,6)*{\srcapr{}};
(-6,0)*{\slineur{}};
(-6,-8)*{\slinenr{}};
(3,-5)*{\srcupr{}};
(6,0)*{\slinenr{}};
(6,8)*{\slineur{}};
(-9,0)*{};(9,0)*{};\endxy}
=\xy 0;/r.18pc/:
 (0,0)*{\sdotur{i}};
 (-4,0)*{};(4,0)*{};
 \endxy
=:\cal{T}'_i \left(\xy 0;/r.18pc/:
 (0,0)*{\sdotdr{i}};
 (6,3)*{ \lambda};
 (-4,0)*{};(10,0)*{};
 \endxy \right) \\
\cal{T}'_i \left(\xy  0;/r.12pc/:
(0,0)*{\sdotu{}};
(-3,6)*{\slcap{}};
(-6,0)*{\slinen{}};
(-6,-8)*{\slined{}};
(3,-5)*{\slcup{}};
(6,0)*{\slined{}};
(6,8)*{\slinen{}};
 (12,3)*{ \lambda};
  (-9,-3)*{\scs j};
 (-9,0)*{};(15,0)*{};
 \endxy \right)
 &=
 {\xy
(0,3)*{(-t_{ij})^{1-(\lambda_i+1)}c_{i,\lambda-\alpha_j} (-1)^{\lambda_j-2}c_{j,\lambda-\alpha_j}^{-1}};
(0,-3)*{(-t_{ij})^{\lambda_i}c_{i,\lambda-\alpha_j}^{-1} (-1)^{\lambda_j}c_{j,\lambda}};
\endxy}
\left({\xy 0;/r.15pc/:
 (-8,-10)*{};(28,10)*{};
(-3,-4)*{\xy \lcapr{} \endxy};
(-3,-4)*{\xy \slinedr{i} \endxy};
(-7,-4)*{\xy \slined{j} \endxy};
(4,-4)*{\xy \lllcupr{} \endxy};
(-8,-4)*{\xy \lllcap{} \endxy};
(11,-4)*{\xy \lcup{} \endxy};
(19,4)*{\xy \slinen{} \endxy};
(23,4)*{\xy \slinenr{} \endxy};
(12.75,-1)*{\bullet};
\endxy}
\; , \;
{\xy 0;/r.15pc/:
 (-8,-10)*{};(28,10)*{};
(-3,-4)*{\xy \lcap{} \endxy};
(-3,-4)*{\xy \slined{j} \endxy};
(-7,-4)*{\xy \slinedr{i} \endxy};
(4,-4)*{\xy \lllcup{} \endxy};
(-8,-4)*{\xy \lllcapr{} \endxy};
(11,-4)*{\xy \lcupr{} \endxy};
(19,4)*{\xy \slinenr{} \endxy};
(23,4)*{\xy \slinen{} \endxy};
(6.75,-1.5)*{\bullet};
\endxy}\right) \\ \\
&=
\left({\xy (0,0)*{\sdotd{j}}; (4,0)*{\slinedr{i}}; (10,2)*{s_i(\lambda)}; \endxy}
\; , \;
{\xy (0,0)*{\slinedr{i}}; (4,0)*{\sdotd{j}}; (10,2)*{s_i(\lambda)}; \endxy}\right)
=: \cal{T}'_i\left({\xy (0,0)*{\sdotd{j}}; (4,2)*{\lambda};\endxy}\right)\\
\cal{T}'_i \left( \xy 0;/r.12pc/:
(0,0)*{\sdotug{}};
(-3,6)*{\slcapg{}};
(-6,0)*{\slineng{}};
(-6,-8)*{\slinedg{}};
(3,-5)*{\slcupg{}};
(6,0)*{\slinedg{}};
(6,8)*{\slineng{}};
 (12,3)*{ \lambda};
  (-9,-3)*{\scs k};
 (-9,0)*{};(15,0)*{};
\endxy \right)
&=
\vcenter{\xy 0;/r.12pc/:
(0,0)*{\sdotug{}};
(-3,6)*{\slcapg{}};
(-6,0)*{\slineng{}};
(-6,-8)*{\slinedg{}};
(3,-5)*{\slcupg{}};
(6,0)*{\slinedg{}};
(6,8)*{\slineng{}};
 (-9,0)*{};(9,0)*{};\endxy}
=\xy 0;/r.18pc/:
 (0,0)*{\sdotdg{k}};
 (-4,0)*{};(4,0)*{};
 \endxy
=: \cal{T}'_i \left(\xy 0;/r.18pc/:
 (0,0)*{\sdotdg{k}};
 (6,3)*{ \lambda};
 (-4,0)*{};(10,0)*{};
 \endxy \right)
\end{align*}
\end{proof}


%
\subsection{Crossing cyclicity}\label{proof:crosscyc}
%

We now verify the crossing cyclicity realtions given in Definition~\ref{defU_cat-cyc} relation~\eqref{def:cross-cyc}.
Note that it suffices to prove cyclicity for the downward crossing,
as the relations for the sideways crossings follow from this and the adjunction relations.
As before, we will use the value of the downward crossing from Section~\ref{sec:down-crossing},
where (by definition) it is given in terms of the upward crossing and \emph{rightward} cap/cup 2-morphisms.

\begin{prop}
For all $\ell,\ell' \in I$, the equation
\[
\cal{T}'_i \left(
\xy0;/r.11pc/:
(0,0)*{\reflectbox{\xy
(-9,0)*{\slinenr{}};
(-9,8)*{\slinenr{}};
(-6,-6)*{\srcupr{}};
(6,7)*{\srcapr{}};
(9,0)*{\slinenr{}};
(9,-8)*{\slinedr{}};
(0,0)*{\ucrossrr{}{}};
(-15,0)*{\slinenr{}};
(-15,8)*{\slinenr{}};
(-6,-9)*{\llrcupr{}};
(6,10)*{\llrcapr{}};
(15,0)*{\slinenr{}};
(15,-8)*{\slinedr{}};
\endxy} };
(-15,-15)*{\scs \ell};(-9,-15)*{\scs \ell'};(18,0)*{};
\endxy \xy 0;/r.12pc/:(0,2)*{\lambda};\endxy \right)
 =
\cal{T}'_i \left(  \xy 0;/r.18pc/:
  (0,0)*{\dcrossrr{\ell}{\ell'}};
 (6,3)*{ \lambda};
 \endxy  \right)
\]
holds in $\Com(\Ucat_Q)$.
\end{prop}

\begin{proof}
We compute the left-hand side,
considering the three possibilities for each $\ell,\ell' \in I$ in relation to the fixed node $i \in I$.
For both stands labeled $i$, we have
\[
\cal{T}'_i \left( \reflectbox{\xy 0;/r.10pc/:
(-9,0)*{\slinenr{}};
(-9,8)*{\slinenr{}};
(-6,-6)*{\srcupr{}};
(6,7)*{\srcapr{}};
(9,0)*{\slinenr{}};
(9,-8)*{\slinedr{}};
(0,0)*{\ucrossrr{}{}};
(-15,0)*{\slinenr{}};
(-15,8)*{\slinenr{}};
(-6,-9)*{\llrcupr{}};
(6,10)*{\llrcapr{}};
(15,0)*{\slinenr{}};
(15,-8)*{\slinedr{}};
(15,-15)*{\reflectbox{$\scs i$}};
(9,-15)*{\reflectbox{$\scs i$}};
(-18,0)*{};(18,0)*{};
\endxy} \xy 0;/r.12pc/:(0,2)*{\lambda};\endxy \right)
=
-c_{i,\lambda}c_{i,\lambda-\alpha_i}c^{-1}_{i,\lambda-2\alpha_i}c^{-1}_{i,\lambda-\alpha_i}
\reflectbox{\xy 0;/r.10pc/:
(-9,0)*{\slinenr{}};
(-9,8)*{\slinenr{}};
(-6,-6)*{\slcupr{}};
(6,7)*{\slcapr{}};
(9,0)*{\slineur{}};
(9,-8)*{\slinenr{}};
(0,0)*{\dcrossrr{}{}};
(-15,0)*{\slinenr{}};
(-15,8)*{\slinenr{}};
(-6,-9)*{\lllcupr{}};
(6,10)*{\lllcapr{}};
(15,0)*{\slineur{}};
(15,-8)*{\slinenr{}};
(15,-15)*{\reflectbox{$\scs i$}};
(9,-15)*{\reflectbox{$\scs i$}};
(-18,0)*{};(18,0)*{};
\endxy}
=     \vcenter{\xy 0;/r.18pc/:
  (-7,0)*{-};(0,0)*{\ucrossrr{i}{i}}; \endxy}
 =:    \cal{T}'_i \left(  \xy 0;/r.18pc/:
  (0,0)*{\dcrossrr{i}{i}};
 (6,3)*{ \lambda};
 \endxy  \right)
\]

For strands labeled $i$ and $j$, we have
\begin{align*}
\cal{T}'_i \left( \reflectbox{\xy 0;/r.10pc/:
(-9,0)*{\slinen{}};
(-9,8)*{\slinen{}};
(-6,-6)*{\srcup{}};
(6,7)*{\srcap{}};
(9,0)*{\slinen{}};
(9,-8)*{\slined{}};
(0,0)*{\ucrossbr{}{}};
(-15,0)*{\slinenr{}};
(-15,8)*{\slinenr{}};
(-6,-9)*{\llrcupr{}};
(6,10)*{\llrcapr{}};
(15,0)*{\slinenr{}};
(15,-8)*{\slinedr{}};
(15,-15)*{\reflectbox{$\scs i$}};
(9,-15)*{\reflectbox{$\scs j$}};
(-18,0)*{};(18,0)*{};
\endxy} \xy 0;/r.12pc/:(0,2)*{\lambda};\endxy \right)
&=
{\xy
(0,3)*{c^{-1}_{i,\l-\alpha_i-\alpha_j} (-t_{ij})^{1-(\l_i+1)} c_{i,\l-\alpha_j} (-1)^{\l_j-2} c^{-1}_{j,\l-\alpha_j}};
(0,-3)*{c_{i,\l} (-t_{ij})^{\l_i-2} c^{-1}_{i,\l-\alpha_i-\alpha_j} (-1)^{\l_j+1} c_{j,\l-\alpha_i}};
\endxy}
 \left(\vcenter{\xy 0;/r.15pc/:
 (-3,0)*{\rcrossrb{i}{j}};
  (6,0)*{\slinedr{i}};
 (3,8)*{\rcrossrr{}{}};
 (-6,8)*{\slinen{}}; (-10,4)*{-};
 \endxy},\vcenter{\xy 0;/r.15pc/:
 (-3,0)*{\rcrossrr{i}{i}};
  (6,0)*{\slined{j}};
 (3,8)*{\rcrossrb{}{}};
 (-6,8)*{\slinenr{}};
 \endxy}\right)\\
&=
\left(\vcenter{\xy 0;/r.15pc/:
 (-3,0)*{\rcrossrb{i}{j}};
  (6,0)*{\slinedr{i}};
 (3,8)*{\rcrossrr{}{}};
 (-6,8)*{\slinen{}}; (-17,3)*{t_{ij}^{-1}t_{ji}^{-1}};
 \endxy},\vcenter{\xy 0;/r.15pc/:
 (-17,3)*{-t_{ij}^{-1}t_{ji}^{-1}}; (-3,0)*{\rcrossrr{i}{i}};
  (6,0)*{\slined{j}};
 (3,8)*{\rcrossrb{}{}};
 (-6,8)*{\slinenr{}};
 \endxy}\right)
=:  \cal{T}'_i \left(  \xy 0;/r.18pc/:
  (0,0)*{\dcrossrb{i}{j}};
 (6,3)*{ \lambda};
 \endxy  \right)\\
\cal{T}'_i \left( \reflectbox{\xy 0;/r.10pc/:
(-9,0)*{\slinenr{}};
(-9,8)*{\slinenr{}};
(-6,-6)*{\srcupr{}};
(6,7)*{\srcapr{}};
(9,0)*{\slinenr{}};
(9,-8)*{\slinedr{}};
(0,0)*{\ucrossrb{}{}};
(-15,0)*{\slinen{}};
(-15,8)*{\slinen{}};
(-6,-9)*{\llrcup{}};
(6,10)*{\llrcap{}};
(15,0)*{\slinen{}};
(15,-8)*{\slined{}};
(15,-15)*{\reflectbox{$\scs j$}};
(9,-15)*{\reflectbox{$\scs i$}};
(-18,0)*{};(18,0)*{};
\endxy} \xy 0;/r.12pc/:(0,2)*{\lambda};\endxy \right)
&=
 {\xy
(0,3)*{(-t_{ij})^{2-\l_i} c_{i,\l-\alpha_j} (-1)^{\l_j-1} c^{-1}_{j,\l-\alpha_i} c^{-1}_{i,\l}};
(0,-3)*{t_{ij} c_{i,\l-\alpha_j} (-t_{ij})^{\l_i} c^{-1}_{i,\l-\alpha_j} (-1)^{\l_j} c_{j,\l} };
\endxy}
\left(\vcenter{\xy 0;/r.15pc/:
 (3,0)*{\lcrossrr{i}{i}};
  (-6,0)*{\slined{j}};
 (-3,8)*{\lcrossbr{}{}};
 (6,8)*{\slinenr{}}; (5,-.5)*[black]{\bullet};
 \endxy}
\vcenter{\xy 0;/r.15pc/:
 (-12,4)*{-}; (3,0)*{\lcrossrr{i}{i}};
  (-6,0)*{\slined{j}};
 (-3,8)*{\lcrossbr{}{}};
 (6,8)*{\slinenr{}}; (1,-.5)*[black]{\bullet};
 \endxy},
 \vcenter{\xy 0;/r.15pc/:
    (3,0)*{\lcrossbr{j}{i}};
    (-6,0)*{\slinedr{i}};
    (-3,8)*{\lcrossrr{}{}};
    (6,8)*{\slinen{}}; (-6,0)*[black]{\bullet};
    \endxy}
 \vcenter{\xy 0;/r.15pc/:
    (-12,4)*{-};
    (3,0)*{\lcrossbr{j}{i}};
    (-6,0)*{\slinedr{i}};
    (-3,8)*{\lcrossrr{}{}};
    (6,8)*{\slinen{}};  (5,-.5)*[black]{\bullet};
    \endxy}\right) \\
&=
t_{ij}^2t_{ji}
\left(\vcenter{\xy 0;/r.15pc/:
 (3,0)*{\lcrossrr{i}{i}};
  (-6,0)*{\slined{j}};
 (-3,8)*{\lcrossbr{}{}};
 (6,8)*{\slinenr{}}; (1,-.5)*[black]{\bullet};
 \endxy}
\vcenter{\xy 0;/r.15pc/:
 (-12,4)*{-}; (3,0)*{\lcrossrr{i}{i}};
  (-6,0)*{\slined{j}};
 (-3,8)*{\lcrossbr{}{}};
 (6,8)*{\slinenr{}}; (5,-.5)*[black]{\bullet};
 \endxy},
 \vcenter{\xy 0;/r.15pc/:
    (3,0)*{\lcrossbr{j}{i}};
    (-6,0)*{\slinedr{i}};
    (-3,8)*{\lcrossrr{}{}};
    (6,8)*{\slinen{}}; (5,-.5)*[black]{\bullet};
    \endxy}
 \vcenter{\xy 0;/r.15pc/:
    (-12,4)*{-};
    (3,0)*{\lcrossbr{j}{i}};
    (-6,0)*{\slinedr{i}};
    (-3,8)*{\lcrossrr{}{}};
    (6,8)*{\slinen{}}; (-6,0)*[black]{\bullet};
    \endxy}\right)
=: \cal{T}'_i \left(  \xy 0;/r.18pc/:
  (0,0)*{\dcrossbr{j}{i}};
 (6,3)*{ \lambda};
 \endxy  \right)
 \end{align*}

For crossings in which at least one strand is $k$-labeled and no strand is $j$-labeled,
the relations are trivial to check. We compute:
\begin{align*}
\cal{T}'_i \left( \reflectbox{\xy 0;/r.10pc/:
(-9,0)*{\slineng{}};
(-9,8)*{\slineng{}};
(-6,-6)*{\srcupg{}};
(6,7)*{\srcapg{}};
(9,0)*{\slineng{}};
(9,-8)*{\slinedg{}};
(0,0)*{\ucrossgg{}{}};
(-15,0)*{\slineng{}};
(-15,8)*{\slineng{}};
(-6,-9)*{\llrcupg{}};
(6,10)*{\llrcapg{}};
(15,0)*{\slineng{}};
(15,-8)*{\slinedg{}};
(15,-15)*{\reflectbox{$\scs k$}};
(9,-15)*{\reflectbox{$\scs k'$}};
(-18,0)*{};(18,0)*{};
\endxy}\xy 0;/r.12pc/:(0,2)*{\lambda};\endxy \right)
&= \reflectbox{\xy 0;/r.10pc/:
(-9,0)*{\slineng{}};
(-9,8)*{\slineng{}};
(-6,-6)*{\srcupg{}};
(6,7)*{\srcapg{}};
(9,0)*{\slineng{}};
(9,-8)*{\slinedg{}};
(0,0)*{\ucrossgg{}{}};
(-15,0)*{\slineng{}};
(-15,8)*{\slineng{}};
(-6,-9)*{\llrcupg{}};
(6,10)*{\llrcapg{}};
(15,0)*{\slineng{}};
(15,-8)*{\slinedg{}};
(15,-15)*{\reflectbox{$\scs k$}};
(9,-15)*{\reflectbox{$\scs k'$}};
(-18,0)*{};(18,0)*{};
\endxy}
=     \vcenter{\xy 0;/r.18pc/:
 (0,0)*{\dcrossgg{k}{k'}}; \endxy}
=: \cal{T}'_i \left(  \xy 0;/r.18pc/:
  (0,0)*{\dcrossgg{k}{k'}};
 (6,3)*{ \lambda};
 \endxy  \right) \\
\cal{T}'_i \left( \reflectbox{\xy 0;/r.10pc/:
(-9,0)*{\slineng{}};
(-9,8)*{\slineng{}};
(-6,-6)*{\srcupg{}};
(6,7)*{\srcapg{}};
(9,0)*{\slineng{}};
(9,-8)*{\slinedg{}};
(0,0)*{\ucrossgr{}{}};
(-15,0)*{\slinenr{}};
(-15,8)*{\slinenr{}};
(-6,-9)*{\llrcupr{}};
(6,10)*{\llrcapr{}};
(15,0)*{\slinenr{}};
(15,-8)*{\slinedr{}};
(15,-15)*{\reflectbox{$\scs i$}};
(9,-15)*{\reflectbox{$\scs k$}};
(-18,0)*{};(18,0)*{};
\endxy}\xy 0;/r.12pc/:(0,2)*{\lambda};\endxy \right)
&=
c^{-1}_{i,\l-\alpha_i-\alpha_k} c_{i,\l} t_{ki}
\reflectbox{\xy 0;/r.10pc/:
(-9,0)*{\slineng{}};
(-9,8)*{\slineng{}};
(-6,-6)*{\srcupg{}};
(6,7)*{\srcapg{}};
(9,0)*{\slineng{}};
(9,-8)*{\slinedg{}};
(0,0)*{\rcrossgr{}{}};
(-15,0)*{\slinenr{}};
(-15,8)*{\slineur{}};
(-6,-9)*{\lllcupr{}};
(6,10)*{\lllcapr{}};
(15,0)*{\slinenr{}};
(15,-8)*{\slinenr{}};
(15,-15)*{\reflectbox{$\scs i$}};
(9,-15)*{\reflectbox{$\scs k$}};
(-18,0)*{};(18,0)*{};
\endxy}
=  \vcenter{\xy 0;/r.18pc/:
 (3,0)*{\rcrossrg{i}{k}}; (-6,0)*{t_{ki}^2};
 \endxy}
 =:  \cal{T}'_i \left(  \xy 0;/r.18pc/:
  (0,0)*{\dcrossrg{i}{k}};
 (6,3)*{ \lambda};
 \endxy  \right)
\\
\cal{T}'_i \left( \reflectbox{\xy 0;/r.10pc/:
(-9,0)*{\slinenr{}};
(-9,8)*{\slinenr{}};
(-6,-6)*{\srcupr{}};
(6,7)*{\srcapr{}};
(9,0)*{\slinenr{}};
(9,-8)*{\slinedr{}};
(0,0)*{\ucrossrg{}{}};
(-15,0)*{\slineng{}};
(-15,8)*{\slineng{}};
(-6,-9)*{\llrcupg{}};
(6,10)*{\llrcapg{}};
(15,0)*{\slineng{}};
(15,-8)*{\slinedg{}};
(15,-15)*{\reflectbox{$\scs k$}};
(9,-15)*{\reflectbox{$\scs i$}};
(-18,0)*{};(18,0)*{};
\endxy}\xy 0;/r.12pc/:(0,2)*{\lambda};\endxy  \right)
&=
c_{i,\l-\alpha_k} c^{-1}_{i,\l-\alpha_i}
\reflectbox{\xy 0;/r.10pc/:
(-9,0)*{\slinenr{}};
(-9,8)*{\slineur{}};
(-6,-6)*{\slcupr{}};
(6,7)*{\slcapr{}};
(9,0)*{\slinenr{}};
(9,-8)*{\slinenr{}};
(0,0)*{\lcrossrg{}{}};
(-15,0)*{\slineng{}};
(-15,8)*{\slineng{}};
(-6,-9)*{\llrcupg{}};
(6,10)*{\llrcapg{}};
(15,0)*{\slineng{}};
(15,-8)*{\slinedg{}};
(15,-15)*{\reflectbox{$\scs k$}};
(9,-15)*{\reflectbox{$\scs i$}};
(-18,0)*{};(18,0)*{};
\endxy}
=\vcenter{\xy 0;/r.18pc/:
 (3,0)*{\lcrossgr{k}{i}}; (-6,0)*{t_{ki}^{-1}};
 \endxy}
=: \cal{T}'_i \left(  \xy 0;/r.18pc/:
  (0,0)*{\dcrossgr{k}{i}};
 (6,3)*{ \lambda};
 \endxy  \right)
\end{align*}

For strands labelled $j$ and $k$, we compute:
\begin{align*}
\cal{T}'_i \left( \reflectbox{\xy 0;/r.10pc/:
(-9,0)*{\slineng{}};
(-9,8)*{\slineng{}};
(-6,-6)*{\srcupg{}};
(6,7)*{\srcapg{}};
(9,0)*{\slineng{}};
(9,-8)*{\slinedg{}};
(0,0)*{\ucrossgb{}{}};
(-15,0)*{\slinen{}};
(-15,8)*{\slinen{}};
(-6,-9)*{\llrcup{}};
(6,10)*{\llrcap{}};
(15,0)*{\slinen{}};
(15,-8)*{\slined{}};
(15,-15)*{\reflectbox{$\scs j$}};
(9,-15)*{\reflectbox{$\scs k$}};
(-18,0)*{};(18,0)*{};
\endxy} \xy  0;/r.12pc/:(0,2)*{\lambda};\endxy \right)
&=
 {\xy
(0,3)*{(-t_{ij})^{1-(\l_i+1)} c_{i,\l-\alpha_j-\alpha_k} (-1)^{\l_j-2-j\cdot k} c^{-1}_{j,\l-\alpha_j-\alpha_k}};
(0,-3)*{t^{-1}_{ki} (-t_{ij})^{\l_i} c^{-1}_{i,\l-\alpha_j} (-1)^{\l_j} c_{j,\l}};
\endxy}
\left(\vcenter{\xy 0;/r.15pc/:
 (3,0)*{\dcrossrg{i}{k}};
  (-6,0)*{\slined{j}};
 (-3,8)*{\ncrossbg{}{}};
 (6,8)*{\slinenr{}};
 \endxy},
 \vcenter{\xy 0;/r.15pc/:
 (3,0)*{\dcrossbg{j}{k}};
  (-6,0)*{\slinedr{i}};
 (-3,8)*{\ncrossrg{}{}};
 (6,8)*{\slinen{}};
 \endxy}\right)\\
&=
(-1)^{j \cdot k} t^{-2}_{ki} t_{jk}
\left(\vcenter{\xy 0;/r.15pc/:
 (3,0)*{\dcrossrg{i}{k}};
  (-6,0)*{\slined{j}};
 (-3,8)*{\ncrossbg{}{}};
 (6,8)*{\slinenr{}};
 \endxy},
 \vcenter{\xy 0;/r.15pc/:
 (3,0)*{\dcrossbg{j}{k}};
  (-6,0)*{\slinedr{i}};
 (-3,8)*{\ncrossrg{}{}};
 (6,8)*{\slinen{}};
 \endxy}\right)
=: \cal{T}'_i \left(  \xy 0;/r.18pc/:
  (0,0)*{\dcrossbg{j}{k}};
 (6,3)*{ \lambda};
 \endxy  \right)\\
\cal{T}'_i \left( \reflectbox{\xy 0;/r.10pc/:
(-9,0)*{\slinen{}};
(-9,8)*{\slinen{}};
(-6,-6)*{\srcup{}};
(6,7)*{\srcap{}};
(9,0)*{\slinen{}};
(9,-8)*{\slined{}};
(0,0)*{\ucrossbg{}{}};
(-15,0)*{\slineng{}};
(-15,8)*{\slineng{}};
(-6,-9)*{\llrcupg{}};
(6,10)*{\llrcapg{}};
(15,0)*{\slineng{}};
(15,-8)*{\slinedg{}};
(15,-15)*{\reflectbox{$\scs k$}};
(9,-15)*{\reflectbox{$\scs j$}};
(-18,0)*{};(18,0)*{};
\endxy}\xy  0;/r.12pc/:(0,2)*{\lambda};\endxy \right)
&=
{\xy
(0,3)*{(-t_{ij})^{1-(\l_i+1)} c_{i,\l-\alpha_j} (-1)^{\l_j-2} c^{-1}_{j,\l-\alpha_j}};
(0,-3)*{(-t_{ij})^{\l_i} c^{-1}_{i,\l-\alpha_k-\alpha_j} (-1)^{\l_j-j \cdot k} c_{j,\l-\alpha_k}};
\endxy}
\left(\vcenter{\xy 0;/r.15pc/:
 (-3,0)*{\dcrossgb{k}{j}};
  (6,0)*{\slinedr{i}};
 (3,8)*{\ncrossgr{}{}};
 (-6,8)*{\slinen{}};
 \endxy},\vcenter{\xy 0;/r.15pc/:
 (-3,0)*{\dcrossgr{k}{i}};
  (6,0)*{\slined{j}};
 (3,8)*{\ncrossgb{}{}};
 (-6,8)*{\slinenr{}};
 \endxy}\right)\\
&=
(-1)^{j \cdot k} t_{ki} t^{-1}_{jk}
\left(\vcenter{\xy 0;/r.15pc/:
 (-3,0)*{\dcrossgb{k}{j}};
  (6,0)*{\slinedr{i}};
 (3,8)*{\ncrossgr{}{}};
 (-6,8)*{\slinen{}};
 \endxy},\vcenter{\xy 0;/r.15pc/:
 (-3,0)*{\dcrossgr{k}{i}};
  (6,0)*{\slined{j}};
 (3,8)*{\ncrossgb{}{}};
 (-6,8)*{\slinenr{}};
 \endxy}\right)
=\cal{T}'_i \left(  \xy 0;/r.18pc/:
  (0,0)*{\dcrossgb{k}{j}};
 (6,3)*{ \lambda};
 \endxy  \right)\end{align*}

Finally, in the case of a crossing between strands labeled $j$ and $j'$,
it's clear that
\[
\cal{T}'_i \left( \reflectbox{\xy 0;/r.10pc/:
(-9,0)*{\slinenp{}};
(-9,8)*{\slinenp{}};
(-6,-6)*{\srcupp{}};
(6,7)*{\srcapp{}};
(9,0)*{\slinenp{}};
(9,-8)*{\slinedp{}};
(0,0)*{\ucrosspb{}{}};
(-15,0)*{\slinen{}};
(-15,8)*{\slinen{}};
(-6,-9)*{\llrcup{}};
(6,10)*{\llrcap{}};
(15,0)*{\slinen{}};
(15,-8)*{\slined{}};
(15,-15)*{\reflectbox{$\scs j$}};
(9,-15)*{\reflectbox{$\scs j'$}};
(-18,0)*{};(18,0)*{};
\endxy} \right)
= C \cdot
\cal{T}'_i \left(  \xy 0;/r.18pc/:
  (0,0)*{\dcrossbp{j}{j'}};
 (6,3)*{ \lambda};
 \endxy  \right)
\]
for some scalar $C$.
A direct computation shows that
\[
C =
t_{ij}^{-1} t_{ij'} c_{i,\l-\alpha_{j'}} c^{-1}_{i,\l-\alpha_j} c_{j,\l} c^{-1}_{j,\l-\alpha_{j'}}
c_{j',\l-\alpha_j} c^{-1}_{j',\l}
(c_{j',\l-\alpha_j} c^{-1}_{j',\l}  c_{j,\l} c^{-1}_{j,\l-\alpha_{j'}})^{-1} = 1
\]
\end{proof}

%
\subsection{Quadratic KLR}
%

\begin{prop}\label{prop:quadKLR}
$\cal{T}_i'$ preserves the quadratic KLR relation.
\end{prop}

\begin{proof}
We verify relation \eqref{def:KLR-R2} in Definition~\ref{defU_cat-cyc}, first considering
the cases that do not require homotopies.
We compute:
\begin{align*}
\cal{T}'_i\left({\xy 0;/r.15pc/: (0,-4)*{\ucrossrr{i}{i}}; (0,4)*{\ucrossrr{}{}}; \endxy}\right)
&=(-1)^2{\xy 0;/r.15pc/: (0,-4)*{\dcrossrr{i}{i}}; (0,4)*{\dcrossrr{}{}}; \endxy}
=0\\
\cal{T}'_i\left({\xy 0;/r.15pc/: (0,-4)*{\ucrossgg{k}{k'}}; (0,4)*{\ucrossgg{}{}}; \endxy}\right)
&={\xy 0;/r.15pc/: (0,-4)*{\ucrossgg{k}{k'}}; (0,4)*{\ucrossgg{}{}}; \endxy}
=\left\{
\begin{array}{cl}
0& \text{if $k=k'$}\\
\cal{T}'_i\left(t_{kk'}{\xy (0,0)*{\slineug{k}}; (3,0)*{\slineug{k'}}; \endxy}\right)& \text{if $k\cdot k'=0$}\\
\cal{T}'_i\left(t_{kk'}{\xy (0,0)*{\sdotug{k}}; (3,0)*{\slineug{k'}}; \endxy}+t_{k'k}{\xy (0,0)*{\slineug{k}}; (3,0)*{\sdotug{k'}}; \endxy}\right)& \text{if $k\cdot k'=-1$}
\end{array}\right.  \\
\cal{T}'_i\left(\xy 0;/r.15pc/: (0,-4)*{\ucrossgr{k}{i}}; (0,4)*{\ucrossrg{}{}}; \endxy\right)
&=t_{ki}\xy 0;/r.15pc/: (0,-4)*{\rcrossgr{k}{i}}; (0,4)*{\lcrossrg{}{}}; \endxy
=t_{ki}\xy  (-2,0)*{\slineug{k}}; (2,0)*{\slinedr{i}}; \endxy
=\cal{T}'_i\left(t_{ki}\xy  (-2,0)*{\slineug{k}}; (2,0)*{\slineur{i}}; \endxy\right)\\ \\
\cal{T}'_i\left(\xy 0;/r.15pc/: (0,-4)*{\ucrossrg{i}{k}}; (0,4)*{\ucrossgr{}{}}; \endxy\right)
&=t_{ki}\xy 0;/r.15pc/: (0,-4)*{\lcrossrg{i}{k}}; (0,4)*{\rcrossgr{}{}}; \endxy
=t_{ki}\xy (-2,0)*{\slinedr{i}}; (2,0)*{\slineug{k}}; \endxy
=\cal{T}'_i\left(t_{ki}\xy (-2,0)*{\slineur{i}}; (2,0)*{\slineug{k}}; \endxy\right)
\end{align*}
Our next four cases concern endomorphisms of chain complexes concentrated in two adjacent homological degrees;
we denote endomorphisms of such complexes using ordered pairs. We compute:
\begin{align*}
\cal{T}'_i \left(\xy 0;/r.15pc/: (0,-4)*{\ucrossgb{k}{j}}; (0,4)*{\ucrossbg{}{}}; \endxy\right)
&=\left(t_{ki}^{-1}\xy 0;/r.10pc/:
 (-3,-10)*{\ncrossgb{k}{j}};
  (6,-10)*{\slinenr{i}};
 (3,-1.5)*{\ucrossgr{}{}};
 (-6,-1.5)*{\slineu{}};
 (3,7)*{\ncrossrg{}{}};
  (-6,7)*{\slinen{}};
 (-3,15.5)*{\ucrossbg{}{}};
 (6,15.5)*{\slineur{}};
 \endxy
 \; , \;
t_{ki}^{-1}\xy 0;/r.10pc/:
 (-3,-10)*{\ncrossgr{k}{i}};
  (6,-10)*{\slinen{j}};
 (3,-1.5)*{\ucrossgb{}{}};
 (-6,-1.5)*{\slineur{}};
 (3,7)*{\ncrossbg{}{}};
  (-6,7)*{\slinenr{}};
 (-3,15.5)*{\ucrossrg{}{}};
 (6,15.5)*{\slineu{}};
 \endxy\right) \\
 & = \left\{\begin{array}{cl}
\left(t_{kj}\xy (-3,0)*{\slineug{k}}; (0,0)*{\slineu{j}}; (3,0)*{\slineur{i}}; \endxy
\; , \;
t_{kj}\xy (-3,0)*{\slineug{k}}; (0,0)*{\slineur{i}}; (3,0)*{\slineu{j}}; \endxy\right) &\text{\ if $j\cdot k=0$} \\ \\
\left(t_{kj}\xy (-3,0)*{\sdotug{k}}; (0,0)*{\slineu{j}}; (3,0)*{\slineur{i}};\endxy + t_{jk}\xy (-3,0)*{\slineug{k}}; (0,0)*{\sdotu{j}}; (3,0)*{\slineur{i}}; \endxy
\; , \;
t_{kj}\xy (-3,0)*{\sdotug{k}}; (0,0)*{\slineur{i}}; (3,0)*{\slineu{j}}; \endxy + t_{jk}\xy (-3,0)*{\slineug{k}}; (0,0)*{\slineur{i}}; (3,0)*{\sdotu{j}}; \endxy\right) &\text{\ if $j\cdot k=-1$}
\end{array}\right. \\
&= \left\{\begin{array}{cl}
\cal{T}'_i\left(t_{kj}\xy (-2,0)*{\slineug{k}}; (2,0)*{\slineu{j}}; \endxy\right) &\text{\ if $j\cdot k=0$} \\
\cal{T}'_i\left( t_{kj}\xy (-2,0)*{\sdotug{k}}; (2,0)*{\slineu{j}}; \endxy + t_{jk}\xy (-2,0)*{\slineug{k}}; (2,0)*{\sdotu{j}}; \endxy \right) &\text{\ if $j\cdot k=-1$}
\end{array}\right.
\end{align*}
and similarly:
\[
\cal{T}'_i \left(\xy 0;/r.15pc/: (0,-4)*{\ucrossbg{j}{k}}; (0,4)*{\ucrossgb{}{}}; \endxy\right)
= \left(t_{ki}^{-1}\xy 0;/r.10pc/:
 (3,-10)*{\ncrossrg{i}{k}};
  (-6,-10)*{\slinen{j}};
 (-3,-1.5)*{\ucrossbg{}{}};
 (6,-1.5)*{\slineur{}};
 (-3,7)*{\ncrossgb{}{}};
  (6,7)*{\slinenr{}};
 (3,15.5)*{\ucrossgr{}{}};
 (-6,15.5)*{\slineu{}};
 \endxy
\; , \;
t_{ki}^{-1}\xy 0;/r.10pc/:
 (3,-10)*{\ncrossbg{j}{k}};
  (-6,-10)*{\slinenr{i}};
 (-3,-1.5)*{\ucrossrg{}{}};
 (6,-1.5)*{\slineu{}};
 (-3,7)*{\ncrossgr{}{}};
  (6,7)*{\slinen{}};
 (3,15.5)*{\ucrossgb{}{}};
 (-6,15.5)*{\slineur{}};
 \endxy\right)  \\
= \left\{\begin{array}{cl}
\cal{T}'_i\left(t_{jk}\xy (-2,0)*{\slineu{j}}; (2,0)*{\slineug{k}}; \endxy\right) &\text{\ if $j\cdot k=0$} \\
\cal{T}'_i\left(t_{jk}\xy (-2,0)*{\sdotu{j}}; (2,0)*{\slineug{k}}; \endxy+t_{kj}\xy (-2,0)*{\slineu{j}}; (2,0)*{\sdotug{k}}; \endxy\right) &\text{\ if $j\cdot k=-1$}
\end{array}\right.
\]

The remaining cases only hold up to chain homotopy. We compute:
\begin{align*}
\cal{T}'_i \left(\xy 0;/r.15pc/: (0,-4)*{\ucrossbr{j}{i}}; (0,4)*{\ucrossrb{}{}}; \endxy\right)
&= t_{ij}\left({\xy (0,-5)*{\xy 0;/r.15pc/:
 (3,0)*{\rcrossrr{i}{i}};
  (-6,0)*{\slinen{j}};
 (-3,8)*{\rcrossbr{}{}};
 (6,8)*{\slinenr{}}; (6,11)*[black]{\bullet};
 \endxy}; (0,5)*{\xy 0;/r.15pc/:
 (-3,0)*{\lcrossrb{}{}};
  (6,0)*{\slinenr{}};
 (3,8)*{\lcrossrr{}{}};
 (-6,8)*{\slineu{}};
 \endxy}; \endxy}
-{\xy (0,-5)*{\xy 0;/r.15pc/:
 (3,0)*{\rcrossrr{i}{i}};
  (-6,0)*{\slinen{j}};
 (-3,8)*{\rcrossbr{}{}};
 (6,8)*{\slineur{}}; (-4.5,10.5)*[black]{\bullet};
 \endxy}; (0,5)*{\xy 0;/r.15pc/:
 (-3,0)*{\lcrossrb{}{}};
  (6,0)*{\slinenr{}};
 (3,8)*{\lcrossrr{}{}};
 (-6,8)*{\slineu{}};
 \endxy}; \endxy}
 ,
{\xy (0,-5)*{\xy 0;/r.15pc/:
    (3,0)*{\rcrossbr{j}{i}};
    (-6,0)*{\slinenr{i}};
    (-3,8)*{\rcrossrr{}{}};
    (6,8)*{\slineu{}};
    \endxy}; (0,5)*{ \xy 0;/r.15pc/:
    (-3,0)*{\lcrossrr{}{}};
  (6,0)*{\slinen{}};
 (3,8)*{\lcrossrb{}{}};
 (-6,8)*{\slineur{}};
 \endxy};
(-.5,2)*[black]{\bullet};
 \endxy}
-{\xy (0,-5)*{\xy 0;/r.15pc/:
    (3,0)*{\rcrossbr{j}{i}};
    (-6,0)*{\slinenr{i}};
    (-3,8)*{\rcrossrr{}{}};
    (6,8)*{\slineu{}};
    \endxy}; (0,5)*{ \xy 0;/r.15pc/:
    (-3,0)*{\lcrossrr{}{}};
  (6,0)*{\slinen{}};
 (3,8)*{\lcrossrb{}{}};
 (-6,8)*{\slineur{}};
 \endxy};
(-3,-1)*[black]{\bullet};
 \endxy}\right) \\
&= t_{ij}\left(
-{\xy 0;/r.18pc/:
 (0,-4)*{\slinenr{i}};
 (0,4)*{\slineur{}};
  (-6,-4)*{\slinen{j}};
  (-6,4)*{\slineu{}}; (0, 0)*[black]{\bullet};
 (6,-4)*{\slinedr{i}};
  (6,4)*{\slinenr{}}; \endxy}
+{\xy 0;/r.18pc/:
 (0,-4)*{\slinenr{i}};
 (0,4)*{\slineur{}};
  (-6,-4)*{\slinen{j}};
  (-6,4)*{\slineu{}}; (6, 0)*[black]{\bullet};
 (6,-4)*{\slinedr{i}};
  (6,4)*{\slinenr{}}; \endxy}
\; , \;
- {\xy 0;/r.18pc/:
 (-6,-4)*{\slinenr{i}};
 (-6,4)*{\slineur{}};
  (0,-4)*{\slinen{j}};
  (0,4)*{\slineu{}}; (-6, 0)*[black]{\bullet};
 (6,-4)*{\slinedr{i}};
  (6,4)*{\slinenr{}}; \endxy}
+{\xy 0;/r.18pc/:
 (-6,-4)*{\slinenr{i}};
 (-6,4)*{\slineur{}};
  (0,-4)*{\slinen{j}};
  (0,4)*{\slineu{}}; (6, 0)*[black]{\bullet};
 (6,-4)*{\slinedr{i}};
  (6,4)*{\slinenr{}}; \endxy}\right) \\
&= \left(- \hspace{-5pt} {\xy 0;/r.18pc/:
 (-3,4)*{\ucrossrb{}{}};
 (-3,-4)*{\ucrossbr{j}{i}};
 (6,-4)*{\slinedr{i}};
  (6,4)*{\slinenr{}}; \endxy}
+t_{ji}{\xy 0;/r.18pc/:
 (0,-4)*{\slinenr{i}};
 (0,4)*{\slineur{}};
  (-6,-4)*{\slinen{j}};
  (-6,4)*{\slineu{}}; (-6, 0)*{\bullet};
 (6,-4)*{\slinedr{i}};
  (6,4)*{\slinenr{}}; \endxy}
+t_{ij}{\xy 0;/r.18pc/:
 (0,-4)*{\slinenr{i}};
 (0,4)*{\slineur{}};
  (-6,-4)*{\slinen{j}};
  (-6,4)*{\slineu{}}; (6, 0)*[black]{\bullet};
 (6,-4)*{\slinedr{i}};
  (6,4)*{\slinenr{}}; \endxy}
\; , \;
- \hspace{-5pt} {\xy 0;/r.18pc/:
 (-3,4)*{\ucrossbr{}{}};
 (-3,-4)*{\ucrossrb{i}{j}};
 (6,-4)*{\slinedr{i}};
  (6,4)*{\slinenr{}}; \endxy}
+t_{ji}{\xy 0;/r.18pc/:
 (-6,-4)*{\slinenr{i}};
 (-6,4)*{\slineur{}};
  (0,-4)*{\slinen{j}};
  (0,4)*{\slineu{}}; (0, 0)*{\bullet};
 (6,-4)*{\slinedr{i}};
  (6,4)*{\slinenr{}}; \endxy}
+t_{ij}{\xy 0;/r.18pc/:
 (-6,-4)*{\slinenr{i}};
 (-6,4)*{\slineur{}};
  (0,-4)*{\slinen{j}};
  (0,4)*{\slineu{}}; (6, 0)*[black]{\bullet};
 (6,-4)*{\slinedr{i}};
  (6,4)*{\slinenr{}}; \endxy}\right)\\
&= \cal{T}'_i\left(t_{ij}{\xy (0,0)*{\slineu{j}}; (4,0)*{\sdotur{i}}; \endxy}+t_{ji}{\xy (0,0)*{\sdotu{j}}; (4,0)*{\slineur{i}}; \endxy}\right)
+\left(-{\xy 0;/r.12pc/:
 (-3,4)*{\ucrossrb{}{}};
 (-3,-4)*{\ucrossbr{j}{i}};
 (6,-4)*{\slinedr{i}};
  (6,4)*{\slinenr{}}; \endxy}
\; , \;
-{\xy 0;/r.12pc/:
 (-3,4)*{\ucrossbr{}{}};
 (-3,-4)*{\ucrossrb{i}{j}};
 (6,-4)*{\slinedr{i}};
  (6,4)*{\slinenr{}}; \endxy}\right)
\end{align*}
where in the second step we make use of the equality
\[
\xy 0;/r.15pc/:
(0,4)*{\lcrossrr{}{}};
(0,-4)*{\rcrossrr{i}{i}};
(2,3)*[black]{\bullet};
\endxy -
\xy 0;/r.15pc/:
(0,4)*{\lcrossrr{}{}};
(0,-4)*{\rcrossrr{i}{i}};
(-2,-1)*[black]{\bullet};
\endxy
= \quad
-{\xy 0;/r.15pc/:
 (0,-4)*{\slinenr{i}};
 (0,4)*{\slineur{}};
(0, 0)*[black]{\bullet};
 (6,-4)*{\slinedr{i}};
  (6,4)*{\slinenr{}}; \endxy}
+{\xy 0;/r.15pc/:
 (0,-4)*{\slinenr{i}};
 (0,4)*{\slineur{}};
(6, 0)*[black]{\bullet};
 (6,-4)*{\slinedr{i}};
  (6,4)*{\slinenr{}}; \endxy}
\]
which holds in any weight.
The result now follows since the chain endomorphism $\left(-{\xy 0;/r.12pc/:
(-3,4)*{\ucrossrb{}{}};
(-3,-4)*{\ucrossbr{j}{i}};
(6,-4)*{\slinedr{i}};
(6,4)*{\slinenr{}}; \endxy}
\; , \;
- {\xy 0;/r.12pc/:
(-3,4)*{\ucrossbr{}{}};
(-3,-4)*{\ucrossrb{i}{j}};
(6,-4)*{\slinedr{i}};
(6,4)*{\slinenr{}}; \endxy}\right)$
is null-homotopic with homotopy $h: \cal{T}_i'(\cal{E}_j\cal{E}_i\onel) \to \cal{T}_i'(\cal{E}_j\cal{E}_i\onel \la 2 \ra)$ given by:
\begin{align*}
     \xy 0;/r.15pc/:
  (-45,15)*+{\cal{E}_j\cal{E}_i\cal{F}_i\onell{s_i(\lambda)}\la -\l_i \ra}="1";
  (-45,-15)*+{\cal{E}_j\cal{E}_i\cal{F}_i\onell{s_i(\lambda)}\la -2-\l_i \ra}="2";
  (45,15)*+{\clubsuit\cal{E}_i\cal{E}_j\cal{F}_i\onell{s_i(\lambda)}\la 1-\l_i \ra}="3";
  (45,-15)*+{\clubsuit\cal{E}_i\cal{E}_j\cal{F}_i\onell{s_i(\lambda)}\la -1-\l_i \ra}="4";
   {\ar_<<<<<{\xy 0;/r.12pc/:
  (-6,0)*{\ucrossrb{}{}};(-13,0)*{-};
  (3,0)*{\slinedr{}}; \endxy}@/_1pc/ "4";"1"};
   {\ar^-{\xy  0;/r.12pc/:(3,0)*{\slinedr{}};(-6,0)*{\ucrossbr{}{}}; (7,0)*{}; \endxy   } "2";"4"};
   {\ar^-{\xy  0;/r.12pc/:(3,0)*{\slinedr{}};(-6,0)*{\ucrossbr{}{}}; (7,0)*{};\endxy   } "1";"3"};
 \endxy \end{align*}
We similarly compute:
\begin{align*}
\cal{T}'_i \left(\xy 0;/r.15pc/: (0,-4)*{\ucrossrb{i}{j}}; (0,4)*{\ucrossbr{}{}}; \endxy\right)
&= t_{ij}\left({\xy (0,5)*{\xy 0;/r.15pc/:
 (3,0)*{\rcrossrr{}{}};
  (-6,0)*{\slinen{}};
 (-3,8)*{\rcrossbr{}{}};
 (6,8)*{\slineur{}};
 \endxy};
(0,-5)*{\xy 0;/r.15pc/:
 (-3,0)*{\lcrossrb{i}{j}};
  (6,0)*{\slinenr{i}};
 (3,8)*{\lcrossrr{}{}};
 (-6,8)*{\slineu{}};
 \endxy};
(4,8)*[black]{\bullet};
\endxy}
-{\xy (0,5)*{\xy 0;/r.15pc/:
 (3,0)*{\rcrossrr{}{}};
  (-6,0)*{\slinen{}};
 (-3,8)*{\rcrossbr{}{}};
 (6,8)*{\slineur{}}; (-5,11)*[black]{\bullet};
 \endxy}; (0,-5)*{\xy 0;/r.15pc/:
 (-3,0)*{\lcrossrb{i}{j}};
  (6,0)*{\slinenr{i}};
 (3,8)*{\lcrossrr{}{}};
 (-6,8)*{\slineu{}};
 \endxy}; \endxy}
 ,
{\xy (0,5)*{\xy 0;/r.15pc/:
    (3,0)*{\rcrossbr{}{}};
    (-6,0)*{\slinenr{}};
    (-3,8)*{\rcrossrr{}{}};
    (6,8)*{\slineu{}}; (-1.5,10.5)*[black]{\bullet};
    \endxy}; (0,-5)*{ \xy 0;/r.15pc/:
    (-3,0)*{\lcrossrr{i}{i}};
  (6,0)*{\slinen{j}};
 (3,8)*{\lcrossrb{}{}};
 (-6,8)*{\slineur{}};
 \endxy}; \endxy}
-{\xy (0,5)*{\xy 0;/r.15pc/:
    (3,0)*{\rcrossbr{}{}};
    (-6,0)*{\slinenr{}};
    (-3,8)*{\rcrossrr{}{}};
    (6,8)*{\slineu{}}; (-5,11)*[black]{\bullet};
    \endxy}; (0,-5)*{ \xy 0;/r.15pc/:
    (-3,0)*{\lcrossrr{i}{i}};
  (6,0)*{\slinen{j}};
 (3,8)*{\lcrossrb{}{}};
 (-6,8)*{\slineur{}};
 \endxy}; \endxy}\right) \\
&= \left(-{\xy 0;/r.18pc/:
 (3,4)*{\ucrossrb{}{}};
 (3,-4)*{\ucrossbr{j}{i}};
 (-6,-4)*{\slinedr{i}};
  (-6,4)*{\slinenr{}}; \endxy} \hspace{-5pt}
+t_{ji}{\xy 0;/r.18pc/:
 (6,-4)*{\slinenr{i}};
 (6,4)*{\slineur{}};
  (0,-4)*{\slinen{j}};
  (0,4)*{\slineu{}}; (0, 0)*{\bullet};
 (-6,-4)*{\slinedr{i}};
  (-6,4)*{\slinenr{}}; \endxy}
+t_{ij}{\xy 0;/r.18pc/:
 (6,-4)*{\slinenr{i}};
 (6,4)*{\slineur{}};
  (0,-4)*{\slinen{j}};
  (0,4)*{\slineu{}}; (-6, 0)*[black]{\bullet};
 (-6,-4)*{\slinedr{i}};
  (-6,4)*{\slinenr{}}; \endxy}
  \; , \;
-{\xy 0;/r.18pc/:
 (3,4)*{\ucrossbr{}{}};
 (3,-4)*{\ucrossrb{i}{j}};
 (-6,-4)*{\slinedr{i}};
  (-6,4)*{\slinenr{}}; \endxy} \hspace{-5pt}
+t_{ji}{\xy 0;/r.18pc/:
 (0,-4)*{\slinenr{i}};
 (0,4)*{\slineur{}};
  (6,-4)*{\slinen{j}};
  (6,4)*{\slineu{}}; (6, 0)*{\bullet};
 (-6,-4)*{\slinedr{i}};
  (-6,4)*{\slinenr{}}; \endxy}
+t_{ij}{\xy 0;/r.18pc/:
 (0,-4)*{\slinenr{i}};
 (0,4)*{\slineur{}};
  (6,-4)*{\slinen{j}};
  (6,4)*{\slineu{}}; (-6, 0)*[black]{\bullet};
 (-6,-4)*{\slinedr{i}};
  (-6,4)*{\slinenr{}}; \endxy}\right)\\
&= \cal{T}'_i\left(t_{ij}{\xy (0,0)*{\slineu{j}}; (-4,0)*{\sdotur{i}}; \endxy}+t_{ji}{\xy (0,0)*{\sdotu{j}}; (-4,0)*{\slineur{i}}; \endxy}\right)
+\left(-{\xy 0;/r.12pc/:
 (3,4)*{\ucrossrb{}{}};
 (3,-4)*{\ucrossbr{j}{i}};
 (-6,-4)*{\slinedr{i}};
  (-6,4)*{\slinenr{}}; \endxy}
  \; , \;
-{\xy 0;/r.12pc/:
 (3,4)*{\ucrossbr{}{}};
 (3,-4)*{\ucrossrb{i}{j}};
 (-6,-4)*{\slinedr{i}};
  (-6,4)*{\slinenr{}}; \endxy}\right)
\end{align*}
where in this case we use the equality
\[
\xy 0;/r.15pc/:
(0,4)*{\rcrossrr{}{}};
(0,-4)*{\lcrossrr{i}{i}};
(2,6.5)*[black]{\bullet};
\endxy -
\xy 0;/r.15pc/:
(0,4)*{\rcrossrr{}{}};
(0,-4)*{\lcrossrr{i}{i}};
(-2,6.5)*[black]{\bullet};
\endxy
= \quad
{\xy 0;/r.15pc/:
 (0,-4)*{\slinenr{i}};
 (0,4)*{\slineur{}};
(0, 0)*[black]{\bullet};
 (6,-4)*{\slinedr{i}};
  (6,4)*{\slinenr{}}; \endxy}
-{\xy 0;/r.15pc/:
 (0,-4)*{\slinenr{i}};
 (0,4)*{\slineur{}};
(6, 0)*[black]{\bullet};
 (6,-4)*{\slinedr{i}};
  (6,4)*{\slinenr{}}; \endxy}
\]
which again holds in any weight.
The relation is verified since the chain endomorphism $\left(-{\xy 0;/r.12pc/:
 (3,4)*{\ucrossrb{}{}};
 (3,-4)*{\ucrossbr{j}{i}};
 (-6,-4)*{\slinedr{i}};
  (-6,4)*{\slinenr{}}; \endxy}
  \; , \;
-{\xy 0;/r.12pc/:
 (3,4)*{\ucrossbr{}{}};
 (3,-4)*{\ucrossrb{i}{j}};
 (-6,-4)*{\slinedr{i}};
  (-6,4)*{\slinenr{}}; \endxy}\right)$
is null-homotopic, with homotopy given by
\[
     \xy 0;/r.15pc/:
  (-45,15)*+{\cal{F}_i\cal{E}_j\cal{E}_i\onell{s_i(\lambda)}\la 1-\l_i \ra}="1";
  (-45,-15)*+{\cal{F}_i\cal{E}_j\cal{E}_i\onell{s_i(\lambda)}\la -1-\l_i \ra}="2";
  (45,15)*+{\clubsuit\cal{F}_i\cal{E}_i\cal{E}_j\onell{s_i(\lambda)}\la 2-\l_i \ra}="3";
  (45,-15)*+{\clubsuit\cal{F}_i\cal{E}_i\cal{E}_j\onell{s_i(\lambda)}\la -\l_i \ra}="4";
  {\ar_<<<<<{\xy 0;/r.12pc/:
  (6,0)*{\ucrossrb{}{}};
  (-3,0)*{\slinedr{}}; \endxy}@/_1pc/ "4";"1"};
   {\ar^-{-\xy  0;/r.12pc/:(-3,0)*{\slinedr{}};(6,0)*{\ucrossbr{}{}}; (7,0)*{}; \endxy   } "2";"4"};
   {\ar^-{-\xy  0;/r.12pc/:(-3,0)*{\slinedr{}};(6,0)*{\ucrossbr{}{}}; (7,0)*{};\endxy   } "1";"3"};
 \endxy
\]

Finally, we compute the case in which strands are labeled $j$ and $j'$ with $i\cdot j = -1 = i \cdot j'$.
In this case,
\[
\cal{T}_i'(\cal{E}_{j}\cal{E}_{j'}\onel) =
\clubsuit \cal{E}_{j} \cal{E}_{i} \cal{E}_{j'} \cal{E}_{i} \onell{s_i(\l)} \to
\cal{E}_{j} \cal{E}_{i} \cal{E}_{i} \cal{E}_{j'} \onell{s_i(\l)} \la 1 \ra \oplus \cal{E}_{i} \cal{E}_{j} \cal{E}_{j'} \cal{E}_{i} \onell{s_i(\l)} \la 1 \ra \to
\cal{E}_{i} \cal{E}_{j} \cal{E}_{i} \cal{E}_{j'} \onell{s_i(\l)} \la 2 \ra
\]
and we denote the relevant endomorphism as an ordered triple. We abuse notation for the component mapping between the terms in homological degree one:
technically this should be given by a $2 \times 2$ matrix, but, in the interest of space, we add all terms in the relevant matrix, as the components are distinguished by their (co)domains.
We compute:
\begin{align*}
&\cal{T}'_i\left(\xy 0;/r.15pc/: (0,-5)*{\ucrossbp{j}{j'}}; (0,4)*{\ucrosspb{}{}}; \endxy\right)
=t_{ij}^{-1}t_{ij'}^{-1}\left(
{\xy  (0,-7)*{\xy 0;/r.15pc/:
    (0,-.75)*{\ncrossrp{i}{j'}}; (9,-.75)*{\slinenr{i}}; (-9,-.75)*{\slinen{j}};
    (-6,8)*{\ncrossbp{}{}};(6,8)*{\ncrossrr{}{}};
    (0,16)*{\ucrossbr{}{}}; (9,16)*{\slineur{}}; (-9,16)*{\slineup{}};  \endxy};
(0,8.5)*{\xy 0;/r.15pc/:
    (0,0)*{\ncrossrb{}{}}; (9,0)*{\slinenr{}}; (-9,0)*{\slinenp{}};
    (-6,8)*{\ncrosspb{}{}};(6,8)*{\ncrossrr{}{}};
    (0,16)*{\ucrosspr{}{}}; (9,16)*{\slineur{}}; (-9,16)*{\slineu{}};  \endxy};
\endxy}, \;
{\xy (-12,0)*{\delta_{jj'}t_{ji}^2}; (0,-3)*{\xy 0;/r.18pc/:
    (0,0)*{\ucrossrr{i}{i}}; (7,0)*{\slinenp{j'}}; (-7,0)*{\slinen{j}};
    \endxy}; (0,3.5)*{\xy 0;/r.18pc/:
    (0,0)*{\ucrossrr{}{}}; (7,0)*{\slineup{}}; (-7,0)*{\slineu{}};
    \endxy}; \endxy}
+{\xy (0,-7)*{\xy 0;/r.15pc/:
     (0,-.75)*{\ncrossrr{i}{i}}; (9,-.75)*{\slinenp{j'}}; (-9,-.75)*{\slinen{j}};
    (-6,8)*{\ncrossbr{}{}};(6,8)*{\ncrossrp{}{}};
    (0,16)*{\ucrossbp{}{}}; (9,16)*{\slineur{}}; (-9,16)*{\slineur{}};\endxy}; (0,8.5)*{\xy 0;/r.15pc/:
    (0,0)*{\ncrosspb{}{}}; (9,0)*{\slinenr{}}; (-9,0)*{\slinenr{}};
    (-6,8)*{\ncrossrb{}{}};(6,8)*{\ncrosspr{}{}};
    (0,16)*{\ucrossrr{}{}}; (9,16)*{\slineup{}}; (-9,16)*{\slineu{}};\endxy}; \endxy}
-{\xy (-12,0)*{t_{ji}\delta_{jj'}}; (0,-7)*{\xy 0;/r.19pc/:
    (0,1)*{\xy 0;/r.15pc/: (0,-.75)*{\ucrossrr{i}{i}}; \endxy}; (7,-.75)*{\slineup{j'}}; (-7,-.75)*{\slineu{j}};
    \endxy}; (0,4.5)*{\xy 0;/r.15pc/:
     (0,0)*{\ncrossrr{}{}}; (9,0)*{\slinenp{}}; (-9,0)*{\slinen{}};
    (-6,8)*{\ncrossbr{}{}};(6,8)*{\ncrossrp{}{}};
    (0,16)*{\ucrossbp{}{}}; (9,16)*{\slineur{}}; (-9,16)*{\slineur{}};\endxy}; \endxy}
+{\xy (-10,0)*{t_{ij}}; (0,-3)*{\xy 0;/r.15pc/:
     (0,-.75)*{\ncrossrr{i}{i}}; (9,-.75)*{\slinenp{j'}}; (-9,-.75)*{\slinen{j}};
    (-6,8)*{\ncrossbr{}{}};(6,8)*{\ncrossrp{}{}};
    (0,16)*{\ucrossbp{}{}}; (9,16)*{\slineur{}}; (-9,16)*{\slineur{}};\endxy}; (0,7.5)*{\xy 0;/r.19pc/:
    (0,0)*{\xy 0;/r.15pc/: (0,0)*{\ucrosspb{}{}}; \endxy}; (7,0)*{\slineur{}}; (-7,0)*{\slineur{}};  \endxy}; \endxy}\right.\\
&\hspace{1.25in}\left.+ {\xy (-11,0)*{t_{ij}t_{ij'}}; (0,-3)*{\xy 0;/r.18pc/:
    (0,-.75)*{\ucrossbp{j}{j'}}; (7,-.75)*{\slinenr{i}}; (-7,-.75)*{\slinenr{i}};
    \endxy}; (0,3.5)*{\xy 0;/r.18pc/:
    (0,0)*{\ucrosspb{}{}}; (7,0)*{\slineur{}}; (-7,0)*{\slineur{}};
    \endxy}; \endxy}
+{\xy (0,-7)*{\xy 0;/r.15pc/:
     (0,-.75)*{\ncrossbp{j}{j'}}; (9,-.75)*{\slinenr{i}}; (-9,-.75)*{\slinenr{i}};
    (-6,8)*{\ncrossrp{}{}};(6,8)*{\ncrossbr{}{}};
    (0,16)*{\ucrossrr{}{}}; (9,16)*{\slineu{}}; (-9,16)*{\slineup{}};\endxy}; (0,8.5)*{\xy 0;/r.15pc/:
    (0,0)*{\ncrossrr{}{}}; (9,0)*{\slinen{}}; (-9,0)*{\slinenp{}};
    (-6,8)*{\ncrosspr{}{}};(6,8)*{\ncrossrb{}{}};
    (0,16)*{\ucrosspb{}{}}; (9,16)*{\slineur{}}; (-9,16)*{\slineur{}};\endxy}; \endxy}
+{\xy (-10,0)*{t_{ij'}}; (0,-7)*{\xy 0;/r.19pc/:
    (0,1)*{\xy 0;/r.15pc/: (0,-.75)*{\ucrossbp{j}{j'}}; \endxy}; (7,.125)*{\slineur{i}}; (-7,.125)*{\slineur{i}};
    \endxy}; (0,4.5)*{\xy 0;/r.15pc/:
     (0,0)*{\ncrosspb{}{}}; (9,0)*{\slinenr{}}; (-9,0)*{\slinenr{}};
    (-6,8)*{\ncrossrb{}{}};(6,8)*{\ncrosspr{}{}};
    (0,16)*{\ucrossrr{}{}}; (9,16)*{\slineup{}}; (-9,16)*{\slineu{}};\endxy}; \endxy}
-{\xy (-12,0)*{t_{ji}\delta_{jj'}}; (0,-3)*{\xy 0;/r.15pc/:
     (0,-.75)*{\ncrossbp{j}{j'}}; (9,-.75)*{\slinenr{i}}; (-9,-.75)*{\slinenr{i}};
    (-6,8)*{\ncrossrp{}{}};(6,8)*{\ncrossbr{}{}};
    (0,16)*{\ucrossrr{}{}}; (9,16)*{\slineu{}}; (-9,16)*{\slineup{}};\endxy}; (0,7.75)*{\xy 0;/r.19pc/:
    (0,0)*{\xy 0;/r.15pc/: (0,0)*{\ucrossrr{}{}}; \endxy}; (7,0)*{\slineu{}}; (-7,0)*{\slineup{}};  \endxy}; \endxy} ,
{\xy (0,-7)*{\xy 0;/r.15pc/:
    (0,-.75)*{\ncrossbr{j}{i}}; (9,-.75)*{\slinenp{j}}; (-9,-.75)*{\slinenr{i}};
    (-6,8)*{\ncrossrr{}{}};(6,8)*{\ncrossbp{}{}};
    (0,16)*{\ucrossrp{}{}}; (9,16)*{\slineu{}}; (-9,16)*{\slineur{}};\endxy};
(0,8.5)*{\xy 0;/r.15pc/:
    (0,0)*{\ncrosspr{}{}}; (9,0)*{\slinen{}}; (-9,0)*{\slinenr{}};
    (-6,8)*{\ncrossrr{}{}};(6,8)*{\ncrosspb{}{}};
    (0,16)*{\ucrossrb{}{}}; (9,16)*{\slineup{}}; (-9,16)*{\slineur{}}; \endxy}; \endxy }\right)\\
&=t_{ij}^{-1}t_{ij'}^{-1}\left({\xy (-10,0)*{t_{ij}}; (0,-7)*{\xy 0;/r.15pc/:
    (0,-.75)*{\ncrossrp{i}{j'}}; (9,-.75)*{\slinenr{i}}; (-9,-.75)*{\slinen{j}};
    (-6,8)*{\ncrossbp{}{}};(6,8)*{\ncrossrr{}{}};
    (3,16)*{\slinenr{}}; (-3,16)*{\slineu{}}; (9,16)*{\slineur{}}; (-9,16)*{\slineup{}};  \endxy};
(0,8)*{\xy 0;/r.15pc/:
    (3,0)*{\sdotr{}}; (-3,0)*{\slinen{}}; (9,0)*{\slinenr{}}; (-9,0)*{\slinenp{}};
    (-6,8)*{\ncrosspb{}{}};(6,8)*{\ncrossrr{}{}};
    (0,16)*{\ucrosspr{}{}}; (9,16)*{\slineur{}}; (-9,16)*{\slineu{}};  \endxy}; \endxy}
+ {\xy (-10,0)*{t_{ji}}; (0,-7)*{\xy 0;/r.15pc/:
    (0,-.75)*{\ncrossrp{i}{j'}}; (9,-.75)*{\slinenr{i}}; (-9,-.75)*{\slinen{j}};
    (-6,8)*{\ncrossbp{}{}};(6,8)*{\ncrossrr{}{}};
    (3,16)*{\slineur{}}; (-3,16)*{\slinen{}}; (9,16)*{\slineur{}}; (-9,16)*{\slineup{}};  \endxy};
(0,8)*{\xy 0;/r.15pc/:
    (3,0)*{\slinenr{}}; (-3,0)*{\sdot{}}; (9,0)*{\slinenr{}}; (-9,0)*{\slinenp{}};
    (-6,8)*{\ncrosspb{}{}};(6,8)*{\ncrossrr{}{}};
    (0,16)*{\ucrosspr{}{}}; (9,16)*{\slineur{}}; (-9,16)*{\slineu{}};  \endxy}; \endxy} ,
{\xy (0,-7)*{\xy 0;/r.15pc/:
     (0,-.75)*{\ncrossrr{i}{i}}; (9,-.75)*{\slinenp{j'}}; (-9,-.75)*{\slinen{j}};
    (-6,8)*{\ncrossbr{}{}};(6,8)*{\ncrossrp{}{}};
    (0,16)*{\ucrossbp{}{}}; (9,16)*{\slineur{}}; (-9,16)*{\slineur{}};\endxy}; (0,8.5)*{\xy 0;/r.15pc/:
    (0,0)*{\ncrosspb{}{}}; (9,0)*{\slinenr{}}; (-9,0)*{\slinenr{}};
    (-6,8)*{\ncrossrb{}{}};(6,8)*{\ncrosspr{}{}};
    (0,16)*{\ucrossrr{}{}}; (9,16)*{\slineup{}}; (-9,16)*{\slineu{}};\endxy}; \endxy}
+{\xy (-10,0)*{t_{ij}}; (0,-3)*{\xy 0;/r.15pc/:
     (0,-.75)*{\ncrossrr{i}{i}}; (9,-.75)*{\slinenp{j'}}; (-9,-.75)*{\slinen{j}};
    (-6,8)*{\ncrossbr{}{}};(6,8)*{\ncrossrp{}{}};
    (0,16)*{\ucrossbp{}{}}; (9,16)*{\slineur{}}; (-9,16)*{\slineur{}};\endxy}; (0,7.5)*{\xy 0;/r.19pc/:
    (0,0)*{\xy 0;/r.15pc/: (0,0)*{\ucrosspb{}{}}; \endxy}; (7,0)*{\slineur{}}; (-7,0)*{\slineur{}};  \endxy}; \endxy}
+ {\xy (-11,0)*{t_{ij}t_{ij'}}; (0,-3)*{\xy 0;/r.18pc/:
    (0,-.75)*{\ucrossbp{j}{j'}}; (7,-.75)*{\slinenr{i}}; (-7,-.75)*{\slinenr{i}};
    \endxy}; (0,3.5)*{\xy 0;/r.18pc/:
    (0,0)*{\ucrosspb{}{}}; (7,0)*{\slineur{}}; (-7,0)*{\slineur{}};
    \endxy}; \endxy}
+{\xy (-10,0)*{t_{ij'}}; (0,-7)*{\xy 0;/r.19pc/:
    (0,1)*{\xy 0;/r.15pc/: (0,-.75)*{\ucrossbp{j}{j'}}; \endxy}; (7,.125)*{\slineur{i}}; (-7,.125)*{\slineur{i}};
    \endxy}; (0,4.5)*{\xy 0;/r.15pc/:
     (0,0)*{\ncrosspb{}{}}; (9,0)*{\slinenr{}}; (-9,0)*{\slinenr{}};
    (-6,8)*{\ncrossrb{}{}};(6,8)*{\ncrosspr{}{}};
    (0,16)*{\ucrossrr{}{}}; (9,16)*{\slineup{}}; (-9,16)*{\slineu{}};\endxy}; \endxy} , \right.\\
&\hspace{1.5in}\left.{\xy (-10,0)*{t_{j'i}}; (0,-7)*{\xy 0;/r.15pc/:
    (0,-.75)*{\ncrossbr{j}{i}}; (9,-.75)*{\slinenp{j'}}; (-9,-.75)*{\slinenr{i}};
    (-6,8)*{\ncrossrr{}{}};(6,8)*{\ncrossbp{}{}};
    (3,16)*{\slinenp{}}; (-3,16)*{\slineur{}}; (9,16)*{\slineu{}}; (-9,16)*{\slineur{}};  \endxy};
(0,8.5)*{\xy 0;/r.15pc/:
    (3,0)*{\sdotp{}}; (-3,0)*{\slinenr{}}; (9,0)*{\slinen{}}; (-9,0)*{\slinenr{}};
    (-6,8)*{\ncrossrr{}{}};(6,8)*{\ncrosspb{}{}};
    (0,16)*{\ucrossrb{}{}}; (9,16)*{\slineup{}}; (-9,16)*{\slineur{}};  \endxy}; \endxy}
+{\xy (-10,0)*{t_{ij'}}; (0,-7)*{\xy 0;/r.15pc/:
    (0,-.75)*{\ncrossbr{j}{i}}; (9,-.75)*{\slinenp{j'}}; (-9,-.75)*{\slinenr{i}};
    (-6,8)*{\ncrossrr{}{}};(6,8)*{\ncrossbp{}{}};
    (3,16)*{\slineup{}}; (-3,16)*{\slinenr{}}; (9,16)*{\slineu{}}; (-9,16)*{\slineur{}};  \endxy};
(0,8.5)*{\xy 0;/r.15pc/:
    (3,0)*{\slinenp{}}; (-3,0)*{\sdotr{}}; (9,0)*{\slinen{}}; (-9,0)*{\slinenr{}};
    (-6,8)*{\ncrossrr{}{}};(6,8)*{\ncrosspb{}{}};
    (0,16)*{\ucrossrb{}{}}; (9,16)*{\slineup{}}; (-9,16)*{\slineur{}};  \endxy}; \endxy}\right)
\end{align*}
which vanishes if $j=j'$, as desired.
If $j\neq j'$, we instead have
\begin{align*}
&=\left({\xy 0;/r.18pc/:
 (-12,0)*{t_{ij'}^{-1}}; (3,-12.75)*{\ncrossrp{i}{j'}};
  (-6,-12.75)*{\slinen{j}};
 (-3,-4)*{\ucrossbp{}{}};
 (6,4)*{\slinenr{}};
(-3,4)*{\ncrosspb{}{}};
  (9,-4)*{\ucrossrr{}{}};
 (3,12)*{\ucrosspr{}{}};
 (-6,12)*{\slineu{}};   (12,4)*{\slinenr{}}; (12,-12.75)*{\slinenr{i}}; (12,12)*{\slineur{}};
 \endxy} ,
 {\xy (-12,0)*{t_{ij}^{-1}t_{ij'}^{-1}}; (0,-7)*{\xy 0;/r.15pc/:
     (0,-.75)*{\ncrossrr{i}{i}}; (9,-.75)*{\slinenp{j'}}; (-9,-.75)*{\slinen{j}};
    (-6,8)*{\ncrossbr{}{}};(6,8)*{\ncrossrp{}{}};
    (0,16)*{\ucrossbp{}{}}; (9,16)*{\slineur{}}; (-9,16)*{\slineur{}};\endxy}; (0,8.5)*{\xy 0;/r.15pc/:
    (0,0)*{\ncrosspb{}{}}; (9,0)*{\slinenr{}}; (-9,0)*{\slinenr{}};
    (-6,8)*{\ncrossrb{}{}};(6,8)*{\ncrosspr{}{}};
    (0,16)*{\ucrossrr{}{}}; (9,16)*{\slineup{}}; (-9,16)*{\slineu{}};\endxy}; \endxy}
+{\xy (-10,0)*{t_{ij'}^{-1}}; (0,-3)*{\xy 0;/r.15pc/:
     (0,-.75)*{\ncrossrr{i}{i}}; (9,-.75)*{\slinenp{j'}}; (-9,-.75)*{\slinen{j}};
    (-6,8)*{\ncrossbr{}{}};(6,8)*{\ncrossrp{}{}};
    (0,16)*{\ucrossbp{}{}}; (9,16)*{\slineur{}}; (-9,16)*{\slineur{}};\endxy}; (0,7.5)*{\xy 0;/r.19pc/:
    (0,0)*{\xy 0;/r.15pc/: (0,0)*{\ucrosspb{}{}}; \endxy}; (7,0)*{\slineur{}}; (-7,0)*{\slineur{}};  \endxy}; \endxy}
+ {\xy (0,-3)*{\xy 0;/r.18pc/:
    (0,-.75)*{\ucrossbp{j}{j'}}; (7,-.75)*{\slinenr{i}}; (-7,-.75)*{\slinenr{i}};
    \endxy}; (0,3.5)*{\xy 0;/r.18pc/:
    (0,0)*{\ucrosspb{}{}}; (7,0)*{\slineur{}}; (-7,0)*{\slineur{}};
    \endxy}; \endxy}
+{\xy (-10,0)*{t_{ij}^{-1}}; (0,-7)*{\xy 0;/r.19pc/:
    (0,1)*{\xy 0;/r.15pc/: (0,-.75)*{\ucrossbp{j}{j'}}; \endxy}; (7,.125)*{\slineur{i}}; (-7,.125)*{\slineur{i}};
    \endxy}; (0,4.5)*{\xy 0;/r.15pc/:
     (0,0)*{\ncrosspb{}{}}; (9,0)*{\slinenr{}}; (-9,0)*{\slinenr{}};
    (-6,8)*{\ncrossrb{}{}};(6,8)*{\ncrosspr{}{}};
    (0,16)*{\ucrossrr{}{}}; (9,16)*{\slineup{}}; (-9,16)*{\slineu{}};\endxy}; \endxy} ,
{\xy 0;/r.18pc/:
 (-20,0)*{-t_{ij}^{-1}}; (-3,-12.75)*{\ncrossbr{j}{i}};
  (6,-12.75)*{\slinenp{j'}};
 (3,-4)*{\ucrossbp{}{}};
 (-9,-4)*{\ucrossrr{}{}};
(-6,4)*{\slinenr{}};
(3,4)*{\ncrosspb{}{}};
 (-3,12)*{\ucrossrb{}{}};
 (6,12)*{\slineup{}}; (-12,4)*{\slinenr{}};  (-12,12)*{\slineur{}}; (-12,-12.75)*{\slinenr{i}};
 \endxy}\right).
\end{align*}

If $j\cdot j'=0$ this simplifies to
\begin{align*}
=\;\; \cal{T}'_i
\left(t_{jj'}{\xy (0,0)*{\slineu{j}}; (4,0)*{\slineup{j'}}; \endxy}\right)
+t_{jj'}\left({t_{ij'}^{-1}\xy (-6,3)*{\slineu{}}; (-6,-4)*{\slinen{}};
(-6,-8.5)*{\scs j};
(-2.5,-8.5)*{\scs i};
(1.5,-8.5)*{\scs j'};
(5.5,-8.5)*{\scs i};
0;/r.15pc/:(2,0)*{\Rthreer{black}{magenta}{black}{}{}{}};
\endxy},
v_{ij'}{\xy (0,0)*{\sucrossrr{i}{i}}; (6,0)*{\sdotup{j'}}; (-6,0)*{\slineu{j}}; \endxy}
-v_{ij}{\xy (0,0)*{\sucrossrr{i}{i}}; (-6,0)*{\sdotu{j}}; (6,0)*{\slineup{j'}}; \endxy}
-{\xy (-6,0)*{\slineu{j}}; (-2,0)*{\slineur{i}}; (2,0)*{\slineur{i}}; (6,0)*{\slineup{j'}}; \endxy}\right.
\\
\left.
\hspace{2in}+{\xy (-10,0)*{t_{ij'}}; (0,0)*{\xy 0;/r.15pc/:
     (0,-.75)*{\ncrossrr{i}{i}}; (9,-.75)*{\slinenp{j'}}; (-9,-.75)*{\slinen{j}};
    (-6,16)*{\ucrossbr{}{}};(6,8)*{\ncrossrp{}{}};
    (3,16)*{\slineup{}}; (-3,8)*{\slinenr{}}; (9,16)*{\slineur{}}; (-9,8)*{\slinen{}};\endxy}; \endxy}
+{\xy (-10,0)*{t_{ij}^{-1}}; (0,0)*{\xy 0;/r.15pc/:
     (0,0)*{\ucrossrr{}{}}; (9,0)*{\slineup{}}; (-9,0)*{\slineu{}};
    (-6,-8)*{\ncrossrb{}{}};(6,-16.75)*{\ncrosspr{j'}{i}};
    (3,-8)*{\slinenr{}}; (-3,-16.75)*{\slinen{j}}; (9,-8)*{\slinenp{}}; (-9,-16.75)*{\slinenr{i}};\endxy}; \endxy},
-t_{ij}^{-1}{\xy (6,3)*{\slineup{}}; (6,-4)*{\slinenp{}};
(6,-8.5)*{\scs j'};
(-5.5,-8.5)*{\scs i};
(-1.5,-8.5)*{\scs j};
(2.5,-8.5)*{\scs i};
0;/r.15pc/:(-2,0)*{\Rthreel{black}{blue}{black}{}{}{}}; \endxy}\right)
\end{align*}
and if $j \cdot j'=-1$ we have
\begin{align*}
& = \cal{T}'_i\left(t_{jj'}{\xy (0,0)*{\sdotu{j}}; (4,0)*{\slineup{j'}}; \endxy}
+t_{j'j}{\xy (0,0)*{\slineu{j}}; (4,0)*{\sdotup{j'}}; \endxy}\right)
+
\left(t_{jj'}{t_{ij'}^{-1}\xy (-6,3)*{\slineu{}}; (-6,-4)*{\slinen{}}; (-6,0)*{\bullet};
(-6,-8.5)*{\scs j};
(-2.5,-8.5)*{\scs i};
(1.5,-8.5)*{\scs j'};
(5.5,-8.5)*{\scs i};
0;/r.15pc/:(2,0)*{\Rthreer{black}{magenta}{black}{}{}{}}; \endxy}
+t_{j'j}{t_{ij'}^{-1}\xy (-6,3)*{\slineu{}};  (-6,-4)*{\slinen{}};
(-6,-8.5)*{\scs j};
(-2.5,-8.5)*{\scs i};
(1.5,-8.5)*{\scs j'};
(5.5,-8.5)*{\scs i};
0;/r.15pc/:(2,0)*{\Rthreer{black}{magenta}{black}{}{}{}};
(8,0)*{\bullet};
\endxy} ,
(t_{jj'}v_{ij'} - t_{j'j}v_{ij}){\xy (0,0)*{\sucrossrr{i}{i}}; (6,0)*{\sdotup{j'}}; (-6,0)*{\sdotu{j}}; \endxy} \right.
\\
& \quad
 \left.
 +t_{j'j}v_{ij'}{\xy (0,0)*{\sucrossrr{i}{i}}; (6,0)*{\sdotup{j'}}; (-6,0)*{\slineu{j}}; (7.5,2)*{\scs 2};\endxy}
-t_{jj'}v_{ij}{\xy (0,0)*{\sucrossrr{i}{i}}; (-6,0)*{\sdotu{j}}; (-4.5,2)*{\scs 2}; (6,0)*{\slineup{j'}}; \endxy}
-t_{jj'}{\xy (-6,0)*{\sdotu{j}}; (-2,0)*{\slineur{i}}; (2,0)*{\slineur{i}}; (6,0)*{\slineup{j'}}; \endxy}
-t_{j'j}{\xy (-6,0)*{\slineu{j}}; (-2,0)*{\slineur{i}}; (2,0)*{\slineur{i}}; (6,0)*{\sdotup{j'}}; \endxy}
+t_{jj'}{\xy (-10,0)*{t_{ij'}^{-1}}; (0,0)*{\xy 0;/r.15pc/:
     (0,0)*{\ncrossrr{}{}}; (9,0)*{\slinenp{}}; (-9,0)*{\sdot{}};
    (-6,16)*{\ucrossbr{}{}};(6,8)*{\ncrossrp{}{}};
    (3,16)*{\slineup{}}; (-3,8)*{\slinenr{}}; (9,16)*{\slineur{}}; (-9,8)*{\slinen{}};\endxy};
    (-6,-8.5)*{\scs j};
(-2.5,-8.5)*{\scs i};
(1.5,-8.5)*{\scs i};
(5.5,-8.5)*{\scs j'};
\endxy}
\right.
\\
& \quad \left.
+t_{j'j}{\xy (-10,0)*{t_{ij'}^{-1}}; (0,0)*{\xy 0;/r.15pc/:
     (0,0)*{\ncrossrr{}{}}; (9,0)*{\sdotp{}}; (-9,0)*{\slinen{}};
    (-6,16)*{\ucrossbr{}{}};(6,8)*{\ncrossrp{}{}};
    (3,16)*{\slineup{}}; (-3,8)*{\slinenr{}}; (9,16)*{\slineur{}}; (-9,8)*{\slinen{}};\endxy};
    (-6,-8.5)*{\scs j};
(-2.5,-8.5)*{\scs i};
(1.5,-8.5)*{\scs i};
(5.5,-8.5)*{\scs j'};
\endxy}
+t_{jj'}{\xy (-10,0)*{t_{ij}^{-1}}; (0,0)*{\xy 0;/r.15pc/:
     (0,0)*{\ucrossrr{}{}}; (9,0)*{\slineup{}}; (-9,0)*{\sdotu{}};
    (-6,-8)*{\ncrossrb{}{}};(6,-16)*{\ncrosspr{}{}};
    (3,-8)*{\slinenr{}}; (-3,-16)*{\slinen{}}; (9,-8)*{\slinenp{}}; (-9,-16)*{\slinenr{}};\endxy};
(-6,-8.5)*{\scs i};
(-2.5,-8.5)*{\scs j};
(1.5,-8.5)*{\scs j'};
(5.5,-8.5)*{\scs i};
\endxy}
+t_{j'j}{\xy (-10,0)*{t_{ij}^{-1}}; (0,0)*{\xy 0;/r.15pc/:
     (0,0)*{\ucrossrr{}{}}; (9,0)*{\slineup{}}; (-9,0)*{\slineu{}};
    (-6,-8)*{\ncrossrb{}{}};(6,-16)*{\ncrosspr{}{}};
    (3,-8)*{\slinenr{}}; (-3,-16)*{\slinen{}}; (9,-8)*{\sdotp{}}; (-9,-16)*{\slinenr{}};\endxy};
(-6,-8.5)*{\scs i};
(-2.5,-8.5)*{\scs j};
(1.5,-8.5)*{\scs j'};
(5.5,-8.5)*{\scs i};
\endxy} ,
-t_{jj'}t_{ij}^{-1}{\xy (6,3)*{\slineup{}}; (6,-4)*{\slinenp{}};
(-5,-8.5)*{\scs i};
(-1.5,-8.5)*{\scs j};
(2.5,-8.5)*{\scs i};
(6.5,-8.5)*{\scs j'};
0;/r.15pc/:(-2,0)*{\Rthreel{black}{blue}{black}{}{}{}};
(-8,0)*{\bullet};
\endxy}
-t_{j'j}t_{ij}^{-1}{\xy (6,3)*{\slineup{}}; (6,-4)*{\slinenp{}}; (6,0)*{\bullet};
(-5,-8.5)*{\scs i};
(-1.5,-8.5)*{\scs j};
(2.5,-8.5)*{\scs i};
(6.5,-8.5)*{\scs j'};
0;/r.15pc/:(-2,0)*{\Rthreel{black}{blue}{black}{}{}{}}; \endxy}\right).
\end{align*}
In both cases,
the second summand (\ie the ``error term'' preventing the relation from holding on the nose) is null-homotopic.
The nonzero terms of both null-homotopies are given in the following diagram by the arrows labeled with the brackets
(with the homotopy for the $j \cdot j'=0$ case on the first line, and the $j \cdot j'=-1$ case on the second).
\begin{align*}
 \xy 0;/r.12pc/:
 (-60,-40)*+{\clubsuit \;\cal{E}_j\cal{E}_i\cal{E}_{j'}\cal{E}_i \onell{s_i(\lambda)}}="bl";
 (-20,-20)*+{\cal{E}_i\cal{E}_j\cal{E}_{j'}\cal{E}_i \onell{s_i(\lambda)}\la 1\ra}="bt";
 (25,-60)*+{\cal{E}_j\cal{E}_i\cal{E}_i\cal{E}_{j'} \onell{s_i(\lambda)}\la 1\ra}="bb";
 (60,-35)*+{\cal{E}_i\cal{E}_j\cal{E}_i\cal{E}_{j'} \onell{s_i(\lambda)}\la 2\ra}="br";
  {\ar^{} "bl";"bt"};
  {\ar_{\;\;\;\xy 0;/r.10pc/:(6,0)*{\ucrosspr{}{}}; (-3,0)*{\slineur{}};
  (-9,0)*{\slineu{}};\endxy \;\;} "bl";"bb"};
  {\ar^{\;\;\;-\xy 0;/r.10pc/:
    (6,0)*{\ucrosspr{}{}}; (-9,0)*{\slineur{}}; (-3,0)*{\slineu{}};\endxy
  } "bt";"br"};
  {\ar_<<<<<<<<<{\xy 0;/r.10pc/:
    (-6,0)*{\ucrossbr{}{}}; (9,0)*{\slineup{}}; (3,0)*{\slineur{}}; \endxy \;\;
 } "bb";"br"};
 (-60,40)*+{\clubsuit \;\cal{E}_{j}\cal{E}_i\cal{E}_{j'}\cal{E}_i \onell{s_i(\lambda)}\la -2j\cdot j'\ra}="tl";
 (-10,60)*+{\cal{E}_i\cal{E}_{j}\cal{E}_{j'}\cal{E}_i \onell{s_i(\lambda)}\la -2j\cdot j'+1\ra}="tt";
 (25,20)*+{\cal{E}_{j}\cal{E}_i\cal{E}_i\cal{E}_{j'} \onell{s_i(\lambda)}\la -2j\cdot j'+1\ra}="tb";
 (60,45)*+{\cal{E}_i\cal{E}_{j}\cal{E}_i\cal{E}_{j'} \onell{s_i(\lambda)}\la -2j\cdot j'+2\ra}="tr";
  {\ar^>>>>>>>>>>>>{\xy 0;/r.10pc/:
    (-6,0)*{\ucrossbr{}{}}; (9,0)*{\slineur{}}; (3,0)*{\slineup{}}; \endxy \;\;\;\;\;\;
  } "tl";"tt"};
  {\ar^>>>>>>{\;\;\;\xy 0;/r.10pc/:(6,0)*{\ucrosspr{}{}}; (-3,0)*{\slineur{}};
  (-9,0)*{\slineu{}};\endxy \;\;} "tl";"tb"};
  {\ar^{} "tt";"tr"};
  {\ar^>>>>>>>>>{\xy 0;/r.10pc/:
    (-6,0)*{\ucrossbr{}{}}; (9,0)*{\slineup{}}; (3,0)*{\slineur{}}; \endxy \;\;
 } "tb";"tr"};
 {\ar@/_.5pc/@{->} "br";"tb" };
{\ar@/_2pc/@{->} "bb";"tl"};
(-32,7)*{\scalebox{5}{\}}};
(-65,-5)*{\xy 0;/r.12pc/: (-24,6)*{+ t_{j'j}t_{ij'}^{-1}}; (0,-.5)*{\ncrossrr{i}{i}}; (9,-.5)*{\sdotp{j'}};  (-9,-.5)*{\slinen{j}};(-9,8.5)*{\slineu{}};(-3,8.5)*{\slineur{}}; (6,8.5)*{\ucrossrp{}{}}; \endxy};
(-112,-5)*{\xy 0;/r.12pc/: (-24,6)*{t_{jj'}t_{ij'}^{-1}}; (0,-.5)*{\ncrossrr{i}{i}}; (9,-.5)*{\slinenp{j'}}; (-9,-.5)*{\sdot{j}}; (-9,8.5)*{\slineu{}};(-3,8.5)*{\slineur{}}; (6,8.5)*{\ucrossrp{}{}}; \endxy};
(-65,20)*{\xy 0;/r.12pc/: (-24,6)*{t_{jj'}t_{ij'}^{-1}}; (0,-.5)*{\ncrossrr{i}{i}}; (9,-.5)*{\slinenp{j'}}; (-9,-.5)*{\slinen{j}};(-9,8.5)*{\slineu{}};(-3,8.5)*{\slineur{}}; (6,8.5)*{\ucrossrp{}{}}; \endxy};
(60,-6)*{\scalebox{5}{\{}};
(90,-18)*{\xy 0;/r.12pc/: (-6,-.5)*{\ncrossrb{i}{j}}; (9,-.5)*{\slinenp{j'}}; (3,-.5)*{\slinenr{i}}; (-9,8.5)*{\sdotu{}}; (0,8.5)*{\ucrossrr{}{}}; (9,8.5)*{\slineup{}}; (-25,6)*{-t_{jj'}t_{ij}^{-1}}; \endxy};
(140,-18)*{\xy 0;/r.12pc/: (-6,-.5)*{\ncrossrb{i}{j}}; (9,-.5)*{\sdotp{j'}}; (3,-.5)*{\slinenr{i}}; (-9,8.5)*{\slineu{}}; (0,8.5)*{\ucrossrr{}{}}; (9,8.5)*{\slineup{}}; (-25,6)*{-t_{j'j}t_{ij}^{-1}}; \endxy};
(90,7)*{\xy 0;/r.12pc/: (-6,-.5)*{\ncrossrb{i}{j}}; (9,-.5)*{\slinenp{j'}}; (3,0)*{\slinenr{i}}; (-9,8.5)*{\slineu{}}; (0,8.5)*{\ucrossrr{}{}}; (9,8.5)*{\slineup{}}; (-25,6)*{-t_{jj'}t_{ij}^{-1}}; \endxy};
 \endxy
\end{align*}
\end{proof}

%
\subsection{Dot slide}
%

\begin{prop}
$\cal{T}_i'$ preserve the KLR dot sliding relation.
\end{prop}

\begin{proof}
We verify relation~\eqref{def:dot-slide} from Definition~\ref{defU_cat-cyc},
only exhibiting the computations for crossings involving $j$- and $j'$-labeled strands
(for $j \cdot i = -1 = j' \cdot i$), as all others are completely straightforward.
For $ij$-crossings with dotted $i$-labeled strand, we compute:
\begin{align*}
\cal{T}'_i\left({\xy (0,0)*{\ucrossrb{i}{j}}; (-2,-.5)*[black]{\bullet}; (12,0)*{\ucrossrb{i}{j}}; (13.5,2.5)*[black]{\bullet}; (6,0)*{-}; \endxy}\right)
&= \left(
{\xy 0;/r.18pc/: (-3,-4.75)*{\lcrossrb{i}{j}}; (6,-4.75)*{\slinenr{i}}; (3,4)*{\lcrossrr{}{}};  (-4.75,-4.5)*[black]{\bullet};
   (-6,4)*{\slineu{}}; \endxy }
-{\xy 0;/r.18pc/: (-3,-4.75)*{\lcrossrb{i}{j}}; (6,-4.75)*{\slinenr{i}}; (3,4)*{\lcrossrr{}{}};  (4.75,6.5)*[black]{\bullet};
   (-6,4)*{\slineu{}}; \endxy } ,\;
-{ \xy 0;/r.18pc/: (-3,-4.75)*{\lcrossrr{i}{i}}; (6,-4.75)*{\slinen{j}}; (3,4)*{\lcrossrb{}{}}; (-4.75,-4.5)*[black]{\bullet};
   (-6,4)*{\slineur{}}; \endxy}
+{ \xy 0;/r.18pc/: (-3,-4.75)*{\lcrossrr{i}{i}}; (6,-4.75)*{\slinen{j}}; (3,4)*{\lcrossrb{}{}}; (4.75,6.5)*[black]{\bullet};
   (-6,4)*{\slineur{}}; \endxy}\right)\\
&= \left(
{\xy 0;/r.18pc/: (-3,-4.75)*{\lcrossrb{i}{j}}; (6.2,-4)*{\slinenr{}}; (6,6)*{\scs i};
(3,7)*{\xy 0;/r.14pc/: (0,0)*{\lcupr{}}; \endxy};
  (3.1,3)*{\xy 0;/r.14pc/: (0,0)*{\lcapr{}}; \endxy};  (-6,4)*{\slineu{}}; \endxy },\;
-{ \xy 0;/r.18pc/: (-3.1,-1.25)*{\xy 0;/r.14pc/: (0,0)*{\lcupr{}}; \endxy};
(-3,-5.75)*{\xy 0;/r.14pc/: (0,0)*{\lcapr{}}; \endxy};
(6,-4.75)*{\slinen{j}}; (1,-1.5)*{\scs i}; (0,-9)*{\scs i};
(3,4)*{\lcrossrb{}{}}; (-6.2,4.2)*{\slineur{}}; \endxy}\right)
=\left(
{\xy 0;/r.18pc/: (3,-4.75)*{\ucrossbr{j}{i}}; (-6.2,-4)*{\slinedr{}}; (6,6)*{\scs i};
(3,7)*{\xy 0;/r.14pc/: (0,0)*{\lcupr{}}; \endxy};
  (-3.1,3)*{\xy 0;/r.14pc/: (0,0)*{\lcapr{}}; \endxy};
  (6,1)*{};(-6,9)*{} **[blue][|(1)]\crv{(6,5) & (-6,6)};
    \endxy
  },\;
-{\xy 0;/r.18pc/:
    (3.1,-1.5)*{
    \xy 0;/r.14pc/: (0,0)*{\lcupr{}}; \endxy
            };
 (6,-7)*{};(-6,1)*{} **[blue][|(1)]\crv{(6,-3) & (-6,-2)};
(-3,-5)*{\xy 0;/r.14pc/: (0,0)*{\lcapr{}}; \endxy};
(-3,4)*{\ucrossbr{}{}};
(6.25,4)*{\slinenr{}};
(.5,3)*{\scs i};
(0,-9)*{\scs i};
(6,-9)*{\scs j};
\endxy}\right)
\end{align*}
which is null-homotopic, as desired, via the homotopy:
\begin{align*}
   \xy 0;/r.15pc/:
  (-40,15)*+{\cal{E}_j\cal{E}_i\cal{F}_i\onell{s_i(\lambda)}\la 1-\l_i \ra}="1";
  (-40,-15)*+{\cal{F}_i\cal{E}_j\cal{E}_i\onell{s_i(\lambda)}\la -1-\l_i \ra}="2";
  (40,15)*+{\clubsuit\ \cal{E}_i\cal{E}_j\cal{F}_i\onell{s_i(\lambda)}\la 2-\l_i \ra}="3";
  (40,-15)*+{\clubsuit\ \cal{F}_i\cal{E}_i\cal{E}_j\onell{s_i(\lambda)}\la -\l_i \ra}="4";
    {\ar_<<<<<{\xy 0;/r.15pc/:
 (3,-1)*{\xy 0;/r.14pc/: (0,0)*{\lcupr{}}; \endxy};
 (6,-7)*{};(-6,2)*{} **[blue][|(1)]\crv{(6,-3) & (-6,-3)};?(1)*[blue][|(1)]\dir{>};
(-3,-5)*{\xy 0;/r.14pc/: (0,0)*{\lcapr{}}; \endxy}; (-10,-3)*{\scs -};
 \endxy}@/^1pc/ "4";"1"};
   {\ar^-{\xy  0;/r.15pc/: (3,0)*{\slinedr{}};(-6,0)*{\ucrossbr{}{}};\endxy   } "1";"3"};
   {\ar_-{\xy 0;/r.15pc/: (-9,0)*{\slinedr{}};(0,0)*{\ucrossbr{}{}}; (-13,0)*{\scs -}; \endxy   } "2";"4"};
 \endxy  \end{align*}
 For the $ij$-crossing with dotted $j$-labeled strand, we have:
\[
\cal{T}'_i\left(
\hspace{-5pt}
{\xy (0,0)*{\ucrossrb{i}{j}}; (-1.25,2.5)*{\bullet}; \endxy}
\hspace{-5pt}
\right)
=\left(
{\xy 0;/r.18pc/: (-3,-4.75)*{\lcrossrb{i}{j}}; (6,-4.75)*{\slinenr{i}}; (3,4)*{\lcrossrr{}{}};
(-6,4)*{\slineu{}}; (-6,5)*{\bullet}; \endxy }, \;
-{ \xy 0;/r.18pc/: (-3,-4.75)*{\lcrossrr{i}{i}}; (6,-4.75)*{\slinen{j}}; (3,4)*{\lcrossrb{}{}}; (1.5,6)*{\bullet};
   (-6,4)*{\slineur{}}; \endxy}\right)
=\left({\xy 0;/r.18pc/: (-3,-4.75)*{\lcrossrb{i}{j}}; (6,-4.75)*{\slinenr{i}}; (3,4)*{\lcrossrr{}{}};  (-1,-5)*{\bullet};
   (-6,4)*{\slineu{}}; \endxy }, \;
-{ \xy 0;/r.18pc/: (-3,-4.75)*{\lcrossrr{i}{i}}; (6,-4.75)*{\slinen{j}}; (3,4)*{\lcrossrb{}{}};
(6,-3)*{\bullet}; (-6,4)*{\slineur{}}; \endxy}\right)
=\cal{T}'_i\left(
\hspace{-5pt}
{\xy (0,0)*{\ucrossrb{i}{j}}; (2,-.25)*{\bullet}; \endxy}
\hspace{-5pt}
\right).
\]
For dotted $ji$-crossings, neither case requires a chain homotopy, so we omit the computations,
which are straightforward.

Finally,we consider dotted $jj'$-crossings.
As in the proof of Proposition \ref{prop:quadKLR}, our chain maps here map between complexes supported in
three adjacent homological degrees, and we denote them as ordered triples.
We have
\begin{align*}
&\cal{T}'_i\left(\xy0;/r.18pc/: (0,0)*{\ucrossbp{j}{j'}}; (-2,-.25)*{\bullet}; (12,0)*{\ucrossbp{j}{j'}};
(13.75,3)*{\bullet}; (6,1)*{-};\endxy\right) =
t_{ij}^{-1}\left(
{\xy 0;/r.14pc/: (0,-8)*{\ncrossrp{}{}}; (9,-8)*{\slinenr{}}; (-9,-8)*{\slinen{}};
    (-6,0)*{\ncrossbp{}{}};(6,0)*{\ncrossrr{}{}};
    (0,8)*{\ucrossbr{}{}}; (9,8)*{\slineur{}}; (-9,8)*{\slineup{}};
(-9,-13)*{\scs j};
(-3,-13)*{\scs i};
(3,-13)*{\scs j'};
(9,-13)*{\scs i};
(-9,-6)*{\bullet};
\endxy}
-{\xy 0;/r.14pc/: (0,-8)*{\ncrossrp{}{}}; (9,-8)*{\slinenr{}}; (-9,-8)*{\slinen{}};
    (-6,0)*{\ncrossbp{}{}};(6,0)*{\ncrossrr{}{}};
    (0,8)*{\ucrossbr{}{}}; (9,8)*{\slineur{}}; (-9,8)*{\slineup{}};
(-9,-13)*{\scs j};
(-3,-13)*{\scs i};
(3,-13)*{\scs j'};
(9,-13)*{\scs i};
(1.5,10)*{\bullet};
\endxy} ,
{\xy 0;/r.14pc/:
     (0,-8)*{\ncrossrr{}{}}; (9,-8)*{\slinenp{}}; (-9,-8)*{\slinen{}};
    (-6,0)*{\ncrossbr{}{}};(6,0)*{\ncrossrp{}{}};
    (0,8)*{\ucrossbp{}{}}; (9,8)*{\slineur{}}; (-9,8)*{\slineur{}};
(-9,-13)*{\scs j};
(-3,-13)*{\scs i};
(3,-13)*{\scs i};
(9,-13)*{\scs j'};
(-9,-6)*{\bullet};
\endxy}
-{\xy 0;/r.14pc/:
     (0,-8)*{\ncrossrr{}{}}; (9,-8)*{\slinenp{}}; (-9,-8)*{\slinen{}};
    (-6,0)*{\ncrossbr{}{}};(6,0)*{\ncrossrp{}{}};
    (0,8)*{\ucrossbp{}{}}; (9,8)*{\slineur{}}; (-9,8)*{\slineur{}};
(-9,-13)*{\scs j};
(-3,-13)*{\scs i};
(3,-13)*{\scs i};
(9,-13)*{\scs j'};
(1.5,10)*{\bullet};
\endxy}
-{\xy  (0,0)*{\sucrossrr{i}{i}}; (5,0)*{\slineu{j}}; (-5,0)*{\sdotu{j}}; (-12,0)*{\delta_{jj'}t_{ji}}; \endxy}
\right. \\
&
\left.
+ {\xy  (0,0)*{\sucrossrr{i}{i}}; (5,0)*{\sdotu{j}}; (-5,0)*{\slineu{j}}; (-12,0)*{\delta_{jj'}t_{ji}}; \endxy}
+{\xy 0;/r.14pc/:
     (0,-8)*{\ncrossbp{}{}}; (9,-8)*{\slinenr{}}; (-9,-8)*{\slinenr{}};
    (-6,0)*{\ncrossrp{}{}};(6,0)*{\ncrossbr{}{}};
    (0,8)*{\ucrossrr{}{}}; (9,8)*{\slineu{}}; (-9,8)*{\slineup{}};
(-9,-13)*{\scs i};
(-3,-13)*{\scs j};
(3,-13)*{\scs j'};
(9,-13)*{\scs i};
(-1.5,-8.5)*{\bullet};
\endxy}
-{\xy 0;/r.14pc/:
     (0,-8)*{\ncrossbp{}{}}; (9,-8)*{\slinenr{}}; (-9,-8)*{\slinenr{}};
    (-6,0)*{\ncrossrp{}{}};(6,0)*{\ncrossbr{}{}};
    (0,8)*{\ucrossrr{}{}}; (9,8)*{\slineu{}}; (-9,8)*{\slineup{}};
(-9,-13)*{\scs i};
(-3,-13)*{\scs j};
(3,-13)*{\scs j'};
(9,-13)*{\scs i};
(9,7)*{\bullet};
\endxy}
+{\xy (0,0)*{\sucrossbp{j}{j'}}; (5,0)*{\slineur{i}}; (-5,0)*{\slineur{i}}; (-1.5,-.75)*{\bullet}; (-8,0)*{t_{ij'}}; \endxy}
-{\xy (0,0)*{\sucrossbp{j}{j'}}; (5,0)*{\slineur{i}}; (-5,0)*{\slineur{i}}; (1.25,3)*{\bullet}; (-8,0)*{t_{ij'}}; \endxy},
-{\xy 0;/r.14pc/:
    (0,-8)*{\ncrossbr{}{}}; (9,-8)*{\slinenp{}}; (-9,-8)*{\slinenr{}};
    (-6,0)*{\ncrossrr{}{}};(6,0)*{\ncrossbp{}{}};
    (0,8)*{\ucrossrp{}{}}; (9,8)*{\slineu{}}; (-9,8)*{\slineur{}};
(-9,-13)*{\scs i};
(-3,-13)*{\scs j};
(3,-13)*{\scs i};
(9,-13)*{\scs j'};
(-1.5,-8.5)*{\bullet};
\endxy}
+{\xy 0;/r.14pc/:
    (0,-8)*{\ncrossbr{}{}}; (9,-8)*{\slinenp{}}; (-9,-8)*{\slinenr{}};
    (-6,0)*{\ncrossrr{}{}};(6,0)*{\ncrossbp{}{}};
    (0,8)*{\ucrossrp{}{}}; (9,8)*{\slineu{}}; (-9,8)*{\slineur{}};
(-9,-13)*{\scs i};
(-3,-13)*{\scs j};
(3,-13)*{\scs i};
(9,-13)*{\scs j'};
(9,7)*{\bullet};
\endxy}
\right) \\
& =
\delta_{jj'}\left(
t_{ij}^{-1}{\xy (-6,3)*{\slineu{}}; (-6,-4)*{\slinen{}}; 0;/r.15pc/:(2,0)*{\Rthreer{black}{blue}{black}{}{}{}};
(-10,-13)*{\scs j};
(-4,-13)*{\scs i};
(2,-13)*{\scs j};
(8,-13)*{\scs i};
\endxy}
+{\xy (-6,0)*{\slineu{j}}; (-2,0)*{\slineur{i}}; (2,0)*{\slineu{j}}; (6,0)*{\slineur{i}}; \endxy} , \;
t_{ij}^{-1}{\xy 0;/r.14pc/:
     (0,-8)*{\ncrossrr{}{}}; (9,-8)*{\slinen{}}; (-9,-8)*{\slinen{}};
    (-6,8)*{\ucrossbr{}{}};(6,0)*{\ncrossrb{}{}};
    (3,8)*{\slineu{}}; (-3,0)*{\slinenr{}}; (9,8)*{\slineur{}}; (-9,0)*{\slinen{}};
(-9,-13)*{\scs j};
(-3,-13)*{\scs i};
(3,-13)*{\scs i};
(9,-13)*{\scs j};
\endxy}
-v_{ij}{\xy (0,0)*{\sucrossrr{i}{i}}; (5,0)*{\slineu{j}}; (-5,0)*{\sdotu{j}}; \endxy}
+v_{ij}{\xy (0,0)*{\sucrossrr{i}{i}}; (5,0)*{\sdotu{j}}; (-5,0)*{\slineu{j}}; \endxy}
\right.
\nn\\
& \hspace{165pt}
\left.
 +t_{ij}^{-1}{\xy 0;/r.14pc/:
     (0,8)*{\ucrossrr{}{}}; (9,8)*{\slineu{}}; (-9,8)*{\slineu{}};
    (-6,0)*{\ncrossrb{}{}};(6,-8)*{\ncrossbr{}{}};
    (3,0)*{\slinenr{}}; (-3,-8)*{\slinen{}}; (9,0)*{\slinen{}}; (-9,-8)*{\slinenr{}};
(-9,-13)*{\scs i};
(-3,-13)*{\scs j};
(3,-13)*{\scs j};
(9,-13)*{\scs i};
\endxy}
+{\xy (-6,0)*{\slineur{i}}; (-2,0)*{\slineu{j}}; (2,0)*{\slineu{j}}; (6,0)*{\slineur{i}}; \endxy},
-t_{ij}^{-1}{\xy (6,3)*{\slineu{}}; (6,-4)*{\slinen{}};0;/r.15pc/:(-2,0)*{\Rthreel{black}{blue}{black}{}{}{}};
(-8,-13)*{\scs i};
(-2,-13)*{\scs j};
(4,-13)*{\scs i};
(10,-13)*{\scs j};
\endxy}
+{\xy (-6,0)*{\slineur{i}}; (-2,0)*{\slineu{j}}; (2,0)*{\slineur{i}}; (6,0)*{\slineu{j}}; \endxy}
\right) \\
& =
\left\{\begin{array}{l}
\cal{T}'_i\left({\xy (0,0)*{\slineu{j}}; (4,0)*{\slineu{j}}; \endxy}\right)
+
\left(
t_{ij}^{-1}{\xy (-6,3)*{\slineu{}}; (-6,-4)*{\slinen{}}; 0;/r.14pc/:(2,0)*{\Rthreer{black}{blue}{black}{}{}{}};
(-10,-13)*{\scs j};
(-4,-13)*{\scs i};
(2,-13)*{\scs j};
(8,-13)*{\scs i};
\endxy},
t_{ij}^{-1}{\xy 0;/r.14pc/:
     (0,-8)*{\ncrossrr{}{}}; (9,-8)*{\slinen{}}; (-9,-8)*{\slinen{}};
    (-6,8)*{\ucrossbr{}{}};(6,0)*{\ncrossrb{}{}};
    (3,8)*{\slineu{}}; (-3,0)*{\slinenr{}}; (9,8)*{\slineur{}}; (-9,0)*{\slinen{}};
(-9,-13)*{\scs j};
(-3,-13)*{\scs i};
(3,-13)*{\scs i};
(9,-13)*{\scs j};
\endxy}
-v_{ij}{\xy (0,0)*{\sucrossrr{i}{i}}; (5,0)*{\slineu{j}}; (-5,0)*{\sdotu{j}}; \endxy}
+v_{ij}{\xy (0,0)*{\sucrossrr{i}{i}}; (5,0)*{\sdotu{j}}; (-5,0)*{\slineu{j}}; \endxy}
-{\xy (-6,0)*{\slineu{j}}; (-2,0)*{\slineur{i}}; (2,0)*{\slineur{i}}; (6,0)*{\slineu{j}}; \endxy}
\right. \\
\hspace{165pt}\left.
+ t_{ij}^{-1}{\xy 0;/r.14pc/:
     (0,8)*{\ucrossrr{}{}}; (9,8)*{\slineu{}}; (-9,8)*{\slineu{}};
    (-6,0)*{\ncrossrb{}{}};(6,-8)*{\ncrossbr{}{}};
    (3,0)*{\slinenr{}}; (-3,-8)*{\slinen{}}; (9,0)*{\slinen{}}; (-9,-8)*{\slinenr{}};
(-9,-13)*{\scs i};
(-3,-13)*{\scs j};
(3,-13)*{\scs j};
(9,-13)*{\scs i};
\endxy},
-t_{ij}^{-1}{\xy (6,3)*{\slineu{}}; (6,-4)*{\slinen{}};0;/r.15pc/:(-2,0)*{\Rthreel{black}{blue}{black}{}{}{}};
(-8,-13)*{\scs i};
(-2,-13)*{\scs j};
(4,-13)*{\scs i};
(10,-13)*{\scs j};
\endxy} \right) \quad \text{if }j=j'\\
0 \quad \text{if }j\neq j'.
\end{array}\right.
\end{align*}
The relation thus holds on the nose unless $j=j'$, in which case the ``error term''
is null-homotopic, with homotopy given by:
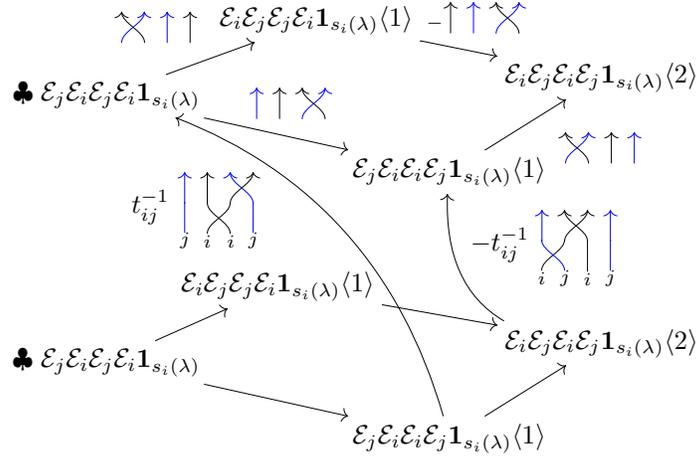
\[
{ \xy 0;/r.12pc/:
 (-65,-40)*+{\clubsuit \;\cal{E}_j\cal{E}_i\cal{E}_{j}\cal{E}_i \onell{s_i(\lambda)}}="bl";
 (-20,-20)*+{\cal{E}_i\cal{E}_j\cal{E}_{j}\cal{E}_i \onell{s_i(\lambda)}\la 1\ra}="bt";
 (25,-60)*+{\cal{E}_j\cal{E}_i\cal{E}_i\cal{E}_{j} \onell{s_i(\lambda)}\la 1\ra}="bb";
 (65,-35)*+{\cal{E}_i\cal{E}_j\cal{E}_i\cal{E}_{j} \onell{s_i(\lambda)}\la 2\ra}="br";
  {\ar^{\xy 0;/r.18pc/:
    \endxy \;\;\;
  } "bl";"bt"};
  {\ar_{\xy 0;/r.18pc/:\endxy \;\;
  } "bl";"bb"};
  {\ar^{\xy 0;/r.18pc/:
    \endxy
  } "bt";"br"};
  {\ar_{\xy 0;/r.18pc/:
   \endxy
 } "bb";"br"};
 (-65,30)*+{\clubsuit \;\cal{E}_{j}\cal{E}_i\cal{E}_j\cal{E}_i \onell{s_i(\lambda)}}="tl";
 (-10,50)*+{\cal{E}_i\cal{E}_{j}\cal{E}_j\cal{E}_i \onell{s_i(\lambda)}\la 1\ra}="tt";
 (25,10)*+{\cal{E}_{j}\cal{E}_i\cal{E}_i\cal{E}_j \onell{s_i(\lambda)}\la 1\ra}="tb";
 (65,35)*+{\cal{E}_i\cal{E}_{j}\cal{E}_i\cal{E}_j \onell{s_i(\lambda)}\la 2\ra}="tr";
  {\ar^{\xy 0;/r.12pc/:
    (-6,0)*{\ucrossbr{}{}}; (9,0)*{\slineur{}}; (3,0)*{\slineu{}}; \endxy \;\;\;\;\;\;\;\;\;\;
  } "tl";"tt"};
  {\ar^{\;\;\;\xy 0;/r.12pc/:(6,0)*{\ucrossbr{}{}}; (-3,0)*{\slineur{}};
  (-9,0)*{\slineu{}};\endxy \;\;} "tl";"tb"};
  {\ar^{\;\;\;-\xy 0;/r.12pc/:
    (6,0)*{\ucrossbr{}{}}; (-9,0)*{\slineur{}}; (-3,0)*{\slineu{}};\endxy
  } "tt";"tr"};
  {\ar_{\;\;\;\;\;\;\; \xy 0;/r.12pc/:
    (-6,0)*{\ucrossbr{}{}}; (9,0)*{\slineu{}}; (3,0)*{\slineur{}}; \endxy } "tb";"tr"};
 {\ar@/_-2pc/@{->} "br";"tb" };
{\ar@/_2.3pc/@{->} "bb";"tl"};
(-40,0)*{\xy 0;/r.12pc/: (-18,4)*{t_{ij}^{-1}}; (0,-.75)*{\ncrossrr{i}{i}}; (9,-.75)*{\slinen{j}}; (-9,-.75)*{\slinen{j}};
    (-9,8)*{\slineu{}};(-3,8)*{\slineur{}}; (6,8)*{\ucrossrb{}{}}; \endxy};
(50,-10)*{\xy 0;/r.12pc/: (-6,-.75)*{\ncrossrb{i}{j}}; (9,-.75)*{\slinen{j}}; (3,-.75)*{\slinenr{i}}; (-9,8)*{\slineu{}};
    (0,8)*{\ucrossrr{}{}}; (9,8)*{\slineu{}}; (-20,4)*{-t_{ij}^{-1}}; \endxy};
 \endxy}
\]
The verification that
$
\cal{T}'_i\left(\xy 0;/r.18pc/: (0,0)*{\ucrossbp{j}{j'}}; (-1.75,3.25)*{\bullet};
(12,0)*{\ucrossbp{j}{j'}}; (14,-.25)*{\bullet}; (6,1)*{-};\endxy\right) \sim
\delta_{jj'}
\cal{T}'_i\left({\xy (0,0)*{\slineu{j}}; (4,0)*{\slineu{j}}; \endxy}\right)$
is almost identical to the above case, so we omit the details.
\end{proof}

%
\subsection{Cubic KLR}
%

\begin{prop}
$\cal{T}_i'$ preserves the cubic KLR relation.
\end{prop}

\begin{proof}
We verify relation~\eqref{def:KLR-R3} in Definition~\ref{defU_cat-cyc}, \ie the ``Reidemeister III''-like KLR relation.
There are 27 cases to consider,
depending on whether the label $\ell$ of each strand satisfies $i \cdot \ell = 2, -1,$ or $0$.
To cover multiple cases at once, we will make use of the following notation, setting
\[
\Delta_{abc} = \begin{cases}
t_{ab} & \text{if } a=c \text{ and } a\cdot b=-1 \\
0 & \text{else} \\
\end{cases}
\]
Note that $\Delta_{abc}=\Delta_{cba}$.

The relation holds on nose (\ie does not require a non-zero homotopy),
except for the strand labelings in the following list:
\[
iji \;\; , \;\; jkj' \;\; , \;\; jij' \;\; , \;\; jj'j 
\]
where we continue with our conventions for strand labelings
($i\cdot j = -1 = i \cdot j'$ and $i \cdot k = 0 =i \cdot k'$).
In the interest of space, we will explicitly exhibit three representative cases
that do not require homotopies (to give the flavor of the computations required),
exhibit the homotopy and verify the relation in the $iji$-labeled case,
and exhibit the homotopy (but not include all the computations involved for the verification)
in the remaining three cases.

In the $jii$-labeled case, the relation holds on the nose via the following computation,
where, as above, we denote the chain map as an ordered pair.
\begin{align*}
&\cal{T}'_i\left({\xy 0;/r.15pc/: (0,0)*{\Rthreel{blue}{black}{black}{}{}{}};
(-6,-13)*{\scs j};
(0,-13)*{\scs i};
(6,-13)*{\scs i};
\endxy}
-{\xy 0;/r.15pc/: (0,0)*{\Rthreer{blue}{black}{black}{}{}{}};
(-6,-13)*{\scs j};
(0,-13)*{\scs i};
(6,-13)*{\scs i};
\endxy}\right)
=
\\
& \; t_{ij}^2\left(
-{\xy 0;/r.12pc/:
   (6,-16)*{\slinedr{}}; (-3,-16)*{\rcrossrr{}{}}; (-12,-16)*{\slinen{}};
   (6,-8)*{\slinenr{}}; (0,-8)*{\sdotr{}}; (-9,-8)*{\rcrossbr{}{}};
   (3,0)*{\rcrossrr{}{}}; (-6,0)*{\slinen{}}; (-12,0)*{\slinenr{}};
   (6,8)*{\slinenr{}}; (-3,8)*{\rcrossbr{}{}}; (-12,8)*{\slinenr{}};
   (6,16)*{\sdotur{}}; (0,16)*{\slineu{}}; (-9,16)*{\dcrossrr{}{}};
(-12,-22)*{\scs j};
(-6,-22)*{\scs i};
(0,-22)*{\scs i};
(6,-22)*{\scs i};
\endxy}
+{\xy 0;/r.12pc/:
   (6,-16)*{\slinedr{}}; (-3,-16)*{\rcrossrr{}{}}; (-12,-16)*{\slinen{}};
   (6,-8)*{\slinenr{}}; (0,-8)*{\sdotr{}}; (-9,-8)*{\rcrossbr{}{}};
   (3,0)*{\rcrossrr{}{}}; (-6,0)*{\slinen{}}; (-12,0)*{\slinenr{}};
   (6,8)*{\slinenr{}}; (-3,8)*{\rcrossbr{}{}}; (-12,8)*{\slinenr{}}; (-5,11)*[black]{\scs \bullet};
   (6,16)*{\slineur{}}; (0,16)*{\slineu{}}; (-9,16)*{\dcrossrr{}{}};
(-12,-22)*{\scs j};
(-6,-22)*{\scs i};
(0,-22)*{\scs i};
(6,-22)*{\scs i};
\endxy}
+{\xy 0;/r.12pc/:
   (6,-16)*{\slinedr{}}; (-3,-16)*{\rcrossrr{}{}}; (-12,-16)*{\slinen{}};
   (6,-8)*{\slinenr{}}; (0,-8)*{\slinenr{}}; (-9,-8)*{\rcrossbr{}{}}; (-11,-5)*[black]{\scs \bullet};
   (3,0)*{\rcrossrr{}{}}; (-6,0)*{\slinen{}}; (-12,0)*{\slinenr{}};
   (6,8)*{\slinenr{}}; (-3,8)*{\rcrossbr{}{}}; (-12,8)*{\slinenr{}};
   (6,16)*{\sdotur{}}; (0,16)*{\slineu{}}; (-9,16)*{\dcrossrr{}{}};
(-12,-22)*{\scs j};
(-6,-22)*{\scs i};
(0,-22)*{\scs i};
(6,-22)*{\scs i};
\endxy}
-{\xy 0;/r.12pc/:
   (6,-16)*{\slinedr{}}; (-3,-16)*{\rcrossrr{}{}}; (-12,-16)*{\slinen{}};
   (6,-8)*{\slinenr{}}; (0,-8)*{\slinenr{}}; (-9,-8)*{\rcrossbr{}{}}; (-11,-5)*[black]{\scs \bullet};
   (3,0)*{\rcrossrr{}{}}; (-6,0)*{\slinen{}}; (-12,0)*{\slinenr{}};
   (6,8)*{\slinenr{}}; (-3,8)*{\rcrossbr{}{}}; (-12,8)*{\slinenr{}}; (-5,11)*[black]{\scs \bullet};
   (6,16)*{\slineur{}}; (0,16)*{\slineu{}}; (-9,16)*{\dcrossrr{}{}};
(-12,-22)*{\scs j};
(-6,-22)*{\scs i};
(0,-22)*{\scs i};
(6,-22)*{\scs i};
\endxy}
  +{\xy 0;/r.12pc/:
   (3,-16)*{\dcrossrr{}{}}; (-6,-16)*{\slinenr{}}; (-12,-16)*{\slinen{}};
   (6,-8)*{\slinenr{}}; (-3,-8)*{\rcrossrr{}{}}; (-12,-8)*{\slinen{}};
   (6,0)*{\slinenr{}}; (0,0)*{\sdotr{}}; (-9,0)*{\rcrossbr{}{}};
   (3,8)*{\rcrossrr{}{}}; (-6,8)*{\slinen{}}; (-12,8)*{\slinenr{}};
   (6,16)*{\sdotur{}}; (-3,16)*{\rcrossbr{}{}}; (-12,16)*{\slinenr{}};
(-12,-22)*{\scs j};
(-6,-22)*{\scs i};
(0,-22)*{\scs i};
(6,-22)*{\scs i};
\endxy}
-{\xy 0;/r.12pc/:
   (3,-16)*{\dcrossrr{}{}}; (-6,-16)*{\slinenr{}}; (-12,-16)*{\slinen{}};
   (6,-8)*{\slinenr{}}; (-3,-8)*{\rcrossrr{}{}}; (-12,-8)*{\slinen{}};
   (6,0)*{\slinenr{}}; (0,0)*{\slinenr{}}; (-9,0)*{\rcrossbr{}{}};
   (3,8)*{\rcrossrr{}{}}; (-6,8)*{\slinen{}}; (-12,8)*{\sdotr{}};
   (6,16)*{\sdotur{}}; (-3,16)*{\rcrossbr{}{}}; (-12,16)*{\slinenr{}};
(-12,-22)*{\scs j};
(-6,-22)*{\scs i};
(0,-22)*{\scs i};
(6,-22)*{\scs i};
\endxy}
-{\xy 0;/r.12pc/:
   (3,-16)*{\dcrossrr{}{}}; (-6,-16)*{\slinenr{}}; (-12,-16)*{\slinen{}};
   (6,-8)*{\slinenr{}}; (-3,-8)*{\rcrossrr{}{}}; (-12,-8)*{\slinen{}};
   (6,0)*{\slinenr{}}; (0,0)*{\sdotr{}}; (-9,0)*{\rcrossbr{}{}};
   (3,8)*{\rcrossrr{}{}}; (-6,8)*{\slinen{}}; (-12,8)*{\slinenr{}};
   (6,16)*{\slineur{}}; (-3,16)*{\rcrossbr{}{}}; (-12,16)*{\slinenr{}};
   (-5,19)*[black]{\scs \bullet};
(-12,-22)*{\scs j};
(-6,-22)*{\scs i};
(0,-22)*{\scs i};
(6,-22)*{\scs i};
\endxy}
+{\xy 0;/r.12pc/:
   (3,-16)*{\dcrossrr{}{}}; (-6,-16)*{\slinenr{}}; (-12,-16)*{\slinen{}};
   (6,-8)*{\slinenr{}}; (-3,-8)*{\rcrossrr{}{}}; (-12,-8)*{\slinen{}};
   (6,0)*{\slinenr{}}; (0,0)*{\slinenr{}}; (-9,0)*{\rcrossbr{}{}};
   (3,8)*{\rcrossrr{}{}}; (-6,8)*{\slinen{}}; (-12,8)*{\sdotr{}};
   (6,16)*{\slineur{}}; (-3,16)*{\rcrossbr{}{}}; (-12,16)*{\slinenr{}};
   (-5,19)*[black]{\scs \bullet};
(-12,-22)*{\scs j};
(-6,-22)*{\scs i};
(0,-22)*{\scs i};
(6,-22)*{\scs i};
\endxy},\right.
\\
&\left. \;\;
 -{\xy 0;/r.12pc/:
   (6,-16)*{\slinedr{}}; (-3,-16)*{\rcrossbr{}{}}; (-12,-16)*{\slinenr{}};
   (6,-8)*{\slinenr{}}; (0,-8)*{\slinen{}}; (-9,-8)*{\rcrossrr{}{}};
   (3,0)*{\rcrossbr{}{}}; (-6,0)*{\slinenr{}}; (-12,0)*{\sdotr{}};
   (6,8)*{\slinen{}}; (-3,8)*{\rcrossrr{}{}}; (-12,8)*{\slinenr{}}; (-5,11)*[black]{\scs \bullet};
   (6,16)*{\slineu{}}; (0,16)*{\slineur{}}; (-9,16)*{\dcrossrr{}{}};
(-12,-22)*{\scs i};
(-6,-22)*{\scs j};
(0,-22)*{\scs i};
(6,-22)*{\scs i};
\endxy}
+{\xy 0;/r.12pc/:
   (6,-16)*{\slinedr{}}; (-3,-16)*{\rcrossbr{}{}}; (-12,-16)*{\slinenr{}};
   (6,-8)*{\slinenr{}}; (0,-8)*{\slinen{}}; (-9,-8)*{\rcrossrr{}{}};
   (3,0)*{\rcrossbr{}{}}; (-6,0)*{\slinenr{}}; (-12,0)*{\sdotr{}};
   (6,8)*{\slinen{}}; (-3,8)*{\rcrossrr{}{}}; (-12,8)*{\slinenr{}};
   (6,16)*{\slineu{}}; (0,16)*{\sdotur{}}; (-9,16)*{\dcrossrr{}{}};
(-12,-22)*{\scs i};
(-6,-22)*{\scs j};
(0,-22)*{\scs i};
(6,-22)*{\scs i};
\endxy}
+{\xy 0;/r.12pc/:
   (6,-16)*{\slinedr{}}; (-3,-16)*{\rcrossbr{}{}}; (-12,-16)*{\slinenr{}};
   (6,-8)*{\slinenr{}}; (0,-8)*{\slinen{}}; (-9,-8)*{\rcrossrr{}{}};
   (3,0)*{\rcrossbr{}{}}; (-6,0)*{\sdotr{}}; (-12,0)*{\slinenr{}};
   (6,8)*{\slinen{}}; (-3,8)*{\rcrossrr{}{}}; (-12,8)*{\slinenr{}}; (-5,11)*[black]{\scs \bullet};
   (6,16)*{\slineu{}}; (0,16)*{\slineur{}}; (-9,16)*{\dcrossrr{}{}};
(-12,-22)*{\scs i};
(-6,-22)*{\scs j};
(0,-22)*{\scs i};
(6,-22)*{\scs i};
\endxy}
-{\xy 0;/r.12pc/:
   (6,-16)*{\slinedr{}}; (-3,-16)*{\rcrossbr{}{}}; (-12,-16)*{\slinenr{}};
   (6,-8)*{\slinenr{}}; (0,-8)*{\slinen{}}; (-9,-8)*{\rcrossrr{}{}};
   (3,0)*{\rcrossbr{}{}}; (-6,0)*{\sdotr{}}; (-12,0)*{\slinenr{}};
   (6,8)*{\slinen{}}; (-3,8)*{\rcrossrr{}{}}; (-12,8)*{\slinenr{}};
   (6,16)*{\slineu{}}; (0,16)*{\sdotur{}}; (-9,16)*{\dcrossrr{}{}};
(-12,-22)*{\scs i};
(-6,-22)*{\scs j};
(0,-22)*{\scs i};
(6,-22)*{\scs i};
\endxy}
 +{\xy 0;/r.12pc/:
   (3,-16)*{\dcrossrr{}{}}; (-6,-16)*{\slinen{}}; (-12,-16)*{\slinenr{}};
   (6,-8)*{\slinenr{}}; (-3,-8)*{\rcrossbr{}{}}; (-12,-8)*{\slinenr{}};
   (6,0)*{\slinenr{}}; (0,0)*{\slinen{}}; (-9,0)*{\rcrossrr{}{}};
   (3,8)*{\rcrossbr{}{}}; (-6,8)*{\slinenr{}}; (-12,8)*{\sdotr{}};(-4.5,18.5)*[black]{\scs \bullet};
   (6,16)*{\slineu{}}; (-3,16)*{\rcrossrr{}{}}; (-12,16)*{\slinenr{}};
(-12,-22)*{\scs i};
(-6,-22)*{\scs j};
(0,-22)*{\scs i};
(6,-22)*{\scs i};
\endxy}
-{\xy 0;/r.12pc/:
   (3,-16)*{\dcrossrr{}{}}; (-6,-16)*{\slinen{}}; (-12,-16)*{\slinenr{}};
   (6,-8)*{\slinenr{}}; (-3,-8)*{\rcrossbr{}{}}; (-12,-8)*{\slinenr{}};
   (6,0)*{\slinenr{}}; (0,0)*{\slinen{}}; (-9,0)*{\rcrossrr{}{}};
   (3,8)*{\rcrossbr{}{}}; (-6,8)*{\sdotr{}}; (-12,8)*{\slinenr{}};(-4.75,18.5)*[black]{\scs \bullet};
   (6,16)*{\slineu{}}; (-3,16)*{\rcrossrr{}{}}; (-12,16)*{\slinenr{}};
(-12,-22)*{\scs i};
(-6,-22)*{\scs j};
(0,-22)*{\scs i};
(6,-22)*{\scs i};
\endxy}
-{\xy 0;/r.12pc/:
   (3,-16)*{\dcrossrr{}{}}; (-6,-16)*{\slinen{}}; (-12,-16)*{\slinenr{}};
   (6,-8)*{\slinenr{}}; (-3,-8)*{\rcrossbr{}{}}; (-12,-8)*{\slinenr{}};
   (6,0)*{\slinenr{}}; (0,0)*{\slinen{}}; (-9,0)*{\rcrossrr{}{}};
   (3,8)*{\rcrossbr{}{}}; (-6,8)*{\slinenr{}}; (-12,8)*{\sdotr{}};(-1.5,18)*[black]{\scs \bullet};
   (6,16)*{\slineu{}}; (-3,16)*{\rcrossrr{}{}}; (-12,16)*{\slinenr{}};
(-12,-22)*{\scs i};
(-6,-22)*{\scs j};
(0,-22)*{\scs i};
(6,-22)*{\scs i};
\endxy}
+{\xy 0;/r.12pc/:
   (3,-16)*{\dcrossrr{}{}}; (-6,-16)*{\slinen{}}; (-12,-16)*{\slinenr{}};
   (6,-8)*{\slinenr{}}; (-3,-8)*{\rcrossbr{}{}}; (-12,-8)*{\slinenr{}};
   (6,0)*{\slinenr{}}; (0,0)*{\slinen{}}; (-9,0)*{\rcrossrr{}{}};
   (3,8)*{\rcrossbr{}{}}; (-6,8)*{\sdotr{}}; (-12,8)*{\slinenr{}};(-1.5,18)*[black]{\scs \bullet};
   (6,16)*{\slineu{}}; (-3,16)*{\rcrossrr{}{}}; (-12,16)*{\slinenr{}};
(-12,-22)*{\scs i};
(-6,-22)*{\scs j};
(0,-22)*{\scs i};
(6,-22)*{\scs i};
\endxy}\right)
\end{align*}
To simplify, we use the dot slide relation to move all dots to the top,
and apply the cubic KLR relation to cancel terms, arriving at:
\begin{align*}
 &= t_{ij}^2\left(
{\xy 0;/r.12pc/:
   (6,-16)*{\slinedr{}}; (-3,-16)*{\rcrossrr{}{}}; (-12,-16)*{\slinen{}};
   (6,-8)*{\slinenr{}}; (0,-8)*{\slinenr{}}; (-9,-8)*{\rcrossbr{}{}};
   (3,-1)*{\ssrcapr{}}; (3,2)*{\ssrcupr{}};(-6,0)*{\slinen{}}; (-12,0)*{\slinenr{}};
   (6,8)*{\slinenr{}}; (-3,8)*{\rcrossbr{}{}}; (-12,8)*{\slinenr{}};
   (6,16)*{\sdotur{}}; (0,16)*{\slineu{}}; (-9,16)*{\dcrossrr{}{}};
(-12,-22)*{\scs j};
(-6,-22)*{\scs i};
(0,-22)*{\scs i};
(6,-22)*{\scs i};
\endxy}
-{\xy 0;/r.12pc/:
   (6,-16)*{\slinedr{}}; (-3,-16)*{\rcrossrr{}{}}; (-12,-16)*{\slinen{}};
   (6,-8)*{\slinenr{}}; (0,-8)*{\slinenr{}}; (-9,-8)*{\rcrossbr{}{}};
   (3,-1)*{\ssrcapr{}}; (3,2)*{\ssrcupr{}};(-6,0)*{\slinen{}}; (-12,0)*{\slinenr{}};
   (6,8)*{\slinenr{}}; (-3,8)*{\rcrossbr{}{}}; (-12,8)*{\slinenr{}}; (-10.25,18)*[black]{\scs \bullet};
   (6,16)*{\slineur{}}; (0,16)*{\slineu{}}; (-9,16)*{\dcrossrr{}{}};
(-12,-22)*{\scs j};
(-6,-22)*{\scs i};
(0,-22)*{\scs i};
(6,-22)*{\scs i};
\endxy}
-{\xy 0;/r.12pc/:
   (6,-16)*{\slinedr{}}; (-3,-16)*{\rcrossrr{}{}}; (-12,-16)*{\slinen{}};
   (6,-8)*{\slinenr{}}; (0,-8)*{\slinenr{}}; (-9,-8)*{\rcrossbr{}{}};
   (3,-1)*{\ssrcapr{}}; (3,2)*{\ssrcupr{}}; (-6,0)*{\slinen{}}; (-12,0)*{\slinenr{}};
   (6,8)*{\slinenr{}}; (-3,8)*{\rcrossbr{}{}}; (-12,8)*{\slinenr{}};
   (6,16)*{\slineur{}}; (0,16)*{\slineu{}}; (-12,16)*{\slinedr{}}; (-6,16)*{\slinedr{}};
(-12,-22)*{\scs j};
(-6,-22)*{\scs i};
(0,-22)*{\scs i};
(6,-22)*{\scs i};
\endxy}
-{\xy 0;/r.12pc/:
   (6,-16)*{\slinedr{}}; (-3,-16)*{\rcrossrr{}{}}; (-12,-16)*{\slinen{}};
   (6,-8)*{\slinenr{}}; (0,-8)*{\slinenr{}}; (-9,-8)*{\rcrossbr{}{}};
   (3,0)*{\rcrossrr{}{}}; (-6,0)*{\slinen{}}; (-12,0)*{\slinenr{}};
   (6,8)*{\slinenr{}}; (-3,8)*{\rcrossbr{}{}}; (-12,8)*{\slinenr{}};
   (6,16)*{\sdotur{}}; (0,16)*{\slineu{}}; (-12,16)*{\slinedr{}}; (-6,16)*{\slinedr{}};
(-12,-22)*{\scs j};
(-6,-22)*{\scs i};
(0,-22)*{\scs i};
(6,-22)*{\scs i};
\endxy}
+{\xy 0;/r.12pc/:
   (6,-16)*{\slinedr{}}; (-3,-16)*{\rcrossrr{}{}}; (-12,-16)*{\slinen{}};
   (6,-8)*{\slinenr{}}; (0,-8)*{\slinenr{}}; (-9,-8)*{\rcrossbr{}{}};
   (3,0)*{\rcrossrr{}{}}; (-6,0)*{\slinen{}}; (-12,0)*{\slinenr{}};
   (6,8)*{\slinenr{}}; (-3,8)*{\rcrossbr{}{}}; (-12,8)*{\slinenr{}};
   (6,16)*{\sdotur{}}; (0,16)*{\slineu{}};  (-12,16)*{\slinedr{}}; (-6,16)*{\slinedr{}};
(-12,-22)*{\scs j};
(-6,-22)*{\scs i};
(0,-22)*{\scs i};
(6,-22)*{\scs i};
\endxy}
-{\xy 0;/r.12pc/:
   (3,-16)*{\dcrossrr{}{}}; (-6,-16)*{\slinenr{}}; (-12,-16)*{\slinen{}};
   (6,-8)*{\slinenr{}}; (-3,-8)*{\rcrossrr{}{}}; (-12,-8)*{\slinen{}};
   (6,0)*{\slinenr{}}; (0,0)*{\slinenr{}}; (-9,0)*{\rcrossbr{}{}};
   (3,7)*{\ssrcapr{}};(3,10)*{\ssrcupr{}}; (-6,8)*{\slinen{}}; (-12,8)*{\slinenr{}};
   (6,16)*{\sdotur{}}; (-3,16)*{\rcrossbr{}{}}; (-12,16)*{\slinenr{}};
(-12,-22)*{\scs j};
(-6,-22)*{\scs i};
(0,-22)*{\scs i};
(6,-22)*{\scs i};
\endxy}
+{\xy 0;/r.12pc/:
   (3,-16)*{\dcrossrr{}{}}; (-6,-16)*{\slinenr{}}; (-12,-16)*{\slinen{}};
   (6,-8)*{\slinenr{}}; (-3,-8)*{\rcrossrr{}{}}; (-12,-8)*{\slinen{}};
   (6,0)*{\slinenr{}}; (0,0)*{\slinenr{}}; (-9,0)*{\rcrossbr{}{}};
   (3,7)*{\ssrcapr{}};(3,10)*{\ssrcupr{}};(-6,8)*{\slinen{}}; (-12,8)*{\slinenr{}};(-5,19)*[black]{\scs \bullet};
   (6,16)*{\slineur{}}; (-3,16)*{\rcrossbr{}{}}; (-12,16)*{\slinenr{}};
(-12,-22)*{\scs j};
(-6,-22)*{\scs i};
(0,-22)*{\scs i};
(6,-22)*{\scs i};
\endxy},
\right.\\
& \qquad \qquad \left.
-{\xy 0;/r.12pc/:
   (6,-16)*{\slinedr{}}; (-3,-16)*{\rcrossbr{}{}}; (-12,-16)*{\slinenr{}};
   (6,-8)*{\slinenr{}}; (0,-8)*{\slinen{}}; (-9,-8)*{\rcrossrr{}{}};
   (3,0)*{\rcrossbr{}{}}; (-6,0)*{\slinenr{}}; (-12,0)*{\slinenr{}};
   (6,8)*{\slinen{}}; (-3,8)*{\rcrossrr{}{}}; (-12,8)*{\slinenr{}};
   (6,16)*{\slineu{}}; (0,16)*{\sdotur{}}; (-12,16)*{\slinedr{}};(-6,16)*{\slinedr{}};
(-12,-22)*{\scs i};
(-6,-22)*{\scs j};
(0,-22)*{\scs i};
(6,-22)*{\scs i};
\endxy}
+{\xy 0;/r.12pc/:
   (6,-16)*{\slinedr{}}; (-3,-16)*{\rcrossbr{}{}}; (-12,-16)*{\slinenr{}};
   (6,-8)*{\slinenr{}}; (0,-8)*{\slinen{}}; (-9,-8)*{\rcrossrr{}{}};
   (3,0)*{\rcrossbr{}{}}; (-6,0)*{\slinenr{}}; (-12,0)*{\slinenr{}};
   (6,8)*{\slinen{}}; (-3,8)*{\rcrossrr{}{}}; (-12,8)*{\slinenr{}};
   (6,16)*{\slineu{}}; (0,16)*{\sdotur{}}; (-12,16)*{\slinedr{}};(-6,16)*{\slinedr{}};
(-12,-22)*{\scs i};
(-6,-22)*{\scs j};
(0,-22)*{\scs i};
(6,-22)*{\scs i};
\endxy}
-{\xy 0;/r.12pc/:
   (6,-16)*{\slinedr{}}; (-3,-16)*{\rcrossbr{}{}}; (-12,-16)*{\slinenr{}};
   (6,-8)*{\slinenr{}}; (0,-8)*{\slinen{}}; (-9,-8)*{\rcrossrr{}{}};
   (3,0)*{\rcrossbr{}{}}; (-6,0)*{\slinenr{}}; (-12,0)*{\slinenr{}};
   (6,8)*{\slinen{}}; (-3,7)*{\ssrcapr{}};(-3,10)*{\ssrcupr{}}; (-12,8)*{\slinenr{}}; (-10.25,18)*[black]{\scs \bullet};
   (6,16)*{\slineu{}}; (0,16)*{\slineur{}}; (-9,16)*{\dcrossrr{}{}};
(-12,-22)*{\scs i};
(-6,-22)*{\scs j};
(0,-22)*{\scs i};
(6,-22)*{\scs i};
\endxy}
-{\xy 0;/r.12pc/:
   (6,-16)*{\slinedr{}}; (-3,-16)*{\rcrossbr{}{}}; (-12,-16)*{\slinenr{}};
   (6,-8)*{\slinenr{}}; (0,-8)*{\slinen{}}; (-9,-8)*{\rcrossrr{}{}};
   (3,0)*{\rcrossbr{}{}}; (-6,0)*{\slinenr{}}; (-12,0)*{\slinenr{}};
   (6,8)*{\slinen{}}; (-3,7)*{\ssrcapr{}};(-3,10)*{\ssrcupr{}}; (-12,8)*{\slinenr{}};
   (6,16)*{\slineu{}}; (0,16)*{\slineur{}}; (-12,16)*{\slinedr{}};(-6,16)*{\slinedr{}};
(-12,-22)*{\scs i};
(-6,-22)*{\scs j};
(0,-22)*{\scs i};
(6,-22)*{\scs i};
\endxy}
+
{\xy 0;/r.12pc/:
   (6,-16)*{\slinedr{}}; (-3,-16)*{\rcrossbr{}{}}; (-12,-16)*{\slinenr{}};
   (6,-8)*{\slinenr{}}; (0,-8)*{\slinen{}}; (-9,-8)*{\rcrossrr{}{}};
   (3,0)*{\rcrossbr{}{}}; (-6,0)*{\slinenr{}}; (-12,0)*{\slinenr{}};
   (6,8)*{\slinen{}}; (-3,7)*{\ssrcapr{}};(-3,10)*{\ssrcupr{}}; (-12,8)*{\slinenr{}};
   (6,16)*{\slineu{}}; (0,16)*{\sdotur{}}; (-9,16)*{\dcrossrr{}{}};
(-12,-22)*{\scs i};
(-6,-22)*{\scs j};
(0,-22)*{\scs i};
(6,-22)*{\scs i};
\endxy}
+{\xy 0;/r.12pc/:
   (3,-16)*{\dcrossrr{}{}}; (-6,-16)*{\slinen{}}; (-12,-16)*{\slinenr{}};
   (6,-8)*{\slinenr{}}; (-3,-8)*{\rcrossbr{}{}}; (-12,-8)*{\slinenr{}};
   (6,0)*{\slinenr{}}; (0,0)*{\slinen{}}; (-9,0)*{\rcrossrr{}{}};
   (3,8)*{\rcrossbr{}{}}; (-6,8)*{\slinenr{}}; (-12,8)*{\slinenr{}};(-4.5,18.5)*[black]{\scs \bullet};
   (6,16)*{\slineu{}}; (-3,15)*{\ssrcapr{}};(-3,18)*{\ssrcupr{}}; (-12,16)*{\slinenr{}};
(-12,-22)*{\scs i};
(-6,-22)*{\scs j};
(0,-22)*{\scs i};
(6,-22)*{\scs i};
\endxy}
-{\xy 0;/r.12pc/:
   (3,-16)*{\dcrossrr{}{}}; (-6,-16)*{\slinen{}}; (-12,-16)*{\slinenr{}};
   (6,-8)*{\slinenr{}}; (-3,-8)*{\rcrossbr{}{}}; (-12,-8)*{\slinenr{}};
   (6,0)*{\slinenr{}}; (0,0)*{\slinen{}}; (-9,0)*{\rcrossrr{}{}};
   (3,8)*{\rcrossbr{}{}}; (-6,8)*{\slinenr{}}; (-12,8)*{\slinenr{}};(-4.5,18.5)*[black]{\scs \bullet};
   (6,16)*{\slineu{}}; (-3,15)*{\ssrcapr{}};(-3,18)*{\ssrcupr{}}; (-12,16)*{\slinenr{}};
(-12,-22)*{\scs i};
(-6,-22)*{\scs j};
(0,-22)*{\scs i};
(6,-22)*{\scs i};
\endxy}\right)\\
& =0.
\end{align*}
For the $jik$-labeled case, we have:
\begin{align*}
\cal{T}'_i\left({\xy 0;/r.15pc/: (0,0)*{\Rthreel{blue}{black}{green}{}{}{}};
(-6,-13)*{\scs j};
(0,-13)*{\scs i};
(6,-13)*{\scs k};
\endxy}\right)
&= t_{ij}\left(
{\xy 0;/r.12pc/:
   (6,-16)*{\slineng{}}; (-3,-16)*{\rcrossrr{}{}}; (-12,-16)*{\slinen{}};
   (6,-8)*{\slineng{}}; (0,-8)*{\sdotr{}}; (-9,-8)*{\rcrossbr{}{}};
   (3,0)*{\ucrossrg{}{}}; (-6,0)*{\slinen{}}; (-12,0)*{\slinenr{}};
   (6,8)*{\slinenr{}}; (-3,8)*{\ucrossbg{}{}}; (-12,8)*{\slinenr{}};
   (6,16)*{\slineur{}}; (0,16)*{\slineu{}}; (-9,16)*{\lcrossrg{}{}};
(-12,-22)*{\scs j};
(-6,-22)*{\scs i};
(0,-22)*{\scs i};
(6,-22)*{\scs k};
\endxy}
-{\xy 0;/r.12pc/:
   (6,-16)*{\slineng{}}; (-3,-16)*{\rcrossrr{}{}}; (-12,-16)*{\slinen{}};
   (6,-8)*{\slineng{}}; (0,-8)*{\slinenr{}}; (-9,-8)*{\rcrossbr{}{}};
   (3,0)*{\ucrossrg{}{}}; (-6,0)*{\slinen{}}; (-12,0)*{\sdotr{}};
   (6,8)*{\slinenr{}}; (-3,8)*{\ucrossbg{}{}}; (-12,8)*{\slinenr{}};
   (6,16)*{\slineur{}}; (0,16)*{\slineu{}}; (-9,16)*{\lcrossrg{}{}};
(-12,-22)*{\scs j};
(-6,-22)*{\scs i};
(0,-22)*{\scs i};
(6,-22)*{\scs k};
\endxy},
{\xy 0;/r.12pc/:
   (6,-16)*{\slineng{}}; (-3,-16)*{\rcrossbr{}{}}; (-12,-16)*{\slinenr{}};
   (6,-8)*{\slineng{}}; (0,-8)*{\slinen{}}; (-9,-8)*{\rcrossrr{}{}};
   (3,0)*{\ucrossbg{}{}}; (-6,0)*{\slinenr{}}; (-12,0)*{\sdotr{}};
   (6,8)*{\slinen{}}; (-3,8)*{\ucrossrg{}{}}; (-12,8)*{\slinenr{}};
   (6,16)*{\slineu{}}; (0,16)*{\slineur{}}; (-9,16)*{\lcrossrg{}{}};
(-12,-22)*{\scs i};
(-6,-22)*{\scs j};
(0,-22)*{\scs i};
(6,-22)*{\scs k};
\endxy}
-{\xy 0;/r.12pc/:
   (6,-16)*{\slineng{}}; (-3,-16)*{\rcrossbr{}{}}; (-12,-16)*{\slinenr{}};
   (6,-8)*{\slineng{}}; (0,-8)*{\slinen{}}; (-9,-8)*{\rcrossrr{}{}};
   (3,0)*{\ucrossbg{}{}}; (-6,0)*{\sdotr{}}; (-12,0)*{\slinenr{}};
   (6,8)*{\slinen{}}; (-3,8)*{\ucrossrg{}{}}; (-12,8)*{\slinenr{}};
   (6,16)*{\slineu{}}; (0,16)*{\slineur{}}; (-9,16)*{\lcrossrg{}{}};
(-12,-22)*{\scs i};
(-6,-22)*{\scs j};
(0,-22)*{\scs i};
(6,-22)*{\scs k};
\endxy}\right)\\
&= t_{ij}\left(
{\xy 0;/r.12pc/:
   (3,-16)*{\lcrossrg{}{}}; (-6,-16)*{\slinenr{}}; (-12,-16)*{\slinen{}};
   (6,-8)*{\slinenr{}}; (-3,-8)*{\ucrossrg{}{}}; (-12,-8)*{\slinen{}};
   (6,0)*{\slinenr{}}; (0,0)*{\slinenr{}}; (-9,0)*{\ucrossbg{}{}};
   (3,8)*{\rcrossrr{}{}}; (-6,8)*{\slinen{}}; (-12,8)*{\slineng{}};
   (6,16)*{\sdotur{}}; (-3,16)*{\rcrossbr{}{}}; (-12,16)*{\slineug{}};
(-12,-22)*{\scs j};
(-6,-22)*{\scs i};
(0,-22)*{\scs i};
(6,-22)*{\scs k};
\endxy}
-{\xy 0;/r.12pc/:
   (3,-16)*{\lcrossrg{}{}}; (-6,-16)*{\slinenr{}}; (-12,-16)*{\slinen{}};
   (6,-8)*{\slinenr{}}; (-3,-8)*{\ucrossrg{}{}}; (-12,-8)*{\slinen{}};
   (6,0)*{\slinenr{}}; (0,0)*{\slinenr{}}; (-9,0)*{\ucrossbg{}{}};
   (3,8)*{\rcrossrr{}{}}; (-6,8)*{\slinen{}}; (-12,8)*{\slineng{}};
   (6,16)*{\slineur{}}; (-3,16)*{\rcrossbr{}{}}; (-12,16)*{\slineug{}};(-4.75,18.5)*[black]{\scs \bullet};
(-12,-22)*{\scs j};
(-6,-22)*{\scs i};
(0,-22)*{\scs i};
(6,-22)*{\scs k};
\endxy},
{\xy 0;/r.12pc/:
   (3,-16)*{\lcrossrg{}{}}; (-6,-16)*{\slinen{}}; (-12,-16)*{\slinenr{}};
   (6,-8)*{\slinenr{}}; (-3,-8)*{\ucrossbg{}{}}; (-12,-8)*{\slinenr{}};
   (6,0)*{\slinenr{}}; (0,0)*{\slinen{}}; (-9,0)*{\ucrossrg{}{}};
   (3,8)*{\rcrossbr{}{}}; (-6,8)*{\slinenr{}}; (-12,8)*{\slineng{}};
   (6,16)*{\slineu{}}; (-3,16)*{\rcrossrr{}{}}; (-12,16)*{\slineug{}};(-4.75,18.5)*[black]{\scs \bullet};
(-12,-22)*{\scs i};
(-6,-22)*{\scs j};
(0,-22)*{\scs i};
(6,-22)*{\scs k};
\endxy}
-{\xy 0;/r.12pc/:
   (3,-16)*{\lcrossrg{}{}}; (-6,-16)*{\slinen{}}; (-12,-16)*{\slinenr{}};
   (6,-8)*{\slinenr{}}; (-3,-8)*{\ucrossbg{}{}}; (-12,-8)*{\slinenr{}};
   (6,0)*{\slinenr{}}; (0,0)*{\slinen{}}; (-9,0)*{\ucrossrg{}{}};
   (3,8)*{\rcrossbr{}{}}; (-6,8)*{\slinenr{}}; (-12,8)*{\slineng{}};
   (6,16)*{\slineu{}}; (-3,16)*{\rcrossrr{}{}}; (-12,16)*{\slineug{}}; (-1.5,18)*[black]{\scs \bullet};
(-12,-22)*{\scs i};
(-6,-22)*{\scs j};
(0,-22)*{\scs i};
(6,-22)*{\scs k};
\endxy}\right)
=\cal{T}'_i\left({\xy 0;/r.15pc/: (0,0)*{\Rthreer{blue}{black}{green}{}{}{}};
(-6,-13)*{\scs j};
(0,-13)*{\scs i};
(6,-13)*{\scs k};
\endxy}\right)
\end{align*}
where in the middle step we use dot sliding, and equation \eqref{eq:otherR3}.

The $kjk'$-labeled case is given by:
\begin{align*}
\cal{T}'_i\left({\xy 0;/r.15pc/: (0,0)*{\Rthreel{green}{blue}{green}{}{}{}};
(-6,-13)*{\scs k};
(0,-13)*{\scs j};
(6,-13)*{\scs k'};
\endxy}\right)
&= \left(
t_{k'i}^{-1}\vcenter{\xy 0;/r.12pc/: (3,0)*{\ucrossgb{}{}};(12,0)*{\slinenr{}};(18,0)*{\slineng{}};
    (0,8)*{\slinen{}};(9,8)*{\ucrossgr{}{}};(18,8)*{\slineng{}};
    (0,16)*{\slinen{}};(6,16)*{\slinenr{}};(15,16)*{\ucrossgg{}{}};
    (0,24)*{\slinen{}};(9,24)*{\ucrossrg{}{}};(18,24)*{\slineng{}};
    (3,32)*{\ucrossbg{}{}};(12,32)*{\slineur{}};(18,32)*{\slineug{}};
(0,-6)*{\scs k};
(6,-6)*{\scs j};
(12,-6)*{\scs i};
(18,-6)*{\scs k'};
\endxy},
t_{k'i}^{-1}\vcenter{\xy 0;/r.12pc/: (3,0)*{\ucrossgr{}{}};(12,0)*{\slinen{}};(18,0)*{\slineng{}};
    (0,8)*{\slinenr{}};(9,8)*{\ucrossgb{}{}};(18,8)*{\slineng{}};
    (0,16)*{\slinenr{}};(6,16)*{\slinen{}};(15,16)*{\ucrossgg{}{}};
    (0,24)*{\slinenr{}};(9,24)*{\ucrossbg{}{}};(18,24)*{\slineng{}};
    (3,32)*{\ucrossrg{}{}};(12,32)*{\slineu{}};(18,32)*{\slineug{}};
(0,-6)*{\scs k};
(6,-6)*{\scs i};
(12,-6)*{\scs j};
(18,-6)*{\scs k'};
\endxy}\right)\\
&=\left(t_{k'i}^{-1}\vcenter{\xy 0;/r.12pc/: (-3,0)*{\ucrossrg{}{}};(-12,0)*{\slinen{}};(-18,0)*{\slineng{}};
    (0,8)*{\slinenr{}};(-9,8)*{\ucrossbg{}{}};(-18,8)*{\slineng{}};
    (0,16)*{\slinenr{}};(-6,16)*{\slinen{}};(-15,16)*{\ucrossgg{}{}};
    (0,24)*{\slinenr{}};(-9,24)*{\ucrossgb{}{}};(-18,24)*{\slineng{}};
    (-3,32)*{\ucrossgr{}{}};(-12,32)*{\slineu{}};(-18,32)*{\slineug{}};
(-18,-6)*{\scs k};
(-12,-6)*{\scs j};
(-6,-6)*{\scs i};
(0,-6)*{\scs k'};
\endxy}
+\Delta_{kjk'}t_{k'i}^{-1}\vcenter{\xy 0;/r.12pc/:
    (-6,0)*{\slineng{}}; (0,0)*{\slinen{}}; (9,0)*{\ucrossrg{}{}};
    (-6,8)*{\slineng{}}; (0,8)*{\slinen{}}; (6,8)*{\slineng{}}; (12,8)*{\slinenr{}};
    (-6,16)*{\slineng{}}; (0,16)*{\slinen{}}; (6,16)*{\slineng{}}; (12,16)*{\slinenr{}};
    (-6,24)*{\slineng{}}; (0,24)*{\slinen{}}; (6,24)*{\slineng{}}; (12,24)*{\slinenr{}};
    (-6,32)*{\slineug{}}; (0,32)*{\slineu{}}; (9,32)*{\ucrossgr{}{}};
(-6,-6)*{\scs k};
(0,-6)*{\scs j};
(6,-6)*{\scs i};
(12,-6)*{\scs k'};
\endxy},
t_{k'i}^{-1}\vcenter{\xy 0;/r.12pc/: (-3,0)*{\ucrossbg{}{}};(-12,0)*{\slinenr{}};(-18,0)*{\slineng{}};
    (0,8)*{\slinen{}};(-9,8)*{\ucrossrg{}{}};(-18,8)*{\slineng{}};
    (0,16)*{\slinen{}};(-6,16)*{\slinenr{}};(-15,16)*{\ucrossgg{}{}};
    (0,24)*{\slinen{}};(-9,24)*{\ucrossgr{}{}};(-18,24)*{\slineng{}};
    (-3,32)*{\ucrossgb{}{}};(-12,32)*{\slineur{}};(-18,32)*{\slineug{}};
(-18,-6)*{\scs k};
(-12,-6)*{\scs i};
(-6,-6)*{\scs j};
(0,-6)*{\scs k'};
\endxy}
+\Delta_{kjk'}t_{k'i}^{-1}\vcenter{\xy 0;/r.12pc/:
    (6,0)*{\slineng{}}; (0,0)*{\slinen{}}; (-9,0)*{\ucrossgr{}{}};
    (6,8)*{\slineng{}}; (0,8)*{\slinen{}}; (-6,8)*{\slineng{}}; (-12,8)*{\slinenr{}};
    (6,16)*{\slineng{}}; (0,16)*{\slinen{}}; (-6,16)*{\slineng{}}; (-12,16)*{\slinenr{}};
    (6,24)*{\slineng{}}; (0,24)*{\slinen{}}; (-6,24)*{\slineng{}}; (-12,24)*{\slinenr{}};
    (6,32)*{\slineug{}}; (0,32)*{\slineu{}}; (-9,32)*{\ucrossrg{}{}};
(-12,-6)*{\scs k};
(-6,-6)*{\scs i};
(0,-6)*{\scs j};
(6,-6)*{\scs k'};
\endxy}\right)\\
&= \cal{T}'_i\left({\xy 0;/r.15pc/: (0,0)*{\Rthreer{green}{blue}{green}{}{}{}};
(-6,-13)*{\scs k};
(0,-13)*{\scs j};
(6,-13)*{\scs k'};
\endxy}
+\Delta_{kjk'}{\xy 0;/r.18pc/:(-6,-6)*{\slineng{}}; (0,-6)*{\slinen{}}; (6,-6)*{\slineng{}};
   (-6,0)*{\slineng{}}; (0,0)*{\slinen{}}; (6,0)*{\slineng{}}; (-6,6)*{\slineug{}}; (0,6)*{\slineu{}}; (6,6)*{\slineug{}};
(-6,-11)*{\scs k};
(0,-11)*{\scs j};
(6,-11)*{\scs k'};
\endxy}\right)
\end{align*}
and all others that don't require a non-zero homotopy are given similarly to these cases.

We now consider the cases listed above that
require chain homotopies.
Considering the
$iji$-labeled case, we compute that
\[
\cal{T}'_i\left({\xy 0;/r.15pc/: (0,0)*{\Rthreel{black}{blue}{black}{}{}{}};
(-6,-13)*{\scs i};
(0,-13)*{\scs j};
(6,-13)*{\scs i};
\endxy}
-{\xy 0;/r.15pc/: (0,0)*{\Rthreer{black}{blue}{black}{}{}{}};
(-6,-13)*{\scs i};
(0,-13)*{\scs j};
(6,-13)*{\scs i};
\endxy}
-t_{ij}{\xy 0;/r.18pc/:(-6,-6)*{\slinenr{i}}; (0,-6)*{\slinen{j}}; (6,-6)*{\slinenr{i}};
   (-6,0)*{\slinenr{}}; (0,0)*{\slinen{}}; (6,0)*{\slinenr{}}; (-6,6)*{\slineur{}}; (0,6)*{\slineu{}}; (6,6)*{\slineur{}}; \endxy}\right)
= (\phi_1,\phi_2)
\]
where
\begin{align*}
\phi_1=&
-t_{ij}{\xy 0;/r.14pc/:
   (-3,-16)*{\lcrossrb{}{}}; (6,-16)*{\slinenr{}{}}; (12,-16)*{\slinedr{}};
   (-6,-8)*{\slinen{}}; (3,-8)*{\lcrossrr{}{}}; (12,-8)*{\slinenr{}};
   (-6,0)*{\slinen{}}; (0,0)*{\slinenr{}}; (9,0)*{\dcrossrr{}{}};
   (-6,8)*{\slinen{}}; (3,8)*{\rcrossrr{}{}}; (12,8)*{\slinenr{}};
   (-3,16)*{\rcrossbr{}{}}; (6,16)*{\sdotur{}}; (12,16)*{\slinenr{}};
(-6,-22)*{\scs i};
(0,-22)*{\scs j};
(6,-22)*{\scs i};
(12,-22)*{\scs i};
\endxy}
+t_{ij}{\xy 0;/r.14pc/:
   (-3,-16)*{\lcrossrb{}{}}; (6,-16)*{\slinenr{}{}}; (12,-16)*{\slinedr{}};
   (-6,-8)*{\slinen{}}; (3,-8)*{\lcrossrr{}{}}; (12,-8)*{\slinenr{}};
   (-6,0)*{\slinen{}}; (0,0)*{\slinenr{}}; (9,0)*{\dcrossrr{}{}};
   (-6,8)*{\slinen{}}; (3,8)*{\rcrossrr{}{}}; (12,8)*{\slinenr{}}; (-5,19)*[black]{\scs \bullet};
   (-3,16)*{\rcrossbr{}{}}; (6,16)*{\slineur{}}; (12,16)*{\slinenr{}};
(-6,-22)*{\scs i};
(0,-22)*{\scs j};
(6,-22)*{\scs i};
(12,-22)*{\scs i};
\endxy}
+t_{ij}{\xy 0;/r.14pc/:
   (3,-16)*{\rcrossrr{}{}}; (-6,-16)*{\slinen{}{}}; (-12,-16)*{\slinedr{}};
   (6,-8)*{\sdotr{}}; (-3,-8)*{\rcrossbr{}{}}; (-12,-8)*{\slinenr{}};
   (6,0)*{\slinenr{}}; (0,0)*{\slinen{}}; (-9,0)*{\dcrossrr{}{}};
   (6,8)*{\slinenr{}}; (-3,8)*{\lcrossrb{}{}}; (-12,8)*{\slinenr{}};
   (3,16)*{\lcrossrr{}{}}; (-6,16)*{\slineu{}}; (-12,16)*{\slinenr{}};
(-12,-22)*{\scs i};
(-6,-22)*{\scs j};
(0,-22)*{\scs i};
(6,-22)*{\scs i};
\endxy}
-t_{ij}{\xy 0;/r.14pc/:
   (3,-16)*{\rcrossrr{}{}}; (-6,-16)*{\slinen{}{}}; (-12,-16)*{\slinedr{}};
   (6,-8)*{\slinenr{}}; (-3,-8)*{\rcrossbr{}{}}; (-12,-8)*{\slinenr{}};
   (6,0)*{\slinenr{}}; (0,0)*{\slinen{}}; (-9,0)*{\dcrossrr{}{}};
   (6,8)*{\slinenr{}}; (-3,8)*{\lcrossrb{}{}}; (-12,8)*{\slinenr{}}; (-4.75,-5.5)*[black]{\scs \bullet};
   (3,16)*{\lcrossrr{}{}}; (-6,16)*{\slineu{}}; (-12,16)*{\slinenr{}};
(-12,-22)*{\scs i};
(-6,-22)*{\scs j};
(0,-22)*{\scs i};
(6,-22)*{\scs i};
\endxy}
-t_{ij}{\xy 0;/r.14pc/:
   (6,-16)*{\slinedr{}}; (0,-16)*{\slinenr{}}; (-6,-16)*{\slinen{}{}}; (-12,-16)*{\slinedr{}};
   (6,-8)*{\slinenr{}}; (0,-8)*{\slinenr{}}; (-6,-8)*{\slinen{}}; (-12,-8)*{\slinenr{}};
   (6,0)*{\slinenr{}}; (0,0)*{\slinenr{}}; (-6,0)*{\slinen{}}; (-12,0)*{\slinenr{}};
   (6,8)*{\slinenr{}}; (0,8)*{\slinenr{}}; (-6,8)*{\slinen{}}; (-12,8)*{\slinenr{}};
   (6,16)*{\slinenr{}}; (0,16)*{\slineur{}}; (-6,16)*{\slineu{}}; (-12,16)*{\slinenr{}};
(-12,-22)*{\scs i};
(-6,-22)*{\scs j};
(0,-22)*{\scs i};
(6,-22)*{\scs i};
\endxy} \\
&= t_{ij}\hspace{-0.4pc}\displaystyle\sum_{\substack{a+b+c+d\\ =\la i,s_i(\lambda) \ra-2}}{\xy 0;/r.13pc/:
   (-3,-16)*{\lcrossrb{}{}}; (9,-16)*{\rcapr{}}; (9,-14)*[black]{\scs \bullet}; (9,-11)*{\scs b};
   (-6,-8)*{\slinen{}}; (0,-8)*{\slinenr{}}; (2,0)*{\scs d};
   (-6,0)*{\slinen{}}; (0,0)*{\sdotr{}}; (12,0)*{\ccbubr{\scs \spadesuit +c}{}};
   (-6,8)*{\slinen{}}; (0,8)*{\slinenr{}};
   (-3,16)*{\rcrossbr{}{}}; (9,17.5)*{\lcupr{}};  (9,16.25)*[black]{\scs \bullet}; (9,13)*{\scs a+1};
(-6,-22)*{\scs i};
(0,-22)*{\scs j};
(5,-22)*{\scs i};
(13,8)*{\scs i};
(14,18)*{\scs i};
\endxy}
+t_{ij}{\xy 0;/r.13pc/:
   (3,-16)*{\rcrossrr{}{}}; (-9,-16)*{\lcrossrb{}{}};
   (6,-8)*{\slinenr{}}; (-12,-8)*{\slinen{}}; (-12,0)*{\slinen{}}; (-6,-8)*{\slinenr{}}; (0,-8)*{\slinenr{}};
   (6,0)*{\slinenr{}}; (-3,0)*{\dcrossrr{}{}}; (3,18)*{\xy 0;/r.11pc/: (0,0)*{\lcupr{}{}}; \endxy};
   (-12,8)*{\slinen{}}; 
    (6,5)*{};(0,5)*{} **\crv{(6,10) & (0,10)}?(0)*\dir{};
   (-6,8)*{\slinenr{}};
   (-9,16)*{\rcrossbr{}{}};
(-12,-22)*{\scs i};
(-6,-22)*{\scs j};
(6,-22)*{\scs i};
(7,18)*{\scs i};
\endxy}
-t_{ij}\hspace{-0.4pc}\displaystyle\sum_{\substack{a+b+c+d\\ =\la i,s_i(\lambda) \ra-2}}{\xy 0;/r.13pc/:
   (-3,-16)*{\lcrossrb{}{}}; (9,-16)*{\rcapr{}}; (9,-14)*[black]{\scs \bullet}; (9,-11)*{\scs b};
   (-6,-8)*{\slinen{}}; (0,-8)*{\slinenr{}}; (6,0)*{\scs d+1};
   (-6,0)*{\slinen{}}; (0,0)*{\sdotr{}}; (16,0)*{\ccbubr{\scs \spadesuit +c}{}};
   (-6,8)*{\slinen{}}; (0,8)*{\slinenr{}};
   (-3,16)*{\rcrossbr{}{}}; (9,17.5)*{\lcupr{}};  (9,16.25)*[black]{\scs \bullet};
   (9,13)*{\scs a};
(-6,-22)*{\scs i};
(0,-22)*{\scs j};
(5,-22)*{\scs i};
(13,8)*{\scs i};
(14,18)*{\scs i};
\endxy}
-t_{ij}\hspace{-0.4pc}\displaystyle\sum_{\substack{a+b+c\\ =\la i,s_i(\lambda) \ra-1}}{\xy 0;/r.13pc/:
   (-6,-16)*{\slinedr{}}; (0,-16)*{\slinen{}}; (9,-16)*{\rcapr{}}; (9,-14)*[black]{\scs \bullet}; (9,-11)*{\scs b};
   (-6,-8)*{\slinenr{}}; (0,-8)*{\slinen{}};
   (-6,0)*{\slinenr{}}; (0,0)*{\slinen{}}; (9,0)*{\ccbubr{\scs \spadesuit +c}{}};
   (-6,8)*{\slinenr{}}; (0,8)*{\slinen{}};
   (-6,16)*{\slinenr{}}; (0,16)*{\slineu{}}; (9,17.5)*{\lcupr{}}; (9,16.25)*[black]{\scs \bullet}; (9,13)*{\scs a};
(-6,-22)*{\scs i};
(0,-22)*{\scs j};
(5,-22)*{\scs i};
(13,8)*{\scs i};
(14,18)*{\scs i};
\endxy} \\
&= -t_{ij}\displaystyle\sum_{\substack{b+c+d\\ =\la i,s_i(\lambda) \ra-1}}{\xy 0;/r.14pc/:
   (-3,-16)*{\lcrossrb{}{}}; (9,-16)*{\rcapr{}}; (9,-14)*[black]{\scs \bullet}; (9,-11)*{\scs b};
   (-6,-8)*{\slinen{}}; (0,-8)*{\slinenr{}}; (2,0)*{\scs d};
   (-6,0)*{\slinen{}}; (0,0)*{\sdotr{}}; (12,0)*{\ccbubr{\scs \spadesuit +c}{}};
   (-6,8)*{\slinen{}}; (0,8)*{\slinenr{}};
   (-3,16)*{\rcrossbr{}{}}; (9,17.5)*{\lcupr{}};
(-6,-22)*{\scs i};
(0,-22)*{\scs j};
(5,-22)*{\scs i};
(13,8)*{\scs i};
(14,18)*{\scs i};
\endxy}
+t_{ij}{\xy 0;/r.14pc/:
   (3,-16)*{\rcrossrr{}{}}; (-9,-16)*{\lcrossrb{}{}};
   (3,-8)*{\lcrossrr{}{}}; (-6,-8)*{\slinenr{}}; (-12,-8)*{\slinen{}};
   (-12,0)*{\slinen{}};
   (6,-3)*{}; (-6,13)*{} **\crv{(6,3) & (-6,7)}?(0)*\dir{};
   (-12,8)*{\slinen{}};
   (0,-3)*{};(-6,-3)*{} **\crv{(0,1) & (-6,1)}?(0)*\dir{};
   (3,18)*{\xy 0;/r.12pc/: (0,0)*{\lcupr{}{}};\endxy}; (-9,16)*{\rcrossbr{}{}};
(-12,-22)*{\scs i};
(-6,-22)*{\scs j};
(0,-22)*{\scs i};
(6,-22)*{\scs i};
(8,18)*{\scs i};
\endxy}
=-t_{ij}\vcenter{\xy 0;/r.17pc/:
   (6,-4)*{\slinedr{}}; (-3,-4)*{\ucrossbr{}{}}; (-12,-4)*{\slinedr{}};
   (-9,20)*{\rcrossbr{}{}};
   (0,1)*{};(-12,17)*{} **[blue]\crv{(0,4) & (-12,14)}?(0)*\dir{};
   (6,1)*{};(-6,17)*{} **[black]\crv{(6,4) & (-6,14)}?(0)*\dir{};
   (-6,1)*{};(-12,1)*{} **\crv{(-6,6) & (-12,6)}?(0)*\dir{};
   (3,22)*{\xy 0;/r.13pc/: (0,0)*{\lcupr{}{}};\endxy};
(-12,-9)*{\scs i};
(-6,-9)*{\scs j};
(0,-9)*{\scs i};
(6,-9)*{\scs i};
(8,20)*{\scs i};
\endxy}
\end{align*}
and
\begin{align*}
\phi_2 &=
-t_{ij}{\xy 0;/r.15pc/:
   (-3,-16)*{\lcrossrr{}{}}; (6,-16)*{\slinen{}{}}; (12,-16)*{\slinedr{}};
   (-6,-8)*{\slinenr{}}; (3,-8)*{\lcrossrb{}{}}; (12,-8)*{\slinenr{}};
   (-6,0)*{\slinenr{}}; (0,0)*{\slinen{}}; (9,0)*{\dcrossrr{}{}};
   (-6,8)*{\slinenr{}}; (3,8)*{\rcrossbr{}{}}; (12,8)*{\slinenr{}}; (-1.5,18)*[black]{\scs \bullet};
   (-3,16)*{\rcrossrr{}{}}; (6,16)*{\slineu{}}; (12,16)*{\slinenr{}};
(-6,-22)*{\scs i};
(0,-22)*{\scs i};
(6,-22)*{\scs j};
(12,-22)*{\scs i};
\endxy}
+t_{ij}{\xy 0;/r.15pc/:
   (-3,-16)*{\lcrossrr{}{}}; (6,-16)*{\slinen{}{}}; (12,-16)*{\slinedr{}};
   (-6,-8)*{\slinenr{}}; (3,-8)*{\lcrossrb{}{}}; (12,-8)*{\slinenr{}};
   (-6,0)*{\slinenr{}}; (0,0)*{\slinen{}}; (9,0)*{\dcrossrr{}{}};
   (-6,8)*{\slinenr{}}; (3,8)*{\rcrossbr{}{}}; (12,8)*{\slinenr{}}; (-5,18.75)*[black]{\scs \bullet};
   (-3,16)*{\rcrossrr{}{}}; (6,16)*{\slineu{}}; (12,16)*{\slinenr{}};
(-6,-22)*{\scs i};
(0,-22)*{\scs i};
(6,-22)*{\scs j};
(12,-22)*{\scs i};
\endxy}
+t_{ij}{\xy 0;/r.15pc/:
   (3,-16)*{\rcrossbr{}{}}; (-6,-16)*{\slinenr{}{}}; (-12,-16)*{\slinedr{}};
   (6,-8)*{\slinen{}}; (-3,-8)*{\rcrossrr{}{}}; (-12,-8)*{\slinenr{}};
   (6,0)*{\slinen{}}; (0,0)*{\slinenr{}}; (-9,0)*{\dcrossrr{}{}};
   (6,8)*{\slinen{}}; (-3,8)*{\lcrossrr{}{}}; (-12,8)*{\slinenr{}}; (-1.5,-6)*[black]{\scs \bullet};
   (3,16)*{\lcrossrb{}{}}; (-6,16)*{\slineur{}}; (-12,16)*{\slinenr{}};
(-12,-22)*{\scs i};
(-6,-22)*{\scs i};
(0,-22)*{\scs j};
(6,-22)*{\scs i};
\endxy}
-t_{ij}{\xy 0;/r.15pc/:
   (3,-16)*{\rcrossbr{}{}}; (-6,-16)*{\slinenr{}{}}; (-12,-16)*{\slinedr{}};
   (6,-8)*{\slinen{}}; (-3,-8)*{\rcrossrr{}{}}; (-12,-8)*{\slinenr{}};
   (6,0)*{\slinen{}}; (0,0)*{\slinenr{}}; (-9,0)*{\dcrossrr{}{}};
   (6,8)*{\slinen{}}; (-3,8)*{\lcrossrr{}{}}; (-12,8)*{\slinenr{}}; (-5,-5.25)*[black]{\scs \bullet};
   (3,16)*{\lcrossrb{}{}}; (-6,16)*{\slineur{}}; (-12,16)*{\slinenr{}};
(-12,-22)*{\scs i};
(-6,-22)*{\scs i};
(0,-22)*{\scs j};
(6,-22)*{\scs i};
\endxy}
-t_{ij}{\xy 0;/r.15pc/:
   (6,-16)*{\slinedr{}}; (-6,-16)*{\slinenr{}}; (0,-16)*{\slinen{}{}}; (-12,-16)*{\slinedr{}};
   (6,-8)*{\slinenr{}}; (-6,-8)*{\slinenr{}}; (0,-8)*{\slinen{}}; (-12,-8)*{\slinenr{}};
   (6,0)*{\slinenr{}}; (-6,0)*{\slinenr{}}; (0,0)*{\slinen{}}; (-12,0)*{\slinenr{}};
   (6,8)*{\slinenr{}}; (-6,8)*{\slinenr{}}; (0,8)*{\slinen{}}; (-12,8)*{\slinenr{}};
   (6,16)*{\slinenr{}}; (-6,16)*{\slineur{}}; (0,16)*{\slineu{}}; (-12,16)*{\slinenr{}};
(-12,-22)*{\scs i};
(-6,-22)*{\scs i};
(0,-22)*{\scs j};
(6,-22)*{\scs i};
\endxy} \\
&=
-t_{ij}\displaystyle\sum_{\substack{a+b+c+d\\ =\la i, s_i(\lambda)+\alpha_j \ra-1}}
{\xy 0;/r.15pc/:
  (3.5,-15)*{\xy 0;/r.15pc/:(0,0)*{\rcrossbr{}{}}; \endxy};
  (-7,-15)*{\xy 0;/r.15pc/:(0,0)*{\slinenr{}}; \endxy};
  (0.5,-10)*{};(-7,-10)*{} **\crv{(0.5,-5) & (-7,-5)}?(0)*\dir{};
  (-3,-6.5)*[black]{\scs \bullet}; (-3,-9)*{\scs b};
  (-13,-8)*{\slinenr{}}; (-13,-16)*{\slinedr{}};
   (6.5,-10)*{}; (6.5,12)*{} **[blue]\dir{-};
   (-13,0)*{\sdotr{}}; (-17,0)*{\scs d};
  (0.5,12)*{};(-7,12)*{} **\crv{(0.5,7) & (-7,7)}?(0)*\dir{};
  (-13,8)*{\slinenr{}};
  (3.5,15)*{\xy 0;/r.15pc/:(0,0)*{\lcrossrb{}{}}; \endxy};
  (-7,15)*{\xy 0;/r.15pc/:(0,0)*{\slineur{}}; \endxy};
 (-7,0)*{\smccbubr{}{}};  (.5,2)*{\scs \spadesuit+c}; (-13,16)*{\slinenr{}};
(-12,-22)*{\scs i};
(-7,-22)*{\scs i};
(0,-22)*{\scs j};
(6,-22)*{\scs i};
\endxy}
+t_{ij}{\xy 0;/r.15pc/:
   (3,-16)*{\rcrossbr{}{}}; (-6,-16)*{\slinenr{}{}}; (-12,-16)*{\slinedr{}};
   (6,-8)*{\slinen{}}; (-3,-8)*{\rcrossrr{}{}}; (-12,-8)*{\slinenr{}};
   (6,0)*{\slinen{}}; (-3,0)*{\lcrossrr{}{}}; (-12,0)*{\slinenr{}};
   (0,5)*{};(-12,13)*{} **\crv{(0,9) & (-12,9)}?(0)*\dir{};
   (6,8)*{\slinen{}};
   (0,13)*{};(-6,13)*{} **\crv{(0,8) & (-6,8)}?(0)*\dir{};
   (3,16)*{\lcrossrb{}{}}; (-6,16)*{\slineur{}}; (-12,16)*{\slinenr{}};
   (-6,5)*{};(-12,5)*{} **\crv{(-6,10) & (-12,10)}?(0)*\dir{};
(-12,-22)*{\scs i};
(-6,-22)*{\scs i};
(0,-22)*{\scs j};
(6,-22)*{\scs i};
\endxy}
=-t_{ij}{\xy 0;/r.15pc/:
    (6,0)*{\slinen{}};
   (0,5)*{};(-6,5)*{} **\crv{(0,1) & (-6,1)}?(0)*\dir{};
   (3,8)*{\lcrossrb{}{}}; (-6,8)*{\slineur{}}; (-12,8)*{\slinenr{}};
   (0,-3)*{}; (-12,5)*{} **\crv{(0,1) & (-12,1)}?(0)*\dir{};
   (-6,-11)*{};(-12,-11)*{} **\crv{(-6,-6) & (-12,-6)}?(1)*\dir{>};
   (3,-8)*{\rcrossbr{}{}};
(-12,-14)*{\scs i};
(-6,-14)*{\scs i};
(0,-14)*{\scs j};
(6,-14)*{\scs i};
\endxy}
\end{align*}
In both computations, we make extensive use of equation \eqref{eq:otherR3}.
It follows that this chain map is null-homotopic, with homotopy given by:
\[
{\xy 0;/r.15pc/:
  (-50,17.5)*+{\cal{F}_i\cal{E}_j\cal{E}_i\cal{F}_i\onell{s_i(\lambda)}\la -2\l_i-5 \ra}="1";
  (-50,-17.5)*+{\cal{F}_i\cal{E}_j\cal{E}_i\cal{F}_i\onell{s_i(\lambda)}\la -2\l_i-5 \ra}="2";
  (50,17.5)*+{\cal{F}_i\cal{E}_i\cal{E}_j\cal{F}_i\onell{s_i(\lambda)}\la -2\l_i-4 \ra}="3";
  (50,-17.5)*+{\cal{F}_i\cal{E}_i\cal{E}_j\cal{F}_i\onell{s_i(\lambda)}\la -2\l_i-4\ra}="4";
 {\ar@/^1pc/ "4";"1"};
   (10,0)*{ t_{ij} \;\;
   \vcenter{\xy 0;/r.15pc/:
   (6,5)*{};(0,5)*{} **\crv{(6,0) & (0,0)}?(0)*\dir{<};
   (-6,-10)*{};(-12,-10)*{} **\crv{(-6,-5) & (-12,-5)}?(1)*\dir{>};
   (0,-10)*{};(-6,-5)*{} **[blue]\crv{(0,-7.5) & (-6,-7.5)}?(1)*[blue]\dir{};
   (0,0)*{};(-6,-5)*{} **[blue]\crv{(0,-2.5) & (-6,-2.5)}?(1)*[blue]\dir{};
   (0,0)*{};(-6,5)*{} **[blue]\crv{(0,2.5) & (-6,0)}?(1)*[blue]\dir{>};
   (6,-10)*{};(-12,5)*{} **\crv{(6,-2.5) & (-12,-2.5)}?(0)*\dir{<};
\endxy}};
   {\ar_-{-\xy  0;/r.15pc/:(-6,0)*{\slinedr{}};(3,0)*{\ucrossbr{}{}}; (12,0)*{\slinedr{}}; (15,0)*{}; \endxy   } "2";"4"};
   {\ar^-{-\xy 0;/r.15pc/: (-6,0)*{\slinedr{}};(3,0)*{\ucrossbr{}{}}; (12,0)*{\slinedr{}}; (15,0)*{};\endxy   } "1";"3"};
 \endxy}
\]

For the remaining cases, we
provide the explicit homotopy between the relevant maps. We have
\[
\cal{T}'_i\left({\xy 0;/r.15pc/: (0,0)*{\Rthreel{blue}{green}{magenta}{}{}{}};
(-6,-13)*{\scs j};
(0,-13)*{\scs k};
(6,-13)*{\scs j'};
\endxy}\right)
\sim
\cal{T}'_i\left({\xy 0;/r.15pc/: (0,0)*{\Rthreer{blue}{green}{magenta}{}{}{}};
(-6,-13)*{\scs j};
(0,-13)*{\scs k};
(6,-13)*{\scs j'};
\endxy}
+\Delta_{jkj'}{\xy 0;/r.18pc/:(-6,-6)*{\slinen{}}; (0,-6)*{\slineng{}}; (6,-6)*{\slinenp{}};
   (-6,0)*{\slinen{}}; (0,0)*{\slineng{}}; (6,0)*{\slinenp{}}; (-6,6)*{\slineu{}}; (0,6)*{\slineug{}}; (6,6)*{\slineup{}};
(-6,-13)*{\scs j};
(0,-13)*{\scs k};
(6,-13)*{\scs j'};
\endxy}\right)
\]
via the homotopy
\[
{\xy 0;/r.15pc/:
	(-100,15)*+{\clubsuit \cal{E}_{j'} \cal{E}_i \cal{E}_k \cal{E}_j \cal{E}_i \onell{s_i(\lambda)}}="1";
	(0,15)*+{\cal{E}_{i} \cal{E}_{j'} \cal{E}_k \cal{E}_j \cal{E}_i \onell{s_i(\lambda)} \la 1 \ra
  		\oplus \cal{E}_{j'} \cal{E}_{i} \cal{E}_k \cal{E}_i \cal{E}_j \onell{s_i(\lambda)} \la 1 \ra}="2";
	(100,15)*+{\cal{E}_{i} \cal{E}_{j'} \cal{E}_k \cal{E}_i \cal{E}_j \onell{s_i(\lambda)} \la 2 \ra}="3";
	(-100,-15)*+{\clubsuit \cal{E}_{j} \cal{E}_i \cal{E}_k \cal{E}_{j'} \cal{E}_i \onell{s_i(\lambda)}}="4";
	(0,-15)*+{\cal{E}_{i} \cal{E}_{j} \cal{E}_k \cal{E}_{j'} \cal{E}_i \onell{s_i(\lambda)} \la 1 \ra
  		\oplus \cal{E}_{j} \cal{E}_{i} \cal{E}_k \cal{E}_i \cal{E}_{j'} \onell{s_i(\lambda)} \la 1 \ra}="5";
	(100,-15)*+{\cal{E}_{i} \cal{E}_{j} \cal{E}_k \cal{E}_i \cal{E}_{j'} \onell{s_i(\lambda)} \la 2 \ra}="6";
	{\ar  "1";"2"};
	{\ar  "2";"3"};
	{\ar  "4";"5"};
	{\ar  "5";"6"};
	{\ar@/^1pc/  "5";"1"};
	{\ar@/^1pc/  "6";"2"};
	(-50,0)*{\left(\begin{smallmatrix} 0 & h^1 \end{smallmatrix}\right)};
	(50,0)*{\left(\begin{smallmatrix} 0 \\ h^2 \end{smallmatrix}\right)};
\endxy}
\]
with
\[
h^1 =
\Delta_{jkj'}t_{ki}^{-1}t_{ij}^{-1}\vcenter{\xy 0;/r.12pc/:
   (0,0)*{\slinen{}};(9,0)*{\ucrossrg{}{}};(18,0)*{\slinenr{}};(24,0)*{\slinenp{}};
   (0,8)*{\slinen{}};(6,8)*{\slineng{}};(15,8)*{\ucrossrr{}{}};(24,8)*{\slinenp{}};
   (0,16)*{\slinen{}};(9,16)*{\ucrossgr{}{}};(18,16)*{\slinenr{}};(24,16)*{\slinenp{}};
   (0,24)*{\slineu{}};(6,24)*{\slineur{}};(12,24)*{\slineug{}};(21,24)*{\ucrossrp{}{}};
(0,-6)*{\scs j}; (6,-6)*{\scs i}; (12,-6)*{\scs k}; (18,-6)*{\scs i};(24,-6)*{\scs j'};
\endxy}
\;\; , \;\;
h^2 =
-\Delta_{jkj'}t_{ki}^{-1}t_{ij}^{-1}\vcenter{\xy 0;/r.12pc/:
   (3,0)*{\ucrossrb{}{}};(12,0)*{\slineng{}};(18,0)*{\slinenr{}};(24,0)*{\slinenp{}};
   (0,8)*{\slinen{}};(9,8)*{\ucrossrg{}{}};(18,8)*{\slinenr{}};(24,8)*{\slinenp{}};
   (0,16)*{\slinen{}};(6,16)*{\slineng{}};(15,16)*{\ucrossrr{}{}};(24,16)*{\slinenp{}};
   (0,24)*{\slineu{}};(9,24)*{\ucrossgr{}{}};(18,24)*{\slineur{}};(24,24)*{\slineup{}};
(0,-6)*{\scs i}; (6,-6)*{\scs j}; (12,-6)*{\scs k}; (18,-6)*{\scs i};(24,-6)*{\scs j'};
\endxy}
\]

For the $jij'$-labeled case, we have
\[
\cal{T}'_i\left({\xy 0;/r.15pc/: (0,0)*{\Rthreel{blue}{black}{magenta}{}{}{}};
(-6,-13)*{\scs j};
(0,-13)*{\scs i};
(6,-13)*{\scs j'};
\endxy}\right)
\sim
\cal{T}'_i\left({\xy 0;/r.15pc/: (0,0)*{\Rthreer{blue}{black}{magenta}{}{}{}};
(-6,-13)*{\scs j};
(0,-13)*{\scs i};
(6,-13)*{\scs j'};
\endxy}
+\Delta_{jij'}{\xy 0;/r.18pc/:(-6,-6)*{\slinen{}}; (0,-6)*{\slinenr{}}; (6,-6)*{\slinenp{}};
   (-6,0)*{\slinen{}}; (0,0)*{\slinenr{}}; (6,0)*{\slinenp{}}; (-6,6)*{\slineu{}}; (0,6)*{\slineur{}}; (6,6)*{\slineup{}};
(-6,-13)*{\scs j};
(0,-13)*{\scs i};
(6,-13)*{\scs j'};
\endxy}\right)
\]
with chain homotopy given by (here we indicate the signs on the differential since they are not the usual ones,
given the homological shift on $\cal{T}_i'(\cal{E}_i \onel)$):
\[
{\xy 0;/r.15pc/:
	(-90,21)*+{\cal{E}_{j'} \cal{E}_i \cal{F}_i \cal{E}_j \cal{E}_i \onell{s_i(\lambda)} \la -1-\l_i \ra }="1";
	(0,27)*+{\cal{E}_{i} \cal{E}_{j'} \cal{F}_i \cal{E}_j \cal{E}_i \onell{s_i(\lambda)} \la -\l_i \ra};
  	(0,21)*+{\oplus};
	(0,15)*+{\cal{E}_{j'} \cal{E}_{i} \cal{F}_i \cal{E}_i \cal{E}_j \onell{s_i(\lambda)} \la -\l_i \ra};
	(90,21)*+{\cal{E}_{i} \cal{E}_{j'} \cal{F}_i \cal{E}_i \cal{E}_j \onell{s_i(\lambda)} \la 1-\l_i \ra}="3";
	(-90,-21)*+{\cal{E}_{j} \cal{E}_i \cal{F}_i \cal{E}_{j'} \cal{E}_i \onell{s_i(\lambda)} \la -1-\l_i \ra}="4";
	(0,-15)*+{\cal{E}_{i} \cal{E}_{j} \cal{F}_i \cal{E}_{j'} \cal{E}_i \onell{s_i(\lambda)} \la -\l_i \ra};
  	(0,-21)*+{\oplus};
  	(0,-27)*+{\cal{E}_{j} \cal{E}_{i} \cal{F}_i \cal{E}_i \cal{E}_{j'} \onell{s_i(\lambda)} \la -\l_i \ra};
	(90,-21)*+{\cal{E}_{i} \cal{E}_{j} \cal{F}_i \cal{E}_i \cal{E}_{j'} \onell{s_i(\lambda)} \la 1-\l_i \ra}="6";
	{\ar^-{\left(\begin{smallmatrix} + \\ -  \end{smallmatrix}\right)}  "1"; (-30,21)*{}};
	{\ar^-{\left(\begin{smallmatrix} + & +  \end{smallmatrix}\right)}  (30,21)*{};"3"};
	{\ar^-{\left(\begin{smallmatrix} + \\ -  \end{smallmatrix}\right)}  "4"; (-30,-21)*{}};
	{\ar^-{\left(\begin{smallmatrix} + & +  \end{smallmatrix}\right)}  (30,-21)*{};"6"};
	{\ar@/_1pc/  (0,-10)*{};"1"};
	{\ar@/^1pc/  "6";(0,10)*{}};
	(-40,0)*{\left(\begin{smallmatrix} h^1_1 & h^1_2 \end{smallmatrix}\right)};
	(40,0)*{\left(\begin{smallmatrix} h^2_1 \\ h^2_2 \end{smallmatrix}\right)};
\endxy}
\]
for
\[
h_1^1 =
- \vcenter{\xy 0;/r.12pc/:
   (0,0)*{\slinenr{}};(6,0)*{\slinen{}};(15,0)*{\lcrossrp{}{}};(24,0)*{\slinenr{}};
   (0,8)*{\slinenr{}};(9,8)*{\ucrossbp{}{}};(18,8)*{\slinenr{}};(24,8)*{\slinenr{}};
   (0,16)*{\slinenr{}};(6,16)*{\slinenp{}};(15,16)*{\rcrossbr{}{}};(24,16)*{\slinenr{}};
   (3,24)*{\ucrossrp{}{}};(12,24)*{\slinenr{}};(18,24)*{\slineu{}};(24,24)*{\slineur{}};
(0,-6)*{\scs i}; (6,-6)*{\scs j}; (12,-6)*{\scs i}; (18,-6)*{\scs j'};(24,-6)*{\scs i};
\endxy}
\;\; , \;\;
h_2^1 =
- \vcenter{\xy 0;/r.15pc/:
	(-6,10)*{};(0,10)*{} **\crv{(-6,5) & (0,5)}?(0)*\dir{<};
	(6,-10)*{};(0,-10)*{} **\crv{(6,-5) & (0,-5)}?(1)*\dir{>};
	(-6,-10)*{};(12,10)*{} **\crv{(-6,0) & (12,0)}?(1)*\dir{>};
	(-12,-10)*{};(6,10)*{} **[blue]\crv{(-12,0) & (6,0)}?(1)*[blue]\dir{>};
	(12,-10)*{};(-12,10)*{} **[magenta]\crv{(12,0) & (-12,0)}?(1)*[magenta]\dir{>};
	(-12,-12)*{\scs j}; (-6,-12)*{\scs i}; (0,12)*{\scs i}; (6,-12)*{\scs i};(12,-12)*{\scs j'};
\endxy}
\;\; , \;\;
h_1^2 =
\vcenter{\xy 0;/r.12pc/:
   (-3,-8)*{\ucrossrp{}{}};(-12,-8)*{\slinedr{}};(-18,-8)*{\slinen{}};(-24,-8)*{\slinenr{}};
   (0,0)*{\slinenr{}};(-6,0)*{\slinenp{}};(-15,0)*{\rcrossbr{}{}};(-24,0)*{\slinenr{}};
   (0,8)*{\slinenr{}};(-9,8)*{\ucrossbp{}{}};(-18,8)*{\slinenr{}};(-24,8)*{\slinenr{}};
   (0,16)*{\slineur{}};(-6,16)*{\slineu{}};(-15,16)*{\lcrossrp{}{}};(-24,16)*{\slineur{}};
(-24,-14)*{\scs i}; (-18,-14)*{\scs j}; (-12,-14)*{\scs i}; (-6,-14)*{\scs i};(0,-14)*{\scs j'};
   \endxy}
\;\; , \;\;
h_2^2 =
- \vcenter{\xy 0;/r.15pc/:
	(-6,10)*{};(0,10)*{} **\crv{(-6,5) & (0,5)}?(0)*\dir{<};
	(6,-10)*{};(0,-10)*{} **\crv{(6,-5) & (0,-5)}?(1)*\dir{>};
	(-6,-10)*{};(12,10)*{} **[blue]\crv{(-6,0) & (12,0)}?(1)*[blue]\dir{>};
	(-12,-10)*{};(6,10)*{} **\crv{(-12,0) & (6,0)}?(1)*\dir{>};
	(12,-10)*{};(-12,10)*{} **[magenta]\crv{(12,0) & (-12,0)}?(1)*[magenta]\dir{>};
	(-12,-12)*{\scs i}; (-6,-12)*{\scs j}; (0,12)*{\scs i}; (6,-12)*{\scs i};(12,-12)*{\scs j'};
\endxy}
\]

For the final case with $jj'j$-labeled strands, we have
\[
\cal{T}'_i\left({\xy 0;/r.15pc/: (0,0)*{\Rthreel{blue}{magenta}{blue}{}{}{}};
(-6,-13)*{\scs j};
(0,-13)*{\scs j'};
(6,-13)*{\scs j};
\endxy}\right)
\sim
\cal{T}'_i\left({\xy 0;/r.15pc/: (0,0)*{\Rthreer{blue}{magenta}{blue}{}{}{}};
(-6,-13)*{\scs j};
(0,-13)*{\scs j'};
(6,-13)*{\scs j};
\endxy}
+\Delta_{jj'j}{\xy 0;/r.18pc/:(-6,-6)*{\slinen{}}; (0,-6)*{\slinenp{}}; (6,-6)*{\slinen{}};
(-6,0)*{\slinen{}}; (0,0)*{\slinenp{}}; (6,0)*{\slinen{}}; (-6,6)*{\slineu{}}; (0,6)*{\slineup{}}; (6,6)*{\slineu{}};
(-6,-13)*{\scs j};
(0,-13)*{\scs j'};
(6,-13)*{\scs j};
\endxy}\right)
\]
The relevant homotopy maps between chain complexes given at the
triple composition of two term chain complexes,
and is non-zero only when $j \neq j'$.
We'll exhibit the homotopy assuming this, and that $j\cdot j' =0$,
as the homotopy is more involved in the $j\cdot j'=-1$ case.
The latter is only possible when the graph $\Gamma$ corresponding to our Cartan datum has
a length-three cycle (which in finite- or affine-type only occurs for $\widehat{\mathfrak{sl}_3}$).

The relevant homotopy is given as follows, where here,
in the interest of space, we employ the notation from \cite{KL3} in denoting
$\cal{E}_{\ell_1 \cdots \ell_k}:= \cal{E}_{\ell_1} \cdots \cal{E}_{\ell_k}$.
We also indicate the signs of the non-zero terms in the differentials,
which are all given up to sign by the relevant $ji$- (or $j'i$-) crossing.
\[
{\xy 0;/r.15pc/:
	(-90,26)*+{\clubsuit \cal{E}_{jij'iji}}="1";
	(-30,37)*+{\cal{E}_{i j j' i j i} \la 1 \ra};
	(-30,31)*+{\oplus};
	(-30,26)*+{\cal{E}_{j i i j' j i} \la 1 \ra}="2";
	(-30,20)*+{\oplus};
	(-30,15)*+{\cal{E}_{j i j' i i j} \la 1 \ra}="2b";
	(30,37)*+{\cal{E}_{i j i j' j i} \la 2 \ra};
	(30,31)*+{\oplus};
	(30,26)*+{\cal{E}_{i j j' i i j} \la 2 \ra}="3";
	(30,20)*+{\oplus};
	(30,15)*+{\cal{E}_{j i i j' i j} \la 2 \ra}="3b";
	(90,26)*+{\clubsuit \cal{E}_{i j i j' i j} \la 3 \ra}="4";
	{\ar^-{\left(\begin{smallmatrix} +  \\ + \\ +  \end{smallmatrix}\right)}  "1"; "2"};
	{\ar^-{\left(\begin{smallmatrix} - & + & 0 \\ - & 0 & + \\ 0 & - & +  \end{smallmatrix}\right)}  "2"; "3"};
	{\ar^-{\left(\begin{smallmatrix} + & - & +  \end{smallmatrix}\right)}  "3"; "4"};
	{\ar^-{\left(\begin{smallmatrix} 0 & 0 & h^1_3  \end{smallmatrix}\right)}@/^1pc/  "2b"; "1"};
	{\ar^-{\left(\begin{smallmatrix} 0 & 0 & 0 \\ 0 & 0 & h^2_{23} \\ 0 & h^2_{32} & h^2_{33}
		\end{smallmatrix}\right)}@/^1pc/  "3b"; "2b"};
	{\ar^-{\left(\begin{smallmatrix} 0 \\ 0 \\ h^3_3  \end{smallmatrix}\right)}@/^1pc/  "4"; "3b"};
\endxy}
\]
Herein, the maps in the homotopy are given by:
\[
h_3^1 = t_{ij}^{-1} v_{ij} t_{ij'}^{-1} t_{jj'}
\vcenter{
\xy 0;/r.12pc/:
	(-15,-1)*{\slinen{j}};
	(-6,-1)*{\ncrossrp{i}{j'}};
	(6,-.7)*{\ncrossrr{i}{i}};
	(15,-1)*{\slinen{j}};
	(-15,8)*{\slinen{}};
	(-9,8)*{\slinenp{}};
	(0,8)*{\ncrossrr{}{}};
	(9,8)*{\slinenr{}};
	(15,8)*{\slinen{}};
	(-15,16)*{\slinen{}};
	(-6,16)*{\ncrosspr{}{}};
	(6,16)*{\ncrossrr{}{}};
	(15,16)*{\slinen{}};
	(-15,24)*{\slineu{}};
	(-9,24)*{\slineur{}};
	(-3,24)*{\slineup{}};
	(3,24)*{\slineur{}};
	(12,24)*{\ucrossrb{}{}};
\endxy
}
\quad , \quad
h_3^3 = - t_{ij}^{-1} v_{ij} t_{ij'}^{-1} t_{jj'}
\vcenter{
\xy 0;/r.12pc/:
	(-12,-1)*{\ncrossrb{i}{j}};
	(-3,-1)*{\slinenr{i}};
	(3,-1)*{\slinenp{j'}};
	(9,-1)*{\slinenr{i}};
	(15,-1)*{\slinen{j}};
	(-15,8)*{\slinen{}};
	(-6,8)*{\ncrossrr{}{}};
	(6,8)*{\ncrosspr{}{}};
	(15,8)*{\slinen{}};
	(-15,16)*{\slinen{}};
	(-9,16)*{\slinenr{}};
	(0,16)*{\ncrossrr{}{}};
	(9,16)*{\slinenp{}};
	(15,16)*{\slinen{}};
	(-15,24)*{\slineu{}};
	(-6,24)*{\ucrossrr{}{}};
	(6,24)*{\ucrossrp{}{}};
	(15,24)*{\slineu{}};
\endxy
}
\]
\[
h_{32}^2 = - t_{ij}^{-1} v_{ij} t_{ij'}^{-1} t_{jj'}
\vcenter{
\xy 0;/r.12pc/:
	(-12,-1)*{\ncrossrb{i}{j}};
	(-3,-1)*{\slinenp{j'}};
	(3,-1)*{\slinenr{i}};
	(9,-1)*{\slinenr{i}};
	(15,-1)*{\slinen{j}};
	(-15,8)*{\slinen{}};
	(-6,8)*{\ncrossrp{}{}};
	(6,8)*{\ncrossrr{}{}};
	(15,8)*{\slinen{}};
	(-15,16)*{\slinen{}};
	(-9,16)*{\slinenp{}};
	(0,16)*{\ncrossrr{}{}};
	(9,16)*{\slinenr{}};
	(15,16)*{\slinen{}};
	(-15,24)*{\slineu{}};
	(-6,24)*{\ucrosspr{}{}};
	(6,24)*{\ucrossrr{}{}};
	(15,24)*{\slineu{}};
\endxy
}
\;\; , \;\;
h_{33}^2 = - v_{ij} t_{ij'}^{-1} t_{jj'}
\vcenter{
\xy 0;/r.12pc/:
	(-15,-1)*{\slinen{j}};
	(-6,-1)*{\ncrossrr{i}{i}};
	(3,-1)*{\slinenp{j'}};
	(9,-1)*{\slinenr{i}};
	(15,-1)*{\slinen{j}};
	(-15,8)*{\slinen{}};
	(-9,8)*{\slinenr{}};
	(0,8)*{\ncrossrp{}{}};
	(9,8)*{\slinenr{}};
	(15,8)*{\slinen{}};
	(-15,16)*{\slineu{}};
	(-9,16)*{\slineur{}};
	(-3,16)*{\slineup{}};
	(6,16)*{\ucrossrr{}{}};
	(15,16)*{\slineu{}};
\endxy
}
\;\; , \;\;
h_{23}^2 = - t_{ij}^{-1} v_{ij} t_{ij'}^{-1} t_{jj'}
\vcenter{
\xy 0;/r.12pc/:
	(-15,-1)*{\slinen{j}};
	(-6,-1)*{\ncrossrr{i}{i}};
	(6,-.7)*{\ncrosspr{i}{i}};
	(15,-1)*{\slinen{j}};
	(-15,8)*{\slinen{}};
	(-9,8)*{\slinenr{}};
	(0,8)*{\ncrossrr{}{}};
	(9,8)*{\slinenp{}};
	(15,8)*{\slinen{}};
	(-15,16)*{\slinen{}};
	(-6,16)*{\ncrossrr{}{}};
	(6,16)*{\ncrossrp{}{}};
	(15,16)*{\slinen{}};
	(-15,24)*{\slineu{}};
	(-9,24)*{\slineur{}};
	(-3,24)*{\slineur{}};
	(3,24)*{\slineup{}};
	(12,24)*{\ucrossrb{}{}};
\endxy
}
\]

\end{proof}

%
\subsection{Bubble relations}
%

We now verify that $\cal{T}_i'$ preserves relation~\eqref{def:UQ-bub} in Definition~\ref{defU_cat-cyc} .

\begin{prop}
\[
\cal{T}_i' \left(
\xy 0;/r.18pc/:
(-12,0)*{\icbub{\l_i-1+m}{\ell}};
(-8,8)*{\lambda};
\endxy
\right)
=
\begin{cases}
c_{\ell,\l} \Id_{\onell{s_{i}(\l)}} & \text{if } m=0 \\
0 & \text{if } m<0
\end{cases}
\qquad , \qquad
\cal{T}_i' \left(
\xy 0;/r.18pc/: (-12,0)*{\iccbub{-\l_i-1+m}{\ell}};
 (-8,8)*{\lambda};
 \endxy
 \right)
 =
\begin{cases}
c_{\ell,\l}^{-1}\Id_{\onell{s_{i}(\l)}} & \text{if } m=0 \\
0 & \text{if } m<0
\end{cases}
\]
\end{prop}

\begin{proof}
We'll give the proof only in the clockwise case, as the counter-clockwise case is completely analogous.
The computations in Section \ref{sec:Ti-bub} show that
\[
  \cal{T}'_i \left(\; \xy 0;/r.18pc/:
 (-12,0)*{\icbub{\la \ell, \lambda \ra -1 + m}{\ell}};
 (-8,8)*{\lambda};
 \endxy \;\right)
=
\left\{
  \begin{array}{ll}
    c_{i,\lambda}^2\xy 0;/r.17pc/:
 (-12,0)*{\ccbubr{-\la i, s_i(\lambda) \ra -1 + m}{i\;}};
 (-5,8)*{s_i(\lambda)};
\endxy & \hbox{if $\ell=i$} \\
& \\
    t_{ki}^{\lambda_i}\xy 0;/r.16pc/:
 (-12,0)*{\cbubg{\la k, s_i(\lambda) \ra -1 + m}{k\; \;}};
 (-5,8)*{s_i(\lambda)};
\endxy & \hbox{if $\ell=k$} \\
& \\
t_{ji}^{\lambda_i}c_{i,\lambda}^{-1}\displaystyle\sum_{h=0}^{m} (-v_{ij})^{-h}
     \xy 0;/r.18pc/:
 (-6,0)*{\cbub{\spadesuit + m-h}{j}};
 (6,0)*{\cbubr{\spadesuit + h}{i}};
 (15,8)*{s_i(\lambda)};
 \endxy & \hbox{if $\ell =j$}
  \end{array}
\right.
\]
which immediately gives the result in the $m<0$ case.

For $m=0$, we compute:
\begin{align*}
c_{i,\lambda}^2
\xy 0;/r.18pc/:
	(-12,0)*{\ccbubr{-\la i, s_i(\lambda) \ra -1}{i}};
	(-8,8)*{s_i(\lambda)};
\endxy
&= c_{i,\lambda}^2c_{i,s_i(\lambda)}^{-1}\Id_{\onell{s_i(\lambda)}}
=c_{i,\lambda}^2c_{i,\lambda-\lambda_i\alpha_i}^{-1} \Id_{\onell{s_i(\lambda)}}
=c_{i,\lambda}^2c_{i,\lambda}^{-1}\Id_{\onell{s_i(\lambda)}}
=c_{i,\lambda}\Id_{\onell{s_i(\lambda)}}
\\ \\
t_{ki}^{\lambda_i}\xy 0;/r.18pc/:
 (-12,0)*{\cbubg{\la k, s_i(\lambda) \ra -1}{k}};
 (-8,8)*{s_i(\lambda)};
 \endxy
&=t_{ki}^{\lambda_i}c_{k,s_i(\lambda)}\Id_{\onell{s_i(\lambda)}}
=t_{ki}^{\lambda_i}c_{k,\lambda-\lambda_i\alpha_i} \Id_{\onell{s_i(\lambda)}}
=t_{ki}^{\lambda_i}t_{ki}^{-\lambda_i}c_{k,\lambda} \Id_{\onell{s_i(\lambda)}}
=c_{k,\lambda}\Id_{\onell{s_i(\lambda)}}
\\ \\
t_{ji}^{\lambda_i}c_{i,\lambda}^{-1}
     \xy 0;/r.18pc/:
 (-6,0)*{\cbub{\clubsuit + 0}{j}};
 (6,0)*{\cbubr{\clubsuit + 0}{i}};
 (18,2)*{s_i(\lambda)};
 \endxy
&=t_{ji}^{\lambda_i}c_{i,\lambda}^{-1} c_{j,s_i(\lambda)}c_{i,s_i(\lambda)}\Id_{\onell{s_i(\lambda)}}
=t_{ji}^{\lambda_i}c_{j,\lambda-\lambda_i\alpha_i} \Id_{\onell{s_i(\lambda)}} \\
&=t_{ji}^{\lambda_i}c_{j,\lambda}t_{ji}^{-\lambda_i} \Id_{\onell{s_i(\lambda)}}
=c_{j,\lambda}\Id_{\onell{s_i(\lambda)}}
\end{align*}
\end{proof}

In Section~\ref{sec:Ti-bub} we verify that the infinite Grassmannian relations from Section~\ref{sec:inf}
are preserved by $\cal{T}_i'$.

%
\subsection{Mixed $EF$ relation}
%

We now verify relation~\eqref{def:mixed} in Definition~\ref{defU_cat-cyc}.

\begin{prop}
$\cal{T}_i'$ preserves the mixed $EF$ relations.
\end{prop}

\begin{proof}
All cases involving $k$-labeled strands hold on the nose, and are trivial to verify.
The following computations exhibit half the requisite checks, and the remaining follow almost identically.
Throughout, we make extensive use of the computations from Section \ref{sec:sideways-crossing}.
\[
\cal{T}'_i\left({\xy 0;/r.15pc/: (0,-4.75)*{\rcrossgg{k}{k'}}; (0,4)*{\lcrossgg{}{}}; \endxy}
-{\xy 0;/r.18pc/: (3,-4)*{\slinedg{k'}}; (3,4)*{\slineng{}}; (-3,-4)*{\slineng{k}}; (-3,4)*{\slineug{}}; \endxy}\right)
={\xy 0;/r.15pc/: (0,-4.75)*{\rcrossgg{k}{k'}}; (0,4)*{\lcrossgg{}{}}; \endxy}
-{\xy 0;/r.18pc/: (3,-4)*{\slinedg{k'}}; (3,4)*{\slineng{}}; (-3,-4)*{\slineng{k}}; (-3,4)*{\slineug{}}; \endxy}=0
\quad , \quad
\cal{T}'_i\left({\xy 0;/r.15pc/: (0,-4.75)*{\rcrossrg{i}{k}}; (0,4)*{\lcrossgr{}{}}; \endxy}
-{\xy 0;/r.18pc/: (3,-4)*{\slinedg{i}}; (3,4)*{\slineng{}}; (-3,-4)*{\slinenr{k}}; (-3,4)*{\slineur{}}; \endxy}\right)
=t_{ki}^{-1}{\xy 0;/r.15pc/: (0,-4.75)*{\dcrossrg{i}{k}}; (0,4)*{\dcrossgr{}{}}; \endxy}
-{\xy 0;/r.18pc/: (3,-4)*{\slinedg{i}}; (3,4)*{\slineng{}}; (-3,-4)*{\slinedr{k}}; (-3,4)*{\slinenr{}}; \endxy}=0
\]
\[
\cal{T}'_i \left({\xy 0;/r.15pc/: (0,-4.75)*{\rcrossgb{k}{j}}; (0,4)*{\lcrossbg{}{}}; \endxy}
-{\xy 0;/r.18pc/: (3,-4)*{\slined{j}}; (3,4)*{\slinen{}}; (-3,-4)*{\slineng{k}}; (-3,4)*{\slineug{}}; \endxy}\right)
=
\left({\xy 0;/r.11pc/:
 (-3,-10.5)*{\rcrossgb{k}{j}};
  (6,-10.5)*{\slinedr{i}};
 (3,-1.25)*{\rcrossgr{}{}};
 (-6,-1.25)*{\slinen{}};
 (3,7)*{\lcrossrg{}{}};
  (-6,7)*{\slined{}};
 (-3,15)*{\lcrossbg{}{}};
 (6,15)*{\slinenr{}}; (15,0)*{};
 \endxy}
-{\xy 0;/r.18pc/:
(0,-4)*{\slined{j}}; (0,4)*{\slinen{}}; (-3,-4)*{\slineng{k}}; (-3,4)*{\slineug{ }};
  (3,-4)*{\slinedr{i}}; (3,4)*{\slinenr{}};
\endxy},
{\xy 0;/r.10pc/:
 (-3,-10.5)*{\rcrossgr{k}{i}};
  (6,-10.5)*{\slined{j}};
 (3,-1.25)*{\rcrossgb{}{}};
 (-6,-1.25)*{\slinenr{}};
 (3,7)*{\lcrossbg{}{}};
  (-6,7)*{\slinenr{}};
 (-3,15)*{\lcrossrg{}{}};
 (6,15)*{\slinen{}};
 \endxy}
-{\xy 0;/r.18pc/: (3,-4)*{\slined{j}}; (3,4)*{\slinen{}}; (-3,-4)*{\slineng{k}}; (-3,4)*{\slineug{}};
  (0,-4)*{\slinedr{i}}; (0,4)*{\slinenr{}}; \endxy}\right)=0
\]
\[
\cal{T}'_i \left(
{\xy 0;/r.15pc/: (0,-4.75)*{\rcrossbg{j}{k}}; (0,4)*{\lcrossgb{}{}}; \endxy}
-{\xy 0;/r.18pc/: (-3,-4)*{\slinen{j}}; (-3,4)*{\slineu{}}; (3,-4)*{\slinedg{k}}; (3,4)*{\slineng{}}; \endxy}\right)
=
\left({\xy 0;/r.11pc/:
 (3,-10.5)*{\rcrossrg{i}{k}};
  (-6,-10.5)*{\slinen{j}};
 (-3,-1.25)*{\rcrossbg{}{}};
 (6,-1.25)*{\slineur{}};
 (-3,7)*{\lcrossgb{}{}};
  (6,7)*{\slinenr{}};
 (3,15)*{\lcrossgr{}{}};
 (-6,15)*{\slineu{}};
 \endxy}
-{\xy 0;/r.18pc/: (-3,-4)*{\slinen{j}}; (-3,4)*{\slineu{}}; (3,-4)*{\slinedg{k}}; (3,4)*{\slineng{}};
  (0,-4)*{\slinenr{i}}; (0,4)*{\slineur{}}; \endxy},
{\xy 0;/r.11pc/:
 (3,-10.5)*{\rcrossbg{j}{k}};
  (-6,-10.5)*{\slinenr{i}};
 (-3,-1.25)*{\rcrossrg{}{}};
 (6,-1.25)*{\slineu{}};
 (-3,7)*{\lcrossgr{}{}};
  (6,7)*{\slinen{}};
 (3,15)*{\lcrossgb{}{}};
 (-6,15)*{\slineur{}}; (-15,0)*{};
 \endxy}
-{\xy 0;/r.18pc/: (0,-4)*{\slinen{j}}; (0,4)*{\slineu{}}; (3,-4)*{\slinedg{k}}; (3,4)*{\slineng{}};
  (-3,-4)*{\slinenr{i}}; (-3,4)*{\slineur{}}; \endxy}\right)=0.
\]

We now consider the cases requiring homotopies.
We compute
\begin{align*}
\cal{T}'_i \left({\xy 0;/r.15pc/: (0,-4.75)*{\rcrossbr{j}{i}}; (0,4)*{\lcrossrb{}{}}; \endxy}
-{\xy 0;/r.18pc/: (-3,-4)*{\slinen{j}}; (-3,4)*{\slineu{}}; (3,-4)*{\slinedr{i}}; (3,4)*{\slinenr{}}; \endxy}\right)
&= \left(t_{ij}^{-1}{\xy 0;/r.14pc/:
(0,-9)*{\xy 0;/r.14pc/:
 (3,-.75)*{\ucrossrr{i}{i}};
  (-6,-.75)*{\slinen{j}};
 (-3,8)*{\ucrossbr{}{}};
 (6,8)*{\slinenr{}};
 \endxy};
(0,8)*{\xy 0;/r.14pc/:
 (-3,0)*{\ucrossrb{}{}};
  (6,0)*{\slinenr{}};
 (3,8)*{\ucrossrr{}{}};
 (-6,8)*{\slineu{}};
(1.5,10.25)*[black]{\scs \bullet};
 \endxy}; \endxy}
-t_{ij}^{-1}{\xy 0;/r.14pc/:
(0,-9)*{\xy 0;/r.14pc/:
 (3,-.75)*{\ucrossrr{i}{i}};
  (-6,-.75)*{\slinen{j}};
 (-3,8)*{\ucrossbr{}{}};
 (6,8)*{\slinenr{}};
 \endxy};
(0,8)*{\xy 0;/r.14pc/:
 (-3,0)*{\ucrossrb{}{}};
  (6,0)*{\slinenr{}};
 (3,8)*{\ucrossrr{}{}};
 (-6,8)*{\slineu{}};
(-5,-.75)*[black]{\scs \bullet};
 \endxy}; \endxy}
-{\xy (0,4)*{\slineur{}}; (-3,4)*{\slineu{}}; (3,4)*{\slineur{}};
        (0,-4)*{\slinenr{i}}; (-3,-4)*{\slinen{j}}; (3,-4)*{\slinenr{i}}; \endxy} , \;
-t_{ij}^{-1}{\xy 0;/r.14pc/:
(0,-9)*{\xy 0;/r.14pc/:
    (3,-.75)*{\ucrossbr{j}{i}};
    (-6,-.75)*{\slinenr{i}};
    (-3,8)*{\ucrossrr{}{}};
    (6,8)*{\slinen{}};    \endxy};
(0,8)*{\xy  (-3,0)*{\ucrossrr{}{}};
  (6,0)*{\slinen{}};
 (3,8)*{\ucrossrb{}{}};
 (-6,8)*{\slineur{}};
(-5,-.75)*[black]{\scs \bullet};
 \endxy}; \endxy}
+t_{ij}^{-1}{\xy 0;/r.14pc/:
(0,-9)*{\xy 0;/r.14pc/:
    (3,-.75)*{\ucrossbr{j}{i}};
    (-6,-.75)*{\slinenr{i}};
    (-3,8)*{\ucrossrr{}{}};
    (6,8)*{\slinen{}};    \endxy};
(0,8)*{\xy (-3,0)*{\ucrossrr{}{}};
  (6,0)*{\slinen{}};
 (3,8)*{\ucrossrb{}{}};
 (-6,8)*{\sdotur{}}; \endxy}; \endxy}
-{\xy (-3,4)*{\slineur{}}; (0,4)*{\slineu{}}; (3,4)*{\slineur{}};
        (-3,-4)*{\slinenr{i}}; (0,-4)*{\slinen{j}}; (3,-4)*{\slinenr{i}}; \endxy}\right) \\
&=\left(-{\xy (-2,0)*{\slineu{j}}; (3,0)*{\sucrossrr{i}{i}}; (1.5,-1)*[black]{\scs \bullet}; \endxy}
-t_{ij}^{-1}t_{ji}{\xy (-2,0)*{\sdotu{j}}; (3,0)*{\sucrossrr{i}{i}}; \endxy} ,
-t_{ij}^{-1}
{\xy 0;/r.11pc/: (0,0)*{\Rthreel{black}{blue}{black}{}{}{}};
(-6,-13)*{\scs i};
(0,-13)*{\scs j};
(6,-13)*{\scs i};
\endxy}
\right)
=\left(-t_{ij}^{-1}{\xy 0;/r.12pc/: (-3,0)*{\ucrossrb{}{}}; (6,0)*{\slinenr{}}; (-6,8)*{\slineu{}}; (3,8)*{\ucrossrr{}{}};
   (-3,-8.75)*{\ucrossbr{j}{i}}; (6,-8.75)*{\slinenr{i}}; \endxy} ,
-t_{ij}^{-1}
{\xy 0;/r.11pc/: (0,0)*{\Rthreel{black}{blue}{black}{}{}{}};
(-6,-13)*{\scs i};
(0,-13)*{\scs j};
(6,-13)*{\scs i};
\endxy}
\right)
\end{align*}
which is null-homotopic, with homotopy given by
\[
{\xy 0;/r.15pc/:
  (-50,15)*+{\cal{E}_j\cal{E}_i\cal{E}_i\onell{s_i(\lambda)}\la \lambda_i\ra}="1";
  (-50,-15)*+{\cal{E}_j\cal{E}_i\cal{E}_i\onell{s_i(\lambda)}\la \lambda_i \ra}="2";
  (50,15)*+{\cal{E}_i\cal{E}_j\cal{E}_i\onell{s_i(\lambda)}\la \lambda_i+1 \ra}="3";
  (50,-15)*+{\cal{E}_i\cal{E}_j\cal{E}_i\onell{s_i(\lambda)}\la \lambda_i+1 \ra}="4";
{\ar@/^1pc/ "4";"1"};
(10,1)*{\xy 0;/r.15pc/: (-3,-4.75)*{\ucrossrb{i}{j}}; (6,-4.75)*{\slinenr{i}}; (-6,4)*{\slineu{}}; (3,4)*{\ucrossrr{}{}};
(-15,0)*{-t_{ij}^{-1}};\endxy};
   {\ar_{\xy  0;/r.15pc/: (3,0)*{\slineur{i}};(-6,0)*{\ucrossbr{j}{i}}; (9,0)*{}; \endxy   } "2";"4"};
   {\ar^{\xy  0;/r.15pc/: (3,0)*{\slineur{i}};(-6,0)*{\ucrossbr{j}{i}}; (9,0)*{}; \endxy   } "1";"3"};
 \endxy}
\]
and
\begin{align*}
\cal{T}'_i \left({\xy 0;/r.15pc/: (0,-4.75)*{\lcrossrb{i}{j}}; (0,4)*{\rcrossbr{}{}}; \endxy}
-{\xy 0;/r.18pc/: (3,-4)*{\slinen{j}}; (3,4)*{\slineu{}}; (-3,-4)*{\slinedr{i}}; (-3,4)*{\slinenr{}}; \endxy}\right)
&= \left(t_{ij}^{-1}{\xy 0;/r.14pc/:
(0,8)*{\xy 0;/r.14pc/:
 (3,0)*{\ucrossrr{}{}};
  (-6,0)*{\slinen{}};
 (-3,8)*{\ucrossbr{}{}};
 (6,8)*{\slineur{}};
 \endxy};
(0,-9)*{\xy 0;/r.14pc/:
 (-3,-.75)*{\ucrossrb{i}{j}};
  (6,-.75)*{\slinenr{i}};
 (3,8)*{\ucrossrr{}{}};
 (-6,8)*{\slinen{}};
(1.75,10)*[black]{\scs \bullet};
 \endxy}; \endxy}
-t_{ij}^{-1}{\xy 0;/r.14pc/:
(0,8)*{\xy 0;/r.14pc/:
 (3,0)*{\ucrossrr{}{}};
  (-6,0)*{\slinen{}};
 (-3,8)*{\ucrossbr{}{}};
 (6,8)*{\slineur{}};
 \endxy};
(0,-9)*{\xy 0;/r.14pc/:
 (-3,-.75)*{\ucrossrb{i}{j}};
  (6,-.75)*{\slinenr{i}};
 (3,8)*{\ucrossrr{}{}};
 (-6,8)*{\slinen{}};
(-5,-.75)*[black]{\scs \bullet};
 \endxy}; \endxy}
-{\xy (-3,4)*{\slineur{}}; (0,4)*{\slineu{}}; (3,4)*{\slineur{}};
        (-3,-4)*{\slinenr{i}}; (0,-4)*{\slinen{j}}; (3,-4)*{\slinenr{i}}; \endxy} , \;
-t_{ij}^{-1}{\xy 0;/r.14pc/:
(0,8)*{\xy 0;/r.14pc/:
    (3,0)*{\ucrossbr{}{}};
    (-6,0)*{\slinenr{}};
    (-3,8)*{\ucrossrr{}{}};
    (6,8)*{\slineu{}};    \endxy};
(0,-9)*{\xy  (-3,-.75)*{\ucrossrr{i}{i}};
  (6,-.75)*{\slinen{j}};
 (3,8)*{\ucrossrb{}{}};
 (-6,8)*{\slinenr{}};
(-5,-.75)*[black]{\scs \bullet};
 \endxy}; \endxy}
+t_{ij}^{-1}{\xy 0;/r.15pc/:
(0,8)*{\xy 0;/r.15pc/:
    (3,0)*{\ucrossbr{}{}};
    (-6,0)*{\slinenr{}};
    (-3,8)*{\ucrossrr{}{}};
    (6,8)*{\slineu{}};    \endxy};
(0,-9)*{\xy (-3,-.75)*{\ucrossrr{i}{i}};
  (6,-.75)*{\slinen{j}};
 (3,8)*{\ucrossrb{}{}};
 (-6,8)*{\sdotr{}}; \endxy}; \endxy}
-{\xy (-3,4)*{\slineur{}}; (3,4)*{\slineu{}}; (0,4)*{\slineur{}};
        (-3,-4)*{\slinenr{i}}; (3,-4)*{\slinen{j}}; (0,-4)*{\slinenr{i}}; \endxy}\right)\\
&= \left(t_{ij}^{-1}
{\xy 0;/r.11pc/: (0,0)*{\Rthreer{black}{blue}{black}{}{}{}};
(-6,-13)*{\scs i};
(0,-13)*{\scs j};
(6,-13)*{\scs i};
\endxy} ,
{\xy (3,0)*{\slineu{j}}; (-3,0)*{\sucrossrr{i}{i}}; (-2,2.5)*[black]{\scs \bullet}; \endxy}
+t_{ij}^{-1}t_{ji}{\xy (3,0)*{\sdotu{j}}; (-3,0)*{\sucrossrr{i}{i}}; \endxy}\right)
=\left(t_{ij}^{-1}
{\xy 0;/r.11pc/: (0,0)*{\Rthreer{black}{blue}{black}{}{}{}};
(-6,-13)*{\scs i};
(0,-13)*{\scs j};
(6,-13)*{\scs i};
\endxy} , \;
t_{ij}^{-1}{\xy 0;/r.12pc/: (3,8)*{\ucrossbr{}{}}; (-6,8)*{\slineur{}}; (6,-8.75)*{\slinen{j}}; (-3,-8.75)*{\ucrossrr{i}{i}};
   (3,0)*{\ucrossrb{}{}}; (-6,0)*{\slinenr{}}; \endxy}\right)
\end{align*}
is null-homotopic with homotopy given by
\begin{align*}
{     \xy 0;/r.15pc/:
  (-50,15)*+{\cal{E}_i\cal{E}_j\cal{E}_i\onell{s_i(\lambda)}\la \lambda_i-1 \ra}="1";
  (-50,-15)*+{\cal{E}_i\cal{E}_j\cal{E}_i\onell{s_i(\lambda)}\la \lambda_i-1 \ra}="2";
  (50,15)*+{\cal{E}_i\cal{E}_i\cal{E}_j\onell{s_i(\lambda)}\la \lambda_i \ra}="3";
  (50,-15)*+{\cal{E}_i\cal{E}_i\cal{E}_j\onell{s_i(\lambda)}\la \lambda_i \ra}="4";
{\ar@/^1pc/ "4";"1"};
(10,1)*{\xy 0;/r.15pc/: (3,4)*{\ucrossrb{}{}}; (-15,1)*{-t_{ij}^{-1}}; (-6,4)*{\slineur{}};
(6,-4.75)*{\slinen{j}}; (-3,-4.75)*{\ucrossrr{i}{i}}; \endxy};
   {\ar_{\xy  0;/r.15pc/: (-3,0)*{\slineur{i}};(6,0)*{\ucrossbr{j}{i}}; (-7,1)*{-};(9,0)*{}; \endxy   } "2";"4"};
   {\ar^{\xy  0;/r.15pc/: (-3,0)*{\slineur{i}};(6,0)*{\ucrossbr{j}{i}}; (-7,1)*{-};(9,0)*{}; \endxy   } "1";"3"};
 \endxy}
\end{align*}

Similar computations show that
\[
\cal{T}'_i \left({\xy 0;/r.15pc/: (0,-4.75)*{\lcrossbr{j}{i}}; (0,4)*{\rcrossrb{}{}}; \endxy}
-{\xy 0;/r.18pc/: (-3,-4)*{\slined{j}}; (-3,4)*{\slinen{}}; (3,-4)*{\slinenr{i}}; (3,4)*{\slineur{}}; \endxy}\right)
=
\left(t_{ij}^{-1}{\xy 0;/r.12pc/: (-3,0)*{\dcrossrb{}{}}; (6,0)*{\slinedr{}}; (-6,8)*{\slinen{}}; (3,8)*{\dcrossrr{}{}};
   (-3,-8.75)*{\dcrossbr{j}{i}}; (6,-8.75)*{\slinedr{i}}; \endxy} , \;
t_{ij}^{-1}{\xy 0;/r.12pc/: (-3,-8.75)*{\dcrossrb{i}{j}}; (6,-8.75)*{\slinedr{i}}; (-6,0)*{\slinen{}}; (3,0)*{\dcrossrr{}{}};
   (-3,8)*{\dcrossbr{}{}}; (6,8)*{\slinedr{}}; \endxy}\right)
\]
which is null-homotopic via the homotopy
\[
     {\xy 0;/r.15pc/:
  (-50,15)*+{\cal{F}_j\cal{F}_i\cal{F}_i\onell{s_i(\lambda)}\la -\lambda_i-3 \ra}="1";
  (-50,-15)*+{\cal{F}_j\cal{F}_i\cal{F}_i\onell{s_i(\lambda)}\la -\lambda_i-3 \ra}="2";
  (50,15)*+{\cal{F}_i\cal{F}_j\cal{F}_i\onell{s_i(\lambda)}\la -\lambda_i-2 \ra}="3";
  (50,-15)*+{\cal{F}_i\cal{F}_j\cal{F}_i\onell{s_i(\lambda)}\la -\lambda_i-2 \ra}="4";
{\ar@/^1pc/ "4";"1"};
(10,2)*{\xy 0;/r.15pc/: (-3,-4.75)*{\dcrossrb{i}{j}}; (-12,0)*{t_{ij}^{-1}}; (6,-4.75)*{\slinedr{i}};
(-6,4)*{\slinen{}}; (3,4)*{\dcrossrr{}{}}; \endxy};
   {\ar_{\xy  0;/r.15pc/: (3,0)*{\slinedr{i}};(-6,0)*{\dcrossbr{j}{i}}; (9,0)*{}; \endxy   } "2";"4"};
   {\ar^{\xy  0;/r.15pc/: (3,0)*{\slinedr{i}};(-6,0)*{\dcrossbr{j}{i}}; (9,0)*{}; \endxy   } "1";"3"};
 \endxy}
\]
and that
\[
\cal{T}'_i \left({\xy 0;/r.15pc/: (0,-4.75)*{\rcrossrb{i}{j}}; (0,4)*{\lcrossbr{}{}}; \endxy}
-{\xy 0;/r.18pc/: (3,-4)*{\slined{j}}; (3,4)*{\slinen{}}; (-3,-4)*{\slinenr{i}}; (-3,4)*{\slineur{}}; \endxy}\right)
=
\left(-t_{ij}^{-1}{\xy 0;/r.12pc/: (3,-8.75)*{\dcrossbr{j}{i}}; (-6,-8.75)*{\slinedr{i}};
(6,0)*{\slinen{}}; (-3,0)*{\dcrossrr{}{}}; (3,8)*{\dcrossrb{}{}}; (-6,8)*{\slinedr{}}; \endxy} , \;
-t_{ij}^{-1}{\xy 0;/r.12pc/: (3,8)*{\dcrossbr{}{}}; (-6,8)*{\slinedr{}};
(6,-8.75)*{\slined{j}}; (-3,-8.75)*{\dcrossrr{i}{i}}; (3,0)*{\dcrossrb{}{}}; (-6,0)*{\slinenr{}}; \endxy}\right)
\]
which is null-homotopic with homotopy given by
\[
{     \xy 0;/r.15pc/:
  (-50,15)*+{\cal{F}_i\cal{F}_j\cal{F}_i\onell{s_i(\lambda)}\la -\lambda_i-4 \ra}="1";
  (-50,-15)*+{\cal{F}_i\cal{F}_j\cal{F}_i\onell{s_i(\lambda)}\la -\lambda_i-4 \ra}="2";
  (50,15)*+{\cal{F}_i\cal{F}_i\cal{F}_j\onell{s_i(\lambda)}\la -\lambda_i-3 \ra}="3";
  (50,-15)*+{\cal{F}_i\cal{F}_i\cal{F}_j\onell{s_i(\lambda)}\la -\lambda_i-3 \ra}="4";
{\ar@/^1pc/ "4";"1"};
(10,2)*{\xy 0;/r.15pc/: (3,4)*{\dcrossrb{}{}}; (-14,0)*{t_{ij}^{-1}}; (-6,4)*{\slinenr{}};
(6,-4.75)*{\slined{j}}; (-3,-4.75)*{\dcrossrr{i}{i}}; \endxy};
   {\ar_{\xy  0;/r.15pc/: (-3,0)*{\slinedr{i}};(6,0)*{\dcrossbr{j}{i}}; (-7,1)*{-};(9,0)*{}; \endxy   } "2";"4"};
   {\ar^{\xy  0;/r.15pc/: (-3,0)*{\slinedr{i}};(6,0)*{\dcrossbr{j}{i}}; (-7,1)*{-};(9,0)*{}; \endxy   } "1";"3"};
 \endxy}
\]
The final case, involving $j$- and $j'$-labeled strands (with $j \neq j'$), will be addressed in
Proposition \ref{prop:extendsl2-jjp} below.
\end{proof}

%
\subsection{Extended \texorpdfstring{$\mathfrak{sl}_2$}{sl(2)} relations}
%
We now verify relation~\eqref{def:EF} in Definition~\ref{defU_cat-cyc}.

\begin{prop}
$\cal{T}_i'$ preserves the the extended $\mathfrak{sl}_2$ relations
in the $i$- and $k$-labeled cases.
\end{prop}

\begin{proof}
In these cases, the relations hold on the nose, as we confirm.
\begin{align*}
\cal{T}'_i\left({\xy 0;/r.18pc/: (0,-4.5)*{\rcrossrr{i}{i}}; (0,4)*{\lcrossrr{}{}}; \endxy}
+{\xy 0;/r.18pc/: (-3,-4)*{\slinenr{i}}; (-3,4)*{\slineur{}}; (3,-4)*{\slinedr{i}}; (3,4)*{\slinenr{}}; \endxy}
-\displaystyle\sum_{\substack{a+b+c\\ =\lambda_i-1}}
{\xy 0;/r.15pc/:
   (9,-8)*{\rcapr{}}; (9,-6.25)*[black]{\scs \bullet}; (9,-3)*{\scs b};
   (15,0)*{\smccbubr{\scs \spadesuit +c}{}};
   (9,9.5)*{\lcupr{}}; (9,8.25)*[black]{\scs \bullet}; (9,5)*{\scs a};
   (13,-6)*{\scs i};(14,9)*{\scs i}; (23,8)*{\l};
\endxy}
\right)
&= (-1)^2{\xy 0;/r.18pc/: (0,-4.5)*{\lcrossrr{i}{i}}; (0,4)*{\rcrossrr{}{}}; \endxy}
+{\xy 0;/r.18pc/: (-3,-4)*{\slinedr{i}}; (-3,4)*{\slinenr{}}; (3,-4)*{\slinenr{i}}; (3,4)*{\slineur{}}; \endxy}
-c_{i,\l}^2c_{i,\l}^{-2}\displaystyle\sum_{\substack{a+b+c\\ =\lambda_i-1}}
{\xy 0;/r.15pc/:
   (9,-8)*{\lcapr{}}; (9,-6.25)*[black]{\scs \bullet}; (9,-3)*{\scs b};
   (15,0)*{\smcbubr{\scs \spadesuit +c}{}};
   (9,9.5)*{\rcupr{}}; (9,8.25)*[black]{\scs \bullet}; (9,5)*{\scs a};
   (13,-6)*{\scs i};(14,9)*{\scs i}; (24,9)*{s_i(\l)};
\endxy}\\
&= {\xy 0;/r.18pc/: (0,-4.5)*{\lcrossrr{i}{i}}; (0,4)*{\rcrossrr{}{}}; \endxy}
+{\xy 0;/r.18pc/: (-3,-4)*{\slinedr{i}}; (-3,4)*{\slinenr{}}; (3,-4)*{\slinenr{i}}; (3,4)*{\slineur{}}; \endxy}
-\displaystyle\sum_{\substack{a+b+c=\\ -\la i,s_i(\lambda)\ra-1}}
{\xy 0;/r.15pc/:
   (9,-8)*{\lcapr{}}; (9,-6.25)*[black]{\scs \bullet}; (9,-3)*{\scs b};
   (15,0)*{\smcbubr{\scs \spadesuit +c}{}};
   (9,9.5)*{\rcupr{}}; (9,8.25)*[black]{\scs \bullet}; (9,5)*{\scs a};
   (13,-6)*{\scs i};(14,9)*{\scs i}; (24,9)*{s_i(\l)};
\endxy}=0 \\
\cal{T}'_i\left({\xy 0;/r.18pc/: (0,-4.5)*{\lcrossrr{i}{i}}; (0,4)*{\rcrossrr{}{}}; \endxy}
+{\xy 0;/r.18pc/: (-3,-4)*{\slinedr{i}}; (-3,4)*{\slinenr{}}; (3,-4)*{\slinenr{i}}; (3,4)*{\slineur{}}; \endxy}
-\displaystyle\sum_{\substack{a+b+c\\ =-\lambda_i-1}}
{\xy 0;/r.15pc/:
   (9,-8)*{\lcapr{}}; (9,-6.25)*[black]{\scs \bullet}; (9,-3)*{\scs b};
   (15,0)*{\smcbubr{\scs \spadesuit +c}{}};
   (9,9.5)*{\rcupr{}}; (9,8.25)*[black]{\scs \bullet}; (9,5)*{\scs a};
   (13,-6)*{\scs i};(14,9)*{\scs i}; (23,8)*{\l};
\endxy}\right)
&= (-1)^2{\xy 0;/r.18pc/: (0,-4.5)*{\rcrossrr{i}{i}}; (0,4)*{\lcrossrr{}{}}; \endxy}
+{\xy 0;/r.18pc/: (-3,-4)*{\slinenr{i}}; (-3,4)*{\slineur{}}; (3,-4)*{\slinedr{i}}; (3,4)*{\slinenr{}}; \endxy}
-c_{i,\l}^2c_{i,\l}^{-2}\displaystyle\sum_{\substack{a+b+c\\ =-\lambda_i-1}}
{\xy 0;/r.15pc/:
   (9,-8)*{\rcapr{}}; (9,-6.25)*[black]{\scs \bullet}; (9,-3)*{\scs b};
   (15,0)*{\smccbubr{\scs \spadesuit +c}{}};
   (9,9.5)*{\lcupr{}}; (9,8.25)*[black]{\scs \bullet}; (9,5)*{\scs a};
   (13,-6)*{\scs i};(14,9)*{\scs i}; (24,9)*{s_i(\l)};
\endxy} \\
&= {\xy 0;/r.18pc/: (0,-4.5)*{\rcrossrr{i}{i}}; (0,4)*{\lcrossrr{}{}}; \endxy}
+{\xy 0;/r.18pc/: (-3,-4)*{\slinenr{i}}; (-3,4)*{\slineur{}}; (3,-4)*{\slinedr{i}}; (3,4)*{\slinenr{}}; \endxy}
-\displaystyle\sum_{\substack{a+b+c=\\ \la i,s_i(\lambda)\ra-1}}
{\xy 0;/r.15pc/:
   (9,-8)*{\rcapr{}}; (9,-6.25)*[black]{\scs \bullet}; (9,-3)*{\scs b};
   (15,0)*{\smccbubr{\scs \spadesuit +c}{}};
   (9,9.5)*{\lcupr{}}; (9,8.25)*[black]{\scs \bullet}; (9,5)*{\scs a};
   (13,-6)*{\scs i};(14,9)*{\scs i}; (24,9)*{s_i(\l)};
\endxy}=0 \\
\cal{T}'_i\left({\xy 0;/r.18pc/: (0,-4.5)*{\rcrossgg{k}{k}}; (0,4)*{\lcrossgg{}{}}; \endxy}
+{\xy 0;/r.18pc/: (-3,-4)*{\slineng{k}}; (-3,4)*{\slineug{}}; (3,-4)*{\slinedg{k}}; (3,4)*{\slineng{}}; \endxy}
-\displaystyle\sum_{\substack{a+b+c\\ =\lambda_k-1}}
{\xy 0;/r.15pc/:
   (9,-8)*{\rcapg{}}; (9,-6.25)*[black]{\scs \bullet}; (9,-3)*{\scs b};
   (15,0)*{\smccbubg{\scs \spadesuit +c}{}};
   (9,9.5)*{\lcupg{}}; (9,8.25)*[black]{\bullet}; (9,5)*{\scs a};
   (14,-6)*{\scs k};(14,9)*{\scs k}; (23,8)*{\l};
\endxy}\right)
&= {\xy 0;/r.18pc/: (0,-4.5)*{\rcrossgg{k}{k}}; (0,4)*{\lcrossgg{}{}}; \endxy}
+{\xy 0;/r.18pc/: (-3,-4)*{\slineng{k}}; (-3,4)*{\slineug{}}; (3,-4)*{\slinedg{k}}; (3,4)*{\slineng{}}; \endxy}
-t_{ki}^{\l_i}t_{ki}^{-\l_i}\displaystyle\sum_{\substack{a+b+c\\ =\lambda_k-1}}
{\xy 0;/r.15pc/:
   (9,-8)*{\rcapg{}}; (9,-6.25)*[black]{\scs \bullet}; (9,-3)*{\scs b};
   (15,0)*{\smccbubg{\scs \spadesuit +c}{}};
   (9,9.5)*{\lcupg{}}; (9,8.25)*[black]{\scs \bullet}; (9,5)*{\scs a};
   (14,-6)*{\scs k};(14,9)*{\scs k}; (24,9)*{s_i(\l)};
\endxy}\\
&= {\xy 0;/r.18pc/: (0,-4.5)*{\rcrossgg{k}{k}}; (0,4)*{\lcrossgg{}{}}; \endxy}
+{\xy 0;/r.18pc/: (-3,-4)*{\slineng{k}}; (-3,4)*{\slineug{}}; (3,-4)*{\slinedg{k}}; (3,4)*{\slineng{}}; \endxy}
-\displaystyle\sum_{\substack{a+b+c=\\ \la k,s_i(\lambda)\ra-1}}
{\xy 0;/r.15pc/:
   (9,-8)*{\rcapg{}}; (9,-6.5)*[black]{\scs \bullet}; (9,-3)*{\scs b};
   (15,0)*{\smccbubg{\scs \spadesuit +c}{}};
   (9,9.5)*{\lcupg{}}; (9,8)*[black]{\scs \bullet}; (9,5)*{\scs a};
   (14,-6)*{\scs k};(14,9)*{\scs k}; (24,9)*{s_i(\l)};
\endxy}=0\\
\cal{T}'_i\left({\xy 0;/r.18pc/: (0,-4.5)*{\lcrossgg{k}{k}}; (0,4)*{\rcrossgg{}{}}; \endxy}
+{\xy 0;/r.18pc/: (-3,-4)*{\slinedg{k}}; (-3,4)*{\slineng{}}; (3,-4)*{\slineng{k}}; (3,4)*{\slineug{}}; \endxy}
-\displaystyle\sum_{\substack{a+b+c\\ =-\lambda_k-1}}
{\xy 0;/r.15pc/:
   (9,-8)*{\lcapg{}}; (9,-6.25)*[black]{\scs \bullet}; (9,-3)*{\scs b};
   (15,0)*{\smcbubg{\scs \spadesuit +c}{}};
   (9,9.5)*{\rcupg{}}; (9,8.25)*[black]{\scs \bullet}; (9,5)*{\scs a};
   (14,-6)*{\scs k};(14,9)*{\scs k}; (23,8)*{\l};
\endxy}\right)
&= {\xy 0;/r.18pc/: (0,-4.5)*{\lcrossgg{k}{k}}; (0,4)*{\rcrossgg{}{}}; \endxy}
+{\xy 0;/r.18pc/: (-3,-4)*{\slinedg{k}}; (-3,4)*{\slineng{}}; (3,-4)*{\slineng{k}}; (3,4)*{\slineug{}}; \endxy}
-t_{ki}^{\l_i}t_{ki}^{-\l_i}\displaystyle\sum_{\substack{a+b+c\\ =-\lambda_k-1}}
{\xy 0;/r.15pc/:
   (9,-8)*{\lcapg{}}; (9,-6.25)*[black]{\scs \bullet}; (9,-3)*{\scs b};
   (15,0)*{\smcbubg{\scs \spadesuit +c}{}};
   (9,9.5)*{\rcupg{}}; (9,8.25)*[black]{\scs \bullet}; (9,5)*{\scs a};
   (14,-6)*{\scs k};(14,9)*{\scs k}; (24,9)*{s_i(\l)};
\endxy}\\
&= {\xy 0;/r.18pc/: (0,-4.5)*{\lcrossgg{k}{k}}; (0,4)*{\rcrossgg{}{}}; \endxy}
+{\xy 0;/r.18pc/: (-3,-4)*{\slinedg{k}}; (-3,4)*{\slineng{}}; (3,-4)*{\slineng{k}}; (3,4)*{\slineug{}}; \endxy}
-\displaystyle\sum_{\substack{a+b+c=\\ -\la k,s_i(\lambda)\ra-1}}
{\xy 0;/r.15pc/:
   (9,-8)*{\lcapg{}}; (9,-6.25)*[black]{\scs \bullet}; (9,-3)*{\scs b};
   (15,0)*{\smcbubg{\scs \spadesuit +c}{}};
   (9,9.5)*{\rcupg{}}; (9,8.25)*[black]{\scs \bullet}; (9,5)*{\scs a};
   (14,-6)*{\scs k};(14,9)*{\scs k}; (24,9)*{s_i(\l)};
\endxy}=0
\end{align*}
\end{proof}

We conclude by considering the outstanding relations,
\ie the $j$-labeled extended $\mathfrak{sl}_2$ relations, and the $jj'$-labeled mixed $EF$ relation.
These are the most-involved relations, in part because the homotopies involved are not-necessarily unique.
Indeed, if $\underline{\Hom}(X,Y)$ denotes the chain complex of all homogeneous maps
(that are not-necessarily degree zero, or chain maps) between complexes $X$ and $Y$,
then given any element $\alpha \in \underline{\Hom}^{-2}(X,Y)$,
the element $d(\alpha)= d_Y \alpha - \alpha d_X$ can be added to any homotopy $h$ without
affecting $d_Y h + h d_X$.
Our previous cases have not admitted such an $\alpha$, but in the present case there exist
(many) such $\alpha$, given by any map
$\cal{E}_i\cal{E}_{j'}\cal{F}_i\cal{F}_j\onell{s_i(\l)}\la 1\ra\to
\cal{E}_{j'}\cal{E}_i\cal{F}_j\cal{F}_i\onell{s_i(\l)}\la -1\ra$.

\begin{prop} \label{prop:extendsl2-jjp}
The relation
 \begin{equation} \label{eq:EFj'j}
\cal{T}'_i\left(-{\xy 0;/r.18pc/: (0,-4.75)*{\rcrosspb{j'}{j}}; (0,4)*{\lcrossbp{}{}}; \endxy}
+(-1)^{\delta_{jj'}}{\xy 0;/r.18pc/: (-3,-4)*{\slinenp{j'}}; (-3,4)*{\slineup{}}; (3,-4)*{\slined{j}}; (3,4)*{\slinen{}}; \endxy}
+\delta_{jj'}\sum_{\substack{a+b+c=\\ \lambda_j-1}}
{\xy 0;/r.19pc/:
   (9,-7)*{\srcap{}}; (8,-6)*{\bullet}; (7,-3)*{\scs b};
   (15,0)*{\medccbub{\scs \spadesuit +c}{}};
   (9,7)*{\slcup{}}; (8,6.25)*{\bullet}; (7,4)*{\scs a};
   (13,-6)*{\scs j};(14,9)*{\scs j}; (21,8)*{\l};
\endxy}
\right) \sim 0
\end{equation}
holds in $\Com(\Ucat_Q)$.
\end{prop}

\begin{proof}
The left-hand side of \eqref{eq:EFj'j} is
given by the following:
\begin{equation}\label{eq:j'jchainmap}
{\xy 0;/r.15pc/:
	(-90,15)*+{\clubsuit \cal{E}_{j'} \cal{E}_i \cal{F}_j \cal{F}_i \onell{s_i(\lambda)} \la -1 \ra}="1";
	(0,15)*+{\cal{E}_{i} \cal{E}_{j'} \cal{F}_j \cal{F}_i \onell{s_i(\lambda)}
  		\oplus \cal{E}_{j'} \cal{E}_{i} \cal{F}_i \cal{F}_j \onell{s_i(\lambda)} }="2";
	(90,15)*+{\cal{E}_{i} \cal{E}_{j'} \cal{F}_i \cal{F}_j \onell{s_i(\lambda)} \la 1 \ra}="3";
	(-90,-15)*+{\clubsuit \cal{E}_{j'} \cal{E}_i \cal{F}_j \cal{F}_i \onell{s_i(\lambda)} \la -1 \ra}="4";
	(0,-15)*+{\cal{E}_{i} \cal{E}_{j'} \cal{F}_j \cal{F}_i \onell{s_i(\lambda)}
  		\oplus \cal{E}_{j'} \cal{E}_{i} \cal{F}_i \cal{F}_j \onell{s_i(\lambda)} }="5";
	(90,-15)*+{\cal{E}_{i} \cal{E}_{j'} \cal{F}_i \cal{F}_j \onell{s_i(\lambda)} \la 1 \ra}="6";
	{\ar  "1";"2"};
	{\ar  "2";"3"};
	{\ar  "4";"5"};
	{\ar  "5";"6"};
	{\ar^-{\phi_1}  "4";"1"};
	{\ar^-{\left(\begin{smallmatrix} \phi_2 & \phi_4 \\ \phi_3 & \phi_5 \end{smallmatrix}\right)}  "5";"2"};
	{\ar^-{\phi_6}  "6";"3"};
\endxy}
\end{equation}
where the components of the chain map are given as follows
(which can be verified by completely simplifying both
sides of the equalities):
\begin{align*}
\phi_1
&=
\vcenter{\xy 0;/r.11pc/:
    (0,40)*{\lcrossbr{}{}}; (9,40)*{\slinenr{}};
    (-9,40)*{\slineup{}};
    (-6,32)*{\ncrossbp{}{}};(6,32)*{\ncrossrr{}{}};
    (0,24)*{\lcrossrp{}{}}; (9,24)*{\slinenr{}}; (-9,24)*{\slined{}};
    (0,16)*{\rcrosspr{}{}}; (9,16)*{\slineur{}}; (-9,16)*{\slinen{}};
    (-6,8)*{\ncrosspb{}{}};(6,8)*{\ncrossrr{}{}};
    (0,-1)*{\rcrossrb{i}{j}}; (9,-1)*{\slinedr{i}}; (-9,-1)*{\slinenp{j'}};\endxy}
+(-1)^{\delta_{jj'}}
\vcenter{\xy 0;/r.11pc/:
    (9,40)*{\slinenr{}};(-9,40)*{\slineup{}};
    (-3,40)*{\slineur{}};(3,40)*{\slinen{}};
    (9,32)*{\slinenr{}};(-9,32)*{\slinenp{}};
    (-3,32)*{\slinenr{}};(3,32)*{\slinen{}};    (9,24)*{\slinenr{}};(-9,24)*{\slinenp{}};
    (-3,24)*{\slinenr{}};(3,24)*{\slinen{}};    (9,16)*{\slinenr{}};(-9,16)*{\slinenp{}};
    (-3,16)*{\slinenr{}};(3,16)*{\slinen{}};
    (9,8)*{\slinenr{}};(-9,8)*{\slinenp{}};
    (-3,8)*{\slinenr{}};(3,8)*{\slinen{}};
    (-3,-1)*{\slinenr{i}};(3,-1)*{\slined{j}}; (9,-1)*{\slinedr{i}}; (-9,-1)*{\slinenp{j'}};\endxy}
=
(-1)^{\delta_{jj'}}
\sum_{\substack{d+e+f\\=-\lambda_i-1}} \;
\vcenter{\xy 0;/r.11pc/:
    (-3,24)*{\slineur{}}; (6,24)*{\dcrossrb{}{}};
    (-9,24)*{\slineup{}};
    (-9,16)*{\slinenp{}};(9,16)*{\slinen{}};(0,18)*{\slcupr{}};
    (0,17.5)*[black]{\scs\bullet};(-3,14)*{\scs d};
    (20,13)*{\ccbubr{\scs\spadesuit+f}{i}};
    (-9,8)*{\slinenp{}};(9,8)*{\slinen{}};(0,7.5)*{\srcapr{}};
    (0,9.25)*[black]{\scs\bullet};(3,11)*{\scs e};
    (-3,-1)*{\slineur{i}}; (6,-1)*{\dcrossbr{j}{}}; (-9,-1)*{\slinenp{j'}};
    (9,32)*{\scs i};
\endxy}
+\delta_{jj'}t_{ji}
\sum_{\substack{a+c=\\ \l_i+\l_j-1}}
\vcenter{\xy 0;/r.11pc/:
    (-6,48)*{\ucrossrb{}{}}; (9,48)*{\slinenr{}};(3,48)*{\slinen{}};
    (0,42)*{\slcup{}};(0,41)*{\scs\bullet};
    (0,44)*{\scriptscriptstyle a};
    (-9,45)*{};(9,29)*{} **[black]\crv{(-9,37)& (9,37)}?(0)*[black]\dir{<};
    (9,45)*{};(-9,29)*{} **[black]\crv{(9,37)& (-9,37)}?(1)*[black]\dir{>};
    (0,25)*{\ccbub{\scriptscriptstyle\spadesuit+c}{j}};
    (9,24)*{\slinenr{}}; (-9,24)*{\slinenr{}};
    (-9,5)*{};(9,21)*{} **[black]\crv{(-9,13)& (9,13)}?(1)*[black]\dir{>};
    (9,5)*{};(-9,21)*{} **[black]\crv{(9,13)& (-9,13)}?(0)*[black]\dir{<};
    (0,6.5)*{\srcap{}};
    (-6,-1.5)*{\ucrossbr{j}{i}}; (3,-1.5)*{\slined{}};
    (9,-1.5)*{\slinedr{i}};
    (3,57)*{\scs j};
\endxy} \\
& \hspace{210pt}
+\delta_{jj'}\sum_{\substack{a+b+c= \\ \l_i+\l_j-2}}
\vcenter{\xy 0;/r.11pc/:
    (-6,48)*{\ucrossrb{}{}}; (9,48)*{\slinenr{}};(3,48)*{\slinen{}};
    (0,42)*{\slcup{}};(0,41)*{\scs\bullet};
    (0,44)*{\scriptscriptstyle a};
    (-9,45)*{};(9,29)*{} **[black]\crv{(-9,37)& (9,37)}?(0)*[black]\dir{<};
    (9,45)*{};(-9,29)*{} **[black]\crv{(9,37)& (-9,37)}?(1)*[black]\dir{>};
    (0,25)*{\ccbub{\scriptscriptstyle\spadesuit+c}{j}};
    (9,24)*{\slinenr{}}; (-9,24)*{\slinenr{}};
    (-9,5)*{};(9,21)*{} **[black]\crv{(-9,13)& (9,13)}?(1)*[black]\dir{>};
    (9,5)*{};(-9,21)*{} **[black]\crv{(9,13)& (-9,13)}?(0)*[black]\dir{<};
    (0,6.5)*{\srcap{}};(0,8.5)*{\scs\bullet};
    (0,5)*{\scriptscriptstyle b};
    (-6,0)*{\ucrossbr{}{}}; (6,0)*{\dcrossrb{}{}};
    (-9,-9)*{\slinen{j}};(-3,-9)*{\slinenr{i}}; (6,-9)*{\dcrossbr{}{i}};
    (3,57)*{\scs j};
\endxy}
+\delta_{jj'}t_{ij}\sum_{\substack{a+b+c= \\ \l_i+\l_j-2}}
\vcenter{\xy 0;/r.11pc/:
    (-6,48)*{\ucrossrb{}{}}; (9,48)*{\slinenr{}};(3,48)*{\slinen{}};
    (0,42)*{\slcup{}};(0,41)*{\scs\bullet};
    (0,44)*{\scriptscriptstyle a};
    (-9,45)*{};(9,29)*{} **[black]\crv{(-9,37)& (9,37)}?(0)*[black]\dir{<};
    (9,45)*{};(-9,29)*{} **[black]\crv{(9,37)& (-9,37)}?(1)*[black]\dir{>};
    (0,25)*{\ccbub{\scriptscriptstyle\spadesuit+c}{j}};
    (-9,29)*{}; (-9,22)*{} **\dir{-};
     (9,29)*{}; (9,22)*{} **\dir{-};
    (0,16)*{\llrcupr{}};
    (9,-3);(-9,-3) **[blue]\crv{(9,10) & (-9,10)}; ?(0)*[blue]\dir{<};
    (0,6.5)*{\scs\bullet};(0,10)*{\scriptscriptstyle b};
    (0,-1)*{\srcapr{}{}};
   (-9,-9)*{\slinen{j}};(-3,-9)*{\slinenr{i}}; (6,-9)*{\dcrossbr{}{i}};
   (3,57)*{\scs j};
\endxy}
\end{align*}
\begin{align*}
\phi_2
&=
-t_{ji}^2\delta_{jj'}\vcenter{\xy 0;/r.12pc/:
    (0,24)*{\slcup{}};
    (0,16)*{\lbigcrossrr{}{}};
    (0,9)*{\srcup{}}; (0,14)*{\slcap{}};
    (0,-1)*{\rbigcrossrr{i}{i}}; (0,-3.5)*{\srcap{}};
    (2.5,7)*{\scs j}; (2.5,30)*{\scs j}; (-2.5,-10)*{\scs j};
    \endxy}
+\vcenter{\xy 0;/r.11pc/:
    (0,40)*{\lcrossbp{}{}}; (9,40)*{\slinenr{}}; (-9,40)*{\slineur{}};
    (-6,32)*{\ncrossbr{}{}};(6,32)*{\ncrossrp{}{}};
    (0,24)*{\lcrossrr{}{}}; (9,24)*{\slinenp{}}; (-9,24)*{\slined{}};
    (0,16)*{\rcrossrr{}{}}; (9,16)*{\slineup{}}; (-9,16)*{\slinen{}};
    (-6,8)*{\ncrossrb{}{}};(6,8)*{\ncrosspr{}{}};
    (0,-1)*{\rcrosspb{j'}{j}}; (9,-1)*{\slinedr{i}}; (-9,-1)*{\slinenr{i}}; \endxy}
+(-1)^{\delta_{jj'}}
\vcenter{\xy 0;/r.11pc/:
    (9,40)*{\slinenr{}};(-9,40)*{\slineur{}};
    (-3,40)*{\slineup{}};(3,40)*{\slinen{}};
    (9,32)*{\slinenr{}};(-9,32)*{\slinenr{}};
    (-3,32)*{\slinenp{}};(3,32)*{\slinen{}};    (9,24)*{\slinenr{}};(-9,24)*{\slinenr{}};
    (-3,24)*{\slinenp{}};(3,24)*{\slinen{}};    (9,16)*{\slinenr{}};(-9,16)*{\slinenr{}};
    (-3,16)*{\slinenp{}};(3,16)*{\slinen{}};
    (9,8)*{\slinenr{}};(-9,8)*{\slinenr{}};
    (-3,8)*{\slinenp{}};(3,8)*{\slinen{}};
    (-3,-1)*{\slinenp{j'}};(3,-1)*{\slined{j}}; (9,-1)*{\slinedr{i}}; (-9,-1)*{\slinenr{i}};\endxy}
-\delta_{jj'}t_{ij}
\sum_{\substack{a+b+c\\ =\lambda_j-1}}
\sum_{h=0}^c (-v_{ij})^{-\l_i-h}
\vcenter{\xy 0;/r.14pc/:
    (0,24)*{\lcupcuprb{}{}};
    (0,24)*{\scs \bullet};(0,26.5)*{\scs a};
    (17.5,20)*{\ccbub{\scriptscriptstyle\spadesuit + c-h}{j}};
    (12,5,6)*{\ccbubr{\scriptscriptstyle\spadesuit + h}{i}};
    (0,-.5)*{\rcapcaprb{}{}};
    (0,0)*{\scs \bullet};(0,-2.5)*{\scs b};
    (-3,-6)*{\scs j}; (-6,-6)*{\scs i}; (3,30)*{\scs j}; (6,30)*{\scs i};
\endxy} \\
&=
\delta_{jj'}t_{ji}
\sum_{\substack{a+c=\\ \l_i+\l_j-1}}
\vcenter{\xy 0;/r.11pc/:
    (-6,57)*{\ucrossbr{}{}}; (9,57)*{\slinenr{}};(3,57)*{\slinen{}};
    (-6,49)*{\ucrossrb{}{}}; (9,49)*{\slinenr{}};(3,49)*{\slinen{}}; (0,43)*{\slcup{}};(0,42)*{\scs\bullet};
    (3,42)*{\scriptscriptstyle a};
    (-9,46)*{};(9,30)*{} **[black]\crv{(-9,38)& (9,38)}?(0)*[black]\dir{<};
    (9,46)*{};(-9,30)*{} **[black]\crv{(9,38)& (-9,38)}?(1)*[black]\dir{>};
    (0,25.5)*{\ccbub{\scriptscriptstyle\spadesuit+c}{j}};
    (9,25)*{\slinenr{}}; (-9,25)*{\slinenr{}};
    (-9,6)*{};(9,22)*{} **[black]\crv{(-9,14)& (9,14)}?(1)*[black]\dir{>};
    (9,6)*{};(-9,22)*{} **[black]\crv{(9,14)& (-9,14)}?(0)*[black]\dir{<};
    (0,8)*{\srcap{}};
    (-9,3)*{\scs i}; (-3,3)*{\scs j}; (9,3)*{\scs i}; (-3,65)*{\scs j};
    \endxy}
+t_{ij'}^{-1}(-1)^{\delta_{jj'}}
\sum_{\substack{d+e+f+g\\=-\l_i-2}}
(-v_{ij'})^{g}\;\;
\vcenter{\xy 0;/r.11pc/:
    (-6,40)*{\ucrosspr{}{}};(6,40)*{\dcrossrb{}{}};
    (-9,32)*{\slineup{}};(9,32)*{\slinen{}};
    (0,34)*{\slcupr{}};
    (0,33)*[black]{\scs \bullet};(-3,31)*{\scriptscriptstyle d};
    (9,24)*{\slinen{}};(-9,24)*{\slinenp{}};
    (20,22)*{\ccbubr{\scriptscriptstyle \spadesuit+f}{i}};
    (-13,16)*{\scriptscriptstyle g};(-9,16)*{\sdotp{}};
    (0,15.5)*[black]{\scs \bullet};(0,18)*{\scriptscriptstyle e};
    (0,14)*{\srcapr{}};(9,16)*{\slinen{}};
    (-6,8)*{\ucrossrp{}{}};
    (3,8)*{\slinedr{}};(9,8)*{\slined{}};
    (6,-1)*{\dcrossbr{j}{}};
    (-3,-1)*{\slineup{j'}}; (-9,-1)*{\slineur{i}};
    (9,48)*{\scs i};
    \endxy}
+\delta_{jj'}t_{ji}
\sum_{\substack{a+b+c= \\ \l_i+\l_j-1}}
\vcenter{\xy 0;/r.11pc/: (0,44)*{\slcup{}};(0,43)*{\scs\bullet};
    (3,42)*{\scriptscriptstyle a};
    (-9,46)*{};(9,30)*{} **[black]\crv{(-9,38)& (9,38)}?(0)*[black]\dir{<};
    (9,46)*{};(-9,30)*{} **[black]\crv{(9,38)& (-9,38)}?(1)*[black]\dir{>};
    (0,26)*{\ccbub{\scriptscriptstyle\spadesuit+c}{j}};
    (9,25)*{\slinenr{}}; (-9,25)*{\slinenr{}};
    (-9,5)*{};(9,22)*{} **[black]\crv{(-9,14)& (9,14)}?(1)*[black]\dir{>};
    (9,5)*{};(-9,22)*{} **[black]\crv{(9,14)& (-9,14)}?(0)*[black]\dir{<};
    (0,6.5)*{\srcap{}};(0,8)*{\scs\bullet};
    (0,4.5)*{\scriptscriptstyle b};(-9,0)*{\slinenr{}};
    (-3,0)*{\slinen{}}; (6,0)*{\dcrossrb{}{}};
    (6,-9)*{\dcrossbr{}{i}};
    (-3,-9)*{\slineu{j}};(-9,-9)*{\slineur{i}};
    (3,50)*{\scs j};
\endxy}\\
& \hspace{270pt}
+\delta_{jj'}t_{ij}
\sum_{\substack{a+b+c= \\ \l_i+\l_j-2}}
\vcenter{\xy 0;/r.11pc/: (0,44)*{\slcup{}};(0,43)*{\scs\bullet};
    (3,42)*{\scriptscriptstyle a};
    (-9,46)*{};(9,30)*{} **[black]\crv{(-9,38)& (9,38)}?(0)*[black]\dir{<};
    (9,46)*{};(-9,30)*{} **[black]\crv{(9,38)& (-9,38)}?(1)*[black]\dir{>};(-6,41)*[black]{\scs \bullet};
    (0,26)*{\ccbub{\scriptscriptstyle\spadesuit+c}{j}};
    (9,25)*{\slinenr{}}; (-9,25)*{\slinenr{}};
    (-9,5)*{};(9,22)*{} **[black]\crv{(-9,14)& (9,14)}?(1)*[black]\dir{>};
    (9,5)*{};(-9,22)*{} **[black]\crv{(9,14)& (-9,14)}?(0)*[black]\dir{<};
    (0,6.5)*{\srcap{}};(0,8)*{\scs\bullet};
    (0,4.5)*{\scriptscriptstyle b};(-9,0)*{\slinenr{}};
    (-3,0)*{\slinen{}}; (6,0)*{\dcrossrb{}{}};
    (6,-9)*{\dcrossbr{}{i}};
    (-3,-9)*{\slineu{j}};(-9,-9)*{\slineur{i}};
    (3,50)*{\scs j};
\endxy}
\end{align*}
\begin{align*}
\phi_3
&=
\delta_{jj'}t_{ji}\vcenter{\xy 0;/r.11pc/:
    (0,32)*{\lcrossrr{}{}}; (9,32)*{\slinen{}}; (-9,32)*{\slineu{}};
    (-6,24)*{\ncrossrb{}{}};(6,24)*{\ncrossbr{}{}};
    (0,16)*{\lcross{}{}}; (9,16)*{\slinenr{}}; (-9,16)*{\slinedr{}}; (0,10)*{\srcup{}};
    (0,2)*{\rbigcrossrr{i}{i}}; (0,-.5)*{\srcap{}};
    (9,40)*{\scs j}; (-2.5,-7)*{\scs j};
    \endxy}
-t_{ij}\vcenter{\xy 0;/r.11pc/:
    (0,35)*{\slcupr{}};
    (0,27)*{\lbigcrossbp{}{}}; (0,22.5)*{\slcapr{}};
    (0,16)*{\rcrossrr{}{}}; (9,16)*{\slineup{}}; (-9,16)*{\slinen{}};
    (-6,8)*{\ncrossrb{}{}};(6,8)*{\ncrosspr{}{}};
    (0,-1)*{\rcrosspb{j'}{j}}; (9,0)*{\slinedr{}}; (-9,-1)*{\slinenr{i}}; (3,41)*{\scs i};
    \endxy}
-\delta_{jj'}t_{ij}
\sum_{\substack{a+b+c\\ =\lambda_j-1}}
\sum_{h=0}^c (-v_{ij})^{-\l_i-h}
\vcenter{\xy 0;/r.14pc/:
    (0,24)*{\lcupcupbr{}{}};
    (0,22)*{\scs \bullet};(0,18)*{\scs a};
    (17.5,20)*{\ccbub{\scriptscriptstyle\spadesuit + c-h}{j}};
    (12.5,6)*{\ccbubr{\scriptscriptstyle\spadesuit + h}{i}};
    (0,-.5)*{\rcapcaprb{}{}};
    (0,0)*{\scs \bullet};(0,-4)*{\scs b};
    (-3,-6)*{\scs j}; (-6,-6)*{\scs i}; (3,30)*{\scs i}; (6,30)*{\scs j};
\endxy}\\
&=
\delta_{jj'}t_{ji}
\sum_{\substack{a+c=\\\l_i+\l_j-1}}
\vcenter{\xy 0;/r.11pc/:
    (-6,57)*{\ucrossrb{}{}};(3,57)*{\slinedr{}};(9,57)*{\slined{}};
    (6,49)*{\dcrossbr{}{}};(-3,49)*{\slinen{}};(-9,49)*{\slinenr{}};
    (0,43)*{\slcup{}};(0,42)*{\scs\bullet};
    (3,41.5)*{\scriptscriptstyle a};
    (-9,46)*{};(9,30)*{} **[black]\crv{(-9,38)& (9,38)}?(0)*[black]\dir{<};
    (9,46)*{};(-9,30)*{} **[black]\crv{(9,38)& (-9,38)}?(1)*[black]\dir{>};
    (0,25.5)*{\ccbub{\scriptscriptstyle\spadesuit+c}{j}};
    (9,25)*{\slinenr{}}; (-9,25)*{\slinenr{}};
    (-9,6)*{};(9,22)*{} **[black]\crv{(-9,14)& (9,14)}?(1)*[black]\dir{>};
    (9,6)*{};(-9,22)*{} **[black]\crv{(9,14)& (-9,14)}?(0)*[black]\dir{<};
    (0,8)*{\srcap{}};
    (-9,3)*{\scs i}; (-3,3)*{\scs j}; (9,3)*{\scs i}; (9,65)*{\scs j};
    \endxy}
-t_{ij}t_{ij'}^{-1}(-1)^{\delta_{jj'}}
\sum_{\substack{e+f+g\\ = -\l_i-1}}
(-v_{ij'})^{g}\;
\vcenter{\xy 0;/r.13pc/:
    (-9,32)*{\slineup{}};(9,32)*{\slinen{}};
    (0,34)*{\slcupr{}};
    (9,24)*{\slinen{}};(-9,24)*{\slinenp{}};
    (18,16)*{\ccbubr{\scriptscriptstyle \spadesuit+f}{i}};
    (-13,16)*{\scriptscriptstyle g};(-9,16)*{\sdotp{}};
    (0,15.5)*[black]{\scs \bullet};(0,18)*{\scriptscriptstyle e};
    (0,14)*{\srcapr{}};(9,16)*{\slinen{}};
    (-6,8)*{\ucrossrp{}{}};(3,8)*{\slinenr{}};
    (9,8)*{\slinen{}};(6,-1)*{\dcrossbr{j}{i}};
    (-3,-1)*{\slineup{j'}};(-9,0)*{\slineur{}};
    (3,40)*{\scs i};
\endxy}
-\delta_{jj'}t_{ij}\sum_{\substack{a+b+c= \\ \l_i+\l_j-2}}
\vcenter{\xy 0;/r.11pc/:
    (0,44)*{\lcupcupbr{}{}};(0,42)*{\scs\bullet};
    (3,41)*{\scriptscriptstyle a};
    (0,35)*{\lllcapr{}};
    (1,26)*{\ccbub{\scriptscriptstyle\spadesuit+c\;\;}{j}};
    (9,25)*{\slinenr{}}; (-9,25)*{\slinenr{}};
    (-9,5)*{};(9,22)*{} **[black]\crv{(-9,14)& (9,14)}?(1)*[black]\dir{>};
    (9,5)*{};(-9,22)*{} **[black]\crv{(9,14)& (-9,14)}?(0)*[black]\dir{<};
    (0,7)*{\srcap{}};(0,9)*{\scs\bullet};
    (-2,10.5)*{\scriptscriptstyle b};(-9,0)*{\slinenr{}};
    (-3,0)*{\slinen{}};(6,0)*{\dcrossrb{}{}};
    (-9,-9)*{\slineur{i}};
    (-3,-9)*{\slineu{j}};(6,-8)*{\dcrossbr{}{}};
    (3,51)*{\scs i}; (6,51)*{\scs j};
\endxy}
\end{align*}
\begin{align*}
\phi_4
&=
t_{ij'}\vcenter{\xy 0;/r.11pc/:
    (0,32)*{\lcrossbp{}{}}; (9,32)*{\slinenr{}}; (-9,32)*{\slineur{}};
    (-6,24)*{\ncrossbr{}{}};(6,24)*{\ncrossrp{}{}};
    (0,16)*{\lcrossrr{}{}}; (9,16)*{\slinenp{}}; (-9,16)*{\slined{}}; (0,10)*{\srcupr{}};
    (-.2,1.5)*{\rbigcrosspb{j'}{j}}; (0,-1)*{\srcapr{}};
    (9,40)*{\scs i}; (-2.5,-7.5)*{\scs i};
    \endxy}
-\delta_{jj'}t_{ji}\vcenter{\xy 0;/r.11pc/:
    (0,34)*{\slcup{}};
    (0,26)*{\lbigcrossrr{}{}}; (0,22.5)*{\slcap{}};
    (0,16)*{\rcross{}{}}; (9,16)*{\slineur{}}; (-9,16)*{\slinenr{}};
    (-6,8)*{\ncrossbr{}{}};(6,8)*{\ncrossrb{}{}};
    (0,-1)*{\rcrossrr{i}{i}}; (9,-.25)*{\slined{}}; (-9,-1)*{\slinen{j}};
    (3,40)*{\scs j};
\endxy}
+\delta_{jj'}t_{ij}
\sum_{\substack{a+b+c\\ =\lambda_j-1}}
\sum_{h=0}^c (-v_{ij})^{-\l_i-h}
\vcenter{\xy 0;/r.14pc/:
    (0,24)*{\lcupcuprb{}{}};
    (0,24.5)*{\scs \bullet};(0,27)*{\scs a};
    (17.5,20)*{\ccbub{\scriptscriptstyle\spadesuit + c-h}{j}};
    (12.5,6)*{\ccbubr{\scriptscriptstyle\spadesuit + h}{i}};
    (0,-.5)*{\rcapcapbr{}{}};
    (0,2)*{\scs \bullet};(0,5)*{\scs b};
    (-3,-6)*{\scs i}; (-6,-6)*{\scs j}; (3,30)*{\scs j}; (6,30)*{\scs i};
\endxy}
    \\
&=
(-1)^{\delta_{jj'}}
\left(\sum_{\substack{d+e+f=\\-\l_i-1}}
\vcenter{\xy 0;/r.12pc/:
    (-6,32)*{\ucrosspr{}{}};(3,32)*{\slined{}};(9,32)*{\slinedr{}};
    (6,24)*{\dcrossrb{}{}};(-3,24)*{\slinenr{}};(-9,24)*{\slinenp{}};
    (0,18)*{\slcupr{}};(0,17.5)*[black]{\scs\bullet};
    (0,21)*{\scriptscriptstyle d};
    (-9,16)*{\slinenp{}};(9,16)*{\slinen{}};
    (-9,7)*{\slinenp{j'}};
    (20,20)*{\ccbubr{\scriptscriptstyle\spadesuit+f}{i}};
    (0,7)*{\srcapr{}};(0,8.5)*[black]{\scs\bullet};
    (0,11)*{\scriptscriptstyle e};(9,7)*{\slined{j}};
    (-3,3)*{\scs i}; (9,40)*{\scs i};
\endxy}
- t_{ij'}^{-1}
\sum_{\substack{d+e+f+g\\=-\l_i-2}}
(-v_{ij'})^{g} \;
\vcenter{\xy 0;/r.12pc/:
    (-6,33)*{\ucrosspr{}{}};(3,33)*{\slined{}};(9,33)*{\slinedr{}};
    (6,25)*{\dcrossrb{}{}};(-3,25)*{\slinenr{}};
    (-9,23)*{\slinenp{}}; (-9,25)*{\slinenp{}};
(-3,22);(3,22) **[black]\crv{(-3,18) & (3,18)};
    (0,19)*[black]{\scs\bullet};
    (-.5,22)*{\scriptscriptstyle d};
    (-9,15)*{\sdotp{}};(9,15)*{\slinen{}}; (9,17)*{\slinen{}};
    (-13,16)*{\scriptscriptstyle g}; (-6,7)*{\ucrossrp{}{}}; (-6,-2)*{\ucrosspr{j'}{i}};
    (20,20)*{\ccbubr{\scriptscriptstyle\spadesuit+f}{i}};
    (-3,12);(3,12) **[black]\crv{(-3,16) & (3,16)};
(0,15)*[black]{\scs\bullet};(9,7)*{\slinen{}};(3,7)*{\slinenr{}};(3,0)*{\slinedr{}};
    (.5,12)*{\scriptscriptstyle e};(9,0)*{\slined{}};
    (9,-6)*{\scs j}; (9,41)*{\scs i};
\endxy}
\right)
+ \delta_{jj'}t_{ij}
\sum_{\substack{a+b+c= \\ \l_i+\l_j-2}}
\vcenter{\xy 0;/r.11pc/:
    (-6,57)*{\ucrossbr{}{}};(3,57)*{\slinen{}};(9,57)*{\slinenr{}};
    (-6,49)*{\ucrossrb{}{}};(3,49)*{\slinen{}};(9,49)*{\slinenr{}};
    (0,43)*{\slcup{}};(0,42)*{\scs\bullet};
    (3,41)*{\scriptscriptstyle a};
    (-9,46)*{};(9,30)*{} **[black]\crv{(-9,38)& (9,38)}?(0)*[black]\dir{<};
    (9,46)*{};(-9,30)*{} **[black]\crv{(9,38)& (-9,38)}?(1)*[black]\dir{>};
    (9,25)*{\slinenr{}}; (-9,25)*{\slinenr{}};
    (1,25)*{\ccbub{\scriptscriptstyle\spadesuit+c\;\;}{j}};
    (0,16)*{\llrcupr{}};
    (0,5)*{\rcapcapbr{}{}};
    (0,8)*{\scs\bullet};(3,10)*{\scriptscriptstyle b};
    (3,65)*{\scs j}; (9,65)*{\scs i}; (3,-.5)*{\scs i}; (6,-.5)*{\scs j};
\endxy} \\
&\hspace{90pt}
-\delta_{jj'}t_{ji}
\sum_{\substack{a+b+c=\\ \l_i+\l_j-1}}
\vcenter{\xy 0;/r.11pc/:
    (0,43)*{\slcup{}};(0,42)*{\scs \bullet};
    (3,41)*{\scriptscriptstyle a};
    (-9,46)*{};(9,30)*{} **[black]\crv{(-9,38)& (9,38)}?(0)*[black]\dir{<};
    (9,46)*{};(-9,30)*{} **[black]\crv{(9,38)& (-9,38)}?(1)*[black]\dir{>};
    (9,25)*{\slinenr{}}; (-9,25)*{\slinenr{}};
    (0,26)*{\ccbub{\scriptscriptstyle\spadesuit+c\;}{j}};
    (-9,5)*{};(9,22)*{} **[black]\crv{(-9,14)& (9,14)}?(1)*[black]\dir{>};
    (9,5)*{};(-9,22)*{} **[black]\crv{(9,14)& (-9,14)}?(0)*[black]\dir{<};
    (0,7)*{\srcap{}};(0,8.5)*{\scs\bullet};
    (0,5.5)*{\scriptscriptstyle b};(6,0)*{\dcrossrb{}{}};
    (-3,0)*{\slinen{}};(-9,0)*{\slinenr{}};
    (-6,-9)*{\ucrossbr{j}{i}};
    (3,-9)*{\slinedr{i}};(9,-8)*{\slined{}};
    (3,49)*{\scs j};
\endxy}
+\delta_{jj'}\sum_{\substack{a+b+c= \\ \l_i+\l_j-2}}
\vcenter{\xy 0;/r.11pc/:
    (-6,57)*{\ucrossbr{}{}}; (9,57)*{\slinenr{}};(3,57)*{\slinen{}};
    (-6,49)*{\ucrossrb{}{}}; (9,49)*{\slinenr{}};(3,49)*{\slinen{}};
    (0,43)*{\slcup{}};(0,42)*{\scs\bullet};
    (0,46)*{\scriptscriptstyle a};
    (-9,46)*{};(9,30)*{} **[black]\crv{(-9,38)& (9,38)}?(0)*[black]\dir{<};
    (9,46)*{};(-9,30)*{} **[black]\crv{(9,38)& (-9,38)}?(1)*[black]\dir{>};
    (1,26)*{\ccbub{\scriptscriptstyle\spadesuit+c\;}{j}};
    (9,25)*{\slinenr{}}; (-9,25)*{\slinenr{}};
    (-9,5)*{};(9,22)*{} **[black]\crv{(-9,14)& (9,14)}?(1)*[black]\dir{>};
    (9,5)*{};(-9,22)*{} **[black]\crv{(9,14)& (-9,14)}?(0)*[black]\dir{<};
    (0,7)*{\srcap{}};(0,8.5)*{\scs\bullet};
    (0,4)*{\scriptscriptstyle b};
    (-6,-1)*{\ucrossbr{j}{i}}; (6,-1)*{\dcrossrb{i}{}};
    (3,65)*{\scs j};
\endxy}
-\delta_{jj'}t_{ij}
\sum_{\substack{a+b+c=\\ \l_i+\l_j-2}}
\vcenter{\xy 0;/r.11pc/:
    (0,43)*{\slcup{}};(0,42)*{\scs \bullet};
    (3,41)*{\scriptscriptstyle a}; (-6,41)*[black]{\scs \bullet};
    (-9,46)*{};(9,30)*{} **[black]\crv{(-9,38)& (9,38)}?(0)*[black]\dir{<};
    (9,46)*{};(-9,30)*{} **[black]\crv{(9,38)& (-9,38)}?(1)*[black]\dir{>};
    (9,25)*{\slinenr{}}; (-9,25)*{\slinenr{}};
    (0,26)*{\ccbub{\scriptscriptstyle\spadesuit+c\;}{j}};
    (-9,5)*{};(9,22)*{} **[black]\crv{(-9,14)& (9,14)}?(1)*[black]\dir{>};
    (9,5)*{};(-9,22)*{} **[black]\crv{(9,14)& (-9,14)}?(0)*[black]\dir{<};
    (0,7)*{\srcap{}};(0,8.5)*{\scs\bullet};
    (0,5.5)*{\scriptscriptstyle b};(6,0)*{\dcrossrb{}{}};
    (-3,0)*{\slinen{}};(-9,0)*{\slinenr{}};
    (-6,-9)*{\ucrossbr{j}{i}};
    (3,-9)*{\slinedr{i}};(9,-8)*{\slined{}};
    (3,49)*{\scs j};
\endxy}
\end{align*}
\begin{align*}
\phi_5
&=
-t_{ij}t_{ij'}\vcenter{\xy 0;/r.11pc/:
    (0,24)*{\slcupr{}};
    (0.2,16.5)*{\lbigcrossbp{}{}};
    (0,9)*{\srcupr{}}; (0,14)*{\slcapr{}};
    (0,-1)*{\rbigcrosspb{j'}{j}}; (0,-3.5)*{\srcapr{}};
    (2.5,30)*{\scs i}; (2.5,7)*{\scs i}; (-2.5,-9)*{\scs i};
    \endxy}
+\vcenter{\xy 0;/r.11pc/:
    (0,40)*{\lcrossrr{}{}}; (9,40)*{\slinen{}}; (-9,40)*{\slineup{}};
    (-6,32)*{\ncrossrp{}{}};(6,32)*{\ncrossbr{}{}};
    (0,24)*{\lcrossbp{}{}}; (9,24)*{\slinenr{}}; (-9,24)*{\slinedr{}};
    (0,16)*{\rcrosspb{}{}}; (9,16)*{\slineur{}}; (-9,16)*{\slinenr{}};
    (-6,8)*{\ncrosspr{}{}};(6,8)*{\ncrossrb{}{}};
    (0,-1)*{\rcrossrr{i}{i}}; (9,-1)*{\slined{j}}; (-9,-1)*{\slinenp{j'}}; \endxy}
+(-1)^{\delta_{jj'}}
\vcenter{\xy 0;/r.11pc/:
    (9,40)*{\slinen{}};(-9,40)*{\slineup{}};
    (-3,40)*{\slineur{}};(3,40)*{\slinenr{}};
    (9,32)*{\slinen{}};(-9,32)*{\slinenp{}};
    (-3,32)*{\slinenr{}};(3,32)*{\slinenr{}};    (9,24)*{\slinen{}};(-9,24)*{\slinenp{}};
    (-3,24)*{\slinenr{}};(3,24)*{\slinenr{}};    (9,16)*{\slinen{}};(-9,16)*{\slinenp{}};
    (-3,16)*{\slinenr{}};(3,16)*{\slinenr{}};
    (9,8)*{\slinen{}};(-9,8)*{\slinenp{}};
    (-3,8)*{\slinenr{}};(3,8)*{\slinenr{}};
    (-3,-1)*{\slinenr{i}};(3,-1)*{\slinedr{i}}; (9,-1)*{\slined{j}}; (-9,-1)*{\slinenp{j'}};\endxy}
+\delta_{jj'}t_{ij}
\sum_{\substack{a+b+c\\ =\lambda_j-1}}
\sum_{h=0}^c (-v_{ij})^{-\l_i-h}
\vcenter{\xy 0;/r.14pc/:
    (0,24)*{\lcupcupbr{}{}};
    (0,22)*{\scs \bullet};(0,19.5)*{\scs a};
    (17.5,20)*{\ccbub{\scriptscriptstyle\spadesuit + c-h}{j}};
    (12.5,6)*{\ccbubr{\scriptscriptstyle\spadesuit + h}{i}};
    (0,-.5)*{\rcapcapbr{}{}};
    (0,2)*{\scs \bullet};(0,5)*{\scs b};
    (-3,-6)*{\scs i}; (-6,-6)*{\scs j}; (3,30)*{\scs i}; (6,30)*{\scs j};
\endxy} \\
&=
t_{ij}t_{ij'}^{-1}(-1)^{\delta_{jj'}}
\sum_{\substack{d+f+g=\\-\l_i-1}}
(-v_{ij'})^{g}\;
\vcenter{\xy 0;/r.12pc/:
    (-9,32)*{\slineup{}};(9,32)*{\slinen{}};
    (0,34)*{\slcupr{}};
    (9,24)*{\slinen{}};(-9,24)*{\slinenp{}};
    (21,16)*{\ccbubr{\scriptscriptstyle \spadesuit+f}{i}};
    (-13,16)*{\scriptscriptstyle g};(-9,16)*{\sdotp{}};
    (0,16.5)*[black]{\scs \bullet};(0,19)*{\scriptscriptstyle e};
    (0,15)*{\srcapr{}};(9,16)*{\slined{}};
    (-6,-1)*{\ucrosspr{j'}{i}};(-6,8)*{\ucrossrp{}{}};
    (3,8)*{\slinenr{}};(9,8)*{\slinen{}};
    (3,0)*{\slinedr{}};(9,-1)*{\slined{j}};
    (3,39)*{\scs i};
    \endxy}
+(-1)^{\delta_{jj'}}
\sum_{\substack{d+e+f\\=-\l_i-1}}
\vcenter{\xy 0;/r.12pc/:
    (-9,32)*{\slineup{}};(-3,32)*{\slineur{}};
    (6,32)*{\dcrossbr{}{}};
    (0,18)*{\slcupr{}};(0,17)*[black]{\scs\bullet};
    (0,14)*{\scriptscriptstyle d};
    (-9,24)*{\slinenp{}}; (-3,24)*{\slinenr{}};
    (-9,16)*{\slinenp{}};(6,24)*{\dcrossrb{}{}};(9,16)*{\slinen{}};
    (20,16)*{\ccbubr{\scriptscriptstyle \spadesuit+f}{i}};
    (-9,8)*{\slinenp{}};(9,8)*{\slinen{}};
    (-9,-1)*{\slinenp{j'}};(9,-1)*{\slined{j}};
    (0,0)*{\srcapr{}};(0,1)*[black]{\scs\bullet};
    (0,4)*{\scriptscriptstyle e};
    (3,39)*{\scs i}; (-3,-4)*{\scs i};
    \endxy}
+\delta_{jj'}\sum_{\substack{a+b+c= \\ \l_i+\l_j-2}}
\vcenter{\xy 0;/r.11pc/:
    (-6,49)*{\ucrossrb{}{}}; (6,49)*{\dcrossbr{}{}};
    (0,43)*{\slcup{}};(0,42)*{\scs\bullet};
    (3,41.5)*{\scriptscriptstyle a};
    (-9,46)*{};(9,30)*{} **[black]\crv{(-9,38)& (9,38)}?(0)*[black]\dir{<};
    (9,46)*{};(-9,30)*{} **[black]\crv{(9,38)& (-9,38)}?(1)*[black]\dir{>};
    (1,26)*{\ccbub{\scriptscriptstyle\spadesuit+c\;}{j}};
    (9,25)*{\slinenr{}}; (-9,25)*{\slinenr{}};
    (-9,5)*{};(9,22)*{} **[black]\crv{(-9,14)& (9,14)}?(1)*[black]\dir{>};
    (9,5)*{};(-9,22)*{} **[black]\crv{(9,14)& (-9,14)}?(0)*[black]\dir{<};
    (0,7)*{\srcap{}};(0,8.5)*{\scs\bullet};
    (3,9.5)*{\scriptscriptstyle b};
    (-6,-1)*{\ucrossbr{j}{i}}; (6,-1)*{\dcrossrb{i}{}};
    (9,57)*{\scs j};
\endxy} \\
&
\hspace{180pt}
+\delta_{jj'}t_{ij}
\sum_{\substack{a+b+c= \\ \l_i+\l_j-2}}
\vcenter{\xy 0;/r.11pc/:
    (0,44)*{\lcupcupbr{}{}};(0,42)*{\scs\bullet};
    (3,41)*{\scriptscriptstyle a};
    (0,35)*{\lllcapr{}};
    (1,26)*{\ccbub{\scriptscriptstyle\spadesuit+c\;\;}{j}};
    (9,25)*{\slinenr{}}; (-9,25)*{\slinenr{}};
    (-9,5)*{};(9,22)*{} **[black]\crv{(-9,14)& (9,14)}?(1)*[black]\dir{>};
    (9,5)*{};(-9,22)*{} **[black]\crv{(9,14)& (-9,14)}?(0)*[black]\dir{<};
    (0,7)*{\srcap{}};(0,9)*{\scs\bullet};
    (-2,10.25)*{\scriptscriptstyle b};(-9,0)*{\slinenr{}};
    (-3,0)*{\slinen{}};(6,0)*{\dcrossrb{}{}};(9,-8)*{\slined{}};
    (-6,-9)*{\ucrossbr{j}{i}};(3,-8)*{\slinedr{}};
    (3,51)*{\scs i}; (6,51)*{\scs j};
\endxy}
+\delta_{jj'}t_{ij}\sum_{\substack{a+b+c= \\ \l_i+\l_j-2}}
\vcenter{\xy 0;/r.11pc/:
    (-9,57)*{\slineu{}};(-3,57)*{\slineur{}};(6,57)*{\dcrossbr{}{}};
    (-6,49)*{\ucrossrb{}{}};(3,49)*{\slinen{}};(9,49)*{\slinenr{}};
    (0,43)*{\slcup{}};(0,42)*{\scs\bullet};
    (3,41)*{\scriptscriptstyle a};
    (-9,46)*{};(9,30)*{} **[black]\crv{(-9,38)& (9,38)}?(0)*[black]\dir{<};
    (9,46)*{};(-9,30)*{} **[black]\crv{(9,38)& (-9,38)}?(1)*[black]\dir{>};
    (9,25)*{\slinenr{}}; (-9,25)*{\slinenr{}};
    (1,25)*{\ccbub{\scriptscriptstyle\spadesuit+c\;\;}{j}};
    (0,16)*{\llrcupr{}};
    (0,5)*{\rcapcapbr{}{}};
    (0,8)*{\scs\bullet};(3,10)*{\scriptscriptstyle b};
    (3,65)*{\scs j}; (9,65)*{\scs i}; (3,-.5)*{\scs i}; (6,-.5)*{\scs j};
\endxy}
\end{align*}
\begin{align*}
\phi_6
&=
\vcenter{\xy 0;/r.11pc/:
    (0,40)*{\lcrossrp{}{}}; (9,40)*{\slinen{}};
    (-9,40)*{\slineur{}};
    (-6,32)*{\ncrossrr{}{}};(6,32)*{\ncrossbp{}{}};
    (0,24)*{\lcrossbr{}{}}; (9,24)*{\slinenp{}}; (-9,24)*{\slinedr{}};
    (0,16)*{\rcrossrb{}{}}; (9,16)*{\slineup{}}; (-9,16)*{\slinenr{}};
    (-6,8)*{\ncrossrr{}{}};(6,8)*{\ncrosspb{}{}};
    (0,-1)*{\rcrosspr{j'}{i}}; (9,-1)*{\slined{j}}; (-9,-1)*{\slinenr{i}};\endxy}
+(-1)^{\delta_{jj'}}
\vcenter{\xy 0;/r.11pc/:
    (9,40)*{\slinen{}};(-9,40)*{\slineur{}};
    (-3,40)*{\slineup{}};(3,40)*{\slinenr{}};
    (9,32)*{\slinen{}};(-9,32)*{\slinenr{}};
    (-3,32)*{\slinenp{}};(3,32)*{\slinenr{}};    (9,24)*{\slinen{}};(-9,24)*{\slinenr{}};
    (-3,24)*{\slinenp{}};(3,24)*{\slinenr{}};    (9,16)*{\slinen{}};(-9,16)*{\slinenr{}};
    (-3,16)*{\slinenp{}};(3,16)*{\slinenr{}};
    (9,8)*{\slinen{}};(-9,8)*{\slinenr{}};
    (-3,8)*{\slinenp{}};(3,8)*{\slinenr{}};
    (-3,-1)*{\slinenp{j'}};(3,-1)*{\slinedr{i}}; (9,-1)*{\slined{j}}; (-9,-1)*{\slinenr{i}};\endxy} \\
& = t_{ij}t_{ij'}^{-1}(-1)^{\delta_{jj'}}
\sum_{\substack{e+f+g\\=-\l_i-1}}
(-v_{ij'})^{g}\;\;
\vcenter{\xy 0;/r.11pc/:
    (-6,40)*{\ucrosspr{}{}};(3,40)*{\slinenr{}};(9,40)*{\slinen{}};
    (-9,32)*{\slineup{}};(9,32)*{\slinen{}};
    (0,34)*{\slcupr{}};
    (9,24)*{\slinen{}};(-9,24)*{\slinenp{}};
    (20,22)*{\ccbubr{\scriptscriptstyle \spadesuit+f}{i}};
    (-13,16)*{\scriptscriptstyle g};(-9,16)*{\sdotp{}};
    (0,16.5)*[black]{\scs \bullet};(2,18.5)*{\scriptscriptstyle e};
    (0,15)*{\srcapr{}};(9,16)*{\slinen{}};
    (-6,7)*{\ucrossrp{i}{j'}};
    (3,8)*{\slinedr{}};(9,7)*{\slined{j}};
    (3,48)*{\scs i};
    \endxy}
+t_{ij'}^{-1}(-1)^{\delta_{jj'}}
\sum_{\substack{d+e+f+g\\=-\l_i-2}}
(-v_{ij'})^{g}\;\;
\vcenter{\xy 0;/r.11pc/:
    (-9,48)*{\slineur{}};(-3,48)*{\slineup{}};
    (6,48)*{\dcrossbr{}{}};
    (-6,40)*{\ucrosspr{}{}};(6,40)*{\dcrossrb{}{}};
    (-9,32)*{\slineup{}};(9,32)*{\slinen{}};
    (0,34)*{\slcupr{}};
    (0,33)*[black]{\scs \bullet};(-3,31)*{\scriptscriptstyle d};
    (9,24)*{\slinen{}};(-9,24)*{\slinenp{}};
    (20,22)*{\ccbubr{\scriptscriptstyle \spadesuit+f}{i}};
    (-13,16)*{\scriptscriptstyle g};(-9,16)*{\sdotp{}};
    (0,16.5)*[black]{\scs \bullet};(2,18.5)*{\scriptscriptstyle e};
    (0,15)*{\srcapr{}};(9,16)*{\slinen{}};
    (-6,7)*{\ucrossrp{i}{j'}};
    (3,8)*{\slinedr{}};(9,7)*{\slined{j}};
    (3,56)*{\scs i};
    \endxy}  \\
& \hspace{90pt}
+\delta_{jj'}t_{ji}
\sum_{\substack{a+b+c= \\ \l_i+\l_j-1}}
\vcenter{\xy 0;/r.11pc/:
    (-9,49)*{\slineur{}};(-3,49)*{\slineu{}}; (6,49)*{\dcrossbr{}{}};
    (0,43)*{\slcup{}};(0,42)*{\scs\bullet};
    (3,41)*{\scriptscriptstyle a};
    (-9,46)*{};(9,30)*{} **[black]\crv{(-9,34)& (9,38)}?(0)*[black]\dir{<};
    (9,46)*{};(-9,30)*{} **[black]\crv{(9,34)& (-9,38)}?(1)*[black]\dir{>};
    (1,26)*{\ccbub{\scriptscriptstyle\spadesuit+c\;\;}{j}};
    (9,25)*{\slinenr{}}; (-9,25)*{\slinenr{}};
    (-9,5)*{};(9,22)*{} **[black]\crv{(-9,14)& (9,14)}?(1)*[black]\dir{>};
    (9,5)*{};(-9,22)*{} **[black]\crv{(9,14)& (-9,14)}?(0)*[black]\dir{<};
    (0,7)*{\srcap{}};(0,8.5)*{\scs\bullet};
    (0,5)*{\scriptscriptstyle b};(-9,-1)*{\slinenr{i}};
    (-3,-1)*{\slinen{j}}; (6,-1)*{\dcrossrb{i}{}};
    (9,57)*{\scs j};
\endxy}
+\delta_{jj'}t_{ij}
\sum_{\substack{a+b+c= \\ \l_i+\l_j-2}}
\vcenter{\xy 0;/r.11pc/:
    (-9,49)*{\slineur{}};(-3,49)*{\slineu{}}; (6,49)*{\dcrossbr{}{}};
    (0,43)*{\slcup{}};(0,42)*{\scs\bullet}; (-6,39.5)*[black]{\scs \bullet};
    (3,41)*{\scriptscriptstyle a};
    (-9,46)*{};(9,30)*{} **[black]\crv{(-9,34)& (9,38)}?(0)*[black]\dir{<};
    (9,46)*{};(-9,30)*{} **[black]\crv{(9,34)& (-9,38)}?(1)*[black]\dir{>};
    (1,26)*{\ccbub{\scriptscriptstyle\spadesuit+c\;\;}{j}};
    (9,25)*{\slinenr{}}; (-9,25)*{\slinenr{}};
    (-9,5)*{};(9,22)*{} **[black]\crv{(-9,14)& (9,14)}?(1)*[black]\dir{>};
    (9,5)*{};(-9,22)*{} **[black]\crv{(9,14)& (-9,14)}?(0)*[black]\dir{<};
    (0,7)*{\srcap{}};(0,8.5)*{\scs\bullet};
    (0,5)*{\scriptscriptstyle b};(-9,-1)*{\slinenr{i}};
    (-3,-1)*{\slinen{j}}; (6,-1)*{\dcrossrb{i}{}};
    (9,57)*{\scs j};
\endxy}
+\delta_{jj'}t_{ij}
\sum_{\substack{a+b+c= \\ \l_i+\l_j-2}}
\vcenter{\xy 0;/r.11pc/:
    (-6,51)*{\ucrossbr{}{}};(3,51)*{\slinenr{}}; (9,51)*{\slinen{}};
    (0,45)*{\slcupr{}};
    (-9,48)*{};(9,48)*{} **[blue]\crv{(-9,38)& (9,38)}?(0)*[blue]\dir{<};
    (0,40.5)*{\scs\bullet}; (2,42.5)*{\scriptscriptstyle a};
    (0,34)*{\lllcapr{}};
    (0,25)*{\ccbub{\scriptscriptstyle\spadesuit+c\;}{j}};
    (9,24)*{\slinenr{}}; (-9,24)*{\slinenr{}};
    (-9,4)*{};(9,21)*{} **[black]\crv{(-9,13)& (9,13)}?(1)*[black]\dir{>};
    (9,4)*{};(-9,21)*{} **[black]\crv{(9,13)& (-9,13)}?(0)*[black]\dir{<};
    (0,6)*{\srcap{}};(0,7.75)*{\scs\bullet};
    (0,4)*{\scriptscriptstyle b};(-9,-2)*{\slinenr{i}};
    (-3,-2)*{\slinen{j}};(6,-1)*{\dcrossrb{}{}};
    (3,59)*{\scs i}; (9,59)*{\scs j};
\endxy}
\end{align*}

The chain map given in equation \eqref{eq:j'jchainmap} is thus null-homotopic,
with homotopy given by
\[
{\xy 0;/r.15pc/:
	(-90,15)*+{\clubsuit \cal{E}_{j'} \cal{E}_i \cal{F}_j \cal{F}_i \onell{s_i(\lambda)} \la -1 \ra}="1";
	(0,15)*+{\cal{E}_{i} \cal{E}_{j'} \cal{F}_j \cal{F}_i \onell{s_i(\lambda)}
  		\oplus \cal{E}_{j'} \cal{E}_{i} \cal{F}_i \cal{F}_j \onell{s_i(\lambda)} }="2";
	(90,15)*+{\cal{E}_{i} \cal{E}_{j'} \cal{F}_i \cal{F}_j \onell{s_i(\lambda)} \la 1 \ra}="3";
	(-90,-15)*+{\clubsuit \cal{E}_{j'} \cal{E}_i \cal{F}_j \cal{F}_i \onell{s_i(\lambda)} \la -1 \ra}="4";
	(0,-15)*+{\cal{E}_{i} \cal{E}_{j'} \cal{F}_j \cal{F}_i \onell{s_i(\lambda)}
  		\oplus \cal{E}_{j'} \cal{E}_{i} \cal{F}_i \cal{F}_j \onell{s_i(\lambda)} }="5";
	(90,-15)*+{\cal{E}_{i} \cal{E}_{j'} \cal{F}_i \cal{F}_j \onell{s_i(\lambda)} \la 1 \ra}="6";
	{\ar  "1";"2"};
	{\ar  "2";"3"};
	{\ar  "4";"5"};
	{\ar  "5";"6"};
	{\ar@/^1pc/_-{\left(\begin{smallmatrix} h^1_1 & h^1_2 \end{smallmatrix}\right)}  "5";"1"};
	{\ar@/^1pc/_-{\left(\begin{smallmatrix} h^2_1 \\ h^2_2 \end{smallmatrix}\right)}  "6";"2"};
\endxy}
\]
where
\begin{align*}
h^1_1
&=
\delta_{jj'}t_{ji}
\sum_{\substack{a+c=\\ \l_i+\l_j-1}}
\vcenter{\xy 0;/r.11pc/:
    (-6,48)*{\ucrossrb{}{}}; (9,48)*{\slinenr{}};(3,48)*{\slinen{}};
    (0,42)*{\slcup{}};(0,41)*{\scs\bullet};
    (0,44)*{\scriptscriptstyle a};
    (-9,45)*{};(9,29)*{} **[black]\crv{(-9,37)& (9,37)}?(0)*[black]\dir{<};
    (9,45)*{};(-9,29)*{} **[black]\crv{(9,37)& (-9,37)}?(1)*[black]\dir{>};
    (0,25)*{\ccbub{\scriptscriptstyle\spadesuit+c}{}};
    (9,24)*{\slinenr{}}; (-9,24)*{\slinenr{}};
    (-9,5)*{};(9,21)*{} **[black]\crv{(-9,13)& (9,13)}?(1)*[black]\dir{>};
    (9,5)*{};(-9,21)*{} **[black]\crv{(9,13)& (-9,13)}?(0)*[black]\dir{<};
    (0,6.5)*{\srcap{}};
    (-9,2)*{\scs i}; (-3,2)*{\scs j}; (3,57)*{\scs j}; (9,2)*{\scs i};
\endxy}
    \\
h^1_2
&=
(-1)^{\delta_{jj'}}
\sum_{\substack{d+e+f\\=-\lambda_i-1}} \;
\vcenter{\xy 0;/r.12pc/:
    (-3,24)*{\slineur{}}; (6,24)*{\dcrossrb{}{}};
    (-9,24)*{\slineup{}};
    (-9,16)*{\slinenp{}};(9,16)*{\slinen{}};(0,18)*{\slcupr{}};
    (0,17.25)*[black]{\scs\bullet};(-3,14)*{\scs d};
    (20,13)*{\ccbubr{\scs\spadesuit+f}{i}};
    (-9,8)*{\slinenp{}};(9,8)*{\slined{}};(0,7.5)*{\srcapr{}};
    (0,9.25)*[black]{\scs\bullet};(3,11)*{\scs e};
    (-9,2)*{\scs j'}; (-3,2)*{\scs i}; (9,2)*{\scs j};  (9,32)*{\scs i};
\endxy}
+\delta_{jj'}\sum_{\substack{a+b+c= \\ \l_i+\l_j-2}}
\vcenter{\xy 0;/r.11pc/:
    (-6,48)*{\ucrossrb{}{}}; (9,48)*{\slinenr{}};(3,48)*{\slinen{}};
    (0,42)*{\slcup{}};(0,41)*{\scs\bullet};
    (0,44)*{\scriptscriptstyle a};
    (-9,45)*{};(9,29)*{} **[black]\crv{(-9,37)& (9,37)}?(0)*[black]\dir{<};
    (9,45)*{};(-9,29)*{} **[black]\crv{(9,37)& (-9,37)}?(1)*[black]\dir{>};
    (0,25)*{\ccbub{\scriptscriptstyle\spadesuit+c}{j}};
    (9,24)*{\slinenr{}}; (-9,24)*{\slinenr{}};
    (-9,5)*{};(9,21)*{} **[black]\crv{(-9,13)& (9,13)}?(1)*[black]\dir{>};
    (9,5)*{};(-9,21)*{} **[black]\crv{(9,13)& (-9,13)}?(0)*[black]\dir{<};
    (0,6.5)*{\srcap{}};(0,8.5)*{\scs\bullet};
    (0,5)*{\scriptscriptstyle b};
    (-6,0)*{\ucrossbr{}{}}; (6,0)*{\dcrossrb{}{}};
    (-9,-6)*{\scs j}; (-3,-6)*{\scs i}; (3,57)*{\scs j}; (3,-6)*{\scs i};
\endxy}
+\delta_{jj'}t_{ij}\sum_{\substack{a+b+c= \\ \l_i+\l_j-2}}
\vcenter{\xy 0;/r.11pc/:
    (-6,48)*{\ucrossrb{}{}}; (9,48)*{\slinenr{}};(3,48)*{\slinen{}};
    (0,42)*{\slcup{}};(0,41)*{\scs\bullet};
    (0,44)*{\scriptscriptstyle a};
    (-9,45)*{};(9,29)*{} **[black]\crv{(-9,37)& (9,37)}?(0)*[black]\dir{<};
    (9,45)*{};(-9,29)*{} **[black]\crv{(9,37)& (-9,37)}?(1)*[black]\dir{>};
    (0,25)*{\ccbub{\scriptscriptstyle\spadesuit+c}{}};
    (-9,29)*{}; (-9,22)*{} **\dir{-};
     (9,29)*{}; (9,22)*{} **\dir{-};
    (0,16)*{\llrcupr{}};
    (9,-3);(-9,-3) **[blue]\crv{(9,10) & (-9,10)}; ?(0)*[blue]\dir{<};
    (0,6.5)*{\scs\bullet};(0,10)*{\scriptscriptstyle b};
    (0,-1)*{\srcapr{}{}};
    (-9,-6)*{\scs j}; (-3,-6)*{\scs i}; (3,57)*{\scs j}; (9,57)*{\scs i};
\endxy}
\\
h^2_1
&=
-t_{ij'}^{-1}(-1)^{\delta_{jj'}}
\sum_{\substack{d+e+f+g\\=-\l_i-2}}
(-v_{ij'})^{g}\;\;
\vcenter{\xy 0;/r.11pc/:
    (-6,40)*{\ucrosspr{}{}};(6,40)*{\dcrossrb{}{}};
    (-9,32)*{\slineup{}};(9,32)*{\slinen{}};
    (0,34)*{\slcupr{}};
    (0,33)*[black]{\scs \bullet};(-3,31)*{\scs d};
    (9,24)*{\slinen{}};(-9,24)*{\slinenp{}};
    (20,22)*{\ccbubr{\scriptscriptstyle \spadesuit+f}{i}};
    (-13,16)*{\scs g};(-9,16)*{\sdotp{}};
    (0,16.5)*[black]{\scs \bullet};(2,18.5)*{\scs e};
    (0,15)*{\srcapr{}};(9,16)*{\slinen{}};
    (-6,7)*{\ucrossrp{i}{j'}};
    (3,8)*{\slinedr{}};(9,7)*{\slined{j}};
    (9,48)*{\scs i};
    \endxy}
-\delta_{jj'}t_{ji}
\sum_{\substack{a+b+c= \\ \l_i+\l_j-1}}
\vcenter{\xy 0;/r.11pc/:
    (0,43)*{\slcup{}};(0,42)*{\scs\bullet};
    (3,41)*{\scs a};
    (-9,46)*{};(9,30)*{} **[black]\crv{(-9,34)& (9,38)}?(0)*[black]\dir{<};
    (9,46)*{};(-9,30)*{} **[black]\crv{(9,34)& (-9,38)}?(1)*[black]\dir{>};
    (1,26)*{\ccbub{\scriptscriptstyle\spadesuit+c\;\;}{j}};
    (9,25)*{\slinenr{}}; (-9,25)*{\slinenr{}};
    (-9,5)*{};(9,22)*{} **[black]\crv{(-9,14)& (9,14)}?(1)*[black]\dir{>};
    (9,5)*{};(-9,22)*{} **[black]\crv{(9,14)& (-9,14)}?(0)*[black]\dir{<};
    (0,7)*{\srcap{}};(0,8.5)*{\scs\bullet};
    (0,5)*{\scs b};(-9,-1)*{\slinenr{i}};
    (-3,-1)*{\slinen{j}}; (6,-1)*{\dcrossrb{i}{}};
    (3,49)*{\scs j};
\endxy}
-\delta_{jj'}t_{ij}
\sum_{\substack{a+b+c= \\ \l_i+\l_j-2}}
\vcenter{\xy 0;/r.11pc/:
    (0,43)*{\slcup{}};(0,42)*{\scs\bullet}; (-6,39.5)*[black]{\scs \bullet};
    (3,41)*{\scs a};
    (-9,46)*{};(9,30)*{} **[black]\crv{(-9,34)& (9,38)}?(0)*[black]\dir{<};
    (9,46)*{};(-9,30)*{} **[black]\crv{(9,34)& (-9,38)}?(1)*[black]\dir{>};
    (1,26)*{\ccbub{\scriptscriptstyle\spadesuit+c\;\;}{j}};
    (9,25)*{\slinenr{}}; (-9,25)*{\slinenr{}};
    (-9,5)*{};(9,22)*{} **[black]\crv{(-9,14)& (9,14)}?(1)*[black]\dir{>};
    (9,5)*{};(-9,22)*{} **[black]\crv{(9,14)& (-9,14)}?(0)*[black]\dir{<};
    (0,7)*{\srcap{}};(0,8.5)*{\scs\bullet};
    (0,5)*{\scs b};(-9,-1)*{\slinenr{i}};
    (-3,-1)*{\slinen{j}}; (6,-1)*{\dcrossrb{i}{}};
    (3,49)*{\scs j};
\endxy} \\
h^2_2
&=
t_{ij}t_{ij'}^{-1}(-1)^{\delta_{jj'}}
\sum_{\substack{e+f+g\\=-\l_i-1}}
(-v_{ij'})^{g}\;\;
\vcenter{\xy 0;/r.12pc/:
    (-9,32)*{\slineup{}};(9,32)*{\slinen{}};
    (0,34)*{\slcupr{}};
    (9,24)*{\slinen{}};(-9,24)*{\slinenp{}};
    (20,22)*{\ccbubr{\scriptscriptstyle \spadesuit+f}{i}};
    (-13,16)*{\scs g};(-9,16)*{\sdotp{}};
    (0,16.5)*[black]{\scs \bullet};(2,18.5)*{\scs e};
    (0,15)*{\srcapr{}};(9,16)*{\slinen{}};
    (-6,7)*{\ucrossrp{i}{j'}};
    (3,8)*{\slinedr{}};(9,7)*{\slined{j}};
    (3,40)*{\scs i};
    \endxy}
+\delta_{jj'}t_{ij}
\sum_{\substack{a+b+c= \\ \l_i+\l_j-2}}
\vcenter{\xy 0;/r.11pc/:
    (0,45)*{\slcupr{}};
    (-9,48)*{};(9,48)*{} **[blue]\crv{(-9,38)& (9,38)}?(0)*[blue]\dir{<};
    (0,40.5)*{\scs\bullet}; (2,42.75)*{\scs a};
    (0,34)*{\lllcapr{}};
    (0,25)*{\ccbub{\scriptscriptstyle\spadesuit+c\;}{j}};
    (9,24)*{\slinenr{}}; (-9,24)*{\slinenr{}};
    (-9,4)*{};(9,21)*{} **[black]\crv{(-9,13)& (9,13)}?(1)*[black]\dir{>};
    (9,4)*{};(-9,21)*{} **[black]\crv{(9,13)& (-9,13)}?(0)*[black]\dir{<};
    (0,6)*{\srcap{}};(0,7.75)*{\scs\bullet};
    (0,4)*{\scs b};(-9,-2)*{\slinenr{i}};
    (-3,-2)*{\slinen{j}};(6,-1)*{\dcrossrb{}{}};
    (3,51)*{\scs i}; (9,51)*{\scs j};
\endxy}
\end{align*}
\end{proof}

It remains to verify the $FE$ version of equation \eqref{eq:EFj'j}.
We can proceed to compute as above, but in this case we can obtain the relation via
a trick using the symmetry $\omega$.
Indeed, note that, up to scalar factors, each map determining
$\cal{T}'_i \left( \hspace{-2.5pt} \xy 0;/r.15pc/:
(0,0)*{\rcrosspb{j'}{j}}; (6,3)*{ \lambda};
\endxy  \right)$ is given by applying $\omega$ to the corresponding component in
$\cal{T}'_i \left( \hspace{-2.5pt} \xy 0;/r.15pc/:
(0,0)*{\lcrossbp{j}{j'}}; (6,3)*{ \lambda};
\endxy  \right)$ and exchanging the roles of $j$ and $j'$.
Upon taking the composition, the discrepancies between the relevant scalars cancel,
and we find that the maps determining
$\cal{T}'_i \left( \xy 0;/r.18pc/: (0,4)*{\rcrosspb{}{}}; (0,-4.75)*{\lcrossbp{j}{j'}}; \endxy \right)$
are given by applying $\omega$ to the $\phi_i$'s.
Similarly, the other terms in the relation are obtained by those in equation \eqref{eq:EFj'j} via
$\omega$. It follows that we can ``apply $\omega$'' to the proof of
Proposition \ref{prop:extendsl2-jjp} (in weight $-\l$) to obtain the following.

\begin{cor}\label{cor:other-extended}
The relation
\[
\cal{T}'_i\left(-{\xy 0;/r.18pc/: (0,-4)*{\lcrosspb{j'}{j}}; (0,4)*{\rcrossbp{}{}}; \endxy}
+(-1)^{\delta_{jj'}}{\xy 0;/r.18pc/: (-3,-4)*{\slinedp{j'}}; (-3,4)*{\slinenp{}};
(3,-4)*{\slinen{j}}; (3,4)*{\slineu{}}; \endxy}
+\delta_{jj'}\sum_{\substack{a+b+c=\\ -\lambda_j-1}}
{\xy 0;/r.15pc/:
   (9,-8)*{\lcap{}}; (9,-6.5)*{\bullet}; (9,-3)*{\scs b};
   (17.5,0)*{\medcbub{\scs \spadesuit +c}{}};
   (9,9.5)*{\rcup{}}; (9,8)*{\bullet}; (9,5)*{\scs a};
   (13.5,-6)*{\scs j};(14,9)*{\scs j}; (21,8)*{\l};
\endxy}\right) \sim 0
\]
holds in $\Com(\Ucat_Q)$.
\end{cor}

\appendix

%
\section{Computation of \texorpdfstring{$\cal{T}'_{i,1}$}{T'i,1} for composite 2-morphisms } \label{sec:composite}
%

In light of Remark \ref{rem:presentation}, we can compute the value of $\cal{T}'_{i,1}$ on
downward dot and sideways and downward crossing $2$-morphisms in terms of the presentation of these $2$-morphisms in terms of
upward dot and crossing $2$-morphisms and cap/cup $2$-morphisms.
In Sections \ref{sec:downdot}, \ref{sec:sideways-crossing}, and \ref{sec:down-crossing}, we compute this value,
and in Section \ref{sec:Ti-bub}, we compute the value of $\cal{T}'_{i,1}$ on bubbles.
Throughout, we employ our conventions that $i\cdot j = -1 = i \cdot j'$ and $i \cdot k=0=i \cdot k'$,
but assume no other relation between $j$, $j'$, $k$, and $k'$.

%
\subsection{Value of \texorpdfstring{$\cal{T}'_{i,1}$}{T'i,1} for downward dot 2-morphisms} \label{sec:downdot}
%

We compute $\cal{T}'_{i,1}$ on downward dot 2-morphisms using the right cyclicity relation.
Each of the following is a direct consequence of the definitions in Sections \ref{def:updot} and \ref{def:capcup}.
\[
\cal{T}'_i \left(\xy 0;/r.18pc/:
 (0,0)*{\sdotdr{i}};
 (6,3)*{ \lambda};
 (-4,0)*{};(9,0)*{};
 \endxy \right)
 :=
\cal{T}'_i \left( \xy 0;/r.12pc/:
(0,0)*{\xy \sdotur{} \endxy};
(0,0)*{\xy \rcapr{} \endxy};
(8,0)*{\xy \slinenr{} \endxy};
(8,-8)*{\xy \slinedr{} \endxy};
(-8,-8)*{\xy \rcupr{} \endxy};
(-8,0)*{\xy \slinenr{} \endxy};
(-8,8)*{\xy \slinedr{} \endxy};
 (15,3)*{ \lambda};
 (-10,0)*{};(16,0)*{};
\endxy \right) \vspace{-0.1in}
=
 \vcenter{\xy 0;/r.17pc/:
 (0,15)*+{\cal{E}_i\onell{s_i(\lambda)}\la 2+\lambda_i \ra}="1";
 (0,-15)*+{\cal{E}_i\onell{s_i(\lambda)}\la \lambda_i \ra}="2";
 {\ar^{\xy
 (0,0)*{\sdotur{i}}; \endxy}"2";"1" };\endxy}
, \;\;
\cal{T}'_i \left(\xy 0;/r.18pc/:
  (0,0)*{\sdotdg{k}};
 (6,3)*{ \lambda};
 (-4,0)*{};(9,0)*{};
 \endxy \right)
 :=
\cal{T}'_i \left( \xy 0;/r.12pc/:
(0,0)*{\xy \sdotug{} \endxy};
(0,0)*{\xy \rcapg{} \endxy};
(8,0)*{\xy \slineng{} \endxy};
(8,-8)*{\xy \slinedg{} \endxy};
(-8,-8)*{\xy \rcupg{} \endxy};
(-8,0)*{\xy \slineng{} \endxy};
(-8,8)*{\xy \slinedg{} \endxy};
 (14,3)*{ \lambda};
 (-10,0)*{};(16,0)*{};
\endxy \right)
= \vspace{-0.1in}
\vcenter{\xy 0;/r.17pc/:
 (0,15)*+{\clubsuit \cal{F}_k\onell{s_i(\lambda)}\la2\ra}="1";
 (0,-15)*+{\clubsuit \cal{F}_k\onell{s_i(\lambda)}}="2";
 {\ar^{\xy (0,0)*{\sdotdg{k}};\endxy}"2";"1" };
  \endxy}
\]
\[
   \cal{T}'_i \left(\xy  0;/r.18pc/:
  (0,0)*{\sdotd{j}};
 (6,3)*{ \lambda};
 (-4,0)*{};(10,0)*{};
 \endxy \right)
 :=
\cal{T}'_i \left( \xy 0;/r.12pc/:
(0,0)*{\xy \sdotu{} \endxy};
(0,0)*{\xy \rcap{} \endxy};
(8,0)*{\xy \slinen{} \endxy};
(8,-8)*{\xy \slined{} \endxy};
(-8,-8)*{\xy \rcup{} \endxy};
(-8,0)*{\xy \slinen{} \endxy};
(-8,8)*{\xy \slined{} \endxy};
 (15,3)*{ \lambda};
 (-10,0)*{};(16,0)*{};
\endxy \right)
=
  \xy  0;/r.17pc/:
  (-25,15)*+{\cal{F}_j\cal{F}_i\onell{s_i(\lambda)}\la1\ra}="1";
  (-25,-15)*+{\cal{F}_j\cal{F}_i\onell{s_i(\lambda)}\la-1\ra}="2";
   {\ar^{\xy (-12,0)*{\sdotd{j}}; (-6,0)*{\slinedr{i}};  \endxy} "2";"1"};
  (25,15)*+{\clubsuit \cal{F}_i\cal{F}_j\onell{s_i(\lambda)}\la2\ra}="3";
  (25,-15)*+{\clubsuit \cal{F}_i\cal{F}_j\onell{s_i(\lambda)}}="4";
    {\ar_{\xy (-12,0)*{\slinedr{i}}; (-6,0)*{\sdotd{j}}; \endxy} "4";"3"};
   {\ar^{\xy (0,0)*{\dcrossbr{j}{i}};\endxy   } "1";"3"};
   {\ar^{\xy (0,0)*{\dcrossbr{j}{i}};\endxy   } "2";"4"};
 \endxy
\]
This agrees with the value in terms of left cyclicity, which is verified in Section \ref{proof:dotcyc}.

%
\subsection{Value of \texorpdfstring{$\cal{T}'_{i,1}$}{T'i,1} on sideways crossing 2-morphisms}
\label{sec:sideways-crossing}

We explicitly compute the value on sideways crossings in terms of the images of upward crossings, caps, and cups.
As above, each follows via a direct (by sometimes tedious) computation using the definitions in Sections \ref{def:upcross} and \ref{def:capcup}.
In the interest of space, we will omit displaying the domain and codomain of the image when they are 1-term compexes,
as, save for the relevant shifts, they can be read from the diagram.
\[
    \cal{T}'_i \left( \hspace{-0.1in}
    \xy 0;/r.18pc/:
  (0,0)*{\rcrossrr{i}{i}};
 (6,3)*{ \lambda};
 \endxy  \right)
:= - \hspace{-.25cm}\lcrossrr{i}{i}
, \quad
\cal{T}'_i \left(\hspace{-0.1in}  \xy 0;/r.18pc/:
  (0,0)*{\lcrossrr{i}{i}};
 (6,3)*{ \lambda};
 \endxy  \right)
:= - \hspace{-.25cm} \rcrossrr{i}{i}
 \]
\[
  \cal{T}'_i \left(\xy 0;/r.17pc/:
  (0,0)*{\rcrossrb{i}{j}};
 (6,3)*{ \lambda};
 \endxy  \right)
 :=
\cal{T}'_i \left( \xy 0;/r.12pc/:
(-9,0)*{\slinen{}};
(-9,8)*{\slinen{}};
(-6,-6)*{\srcup{}};
(6,7)*{\srcap{}};
(9,0)*{\slinen{}};
(9,-8)*{\slined{}};
(0,0)*{\ucrossbr{}{}};
(3,-8)*{\slinenr{}};
(-3,8)*{\slineur{}};
(-12,0)*{};(12,0)*{};
\endxy \right)=
   \xy 0;/r.17pc/:
  (-35,22)*+{\cal{F}_j\cal{F}_i\cal{F}_i\onell{s_i(\lambda)}\la-3 -\lambda_i \ra}="1";
  (-35,-22)*+{\cal{F}_i\cal{F}_j\cal{F}_i\onell{s_i(\lambda)}\la-4 -\lambda_i \ra}="2";
   {\ar^{\vcenter{\xy 0;/r.15pc/:
(0,10)*{\xy
(-14,4)*{t_{ij}t_{ji}};
 (-3,0)*{\dcrossrb{i}{j}};
  (6,0)*{\slinedr{i}};
 (3,8)*{\dcrossrr{}{}};
 (-6,8)*{\slinen{}};
(6,2)*[black]{\bullet};
 \endxy};
(0,-10)*{\xy
(-16,4)*{-t_{ij}t_{ji}};
 (-3,0)*{\dcrossrb{i}{j}};
  (6,0)*{\slinedr{i}};
 (3,8)*{\dcrossrr{}{}};
 (-6,8)*{\slinen{}};
(5,10.5)*[black]{\bullet};
 \endxy}; \endxy }} "2";"1"};
  (35,22)*+{\clubsuit\cal{F}_i\cal{F}_j\cal{F}_i\onell{s_i(\lambda)}\la-2-\lambda_i \ra}="3";
  (35,-22)*+{\clubsuit\cal{F}_i\cal{F}_i\cal{F}_j\onell{s_i(\lambda)}\la-3-\lambda_i \ra}="4";
   {\ar_{\vcenter{\xy 0;/r.15pc/:
(0,10)*{\xy
(-14,4)*{t_{ij}t_{ji}};
 (-3,0)*{\dcrossrr{i}{i}};
  (6,0)*{\slined{j}};
 (3,8)*{\dcrossrb{}{}};
 (-6,8)*{\slinenr{}};
(5,10.5)*[black]{\bullet};
 \endxy};
(0,-10)*{\xy
(-16,4)*{-t_{ij}t_{ji}};
 (-3,0)*{\dcrossrr{i}{i}};
  (6,0)*{\slined{j}};
 (3,8)*{\dcrossrb{}{}};
 (-6,8)*{\slinenr{}};
(-1,.5)*[black]{\bullet};
 \endxy}; \endxy }} "4";"3"};
   {\ar^-{\xy  0;/r.15pc/: (3,0)*{\slinedr{i}};(-6,0)*{\dcrossbr{j}{i}};\endxy   } "1";"3"};
   {\ar^-{\xy 0;/r.15pc/: (-9,0)*{\slinedr{i}};(-13,1)*{-};(0,0)*{\dcrossbr{j}{i}}; \endxy   } "2";"4"};
 \endxy
 \]
 \[
\cal{T}'_i \left(  \xy 0;/r.17pc/:
  (0,0)*{\lcrossbr{j}{i}};
 (6,3)*{ \lambda};
 \endxy  \right)
  :=
\cal{T}'_i \left( \xy 0;/r.12pc/:
(-3,-8)*{\slinenr{}};
(3,8)*{\slineur{}};
(0,0)*{\ucrossrb{}{}};
(9,0)*{\slinen{}};
(9,8)*{\slinen{}};
(6,-6)*{\slcup{}};
(-6,7)*{\slcap{}};
(-9,0)*{\slinen{}};
(-9,-8)*{\slined{}};
(-12,0)*{};(12,0)*{};
\endxy \right)
=
     \xy 0;/r.17pc/:
  (-35,18)*+{\cal{F}_i\cal{F}_j\cal{F}_i\onell{s_i(\lambda)}\la-4 -\lambda_i \ra}="1";
  (-35,-18)*+{\cal{F}_j\cal{F}_i\cal{F}_i\onell{s_i(\lambda)}\la -3 -\lambda_i \ra}="2";
   {\ar^{\vcenter{\xy 0;/r.16pc/:
 (3,0)*{\dcrossrr{i}{i}};
  (-6,0)*{\slined{j}};
 (-3,8.5)*{\dcrossbr{}{}};
 (6,8.5)*{\slinenr{}}; (-16,4)*{t_{ij}^{-2}t_{ji}^{-1}};
 \endxy}}  "2";"1"};
  (35,18)*+{\clubsuit\;\cal{F}_i\cal{F}_i\cal{F}_j\onell{s_i(\lambda)}
  \la -3-\lambda_i \ra}="3";
  (35,-18)*+{\clubsuit\;\cal{F}_i\cal{F}_j\cal{F}_i\onell{s_i(\lambda)}
  \la -2-\lambda_i \ra}="4";
  {\ar_{\vcenter{\xy 0;/r.16pc/:
    (3,0)*{\dcrossbr{j}{i}};
    (-6,0)*{\slinedr{i}};
    (-3,8.5)*{\dcrossrr{}{}};
    (6,8.5)*{\slinen{}}; (-17,4)*{-t_{ij}^{-2}t_{ji}^{-1}}; (-14,0)*{};
    \endxy}} "4";"3"};
   {\ar^-{\xy  0;/r.15pc/: (3,0)*{\slinedr{i}};(-6,0)*{\dcrossbr{j}{i}};\endxy   } "2";"4"};
   {\ar^-{\xy 0;/r.15pc/: (-9,0)*{\slinedr{i}};(0,0)*{\dcrossbr{j}{i}}; (-13,1)*{-}; \endxy   } "1";"3"};
 \endxy
 \]
 \[
\cal{T}'_i \left(  \xy 0;/r.17pc/:
  (0,0)*{\rcrossbr{j}{i}};
 (6,3)*{ \lambda};
 \endxy  \right)
  :=
\cal{T}'_i \left( \xy 0;/r.12pc/:
(-9,0)*{\slinenr{}};
(-9,8)*{\slinenr{}};
(-6,-6)*{\srcupr{}};
(6,7)*{\srcapr{}};
(9,0)*{\slinenr{}};
(9,-8)*{\slinedr{}};
(0,0)*{\ucrossrb{}{}};
(3,-8)*{\slinen{}};
(-3,8)*{\slineu{}};
(-12,0)*{};(12,0)*{};
\endxy \right)
=
     \xy 0;/r.17pc/:
  (-30,18)*+{\cal{E}_i\cal{E}_j\cal{E}_i\onell{s_i(\lambda)}\la \lambda_i-1 \ra}="1";
  (-30,-18)*+{\cal{E}_j\cal{E}_i\cal{E}_i\onell{s_i(\lambda)}\la \lambda_i \ra}="2";
   {\ar^{\vcenter{\xy 0;/r.16pc/:
 (3,0)*{\ucrossrr{i}{i}};
  (-6,0)*{\slinen{j}};(-12,3)*{t_{ij}^{-1}};
 (-3,8.5)*{\ucrossbr{}{}};
 (6,8.5)*{\slineur{}};
 \endxy}}  "2";"1"};
  (35,18)*+{\cal{E}_i\cal{E}_i\cal{E}_j\onell{s_i(\lambda)}\la \lambda_i \ra}="3";
  (35,-18)*+{\cal{E}_i\cal{E}_j\cal{E}_i\onell{s_i(\lambda)}\la \lambda_i+1 \ra}="4";
  {\ar_{\vcenter{\xy 0;/r.16pc/:
    (3,0)*{\ucrossbr{j}{i}};
    (-6,0)*{\slinenr{i}};
    (-3,8.5)*{\ucrossrr{}{}};
    (6,8.5)*{\slineu{}}; (-13,3)*{-t_{ij}^{-1}};
    \endxy}} "4";"3"};
   {\ar^-{\xy  0;/r.15pc/:
   (3,0)*{\slineur{i}};(-6,0)*{\ucrossbr{j}{i}};\endxy   } "2";"4"};
   {\ar^-{\xy 0;/r.15pc/:
   (-9,0)*{\slineur{i}};(0,0)*{\ucrossbr{j}{i}}; (-13,1)*{-}; \endxy   } "1";"3"};
 \endxy
 \]
 \[
\cal{T}'_i \left(  \xy 0;/r.17pc/:
  (0,0)*{\lcrossrb{i}{j}};
 (6,3)*{ \lambda};
 \endxy  \right)
 :=
\cal{T}'_i \left( \xy 0;/r.12pc/:
(-3,-8)*{\slinen{}};
(3,8)*{\slineu{}};
(0,0)*{\ucrossbr{}{}};
(9,0)*{\slinenr{}};
(9,8)*{\slinenr{}};
(6,-6)*{\slcupr{}};
(-6,7)*{\slcapr{}};
(-9,0)*{\slinenr{}};
(-9,-8)*{\slinedr{}};
(-12,0)*{};(12,0)*{};
\endxy \right)=
   \xy 0;/r.17pc/:
  (-30,22)*+{\cal{E}_j\cal{E}_i\cal{E}_i\onell{s_i(\lambda)}\la \lambda_i \ra}="1";
  (-30,-22)*+{\cal{E}_i\cal{E}_j\cal{E}_i\onell{s_i(\lambda)}\la \lambda_i-1 \ra}="2";
   {\ar^{\vcenter{\xy 0;/r.15pc/: (-15,0)*{};
(0,10)*{\xy
 (-3,0)*{\ucrossrb{i}{j}};
  (6,0)*{\slinenr{i}};
 (3,8)*{\ucrossrr{}{}};
 (-6,8)*{\slineu{}};
(1.5,10)*[black]{\bullet};
 \endxy};
(0,-10)*{\xy
(-12,4)*{-};
 (-3,0)*{\ucrossrb{i}{j}};
  (6,0)*{\slinenr{i}};
 (3,8)*{\ucrossrr{}{}};
 (-6,8)*{\slineu{}};
(-5,0)*[black]{\bullet};
 \endxy}; \endxy }} "2";"1"};
  (35,22)*+{\cal{E}_i\cal{E}_j\cal{E}_i\onell{s_i(\lambda)}\la \lambda_i+1 \ra}="3";
  (35,-22)*+{\cal{E}_i\cal{E}_i\cal{E}_j\onell{s_i(\lambda)}\la \lambda_i \ra}="4";
   {\ar_{\vcenter{\xy 0;/r.15pc/:
(0,10)*{\xy
 (-3,0)*{\ucrossrr{i}{i}};
  (6,0)*{\slinen{j}};
 (3,8)*{\ucrossrb{}{}};
 (-6,8)*{\slineur{}};
(-5,0)*[black]{\bullet};
 \endxy};
(0,-10)*{\xy
(-12,4)*{-};
 (-3,0)*{\ucrossrr{i}{i}};
  (6,0)*{\slinen{j}};
 (3,8)*{\ucrossrb{}{}};
 (-6,8)*{\slineur{}};
(-6,8)*[black]{\bullet};
 \endxy}; \endxy }} "4";"3"};
   {\ar^-{\xy  0;/r.15pc/: (3,0)*{\slineur{i}};(-6,0)*{\ucrossbr{j}{i}};\endxy   } "1";"3"};
   {\ar^-{-\xy 0;/r.15pc/: (-9,0)*{\slineur{i}};(0,0)*{\ucrossbr{j}{i}}; \endxy   } "2";"4"};
 \endxy
\]
\[
  \cal{T}'_i \left(  \xy 0;/r.18pc/:
  (0,0)*{\rcrossrg{i}{k}};
 (6,3)*{ \lambda};
 \endxy  \right)
 :=
\cal{T}'_i \left( \xy 0;/r.12pc/:
(-9,0)*{\slineng{}};
(-9,8)*{\slineng{}};
(-6,-6)*{\srcupg{}};
(6,7)*{\srcapg{}};
(9,0)*{\slineng{}};
(9,-8)*{\slinedg{}};
(0,0)*{\ucrossgr{}{}};
(3,-8)*{\slinenr{}};
(-3,8)*{\slineur{}};
(-12,0)*{};(12,0)*{};
\endxy \right)
= t_{ki}^{-2} \hspace{-.25cm} \dcrossrg{i}{k}
, \quad
\cal{T}'_i \left(  \xy 0;/r.18pc/:
  (0,0)*{\lcrossgr{k}{i}};
 (6,3)*{ \lambda};
 \endxy  \right)
 :=
\cal{T}'_i \left( \xy 0;/r.12pc/:
(-3,-8)*{\slinenr{}};
(3,8)*{\slineur{}};
(0,0)*{\ucrossrg{}{}};
(9,0)*{\slineng{}};
(9,8)*{\slineng{}};
(6,-6)*{\slcupg{}};
(-6,7)*{\slcapg{}};
(-9,0)*{\slineng{}};
(-9,-8)*{\slinedg{}};
(-12,0)*{};(12,0)*{};
\endxy \right)
= t_{ki} \hspace{-.25cm} \dcrossgr{k}{i}
\]
\[
\cal{T}'_i \left(\xy 0;/r.18pc/:
  (0,0)*{\rcrossgr{k}{i}};
 (6,3)*{ \lambda};
 \endxy  \right)
 :=
\cal{T}'_i \left( \xy 0;/r.12pc/:
(-9,0)*{\slinenr{}};
(-9,8)*{\slinenr{}};
(-6,-6)*{\srcupr{}};
(6,7)*{\srcapr{}};
(9,0)*{\slinenr{}};
(9,-8)*{\slinedr{}};
(0,0)*{\ucrossrg{}{}};
(3,-8)*{\slineng{}};
(-3,8)*{\slineug{}};
(-12,0)*{};(12,0)*{};
\endxy \right)
= \hspace{-.25cm} \ucrossgr{k}{i}
, \quad
\cal{T}'_i \left(  \xy 0;/r.18pc/:
  (0,0)*{\lcrossrg{i}{k}};
 (6,3)*{ \lambda};
 \endxy  \right)
 :=
\cal{T}'_i \left( \xy 0;/r.12pc/:
(-3,-8)*{\slineng{}};
(3,8)*{\slineug{}};
(0,0)*{\ucrossgr{}{}};
(9,0)*{\slinenr{}};
(9,8)*{\slinenr{}};
(6,-6)*{\slcupr{}};
(-6,7)*{\slcapr{}};
(-9,0)*{\slinenr{}};
(-9,-8)*{\slinedr{}};
(-12,0)*{};(12,0)*{};
\endxy \right)
= t_{ik}^{-1} \hspace{-.25cm} \ucrossrg{i}{k}
\]
\[
      \cal{T}'_i \left(  \xy 0;/r.18pc/:
  (0,0)*{\rcrossgg{k}{k'}};
 (6,3)*{ \lambda};
 \endxy  \right)
 :=
\cal{T}'_i \left( \xy 0;/r.12pc/:
(-9,0)*{\slineng{}};
(-9,8)*{\slineng{}};
(-6,-6)*{\srcupg{}};
(6,7)*{\srcapg{}};
(9,0)*{\slineng{}};
(9,-8)*{\slinedg{}};
(0,0)*{\ucrossgg{}{}};
(3,-8)*{\slineng{}};
(-3,8)*{\slineug{}};
(-12,0)*{};(12,0)*{};
\endxy \right)
= \hspace{-.25cm} \rcrossgg{k}{k'}
, \quad
\cal{T}'_i \left(  \xy 0;/r.18pc/:
  (0,0)*{\lcrossgg{k'}{k}};
 (6,3)*{ \lambda};
 \endxy  \right)
 :=
\cal{T}'_i \left( \xy 0;/r.12pc/:
(-3,-8)*{\slineng{}};
(3,8)*{\slineug{}};
(0,0)*{\ucrossgg{}{}};
(9,0)*{\slineng{}};
(9,8)*{\slineng{}};
(6,-6)*{\slcupg{}};
(-6,7)*{\slcapg{}};
(-9,0)*{\slineng{}};
(-9,-8)*{\slinedg{}};
(-12,0)*{};(12,0)*{};
\endxy \right)
= \lcrossgg{k'}{k}
\]
\[
     \cal{T}'_i \left(  \xy 0;/r.18pc/:
  (0,0)*{\rcrossbg{j}{k}};
 (6,3)*{ \lambda};
 \endxy  \right)
:=
\cal{T}'_i \left( \xy 0;/r.12pc/:
(-9,0)*{\slineng{}};
(-9,8)*{\slineng{}};
(-6,-6)*{\srcupg{}};
(6,7)*{\srcapg{}};
(9,0)*{\slineng{}};
(9,-8)*{\slinedg{}};
(0,0)*{\ucrossgb{}{}};
(3,-8)*{\slinen{}};
(-3,8)*{\slineu{}};
(-12,0)*{};(12,0)*{};
\endxy \right)
=
 \xy 0;/r.18pc/:
  (-25,15)*+{\clubsuit\;\cal{F}_k\cal{E}_j\cal{E}_i
  \onell{s_i(\lambda)}}="1";
  (-25,-15)*+{\clubsuit \;\cal{E}_j\cal{E}_i\cal{F}_k\onell{s_i(\lambda)}}="2";
   {\ar^{ \vcenter{\xy 0;/r.18pc/:
 (3,0)*{\rcrossrg{i}{k}};
  (-6,0)*{\slinen{j}};
 (-3,8.5)*{\rcrossbg{}{}}; (-12,4)*{t_{ki}};
 (6,8.5)*{\slineur{}};
 \endxy}} "2";"1"};
  (30,15)*+{\cal{F}_k\cal{E}_i\cal{E}_j\onell{s_i(\lambda)}\la1\ra}="3";
  (30,-15)*+{\cal{E}_i\cal{E}_j\cal{F}_k\onell{s_i(\lambda)}\la1\ra}="4";
  {\ar_{ \vcenter{\xy 0;/r.18pc/:
 (3,0)*{\rcrossbg{j}{k}};
  (-6,0)*{\slinenr{i}};
 (-3,8.5)*{\rcrossrg{}{}}; (-12,4)*{t_{ki}};
 (6,8.5)*{\slineu{}};
 \endxy}} "4";"3"};
   {\ar^-{\xy 0;/r.15pc/: (3,0)*{\slinedg{k}};(-6,0)*{\ucrossbr{j}{i}};\endxy   } "2";"4"};
   {\ar^-{\xy 0;/r.15pc/:
   (-9,0)*{\slinedg{k}};(0,0)*{\ucrossbr{j}{i}};\endxy   } "1";"3"};
 \endxy
\]
\[
\cal{T}'_i \left(\xy 0;/r.18pc/:
  (0,0)*{\lcrossgb{k}{j}};
 (6,3)*{ \lambda};
 \endxy  \right)
:=
\cal{T}'_i \left( \xy 0;/r.12pc/:
(-3,-8)*{\slinen{}};
(3,8)*{\slineu{}};
(0,0)*{\ucrossbg{}{}};
(9,0)*{\slineng{}};
(9,8)*{\slineng{}};
(6,-6)*{\slcupg{}};
(-6,7)*{\slcapg{}};
(-9,0)*{\slineng{}};
(-9,-8)*{\slinedg{}};
(-12,0)*{};(12,0)*{};
\endxy \right)
=
 \xy 0;/r.18pc/:
  (-25,15)*+{\clubsuit \;\cal{E}_j\cal{E}_i\cal{F}_k\onell{s_i(\lambda)}}="1";
  (-25,-15)*+{\clubsuit \;\cal{F}_k\cal{E}_j\cal{E}_i\onell{s_i(\lambda)}}="2";
   {\ar^{\; \vcenter{\xy 0;/r.18pc/:
 (-3,0)*{\lcrossgb{k}{j}};
  (6,0)*{\slinenr{i}};
 (3,8.5)*{\lcrossgr{}{}};
 (-6,8.5)*{\slineu{}}; (-13,4)*{t_{ki}^{-1}};
 \endxy}} "2";"1"};
  (30,15)*+{\cal{E}_i\cal{E}_j\cal{F}_k\onell{s_i(\lambda)}\la1\ra}="3";
  (30,-15)*+{\cal{F}_k\cal{E}_i\cal{E}_j\onell{s_i(\lambda)}\la1\ra}="4";
  {\ar_{\;\vcenter{\xy 0;/r.18pc/:
 (-3,0)*{\lcrossgr{k}{i}};
  (6,0)*{\slinen{j}};
 (3,8.5)*{\lcrossgb{}{}};
 (-6,8.5)*{\slineur{}}; (-13,4)*{t_{ki}^{-1}};
 \endxy} } "4";"3"};
   {\ar^-{\xy 0;/r.15pc/:
   (-3,0)*{\slinedg{k}};(6,0)*{\ucrossbr{j}{i}};\endxy   } "2";"4"};
   {\ar^-{\xy  0;/r.15pc/:
   (9,0)*{\slinedg{k}};(0,0)*{\ucrossbr{j}{i}};\endxy   } "1";"3"};
 \endxy
\]
\[      \cal{T}'_i \left(\xy 0;/r.18pc/:
  (0,0)*{\rcrossgb{k}{j}};
 (6,3)*{ \lambda};
 \endxy  \right)
:=
\cal{T}'_i \left( \xy 0;/r.12pc/:
(-9,0)*{\slinen{}};
(-9,8)*{\slinen{}};
(-6,-6)*{\srcup{}};
(6,7)*{\srcap{}};
(9,0)*{\slinen{}};
(9,-8)*{\slined{}};
(0,0)*{\ucrossbg{}{}};
(3,-8)*{\slineng{}};
(-3,8)*{\slineug{}};
(-12,0)*{};(12,0)*{};
\endxy \right)
=
 \vcenter{\xy 0;/r.16pc/:
  (-30,15)*+{\cal{F}_j\cal{F}_i\cal{E}_k
  \onell{s_i(\lambda)}\la-1\ra}="1";
  (-30,-15)*+{\cal{E}_k\cal{F}_j\cal{F}_i
  \onell{s_i(\lambda)}\la-1\ra}="2";
   {\ar^-{ \vcenter{\xy 0;/r.15pc/:
 (-3,0)*{\rcrossgb{k}{j}};
  (6,0)*{\slinedr{i}};
 (3,8.5)*{\rcrossgr{}{}};
 (-6,8.5)*{\slinen{}};
(-20,10)*{(-1)^{j\cdot k}};
(-16,2)*{t_{ki}^{-1}t_{jk}};
 \endxy}} "2";"1"};
  (30,15)*+{\clubsuit \;\cal{F}_i\cal{F}_j\cal{E}_k\onell{s_i(\lambda)}}="3";
  (30,-15)*+{\clubsuit \;\cal{E}_k\cal{F}_i\cal{F}_j\onell{s_i(\lambda)}}="4";
  {\ar_-{\vcenter{\xy 0;/r.15pc/:
 (-7,0)*{\rcrossgr{k}{i}};
  (2,0)*{\slined{j}};
 (-1,8.5)*{\rcrossgb{}{}};
 (-10,8.5)*{\slinenr{}};
(-24,10)*{(-1)^{j\cdot k}};
(-20,2)*{t_{ki}^{-1}t_{jk}};
 \endxy} } "4";"3"};
   {\ar^-{\xy 0;/r.15pc/: (-3,0)*{\slineug{k}};(6,0)*{\dcrossbr{j}{i}};\endxy   } "2";"4"};
   {\ar^-{\xy 0;/r.15pc/: (9,0)*{\slineug{k}};(0,0)*{\dcrossbr{j}{i}};\endxy   } "1";"3"};
 \endxy}
\]
\[
\cal{T}'_i \left(  \xy 0;/r.18pc/:
  (0,0)*{\lcrossbg{j}{k}};
 (6,3)*{ \lambda};
 \endxy  \right)
:=
\cal{T}'_i \left( \xy 0;/r.12pc/:
(-3,-8)*{\slineng{}};
(3,8)*{\slineug{}};
(0,0)*{\ucrossgb{}{}};
(9,0)*{\slinen{}};
(9,8)*{\slinen{}};
(6,-6)*{\slcup{}};
(-6,7)*{\slcap{}};
(-9,0)*{\slinen{}};
(-9,-8)*{\slined{}};
(-12,0)*{};(12,0)*{};
\endxy \right)
=
 \vcenter{\xy 0;/r.16pc/:
  (-30,15)*+{\cal{E}_k\cal{F}_j\cal{F}_i
  \onell{s_i(\lambda)}\la-1\ra}="1";
  (-30,-15)*+{\cal{F}_j\cal{F}_i\cal{E}_k
  \onell{s_i(\lambda)}\la-1\ra}="2";
   {\ar^-{ \vcenter{\xy 0;/r.15pc/:
 (3,0)*{\lcrossrg{i}{k}};
  (-6,0)*{\slined{j}};
 (-3,8.5)*{\lcrossbg{}{}};
 (6,8.5)*{\slinenr{}};
 (-20,10)*{(-1)^{j\cdot k}};
(-16,2)*{t_{ik}t_{jk}^{-1}};
 \endxy}} "2";"1"};
  (30,15)*+{\clubsuit \;\cal{E}_k\cal{F}_i\cal{F}_j\onell{s_i(\lambda)}}="3";
  (30,-15)*+{\clubsuit \;\cal{F}_i\cal{F}_j\cal{E}_k\onell{s_i(\lambda)}}="4";
  {\ar_-{ \vcenter{\xy 0;/r.15pc/:
 (3,0)*{\lcrossbg{j}{k}};
  (-6,0)*{\slinedr{i}};
 (-3,8.5)*{\lcrossrg{}{}};
 (6,8.5)*{\slinen{}};
 (-20,10)*{(-1)^{j\cdot k}};
(-16,2)*{t_{ik}t_{jk}^{-1}};
 \endxy}} "4";"3"};
   {\ar^-{\xy  0;/r.15pc/: (3,0)*{\slineug{k}};(-6,0)*{\dcrossbr{j}{i}};\endxy   } "2";"4"};
   {\ar^-{\xy 0;/r.15pc/: (-9,0)*{\slineug{k}};(0,0)*{\dcrossbr{j}{i}};\endxy   } "1";"3"};
 \endxy}
\]

$\cal{T}'_i \left(  \xy 0;/r.15pc/:
  (0,0)*{\rcrosspb{j'}{j}}; (-20,0)*{(-1)^{j\cdot j'}t_{jj'}^{-1}};
 (6,3)*{ \lambda};
 \endxy  \right)
 :=
\cal{T}'_i \left( \xy 0;/r.10pc/:
(-9,0)*{\slinen{}};
(-9,8)*{\slinen{}};
(-6,-6)*{\srcup{}};
(6,7)*{\srcap{}};
(9,0)*{\slinen{}};
(9,-8)*{\slined{}};
(0,0)*{\ucrossbp{}{}};
(3,-8)*{\slinenp{}};
(-3,8)*{\slineup{}};
(-12,0)*{};(12,0)*{};(-32,0)*{(-1)^{j\cdot j'}t_{jj'}^{-1}};
\endxy \right) =$
\[
 \xy 0;/r.15pc/:
 (-65,-35)*+{\cal{E}_{j'}\cal{E}_i\cal{F}_{j}\cal{F}_i \onell{s_i(\lambda)}\la-1\ra}="bl";
 (-20,-15)*+{\clubsuit\;\cal{E}_i\cal{E}_{j'}\cal{F}_{j}\cal{F}_i \onell{s_i(\lambda)}}="bt";
 (25,-55)*+{\clubsuit\;\cal{E}_{j'}\cal{E}_i\cal{F}_i\cal{F}_{j} \onell{s_i(\lambda)}}="bb";
 (65,-30)*+{\cal{E}_i\cal{E}_{j'}\cal{F}_i\cal{F}_{j} \onell{s_i(\lambda)}\la1\ra}="br";
  {\ar_<<<<<<<<<<<<<<<{\xy 0;/r.12pc/:
    (-6,0)*{\ucrosspr{j'}{i}}; (9,0)*{\slinedr{i}}; (3,0)*{\slined{j}}; \endxy \;\;\;
  } "bl";"bt"};
  {\ar_{\xy 0;/r.12pc/:(6,0)*{\dcrossbr{j}{i}}; (-3,0)*{\slineur{i}};
    (-9,0)*{\slineup{j'}};\endxy \;\;
  } "bl";"bb"};
  {\ar^>>{\xy 0;/r.12pc/:
    (6,0)*{\dcrossbr{j}{i}}; (-9,0)*{\slineur{i}}; (-13,1)*{-}; (-3,0)*{\slineup{j'}};\endxy
  } "bt";"br"};
  {\ar_>>>>{\xy 0;/r.12pc/:
    (-6,0)*{\ucrosspr{j'}{i}}; (9,0)*{\slined{j}}; (3,0)*{\slinedr{i}};\endxy
 } "bb";"br"};
  (-65,40)*+{\cal{F}_{j}\cal{F}_i\cal{E}_{j'}\cal{E}_i \onell{s_i(\lambda)}\la-1\ra}="tl";
 (-20,60)*+{\clubsuit \;\cal{F}_i\cal{F}_{j}\cal{E}_{j'}\cal{E}_i \onell{s_i(\lambda)}}="tt";
 (25,20)*+{\clubsuit \;\cal{F}_{j}\cal{F}_i\cal{E}_i\cal{E}_{j'} \onell{s_i(\lambda)}}="tb";
 (65,45)*+{\cal{F}_i\cal{F}_{j}\cal{E}_i\cal{E}_{j'} \onell{s_i(\lambda)}\la1\ra}="tr";
  {\ar^<<<<<<<<<<<<{\xy 0;/r.12pc/:
    \endxy} "tl";"tt"};
  {\ar^<<<<<<<<<<<<{\;\;\;\xy 0;/r.12pc/:\endxy \;\;} "tl";"tb"};
  {\ar^{\;\;\;\xy 0;/r.12pc/:\endxy} "tt";"tr"};
  {\ar_{\xy 0;/r.12pc/:
    \endxy
 } "tb";"tr"};
    {\ar^{\vcenter{\xy 0;/r.10pc/:
    (0,0)*{\rcrossrb{i}{j}}; (9,0)*{\slinedr{i}}; (-9,0)*{\slinenp{j'}};
    (-6,8.5)*{\ncrosspb{}{}};(6,8.5)*{\ncrossrr{}{}};
    (0,16.5)*{\rcrosspr{}{}}; (9,16.5)*{\slineur{}}; (-9,16.5)*{\slinen{}}; (-16,6)*{t_{ij}^{-1}}; \endxy}
  } "bl";"tl"};
  {\ar_{\vcenter{\xy 0;/r.10pc/:
    (0,0)*{\rcrosspr{j'}{i}}; (9,0)*{\slined{j}}; (-9,0)*{\slinenr{i}};
    (-6,8.5)*{\ncrossrr{}{}};(6,8.5)*{\ncrosspb{}{}};
    (0,16.5)*{\rcrossrb{}{}}; (9,16.5)*{\slineup{}}; (-9,16.5)*{\slinenr{}}; (-16,6)*{t_{ij}^{-1}}; \endxy}
 } "br";"tr"};
   {\ar_>>>>>>>>>{\xy 0;/r.10pc/:
    (0,0)*{\rbigcrosspb{j'}{j}}; (0,-2.5)*{\srcapr{}}; (0,8)*{\srcupr{}}; (-22,2)*{-t_{ij}^{-1}t_{ij'}};
(-2.5,15)*{\scs i}; (-2.5,-9)*{\scs i};
    \endxy} "bb";"tb"};
  {\ar^->>>>>>>>>>>>>>>>>>>>>>{\xy 0;/r.10pc/:
    (0,0)*{\rbigcrossrr{i}{i}}; (0,-2.5)*{\srcap{}}; (0,8)*{\srcup{}}; (-20,2)*{\delta_{jj'}v_{ij}};
(-2.5,-9)*{\scs j}; (-2.5,15)*{\scs j};
    \endxy} "bt";"tt"};
    {\ar@/_1.3pc/@{->} "bt";"tb"};
  {\ar@/_.9pc/@{->} "bb";"tt"};
  (-6.5,2.5)*{\xy 0;/r.10pc/:
    (-18,6)*{\scs -t_{ij}^{-1}}; (0,0)*{\rcrosspb{j'}{j}}; (9,0)*{\slinedr{i}}; (-9,0)*{\slinenr{i}};
    (-6,8.5)*{\ncrossrb{}{}};(6,8.5)*{\ncrosspr{}{}};
    (0,16.5)*{\rcrossrr{}{}}; (9,16.5)*{\slineup{}}; (-9,16.5)*{\slinen{}};\endxy};
   (6,41)*{\xy 0;/r.10pc/:
    (-16,6)*{t_{ij}^{-1}}; (0,0)*{\rcrossrr{i}{i}}; (9,0)*{\slined{j}}; (-9,0)*{\slinenp{j'}};
    (-6,8.5)*{\ncrosspr{}{}};(6,8.5)*{\ncrossrb{}{}};
    (0,16.5)*{\rcrosspb{}{}}; (9,16.5)*{\slineur{}}; (-9,16.5)*{\slinenr{}};\endxy};
 \endxy
\]

$\cal{T}'_i \left(  \xy 0;/r.15pc/:
  (0,0)*{\lcrossbp{j}{j'}}; (-20,0)*{(-1)^{j\cdot j'+1}t_{jj'}};
 (6,3)*{ \lambda};
 \endxy  \right)
 :=
\cal{T}'_i \left( \xy 0;/r.10pc/:
(-3,-8)*{\slinenp{}};
(3,8)*{\slineup{}};
(0,0)*{\ucrosspb{}{}};
(9,0)*{\slinen{}};
(9,8)*{\slinen{}};
(6,-6)*{\slcup{}};
(-6,7)*{\slcap{}};
(-9,0)*{\slinen{}};
(-9,-8)*{\slined{}};
(-12,0)*{};(12,0)*{};
(-34,0)*{(-1)^{j\cdot j'+1}t_{jj'}};
\endxy \right)=$
\[
 \xy 0;/r.15pc/:
 (-65,-35)*+{\cal{F}_{j}\cal{F}_i\cal{E}_{j'}\cal{E}_i \onell{s_i(\lambda)}\la-1\ra}="bl";
 (-20,-15)*+{\clubsuit\;\cal{F}_i\cal{F}_{j}\cal{E}_{j'}\cal{E}_i \onell{s_i(\lambda)}}="bt";
 (25,-55)*+{\clubsuit\;\cal{F}_{j}\cal{F}_i\cal{E}_i\cal{E}_{j'} \onell{s_i(\lambda)}}="bb";
 (65,-30)*+{\cal{F}_i\cal{F}_{j}\cal{E}_i\cal{E}_{j'} \onell{s_i(\lambda)}\la1\ra}="br";
  {\ar_<<<<<<<<<<<<<<<{\xy 0;/r.12pc/:
    (-6,0)*{\dcrossbr{j}{i}}; (9,0)*{\slineur{i}}; (3,0)*{\slineup{j'}}; \endxy \;\;\;
  } "bl";"bt"};
  {\ar_{\xy 0;/r.12pc/:(6,0)*{\ucrosspr{j'}{i}}; (-13,1)*{-}; (-3,0)*{\slinedr{i}};
    (-9,0)*{\slined{j}};\endxy \;\;
  } "bl";"bb"};
  {\ar^>>{\xy 0;/r.12pc/:
    (6,0)*{\ucrosspr{j'}{i}}; (-9,0)*{\slinedr{i}}; (-3,0)*{\slined{j}};\endxy
  } "bt";"br"};
  {\ar_>>>>{\xy 0;/r.12pc/:
    (-6,0)*{\dcrossbr{j}{i}}; (9,0)*{\slineup{j'}}; (3,0)*{\slineur{i}};\endxy
 } "bb";"br"};
 (-65,40)*+{\cal{E}_{j'}\cal{E}_i\cal{F}_{j}\cal{F}_i \onell{s_i(\lambda)}\la-1\ra}="tl";
 (-20,60)*+{\clubsuit \;\cal{E}_i\cal{E}_{j'}\cal{F}_{j}\cal{F}_i \onell{s_i(\lambda)}}="tt";
 (25,20)*+{\clubsuit \;\cal{E}_{j'}\cal{E}_i\cal{F}_i\cal{F}_{j} \onell{s_i(\lambda)}}="tb";
 (65,45)*+{\cal{E}_i\cal{E}_{j'}\cal{F}_i\cal{F}_{j} \onell{s_i(\lambda)}\la1\ra}="tr";
 {\ar^<<<<<<<<<<<<{\xy 0;/r.12pc/:
    \endxy \;\;\;
  } "tl";"tt"};
  {\ar^<<<<<<<<<<<<{\;\;\;\xy 0;/r.12pc/:\endxy \;\;} "tl";"tb"};
  {\ar^{\;\;\;\xy 0;/r.12pc/:
   \endxy
  } "tt";"tr"};
  {\ar_{\xy 0;/r.12pc/:
    \endxy
 } "tb";"tr"};
  {\ar^{\vcenter{\xy 0;/r.10pc/:
    (0,0)*{\lcrossrp{i}{j'}}; (9,0)*{\slinenr{i}}; (-9,0)*{\slined{j}};
    (-6,8.5)*{\ncrossbp{}{}};(6,8.5)*{\ncrossrr{}{}};
    (0,16.5)*{\lcrossbr{}{}}; (9,16.5)*{\slinenr{}}; (-9,16.5)*{\slineup{}}; (-16,6)*{t_{ij}}; \endxy}
  } "bl";"tl"};
  {\ar_{\vcenter{\xy 0;/r.10pc/:
    (0,0)*{\lcrossbr{j}{i}}; (9,0)*{\slinenp{j'}}; (-9,0)*{\slinedr{i}};
    (-6,8.5)*{\ncrossrr{}{}};(6,8.5)*{\ncrossbp{}{}};
    (0,16.5)*{\lcrossrp{}{}}; (9,16.5)*{\slinen{}}; (-9,16.5)*{\slineur{}}; (-16,6)*{t_{ij}}; \endxy}
 } "br";"tr"};
  {\ar_>>>>>>>>>{\xy 0;/r.10pc/:
    (0,0)*{\lbigcrossbp{j}{j'}}; (0,-2.5)*{\slcapr{}}; (0,8)*{\slcupr{}}; (-16,2)*{t_{ij}^{2}};
(2.5,-9)*{\scs i}; (2.5,15)*{\scs i};
    \endxy} "bb";"tb"};
  {\ar^->>>>>>>>>>>>>>>>>>>>>>{\xy 0;/r.10pc/:
    (0,0)*{\lbigcrossrr{i}{i}}; (0,-2.5)*{\slcap{}}; (0,8)*{\slcup{}}; (-27,2)*{-\delta_{jj'}t_{ij}t_{j'i}};
(2.5,-9)*{\scs j}; (2.5,15)*{\scs j};
    \endxy} "bt";"tt"};
  {\ar@/_1.2pc/@{->} "bt";"tb"};
  {\ar@/_.9pc/@{->} "bb";"tt"};
  (-6,4)*{\xy 0;/r.10pc/:
    (-15,6)*{t_{ij}}; (0,0)*{\lcrossbp{j}{j'}}; (9,0)*{\slinenr{i}}; (-9,0)*{\slinedr{i}};
    (-6,8.5)*{\ncrossrp{}{}};(6,8.5)*{\ncrossbr{}{}};
    (0,16.5)*{\lcrossrr{}{}}; (9,16.5)*{\slinen{}}; (-9,16.5)*{\slineup{}};\endxy};
   (6,41)*{\xy 0;/r.10pc/:
    (-18,6)*{-t_{ij}}; (0,0)*{\lcrossrr{i}{i}}; (9,0)*{\slinenp{j'}}; (-9,0)*{\slined{j}};
    (-6,8.5)*{\ncrossbr{}{}};(6,8.5)*{\ncrossrp{}{}};
    (0,16.5)*{\lcrossbp{}{}}; (9,16.5)*{\slinenr{}}; (-9,16.5)*{\slineur{}};\endxy};
 \endxy
\]

%
\subsection{Value of \texorpdfstring{$\cal{T}'_{i,1}$}{T'i,1} on downwards crossing 2-morphisms} \label{sec:down-crossing}
%

\[
    \cal{T}'_i \left(  \xy 0;/r.18pc/:
  (0,0)*{\dcrossrr{i}{i}};
 (6,3)*{ \lambda};
 \endxy  \right)
:=
\cal{T}'_i \left( \xy 0;/r.10pc/:
(6,7)*{\srcapr{}};
(0,0)*{\ucrossrr{}{}};
(9,0)*{\slinenr{}};
(9,-8)*{\slinedr{}};
(15,0)*{\slinenr{}};
(15,-8)*{\slinedr{}};
(-6,-9)*{\llrcupr{}};
(6,10)*{\llrcapr{}};
(-6,-6)*{\srcupr{}};
(-9,0)*{\slinenr{}};
(-9,8)*{\slinenr{}};
(-15,0)*{\slinenr{}};
(-15,8)*{\slinenr{}};
(-18,0)*{};(18,0)*{};
\endxy \right)
= - \hspace{-.25cm}\ucrossrr{i}{i}
, \quad
 \cal{T}'_i \left(  \xy 0;/r.18pc/:
  (0,0)*{\dcrossgg{k}{k'}};
 (6,3)*{ \lambda};
 \endxy  \right)
:=
\cal{T}'_i \left( \xy 0;/r.10pc/:
(-9,0)*{\slineng{}};
(-9,8)*{\slineng{}};
(-6,-6)*{\srcupg{}};
(6,7)*{\srcapg{}};
(9,0)*{\slineng{}};
(9,-8)*{\slinedg{}};
(0,0)*{\ucrossgg{}{}};
(-15,0)*{\slineng{}};
(-15,8)*{\slineng{}};
(-6,-9)*{\llrcupg{}};
(6,10)*{\llrcapg{}};
(15,0)*{\slineng{}};
(15,-8)*{\slinedg{}};
(-18,0)*{};(18,0)*{};
\endxy \right)
= \hspace{-.25cm} \dcrossgg{k}{k'}
\]

\[
  \cal{T}'_i \left(  \xy 0;/r.18pc/:
  (0,0)*{\dcrossrg{i}{k}};
 (6,3)*{ \lambda};
 \endxy  \right)
:=
\cal{T}'_i \left( \xy 0;/r.10pc/:
(-9,0)*{\slinenr{}};
(-9,8)*{\slinenr{}};
(-6,-6)*{\srcupr{}};
(6,7)*{\srcapr{}};
(9,0)*{\slinenr{}};
(9,-8)*{\slinedr{}};
(0,0)*{\ucrossrg{}{}};
(-15,0)*{\slineng{}};
(-15,8)*{\slineng{}};
(-6,-9)*{\llrcupg{}};
(6,10)*{\llrcapg{}};
(15,0)*{\slineng{}};
(15,-8)*{\slinedg{}};
(-18,0)*{};(18,0)*{};
\endxy \right)
= t_{ki}^2 \hspace{-.25cm} \rcrossrg{i}{k}
, \quad
\cal{T}'_i \left(  \xy 0;/r.18pc/:
  (0,0)*{\dcrossgr{k}{i}};
 (6,3)*{ \lambda};
 \endxy  \right)
  :=
\cal{T}'_i \left( \xy 0;/r.10pc/:
(-9,0)*{\slineng{}};
(-9,8)*{\slineng{}};
(-6,-6)*{\srcupg{}};
(6,7)*{\srcapg{}};
(9,0)*{\slineng{}};
(9,-8)*{\slinedg{}};
(0,0)*{\ucrossgr{}{}};
(-15,0)*{\slinenr{}};
(-15,8)*{\slinenr{}};
(-6,-9)*{\llrcupr{}};
(6,10)*{\llrcapr{}};
(15,0)*{\slinenr{}};
(15,-8)*{\slinedr{}};
(-18,0)*{};(18,0)*{};
\endxy \right) = t_{ki}^{-1} \hspace{-.25cm} \lcrossgr{k}{i}
\]

\[
  \cal{T}'_i \left(  \xy 0;/r.15pc/:
  (0,0)*{\dcrossrb{i}{j}};
 (6,3)*{ \lambda};
 \endxy  \right)
 :=
\cal{T}'_i \left( \xy 0;/r.10pc/:
(-9,0)*{\slinenr{}};
(-9,8)*{\slinenr{}};
(-6,-6)*{\srcupr{}};
(6,7)*{\srcapr{}};
(9,0)*{\slinenr{}};
(9,-8)*{\slinedr{}};
(0,0)*{\ucrossrb{}{}};
(-15,0)*{\slinen{}};
(-15,8)*{\slinen{}};
(-6,-9)*{\llrcup{}};
(6,10)*{\llrcap{}};
(15,0)*{\slinen{}};
(15,-8)*{\slined{}};
(-18,0)*{};(18,0)*{};
\endxy \right)
=
   \xy 0;/r.15pc/:
  (-35,20)*+{\clubsuit\ \cal{F}_j\cal{F}_i\cal{E}_i\onell{s_i(\lambda)}\la\lambda_i-1 \ra}="1";
  (-35,-15)*+{\clubsuit\ \cal{E}_i\cal{F}_j\cal{F}_i\onell{s_i(\lambda)}\la\lambda_i-1 \ra}="2";
   {\ar^{\vcenter{\xy 0;/r.15pc/:
 (-3,0)*{\rcrossrb{i}{j}};
  (6,0)*{\slinedr{i}};
 (3,8.5)*{\rcrossrr{}{}};
 (-6,8.5)*{\slinen{}}; (-17,3)*{t_{ij}^{-1}t_{ji}^{-1}};
 \endxy }} "2";"1"};
  (35,20)*+{\cal{F}_i\cal{F}_j\cal{E}_i\onell{s_i(\lambda)}\la \lambda_i \ra}="3";
  (35,-15)*+{\cal{E}_i\cal{F}_i\cal{F}_j\onell{s_i(\lambda)}\la \lambda_i \ra}="4";
   {\ar_{\vcenter{ \xy 0;/r.15pc/:
 (-17,3)*{-t_{ij}^{-1}t_{ji}^{-1}}; (-3,0)*{\rcrossrr{i}{i}};
  (6,0)*{\slined{j}};
 (3,8.5)*{\rcrossrb{}{}};
 (-6,8.5)*{\slinenr{}};
 \endxy}} "4";"3"};
   {\ar^-{\xy  0;/r.15pc/: (3,0)*{\slineur{i}};(-6,0)*{\dcrossbr{j}{i}};\endxy   } "1";"3"};
   {\ar^-{\xy 0;/r.15pc/: (-9,0)*{\slineur{i}};(0,0)*{\dcrossbr{j}{i}}; (-13,1)*{-}; \endxy   } "2";"4"};
 \endxy
\]
\[
\cal{T}'_i \left(  \xy 0;/r.15pc/:
  (0,0)*{\dcrossbr{j}{i}};
 (6,3)*{ \lambda};
 \endxy  \right)
:=
\cal{T}'_i \left( \xy 0;/r.10pc/:
(-9,0)*{\slinen{}};
(-9,8)*{\slinen{}};
(-6,-6)*{\srcup{}};
(6,7)*{\srcap{}};
(9,0)*{\slinen{}};
(9,-8)*{\slined{}};
(0,0)*{\ucrossbr{}{}};
(-15,0)*{\slinenr{}};
(-15,8)*{\slinenr{}};
(-6,-9)*{\llrcupr{}};
(6,10)*{\llrcapr{}};
(15,0)*{\slinenr{}};
(15,-8)*{\slinedr{}};
(-18,0)*{};(18,0)*{};
\endxy \right)
=
     \vcenter{\xy 0;/r.15pc/:
  (-30,27)*+{\clubsuit\ \cal{E}_i\cal{F}_j\cal{F}_i\onell{s_i(\lambda)}\la \lambda_i+1 \ra}="1";
  (-30,-27)*+{\clubsuit\ \cal{F}_j\cal{F}_i\cal{E}_i\onell{s_i(\lambda)}\la \lambda_i-1 \ra}="2";
   {\ar^{\vcenter{\xy 0;/r.15pc/:
 (0,12)*{\xy 0;/r.15pc/: (3,0)*{\lcrossrr{i}{i}}; (-16,4)*{t_{ij}^2t_{ji}};
  (-6,0)*{\slined{j}};
 (-3,8.5)*{\lcrossbr{}{}};
 (6,8.5)*{\slinenr{}};
 (1,0)*[black]{\bullet};
 \endxy};
 (0,-12)*{\xy 0;/r.15pc/:
 (-16,4)*{-t_{ij}^2t_{ji}}; (3,0)*{\lcrossrr{i}{i}};
  (-6,0)*{\slined{j}};
 (-3,8.5)*{\lcrossbr{}{}};
 (6,8.5)*{\slinenr{}};
 (5,0)*[black]{\bullet};
 \endxy};\endxy}}  "2";"1"};
  (35,27)*+{ \cal{E}_i\cal{F}_i\cal{F}_j\onell{s_i(\lambda)}\la\lambda_i+2 \ra}="3";
  (35,-27)*+{ \cal{F}_i\cal{F}_j\cal{E}_i\onell{s_i(\lambda)}\la \lambda_i \ra}="4";
  {\ar_{\vcenter{\xy 0;/r.15pc/:
  (0,12)*{\xy 0;/r.15pc/:
    (3,0)*{\lcrossbr{j}{i}}; (-16,4)*{t_{ij}^2t_{ji}};
    (-6,0)*{\slinedr{i}};
    (-3,8.5)*{\lcrossrr{}{}};
    (6,8.5)*{\slinen{}}; (5,0)*[black]{\bullet};
    \endxy};
    (0,-12)*{\xy 0;/r.15pc/:
    (-16,4)*{-t_{ij}^2t_{ji}};
    (3,0)*{\lcrossbr{j}{i}};
    (-6,0)*{\slinedr{i}};
    (-3,8)*{\lcrossrr{}{}};
    (6,8)*{\slinen{}};  (-6,2)*[black]{\bullet}; \endxy};\endxy}} "4";"3"};
   {\ar_{\xy  0;/r.15pc/:
   (3,0)*{\slineur{i}};(-6,0)*{\dcrossbr{j}{i}};\endxy   } "2";"4"};
   {\ar^-{\xy 0;/r.15pc/:(-9,0)*{\slineur{i}};(0,0)*{\dcrossbr{j}{i}}; (-13,1)*{-}; \endxy   } "1";"3"};
 \endxy}
\]

$
     \cal{T}'_i \left(  \xy 0;/r.18pc/:
  (0,0)*{\dcrossbg{j}{k}};
 (6,3)*{ \lambda};
 \endxy  \right)
:=
\cal{T}'_i \left( \xy 0;/r.10pc/:
(-9,0)*{\slinen{}};
(-9,8)*{\slinen{}};
(-6,-6)*{\srcup{}};
(6,7)*{\srcap{}};
(9,0)*{\slinen{}};
(9,-8)*{\slined{}};
(0,0)*{\ucrossbg{}{}};
(-15,0)*{\slineng{}};
(-15,8)*{\slineng{}};
(-6,-9)*{\llrcupg{}};
(6,10)*{\llrcapg{}};
(15,0)*{\slineng{}};
(15,-8)*{\slinedg{}};
(-18,0)*{};(18,0)*{};
\endxy \right)
$
\[= 
\xy 0;/r.18pc/:
  (-35,15)*+{\clubsuit\;\cal{F}_k\cal{F}_j\cal{F}_i
  \onell{s_i(\lambda)}\la -1-j\cdot k\ra}="1";
  (-35,-15)*+{\clubsuit \;\cal{F}_j\cal{F}_i\cal{F}_k\onell{s_i(\lambda)}\la-1\ra}="2";
   {\ar^{ \vcenter{\xy 0;/r.15pc/:
(-20,10)*{(-1)^{j\cdot k}};
(-16,2)*{t_{ki}^{-2}t_{jk}};
(3,0)*{\dcrossrg{i}{k}};
  (-6,0)*{\slined{j}};
 (-3,8.5)*{\ncrossbg{}{}};
 (6,8.5)*{\slinenr{}};
 \endxy}} "2";"1"};
  (35,15)*+{\cal{F}_k\cal{F}_i\cal{F}_j\onell{s_i(\lambda)}\la -j\cdot k\ra}="3";
  (35,-15)*+{\cal{F}_i\cal{F}_j\cal{F}_k\onell{s_i(\lambda)}}="4";
  {\ar_{ \vcenter{\xy 0;/r.15pc/:
(-20,10)*{(-1)^{j\cdot k}};
(-16,2)*{t_{ki}^{-2}t_{jk}};
 (3,0)*{\dcrossbg{j}{k}};
  (-6,0)*{\slinedr{i}};
 (-3,8.5)*{\ncrossrg{}{}};
 (6,8.5)*{\slinen{}};
 \endxy}} "4";"3"};
   {\ar^-{\xy  0;/r.15pc/:
   (3,0)*{\slinedg{k}};(-6,0)*{\dcrossbr{j}{i}};\endxy   } "2";"4"};
   {\ar^-{\xy 0;/r.15pc/:
   (-9,0)*{\slinedg{k}};(0,0)*{\dcrossbr{j}{i}};\endxy   } "1";"3"};
 \endxy
\]
$
\cal{T}'_i \left(  \xy 0;/r.18pc/:
  (0,0)*{\dcrossgb{k}{j}};
 (6,3)*{ \lambda};
 \endxy  \right)
 :=
\cal{T}'_i \left( \xy 0;/r.10pc/:
(-9,0)*{\slineng{}};
(-9,8)*{\slineng{}};
(-6,-6)*{\srcupg{}};
(6,7)*{\srcapg{}};
(9,0)*{\slineng{}};
(9,-8)*{\slinedg{}};
(0,0)*{\ucrossgb{}{}};
(-15,0)*{\slinen{}};
(-15,8)*{\slinen{}};
(-6,-9)*{\llrcup{}};
(6,10)*{\llrcap{}};
(15,0)*{\slinen{}};
(15,-8)*{\slined{}};
(-18,0)*{};(18,0)*{};
\endxy \right)
$
\[
=
 \xy 0;/r.18pc/:
  (-35,15)*+{\clubsuit\;\cal{F}_j\cal{F}_i\cal{F}_k
  \onell{s_i(\lambda)}\la -1-j\cdot k\ra}="1";
  (-35,-15)*+{\clubsuit \;\cal{F}_k\cal{F}_j\cal{F}_i\onell{s_i(\lambda)}\la-1\ra}="2";
   {\ar^{ \vcenter{\xy 0;/r.15pc/:
 (-3,0)*{\dcrossgb{k}{j}};
(-20,10)*{(-1)^{j\cdot k}};
(-16,2)*{t_{ki} t_{jk}^{-1}};
  (6,0)*{\slinedr{i}};
 (3,8.5)*{\ncrossgr{}{}};
 (-6,8.5)*{\slinen{}};
 \endxy}} "2";"1"};
  (35,15)*+{\cal{F}_i\cal{F}_j\cal{F}_k\onell{s_i(\lambda)}\la -j\cdot k\ra}="3";
  (35,-15)*+{\cal{F}_k\cal{F}_i\cal{F}_j\onell{s_i(\lambda)}}="4";
  {\ar_{\vcenter{\xy 0;/r.15pc/:
 (-3,0)*{\dcrossgr{k}{i}};
(-20,10)*{(-1)^{j\cdot k}};
(-16,2)*{t_{ki} t_{jk}^{-1}};
  (6,0)*{\slined{j}};
 (3,8.5)*{\ncrossgb{}{}};
 (-6,8.5)*{\slinenr{}};
 \endxy} } "4";"3"};
   {\ar^-{\xy 0;/r.15pc/:  (-3,0)*{\slinedg{k}};(6,0)*{\dcrossbr{j}{i}};\endxy   } "2";"4"};
   {\ar^-{\xy 0;/r.15pc/:
   (9,0)*{\slinedg{k}};(0,0)*{\dcrossbr{j}{i}};\endxy   } "1";"3"};
 \endxy
\]

$\cal{T}'_i \left(  \xy 0;/r.18pc/:
  (0,0)*{\dcrossbp{j}{j'}};
 (6,3)*{ \lambda};(-15,0)*{t_{jj'}^{-1}t_{j'j}};
 \endxy  \right)
=\cal{T}'_i \left( \xy 0;/r.10pc/:
(-9,0)*{\slinen{}};
(-9,8)*{\slinen{}};
(-6,-6)*{\srcup{}};
(6,7)*{\srcap{}};
(9,0)*{\slinen{}};
(9,-8)*{\slined{}};
(0,0)*{\ucrossbp{}{}};
(-15,0)*{\slinenp{}};
(-15,8)*{\slinenp{}};
(-6,-9)*{\llrcupp{}};
(6,10)*{\llrcapp{}};
(15,0)*{\slinenp{}};
(15,-8)*{\slinedp{}};
(-18,0)*{};(18,0)*{};
(-32,0)*{t_{jj'}^{-1}t_{j'j}};
\endxy \right)=$
\begin{align*}
 \xy 0;/r.20pc/:
 (-60,-40)*+{\cal{F}_j\cal{F}_i\cal{F}_{j'}\cal{F}_i \onell{s_i(\lambda)}\la-2\ra}="bl";
 (-20,-20)*+{\cal{F}_i\cal{F}_j\cal{F}_{j'}\cal{F}_i \onell{s_i(\lambda)}\la-1\ra}="bt";
 (25,-60)*+{\cal{F}_j\cal{F}_i\cal{F}_i\cal{F}_{j'} \onell{s_i(\lambda)}\la-1\ra}="bb";
 (60,-35)*+{\clubsuit\;\cal{F}_i\cal{F}_j\cal{F}_i\cal{F}_{j'} \onell{s_i(\lambda)}}="br";
  {\ar^<<<<<<<<<<<<{\xy 0;/r.14pc/:
    (-6,0)*{\dcrossbr{j}{i}}; (9,0)*{\slinedr{i}}; (3,0)*{\slinedp{j'}}; \endxy \;\;\;
  } "bl";"bt"};
  {\ar_{\xy 0;/r.14pc/:(6,0)*{\dcrosspr{j'}{i}}; (-3,0)*{\slinedr{i}};(-13,1)*{-};
    (-9,0)*{\slined{j}};\endxy \;\;
  } "bl";"bb"};
  {\ar^>>>>>>>>{\xy 0;/r.14pc/:
    (6,0)*{\dcrosspr{j'}{i}}; (-9,0)*{\slinedr{i}}; (-3,0)*{\slined{j}};\endxy
  } "bt";"br"};
  {\ar_{\xy 0;/r.14pc/:
    (-6,0)*{\dcrossbr{j}{i}}; (9,0)*{\slinedp{j'}}; (3,0)*{\slinedr{i}};\endxy
 } "bb";"br"};
 (-60,40)*+{\cal{F}_{j'}\cal{F}_i\cal{F}_j\cal{F}_i \onell{s_i(\lambda)}\la-2-j\cdot j'\ra}="tl";
 (-20,60)*+{\cal{F}_i\cal{F}_{j'}\cal{F}_j\cal{F}_i \onell{s_i(\lambda)}\la-1-j\cdot j'\ra}="tt";
 (25,20)*+{\cal{F}_{j'}\cal{F}_i\cal{F}_i\cal{F}_j \onell{s_i(\lambda)}\la-1-j\cdot j'\ra}="tb";
 (60,45)*+{\clubsuit\;\cal{F}_i\cal{F}_{j'}\cal{F}_i\cal{F}_j \onell{s_i(\lambda)}\la-j\cdot j'\ra}="tr";
 {\ar^<<<<<<<<<<<<{\xy 0;/r.14pc/:
    \endxy \;\;\;
  } "tl";"tt"};
  {\ar^<<<<<<<<<<<<{\;\;\;\xy 0;/r.14pc/:\endxy \;\;} "tl";"tb"};
  {\ar^{\;\;\;\xy 0;/r.14pc/:
   \endxy
  } "tt";"tr"};
  {\ar_{\xy 0;/r.14pc/:
    \endxy
 } "tb";"tr"};
  {\ar^{\vcenter{\xy 0;/r.14pc/:
        (0,0)*{\dcrossrp{i}{j'}};
        (9,0)*{\slinedr{i}};
        (-9,0)*{\slined{j}};
        (-6,8.5)*{\ncrossbp{}{}};
        (6,8.5)*{\ncrossrr{}{}};
        (0,16)*{\ncrossbr{}{}};
        (9,16)*{\slinenr{}};
        (-9,16)*{\slinenp{}};
        (-16,6)*{\scs -t_{ij}^{-1}}; \endxy}
  } "bl";"tl"};
  {\ar_{\vcenter{\xy 0;/r.14pc/:
        (0,0)*{\dcrossbr{j}{i}};
        (9,0)*{\slinedp{j'}};
        (-9,0)*{\slinedr{i}};
        (-6,8.5)*{\ncrossrr{}{}};
        (6,8.5)*{\ncrossbp{}{}};
        (0,16)*{\ncrossrp{}{}};
        (9,16)*{\slinen{}};
        (-9,16)*{\slinenr{}};
        (-14,6)*{\scs t_{ij}^{-1}};
    \endxy}
 } "br";"tr"};
  {\ar_>>>>>>>>>>>>>>{\xy 0;/r.14pc/:
        (0,0)*{\dcrossrr{i}{i}};
        (7,0)*{\slined{j}};
        (-7,0)*{\slined{j}};
        (-19,0)*{\scs -\delta_{jj'}v_{ij}};
    \endxy} "bb";"tb"};
  {\ar^-{\xy 0;/r.14pc/:
        (0,0)*{\dcrossbp{j}{j'}};
        (7,0)*{\slinedr{i}};
        (-7,0)*{\slinedr{i}};
        (-17,0)*{\scs t_{ij}^{-1}t_{ij'}};
    \endxy} "bt";"tt"};
  {\ar^{} "bt";"tb"};
  {\ar@/_.7pc/@{->} "bb";"tt"};
  (-8.5,8)*{\xy 0;/r.13pc/:
        (-14,6)*{\scs t_{ij}^{-1}};
        (0,0)*{\dcrossbp{j}{j'}};
        (9,0)*{\slinedr{i}};
        (-9,0)*{\slinedr{i}};
        (-6,8.5)*{\ncrossrp{}{}};
        (6,8.5)*{\ncrossbr{}{}};
        (0,16)*{\ncrossrr{}{}};
        (9,16)*{\slinen{}};
        (-9,16)*{\slinenp{}};
    \endxy};
   (5,41)*{\xy 0;/r.14pc/:
        (-14,6)*{\scs t_{ij}^{-1}};
        (0,0)*{\dcrossrr{i}{i}};
        (9,0)*{\slinedp{j'}};
        (-9,0)*{\slined{j}};
        (-6,8.5)*{\ncrossbr{}{}};
        (6,8.5)*{\ncrossrp{}{}};
        (0,16)*{\ncrossbp{}{}};
        (9,16)*{\slinenr{}};
        (-9,16)*{\slinenr{}};\endxy};
 \endxy
\end{align*}

%
\subsection{Computation of \texorpdfstring{$\cal{T}'_{i,1}$}{T'i,1} on bubble 2-morphisms} \label{sec:Ti-bub}
%

We compute the image of bubble $2$-morphisms,
and use them to  explicitly verify that $\cal{T}'_{i,1}$ preserves the infinite Grassmannian relation.

\begin{align*}
\cal{T}'_i \left(\; \xy 0;/r.18pc/:
 (-12,0)*{\cbubr{\la i, \lambda \ra -1 + \alpha}{i}};
 (-8,8)*{\lambda};
 \endxy \;\right)
\; =\; &c_{i,\lambda}^2\xy 0;/r.18pc/:
 (-12,0)*{\ccbubr{\la i, \lambda \ra -1 + \alpha}{i}};
 (-8,8)*{s_i(\lambda)};
 \endxy
\; =\;c_{i,\lambda}^2\xy 0;/r.18pc/:
 (-12,0)*{\ccbubr{-\la i, s_i(\lambda) \ra -1 + \alpha}{i}};
 (-8,8)*{s_i(\lambda)};
 \endxy\\
 \cal{T}'_i \left(\; \xy 0;/r.18pc/:
 (-12,0)*{\ccbubr{-\la i, \lambda \ra -1 + \alpha}{i}};
 (-8,8)*{\lambda};
 \endxy \;\right)
\; =\;&c_{i,\lambda}^{-2}\xy 0;/r.18pc/:
 (-12,0)*{\cbubr{-\la i, \lambda \ra -1 + \alpha}{i}};
 (-8,8)*{s_i(\lambda)};
 \endxy
\; =\;c_{i,\lambda}^{-2}\xy 0;/r.18pc/:
 (-12,0)*{\cbubr{\la i, s_i(\lambda) \ra -1 + \alpha}{i}};
 (-8,8)*{s_i(\lambda)};
 \endxy \\
\cal{T}'_i \left(\; \xy 0;/r.18pc/:
 (-12,0)*{\cbubg{\la k, \lambda \ra -1 + \alpha}{k}};
 (-8,8)*{\lambda};
 \endxy \;\right)
\; =\;&t_{ki}^{\lambda_i}\xy 0;/r.18pc/:
 (-12,0)*{\cbubg{\la k, \lambda \ra -1 + \alpha}{k}};
 (-8,8)*{s_i(\lambda)};
 \endxy
\; =\;t_{ki}^{\lambda_i}\xy 0;/r.18pc/:
 (-12,0)*{\cbubg{\la k, s_i(\lambda) \ra -1 + \alpha}{k}};
 (-8,8)*{s_i(\lambda)};
 \endxy\\
\cal{T}'_i \left(\; \xy 0;/r.18pc/:
 (-12,0)*{\ccbubg{-\la k, \lambda \ra -1 + \alpha}{k}};
 (-8,8)*{\lambda};
 \endxy \;\right)
\; =\;&t_{ki}^{-\lambda_i}\xy 0;/r.18pc/:
 (-12,0)*{\ccbubg{-\la k, \lambda \ra -1 + \alpha}{k}};
 (-8,8)*{s_i(\lambda)};
 \endxy
\; =\;t_{ki}^{-\lambda_i}\xy 0;/r.18pc/:
 (-12,0)*{\ccbubg{-\la k, s_i(\lambda) \ra -1 + \alpha}{k}};
 (-8,8)*{s_i(\lambda)};
 \endxy\end{align*}

For $j$-labelled bubbles, we make use of the bubble sliding relations from Section~\ref{sec:bubble}.
(Note that, in the first equation, the number of dots on the black circles equals zero for both summands.)
\begin{align*}
\cal{T}'_i \left(\; \xy 0;/r.18pc/:
 (-12,0)*{\cbub{\la j, \lambda \ra -1 + \alpha}{j}};
 (-8,8)*{\lambda};
 \endxy \;\right)
=&
(-t_{ij})^{1+\lambda_i}c_{i,\lambda}^{-1}
\left(
\xy 0;/r.20pc/:(0,0)*{\xybox{
    (-3,0)*{}="t1"; (3,0)*{}="t2";
    "t1";"t2" **[blue]\crv{(-3,-4) & (3,-4)} ?(1)*[blue]\dir{<}?(.5)*\dir{}+(0,0)*{\bullet}
    +(0,-2.5)*{\scriptscriptstyle \la j, s_i(\lambda)-\alpha_i \ra -1 + \alpha-\lambda_i-1}   ;
    "t1";"t2" **[blue]\crv{(-3,4) & (3,4)}  ?(0)*[blue]\dir{>};
    (-20,0)*{}="t1"; (20,0)*{}="t2";
    "t1";"t2" **[black]\crv{(-20,-20) & (20,-20)};?(0)*[black]\dir{<}?(.5)*\dir{}+(0,0)*[black]
    {\bullet}+(0,-2.5)*{\scs \la i, s_i(\lambda) \ra -1 +\lambda_i+1};
    "t1";"t2" **[black]\crv{(-20,20) & (20,20)}?(1)*[black]\dir{>};
    (3,4)*{\scs j};(20,8)*{\scs i};}};
     (21,-10)*{s_i(\lambda)}; (-14,0)*{};
     \endxy
  - \quad
     \xy 0;/r.20pc/:(0,0)*{\xybox{
    (-3,0)*{}="t1"; (3,0)*{}="t2";
    "t1";"t2" **[black]\crv{(-3,-4) & (3,-4)} ?(1)*[black]\dir{<}?(.5)*\dir{}+(0,0)*[black]{\bullet}+(0,-2.5)*
    {\scriptscriptstyle \la i, s_i(\lambda)-\alpha_j \ra -1 +\lambda_i};
    "t1";"t2" **[black]\crv{(-3,4) & (3,4)}  ?(0)*[black]\dir{>}?(.5)*\dir{};
    (-16,0)*{}="t1"; (16,0)*{}="t2";
    "t1";"t2" **[blue]\crv{(-16,-16) & (16,-16)};?(0)*[blue]\dir{<}?(.5)*\dir{}+(0,0)*{\bullet}
    +(0,-2.5)*{\scs \la j, s_i(\lambda)\ra-1+ \alpha-\lambda_i};
    "t1";"t2" **[blue]\crv{(-16,16) & (16,16)}?(1)*[blue]\dir{>};
    (3,4)*{\scs i};(17,5)*{\scs j};
    (18,-10)*{s_i(\lambda)};}    };
     \endxy
\right) \\
=&
%
(-t_{ij})^{1+\lambda_i}c_{i,\lambda}^{-1}\left(
\sum_{\substack{h= \\ \text{max}(0,\lambda_i+1)}}^{\alpha}
t_{ji}^{-1} (-v_{ji})^{h-\lambda_i-1}\;
     \xy 0;/r.18pc/:
 (-6,0)*{\cbub{\spadesuit + \alpha-h}{j}};
 (6,0)*{\cbubr{\spadesuit + h}{i}};
 (19,4)*{s_i(\lambda)};
 \endxy
\right.\\
&\qquad \qquad \qquad \qquad \qquad \qquad
\left. -
\sum_{h=0}^{\text{min}(\lambda_i, \alpha)} t_{ij}^{-1}(-v_{ij})^{\lambda_i-h}\;
     \xy 0;/r.18pc/:
 (-6,0)*{\cbub{\spadesuit + \alpha-h}{j}};
 (6,0)*{\cbubr{\spadesuit + h}{i}}; (-15,0)*{};
 (19,4)*{s_i(\lambda)};
 \endxy
\right)
\\
=&
t_{ji}^{\lambda_i}c_{i,\lambda}^{-1}
\sum_{h=0}^{\alpha} (-v_{ij})^{-h}\;
     \xy 0;/r.18pc/:
 (-6,0)*{\cbub{\spadesuit + \alpha-h}{j}};
 (6,0)*{\cbubr{\spadesuit + h}{i}};
 (19,4)*{s_i(\lambda)};
 \endxy\end{align*}
Similarly, the image of the counter-clockwise bubble is given by:
\begin{align*}
\cal{T}'_i \left(\; \xy 0;/r.18pc/:
 (-12,0)*{\ccbub{-\la j, \lambda \ra -1 + \alpha}{j}};
 (-8,8)*{\lambda};
 \endxy \;\right)
=&
(-t_{ij})^{-\lambda_i}c_{i,\lambda}t_{ij}\left(
\sum_{h=0}^{\text{min}(-\lambda_i, \alpha)}
t_{ij}^{-1}(-v_{ij})^{-\lambda_i-h}\;
     \xy 0;/r.18pc/:
 (-6,0)*{\ccbub{\spadesuit + \alpha-h}{j}};
 (6,0)*{\ccbubr{\spadesuit + h}{i}}; (-15,0)*{};
 (19,2)*{s_i(\lambda)};
 \endxy\right.\\
& \qquad \qquad \qquad \qquad \qquad \qquad
\left. -
\sum_{\substack{h= \\ \text{max}(0,-\lambda_i)}}^{\alpha} t_{ji}^{-1}(-v_{ji})^{\lambda_i-1+h}\;
     \xy 0;/r.18pc/:
 (-6,0)*{\ccbub{\spadesuit + \alpha-h}{j}};
 (6,0)*{\ccbubr{\spadesuit + h}{i}};
 (19,2)*{s_i(\lambda)};
 \endxy\right)\\
=&
t_{ji}^{-\lambda_i}c_{i,\lambda}
\sum_{h=0}^{\alpha} (-v_{ij})^{-h}\;
     \xy 0;/r.18pc/:
 (-6,0)*{\ccbub{\spadesuit + \alpha-h}{j}};
 (6,0)*{\ccbubr{\spadesuit + h}{i}};
 (19,2)*{s_i(\lambda)};
 \endxy\end{align*}
(In both cases, recall our convention that any sums with non-increasing index are by definition zero.)

These computations for the images of bubbles under $\cal{T}'_{i,1}$ are only valid when the number of dots is positive; however, our next result shows that they also hold for bubbles with
a negative number of dots (\ie for fake bubbles, see Definition~\ref{defU_cat-cyc} \eqref{def:UQ-bub}).

\begin{lemma}
 $\cal{T}'_{i,1}$ preserves the infinite Grassmannian relation,
 \ie
\[
\cal{T}'_i\left(\left(
\xy  0;/r.15pc/:(0,0)*{\ccbubr{-\la i,\lambda\ra -1}{\ell}}; (6,5)*{\lambda};\endxy +
\cdots +
\xy  0;/r.15pc/:(0,0)*{\ccbubr{-\la i,\lambda\ra -1+\alpha}{\ell}}; (6,5)*{\lambda};\endxy t^\alpha + \cdots\right)
\left(\xy  0;/r.15pc/:(0,0)*{\cbubr{\la i,\lambda\ra -1}{\ell}}; (6,5)*{\lambda};\endxy +
\cdots +
\xy  0;/r.15pc/:(0,0)*{\cbubr{\la i,\lambda\ra -1+\alpha}{\ell}}; (6,5)*{\lambda};\endxy t^\alpha + \cdots\right)\right) = \Id_{\onell{s_i(\l)}}
\]
\end{lemma}

\begin{proof}
The only non-trivial case is when the bubbles are $j$-labeled (for $i\cdot j=-1$),
and here we compute the relation in degree $\alpha$ as follows.
\begin{align*}
\cal{T}'_i\left(
\sum_{g+h = \alpha}
\;\; \vcenter{\xy 0;/r.18pc/:
    (2,0)*{\cbub{\l_i-1+g}{j}};
    (24,0)*{\ccbub{-\l_i-1+h}{j}};
  (12,8)*{\lambda};
 \endxy}
\right)
&= \displaystyle\sum_{r+s+t+u=\alpha}
(-v_{ij})^{-s-u}\;
\xy  0;/r.15pc/: (-36,0)*{\cbub{\spadesuit+r}{j}}; (-18,0)*{\cbubr{\spadesuit+s}{i}}; (0,0)*{\ccbub{\spadesuit+t}{j}}; (18,0)*{\ccbubr{\spadesuit+u}{i}}; (30,4)*{s_i(\lambda)};\endxy \\
&= \displaystyle\sum_{k+s+u=\alpha }(-v_{ij})^{-s-u}\;
\xy  0;/r.15pc/: (0,0)*{\cbubr{\spadesuit+s}{i}}; (18,0)*{\ccbubr{\spadesuit+u}{i}}; (30,4)*{s_i(\lambda)};\endxy
\left(\sum_{r+t=k}\;
\xy  0;/r.15pc/: (0,0)*{\cbub{\spadesuit+r}{j}};(18,0)*{\ccbub{\spadesuit+t}{j}}; (30,4)*{s_i(\lambda)}; \endxy\right) \\
&= \displaystyle\sum_{k+s+u=\alpha } \delta_{0,k} (-v_{ij})^{-s-u}\;
\xy  0;/r.15pc/: (0,0)*{\cbubr{\spadesuit+s}{i}}; (18,0)*{\ccbubr{\spadesuit+u}{i}}; (30,4)*{s_i(\lambda)};\endxy \\
&= \displaystyle\sum_{s+u=\alpha } (-v_{ij})^{-s-u}\;
\xy  0;/r.15pc/: (0,0)*{\cbubr{\spadesuit+s}{i}}; (18,0)*{\ccbubr{\spadesuit+u}{i}}; (30,4)*{s_i(\lambda)};\endxy
= (-v_{ij})^{-\alpha} \delta_{0,\alpha} \Id_{\onell{s_i(\l)}}
\end{align*}

\end{proof}


\bibliographystyle{plain}
\bibliography{bib-braid}

%

%
\end{document}